\documentclass[11pt]{aart}

\usepackage[letterpaper, hmargin=0.9in, top=0.9in, bottom=1.1in, footskip=0.6in]{geometry}

\usepackage{titlesec}
\titleformat{\section}[block]{\filcenter\normalfont\bfseries\large}{\thesection.}{.5em}{}\titlespacing*{\section}{0pt}{2\baselineskip}{1\baselineskip}
\titleformat{\subsection}[runin]{\normalfont\bfseries}{\thesubsection.}{.4em}{}[.]\titlespacing{\subsection}{0pt}{2ex plus .1ex minus .2ex}{.8em}
\titleformat{\subsubsection}[runin]{\normalfont\itshape}{\thesubsubsection.}{.3em}{}[.]\titlespacing{\subsubsection}{0pt}{1ex plus .1ex minus .2ex}{.5em}
\titleformat{\paragraph}[runin]{\normalfont\itshape}{\theparagraph.}{.3em}{}[.]\titlespacing{\paragraph}{0pt}{1ex plus .1ex minus .2ex}{.5em}

\usepackage[T1]{fontenc}
\usepackage[utf8]{inputenc}

\usepackage{mlmodern}

\usepackage{wrapfig}

\usepackage{tabularx,array}
\usepackage[bb=libus]{mathalpha}
\usepackage{geometry}
\usepackage{yfonts}
\usepackage[T1]{fontenc} 
\usepackage[utf8]{inputenc} 
\usepackage{mathtools}
\usepackage{amssymb}
\usepackage{amsthm}
\usepackage{bbm}
\usepackage{mathrsfs}
\usepackage[english]{babel} 
\usepackage{enumitem}
\usepackage{dsfont}
\usepackage{url}
\usepackage{graphicx}
\usepackage{braket}
\usepackage{nicematrix}
\usepackage{cancel}

\definecolor{vdarkred}{rgb}{0.6,0,0.2}
\definecolor{vdarkblue}{rgb}{0,0.2,0.6}
\usepackage[pdftex, colorlinks, linkcolor=vdarkblue,citecolor=vdarkred,pagebackref]{hyperref}

\renewcommand{\leq}{\leqslant}
\renewcommand{\geq}{\geqslant}

\newcommand{\N}{\mathbb{N}} % natuerliche Zahlen
\newcommand{\Z}{\mathbb{Z}} % ganze Zahlen
\newcommand{\Q}{\mathbb{Q}} % rationale Zahlen
\newcommand{\R}{\mathbb{R}} % reelle Zahlen
\newcommand{\C}{\mathbb{C}} % reelle Zahlen

\newcommand{\1}{\mathds{1}}

\theoremstyle{plain}
\newtheorem{thrm}{Theorem}[section]
\newtheorem{prop}[thrm]{Proposition}
\newtheorem{cor}[thrm]{Corollary}
\newtheorem{lemma}[thrm]{Lemma}

\newtheorem{ex}[thrm]{Example}

\theoremstyle{definition}
\newtheorem{mydef}[thrm]{Definition}
\newtheorem{assumption}[thrm]{Assumption}
\newtheorem{remark}[thrm]{Remark}

\usepackage{booktabs}

\numberwithin{equation}{section}
\numberwithin{figure}{section}

\begin{document}
	
	\title{Density of states of randomly sprinkled graphs}
%	\date{\vspace{-10ex}}
%	\date{}
	
	\author{Antti Knowles \and Steffen Polzer}
	\maketitle

\begin{abstract}
Independently at every vertex $x$ of a large core graph $\mathbb G$, we draw a random graph and connect $x$ to a subset of its vertices. This model of a randomly sprinkled graph captures qualitative features of commonly observed graphs; it can also be regarded as a model of quantum disorder, where disorder arises from local perturbations to the graph geometry. We show that it is naturally connected both to the Anderson model and to site percolation on $\mathbb G$ by deriving an Anderson-percolation representation for its spectrum. We characterize the support of the spectrum and derive quantitative bounds on the integrated density of states. We also analyse in detail the regime of sparse sprinkling, in particular investigating the behaviour of the density of states in the vicinity of atoms.
\end{abstract}

\setcounter{tocdepth}{1}
\tableofcontents
	
	\section{Introduction}
	
One of the most fundamental spectral properties of a graph is its \emph{empirical eigenvalue measure}, which describes its eigenvalue density. For a finite graph $\mathcal G$ with vertex set $\cal V$, it is defined as
\begin{equation} \label{def_DOS}
\mu \coloneqq \frac{1}{|\mathcal V|} \sum_{i = 1}^{|\mathcal V|} \delta_{\lambda_i}\,,
\end{equation}
where $\lambda_1, \dots, \lambda_{|\cal V|} \in \R$ is an enumeration of the eigenvalues of the adjacency matrix of $\mathcal G$. In a suitable infinite-volume limit, where $|\mathcal V| \to \infty$, the measure $\mu$ typically converges to a limiting measure -- the \emph{density of states} of the limiting graph. Depending on the complexity of the graph, the density of states can be very complicated. Complexity in graphs is often modelled using randomness, by drawing the edges of the graph from some prescribed probability distribution. A simple example of a random graph is the $d$-regular graph on $N$ vertices, where the graph is chosen uniformly at random from the set of all graphs on $N$ vertices such that each vertex has degree $d$. As $N \to \infty$, the empirical eigenvalue measure converges to the density of states of the Bethe lattice (the infinite $d$-regular tree),
 which is the Kesten--McKay law\footnote{See the imaginary part of \eqref{Equation: ST Kesten-McKay} below.}. Another simple example is the Erd\H{o}s--Rényi graph, where each edge of the complete graph on $N$ vertices is kept independently with probability $p \equiv p_N > 0$. If $Np \to \infty$ as $N \to \infty$, then the rescaled empirical eigenvalue measure converges to the semicircle law\footnote{Here we suppose that $p \leq 1/2$.}. On the other hand, if $d = Np$ is fixed as $N \to \infty$ then the limiting density of states is far more complicated with no known explicit form, and is only partially understood. In particular, if $d > 0$ it is known that the density of states has an infinite set of atoms at so-called totally real algebraic integers, corresponding to eigenvalues of finite trees\footnote{These eigenvalues are associated with \emph{tuning fork} eigenvectors; see below.}. For $d > 1$, the density of states has a nonzero continuous component, and for large enough $d$ is has a nonzero absolutely continuous component. We refer the reader to \cite[Theorem 3.3]{Bor26} for a recent summary of the extensive literature on this topic\footnote{Some key references are \cite{AB23, Bor16, BL10, BLS11, BSV17, CS21, Salez15, Salez19, Salez}.}.  See Figure \ref{fig:ER} for a simulation.

	\begin{figure}[ht!]
		\centering
		\includegraphics[width=0.5\linewidth]{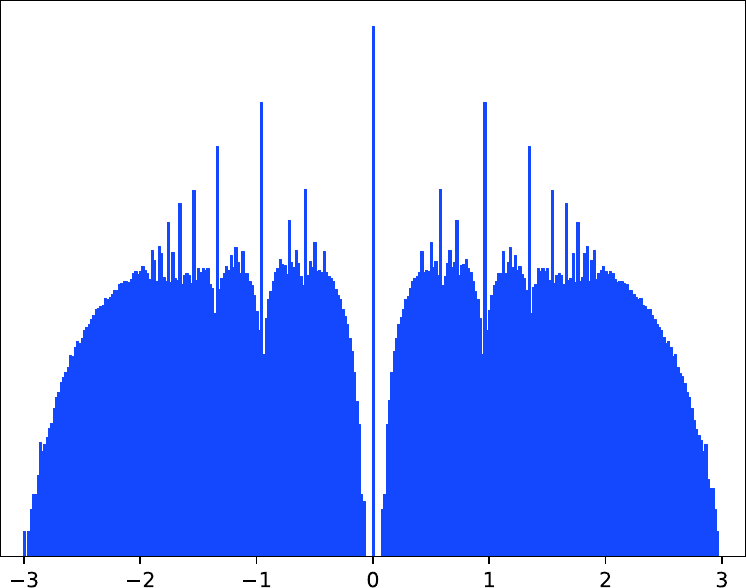}
		\caption{A Histogram of the eigenvalue density of the Erd\H{o}s--Rényi graph. The parameters are $N = 15\,000$ and $d = Np = 1.1$. The histogram was obtained by averaging over 50 realizations. The vertixal axis is logarithmic.}
		\label{fig:ER}
	\end{figure}

Many random graph models, as well numerous graphs observed in the real world, share the following qualitative structure with the Erd\H{o}s--Rényi graph with $d > 1$: a macroscopic \emph{core} to whose vertices are connected a set of smaller \emph{attached graphs}; see Figure \ref{fig:sprinkling}. One commonly used construction for the core of a graph is the Karp--Sipser core, obtained by iteratively removing a leaf and its unique neighbour. Typically, the core of the graph is expected to have a simple density of states\footnote{For instance, for the Erd\H{o}s-Rényi graph with large enough $d = Np$, it is known \cite{AB23} that the density of states of the Karp--Sipser core is purely absolutely continuous away from the spectral edge.}, and much of the complexity in the density of states of the original graph is expected to arise from the addition of the attached graphs.

	\begin{figure}[ht!]
		\centering
		\includegraphics[width=0.4\linewidth]{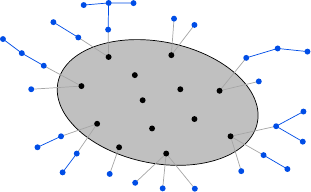}
		\caption{An illustration of the splitting of a graph into its core (in black) and the attached graphs (in blue).}
		\label{fig:sprinkling}
	\end{figure}

How this graph sprinkling procedure influences the density of states is a delicate question that has, to the best of our knowledge, not so far been investigated in the literature. The goal of this paper is to address it. To that end, we introduce a natural and general model of random graphs $\mathcal G$ that captures the qualitative structure described above and illustrated in Figure \ref{fig:sprinkling}, and we investigate its density of states. In this model, we choose a core -- a large (possibly random) graph $\mathbb G$ -- and, independently and identically distributed at every vertex $x$ of $\mathbb G$, we draw a random (possibly empty) graph $\mathrm G^x$ and connect $x$ to a subset of its vertices. This models a macroscopic core sprinkled with smaller attached graphs\footnote{Note that our model does not encompass the Erd\H{o}s--Rényi graph, because of the lack of independence of the core and the attached graphs, although we expect it to share many qualitative features with the Erd\H{o}s--Rényi graph.}.

Another important motivation for our model stems from the theory of disordered quantum systems. One of the most studied models of quantum disorder, the Anderson model, is the sum of the adjacency matrix of a lattice and a diagonal matrix whose entries are independent and identically distributed random variables. Its density of states plays an important role in many central questions of the field, entering in formulas for e.g.\ the free energy and the conductance; see e.g.\ \cite{AizenmanWarzel.2015}. The random diagonal matrix represents a random potential, which perturbs the adjacency matrix of the atomic lattice. A different perturbation of the lattice can be obtained by grafting vertices or small graphs to the vertices of the lattice, as in Figure \ref{fig:lattice}. The physical interpretation is that the disorder arises from a random sprinkling of supplementary atoms attached to the lattice (in contrast to the randomly scattered impurities within the lattice in the usual Anderson model).

	\begin{figure}[ht!]
		\centering
		\includegraphics[width=0.4\linewidth]{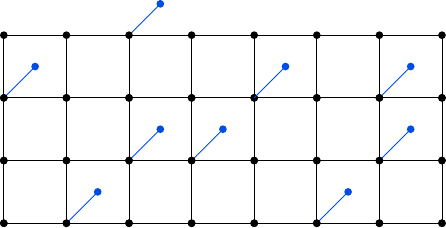}
		\caption{The lattice $\mathbb Z^2$ (in black) sprikled with attached leaves (in blue).}
		\label{fig:lattice}
	\end{figure}

Denote by $A$ the adjacency matrix of the sprinkled graph $\cal G$.
We shall analyse the empirical eigenvalue measure of $A$ in the limit of large graphs, i.e.\ in the infinite-volume limit $|\mathcal V| \to \infty$. To that end, it is helpful to rewrite the expectation of the empirical eigenvalue measure \eqref{def_DOS} for any $I \in \mathcal B(\R)$ as
\begin{equation} \label{DOS-rewriting}
\mathbb E[\mu(I)] = \mathbb E\bigg[\frac{1}{|\mathcal V|} \sum_{i = 1}^{|\mathcal V|} \1_I(\lambda_i)\bigg] = \mathbb E\bigg[\frac{1}{|\mathcal V|} \operatorname{tr} (\1_I(A)) \bigg] = \mathbb E\big[\langle \delta_o, \1_I(A) \delta_o \rangle\big],
\end{equation}
where the vertex $o$ is uniformly distributed over $\cal V$, $\1_I(A)$ denotes the spectral projection of $A$ onto the set $I$, and $\delta_{o} \coloneqq (\delta_{ox})_x$ is the unit vector supported on the vertex $o \in \cal V$. While the left-hand side of \eqref{DOS-rewriting} only makes sense for finite graphs, the right-hand side makes sense for an infinite random rooted graph $(\mathcal G, o)$ and is commonly called the {\em density of states} of $(\mathcal G, o)$.

The infinite-volume limit is naturally formulated in the sense of Benjamini--Schramm, i.e.\ with respect to the local topology on the set of rooted locally finite graphs; see Subsection \ref{Subsection: Benjamini Schramm convergence} below. We suppose that the core $\mathbb G_N$ with uniformly chosen root $\mathbb o_N$ converges to some infinite rooted graph $(\mathbb G, \mathbb o)$, and show that the corresponding sprinkled graph $\mathcal G_N$ with uniformly chosen root $o_N$ converges to an infinite rooted graph $(\mathcal G, o)$, which we also construct explicitly. As a consequence, the expected\footnote{Assuming a stronger mode of convergence for the core $\mathbb G_N$, one can also show convergence of the empirical eigenvalue measure to the density of states, in a suitable sense.} empirical eigenvalue measure of $(\mathcal G_N, o_N)$ converges to the density of states of $(\mathcal G, o)$, which we denote by $\nu$. The convergence is with respect to the Kolmogorov-Smirnov metric (uniform convergence of cumulative distribution functions), which in particular implies convergence of atoms.

It is easy to see that $\nu$ may have atoms coming from eigenvectors that are exclusively supported on the attached graphs. Such eigenvectors are sometimes called {\em tuning forks}, in analogy to the vibration modes of a tuning fork, which are supported on the prongs of the fork and vanish on the handle; see Figure \ref{fig:example_TF} below for some examples of tuning forks. Eigenvectors of a similar form have for example been applied in \cite{CCFST86} or \cite{AltDucatezKnowles.2021} in order to study the spectrum of site percolation on the integer lattice and the Erd\H{o}s--Rényi graph. The contribution of the tuning forks to the density of states, denoted by $\nu_{\operatorname{TF}}$, is relatively easy to analyse (see Proposition \ref{Proposition: Tuning fork contribution} below), and we shall therefore mainly focus on the density of states from which their contribution is removed. While it does not appear that the latter admits a simple explicit characterization, it turns out that a great deal about it can nevertheless be understood in the form of quantitative estimates, which constitute the core of our paper.

Let us now outline the main contributions of this paper; for precise statements, we refer to Subsections \ref{Subsection: notation spinkled graphs}  and \ref{Subsection: Benjamini Schramm convergence} for the definition of our model and Section \ref{Section: Results sprinkled graphs} for the results.

\begin{enumerate}
\item
We show that our model is naturally connected both to the Anderson model and to site percolation on $\mathbb G$. This yields a representation of the spectrum of the adjacency operator, excluding the tuning fork eigenvalues, which we call the {\em Anderson-percolation representation}. It is formulated in terms of the Anderson percolation Hamiltonian, which is an Anderson Hamiltonian on the site percolation subgraph of the core graph $\mathbb G$, with an explicit i.i.d.\ discrete random potential. Both the percolation probability and the random potential depend on an energy parameter $\lambda$; see \eqref{def_AH} below. In finite volume, this correspondence is an exact algebraic correspondence (see Proposition \ref{Proposition: Anderson representation in finite volume} below), which has a broad range of consequences in infinite volume for the measure $\nu - \nu_{\operatorname{TF}}$.
\end{enumerate}

As a simple example, we show (under suitable assumptions) that $\nu$ may have atoms which do not arise from tuning forks. For $\mathbb G = \mathbb Z^d$, we additionally show that atoms of $\nu$ can only appear at eigenvalues of the attached graphs which yields, for simple sprinkling distributions, a nontrivial lower bound on the continuous part of $\nu$. Further applications are listed below.

\begin{enumerate} \setcounter{enumi}{1}
\item
Assuming that $\mathbb G$ is deterministic and vertex transitive, and that its spectrum is a symmetric interval (satisfied e.g.\ when $\mathbb G$ is a suitable bipartite lattice), we characterize the support of $\nu$ explicitly. The proof uses the Anderson-percolation representation, and is inspired by the classical Kunz-Souillard argument \cite{KS80} for the Anderson model. See Figure \ref{fig:support} below for an illustration.
\item
We derive explicit upper and lower bounds on the integrated density of states of $\nu - \nu_{\operatorname{TF}}$; see Figure \ref{fig:comparedfs} below for an illustration.
\end{enumerate}

In finite volume, the Anderson-percolation representation reduces the effect of the sprinkling to a family of iterated rank-one perturbations of the adjacency matrix. This results in a qualitative picture of the effect of the sprinkling on the spectrum, which describes two competing effects: the sprinkling generates eigenvalues close to (so-called \emph{admissible}) eigenvalues of the attached graphs, while repelling the existing eigenvalues away from them, see Figure \ref{Figure: Smoothening is necessary} for an illustration of this effect. Compare also to Figure \ref{fig:ER} where we can see a similar effect in the spectrum of the Erd\H{o}s--Rényi graph.
Depending on which effect prevails, the mass of $\nu - \nu_{\operatorname{TF}}$ in a neighbourhood of an eigenvalue $\lambda$ of an attached graph can be either greater or smaller than in the original core, in which case we call $\lambda$ a \emph{source} or a \emph{sink}, respectively.

\begin{enumerate} \setcounter{enumi}{3}
\item
We investigate rigorously which of these effects prevails in the regime of sparse sprinkling, where $p \coloneqq \mathrm P(\mathrm G^x \neq \emptyset)$ is small. We find that, depending on the law of the attached graphs and the density of states of the core $\mathbb G$, eigenvalues $\lambda$ of the attached graphs can be sources or sinks.

In order to do so, we derive an expansion of $\nu - \nu_{\operatorname{TF}}$ in the limit of sparse sprinkling by computing it to first order in $p$ and estimating the second-order contribution. The first-order term $\phi$, which describes the leading contribution for sparse sprinkling, is in general not a signed measure but a genuine distribution. In addition to addressing the dichotomy of sources and sinks, we find that $\phi$ gives a rather accurate approximation to $\nu - \nu_{\operatorname{TF}}$; see Figure \ref{Figure: Example small p expansion} below for an illustration.

Moreover, we show that if the degree of the core $\mathbb G$ is large enough, the effect of the sprinkling near the eigenvalues of the attached graphs amounts to Cauchy peaks in $\nu - \nu_{\operatorname{TF}}$, which are centred around the eigenvalues of the attached graphs; see Figure \ref{fig:smallplargedglobal} below for an illustration.

Finally, in Appendix \ref{Subsection: Expansion around p=1}, we derive a similar expansion around $p = 1$. Unlike the trivial endpoint $p = 0$, the other endpoint $p = 1$ often corresponds to a nontrivial but explicitly solvable deterministic model. See Figure \ref{fig:expansionaround1} below for an illustration.
\end{enumerate}

We conclude this overview by remarking that, in this paper, we focus on the adjacency matrix $A$ of the graph $\mathcal G$, because it is the simplest matrix associated with $\cal G$ and also because it emerges naturally from the tight-binding approximation of a quantum particle hopping on a random graph, as in Figure \ref{fig:lattice}. Other matrices associated with $\mathcal G$, such as the Laplacian, are also of interest, and we believe our techniques should also be applicable to them.

\paragraph{Acknowledgements} We are grateful to Gérard Ben Arous for inspiring discussions on the problem and for helpful comments. We acknowledge funding from the European Research Council (ERC) and the Swiss State Secretariat for Education, Research and Innovation (SERI) through the consolidator grant ProbQuant, as well as funding from the Swiss National Science Foundation through the NCCR SwissMAP grant.

	\subsection{The model and basic notations}
	\label{Subsection: notation spinkled graphs} 
	Let $\mathbb G$ be a (large) finite (random or deterministic) graph\footnote{By a graph, we always mean a simple graph.} on the vertex set $\mathbb V$. For simplicity, we assume that $\mathbb V$ is deterministic. 
	Independently and identically distributed at any vertex $x\in \mathbb V$, we draw a random element $(\mathrm G^x, \varphi^x)$ consisting of a finite random graph $\mathrm G^x$ with vertex set $\mathrm V^x \subset \N$ and a vector $\varphi^x \in \{0, 1\}^{\mathrm V^x}$. We call $\mathrm G^x$ the {\em attached graph} in $x$ and specifically allow it be the empty graph (which we denote by $\emptyset$), i.e.\ the graph with empty vertex and edge set. We attach\footnote{While not necessary for our discussion, the reader might assume that every connected component of $\mathrm G^x$ contains at least at one vertex $t$ with $\varphi^x_t = 1$.} $\mathrm G^x$ to $\mathbb G$ by connecting $x$ to every vertex $t\in \mathrm V^x$ such that $\varphi^x_t = 1$; see Figure \ref{fig:notation} for an illustration. 
	\begin{figure}[ht!]
		\centering
		\includegraphics[width=0.3\linewidth]{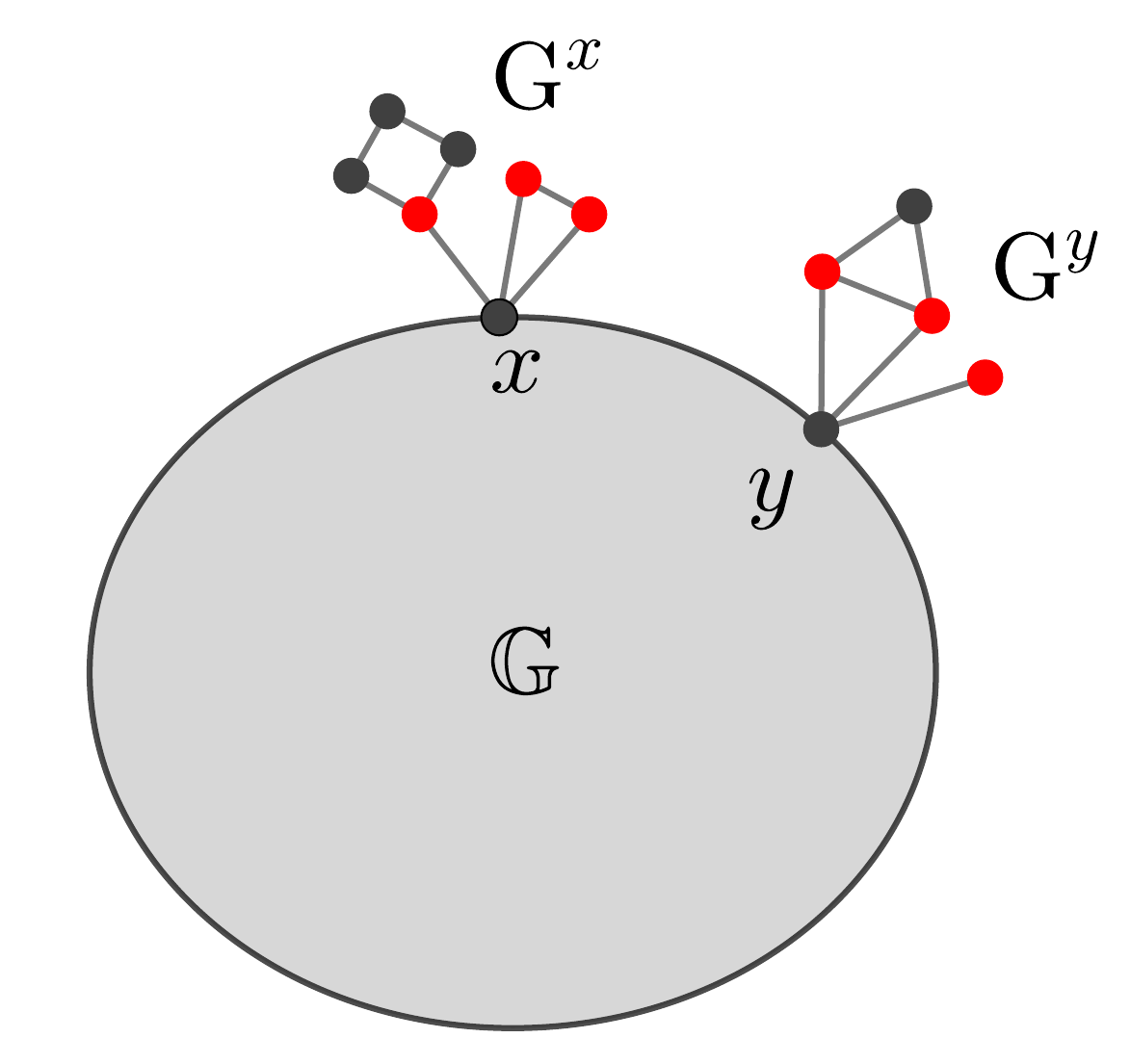}
		\caption{Illustration of our notation. All vertices $(zt)$ for which $\varphi^z_t = 1$ are drawn in red.}
		\label{fig:notation}
	\end{figure}
	We denote for $x\in \mathbb V$ and $t\in \mathrm V^x$ by $(xt)$ the vertex $t$ of $\mathrm G^x$ that has been added to $\mathbb G$. Hence, the vertex set $\mathcal V$ of the resulting sprinkled graph $\mathcal G$ is
	\begin{equation*}
		\mathcal V \coloneqq \mathbb V \cup \{(xt): x \in \mathbb V^+, \, t\in \mathrm V^x\} \quad \text{ where } \quad \mathbb V^+ \coloneqq \{x\in \mathbb V:\, \mathrm G^x \neq \emptyset\}.
	\end{equation*}
	If we denote by $B$ the adjacency matrix of $\mathbb G$ and, for $x\in \mathbb V$, by $\mathrm  T^x$ the adjacency matrix of $\mathrm G^x$, then
	the adjacency matrix $A$ of $\mathcal G$ has the block form
	\begin{equation}
		\label{Equation: blockform adjacency matrix}
		A = 
		\begin{pmatrix}
			B & E \\
			E^* & T		
		\end{pmatrix}
		\text{ with }
		T_{(xs)(yt)} \coloneqq \delta_{xy} \mathrm T^{x}_{st}, \quad E_{x (yt)} \coloneqq \delta_{xy} \varphi^x_{t}
	\end{equation}
	i.e.\ $T$ itself has block-diagonal form.
	We denote by
	\begin{equation*}
		m_x \coloneqq |\mathrm V^x|, \quad m_x^* \coloneqq \sum_{t\in \mathrm V^x} \varphi^x_t
	\end{equation*}
	the total number of vertices contained in $\mathrm G^x$ and the number of vertices in $\mathrm G^x$ connected to $x$ respectively.
	We call the distribution $\mathrm P$ of $(\mathrm G^x, \varphi^x)$ (for some arbitrary $x\in \mathbb V$) the sprinkling distribution, and denote by $\mathrm E$ the corresponding expectation.
	We drop the sub- and superscript $x$ if we draw a generic element from the sprinkling distribution, independent of all other sources of randomness; e.g.\ we denote by $(\mathrm G, \varphi)$ a generic element drawn from  $\mathrm P$ and by $m$ the total number of vertices contained in $\mathrm G$.
	
Finally, we denote by $\nu$ the expected empirical eigenvalue measure of $\mathcal G$,
	\begin{equation*}
		\nu \coloneqq \mathbb E\bigg[ \frac{1}{|\mathcal V|} \sum_{\lambda \in \sigma(A)} \gamma(A,\lambda) \delta_{\lambda}\bigg] \quad \text{ where } \gamma(A, \lambda) \coloneqq \dim \operatorname{ker}(A- \lambda),
	\end{equation*}
where $\mathbb E$ denotes expectation over the law $\mathbb P$ of the whole random graph $\mathcal G$.
	
For the convenience of the reader, we summarize the main notations of this paper in Appendix \ref{sec:symbols}.
	
	\subsection{The limit of large graphs}
	\label{Subsection: Benjamini Schramm convergence}
	We shall study $\nu$ in the limit of large graphs $\mathbb G$. To that end, we work in the framework of Benjamini-Schramm convergence, which we now briefly review. We refer the reader to \cite[Section 2]{Bor16} for a more in-depth introduction into the topic. For a rooted graph\footnote{A rooted graph $(G, r)$ is a connected graph $G$ together with a distinguished vertex $r\in V$, the root. In case the graph $G$ is not connected and $r\in V$, we sometimes abuse notation and denote by $(G, r)$ the rooted graph $(G(r), r)$ where $G(r)$ is the connected component containing $r$.}  $(G, r)$ and $R>0$ we denote by $(G, r)_R$ the ball in $G$ around $r$ with radius $R$ (with respect to the graph distance). For two rooted graphs $(G_1, r_1)$, $(G_2, r_2)$ we write $(G_1, r_1) \cong (G_2, r_2)$ if there exists a graph isomorphism between $G_1$ and $G_2$ that identifies the roots i.e.\ which maps $r_1$ to $r_2$. 
	An unlabeled rooted graph is the equivalence class of a rooted graph $(G, r)$ modulo graph isomorphisms that identify the roots. The set $\mathscr G$ of locally finite, unlabelled, rooted graphs equipped with the topology generated by the metric
	\begin{equation*}
		d\big((G_1, r_1), (G_2, r_2)\big) \coloneqq \sum_{R=1}^\infty 2^{-R} \1_{\{ (G_1, r_1)_R \ncong (G_2, r_2)_R\}  }
	\end{equation*} 
	is a Polish space whose topology is called the local topology. Notice that a sequence $((G_N, r_N))_N$ converges to $(G, r)$ with respect to $d$ if and only if $(G_N, r_N)_R \cong (G, r)_R$ eventually for every $R>0$.
	
	In order to study the sprinkling in the limit of large graphs, we assume that $\mathbb G_N$ is for every $N\in \mathbb N$ a finite (random or deterministic) graph, that $\mathbb V_N$ is deterministic, and that $|\mathbb V_N| \to \infty$ as $N \to \infty$. After uniformly choosing a vertex $\mathbb o_N$ from $\mathbb V_N$ as a root (independently from $\mathbb G_N$), we obtain a sequence of random rooted graphs $((\mathbb G_N, \mathbb o_N))_N$, which we assume to converge in distribution (with respect to the local topology) to some random rooted graph $(\mathbb G, \mathbb o)$. Two simple examples are the integer lattice $\mathbb Z^d$, which is the limit of $ \mathbb G_N = [-N, N]^d \cap \mathbb Z^d$, and the Bethe lattice $\mathbb B_d$, i.e.\ the infinite $d$-regular tree, which is the limit\footnote{Notice that the the ball of radius $N$ around a fixed vertex $\mathbb o\in \mathbb B_d$ does not converge to $\mathbb B_d$ as the latter is non-amenable: the number of vertices in the boundary is non-negligible compared to the total number of vertices in the ball.} of the random $d$-regular graph on $2N$ vertices.
	
	We now sprinkle our finite graphs $\mathbb G_N$ as in Subsection \ref{Subsection: notation spinkled graphs} (with the distribution of attached graphs being chosen constant in $N\in \mathbb N$) to obtain the sprinkled graph $\mathcal G_N$ with vertex set $\mathcal V_N$. Let $o_N$ be a vertex chosen uniformly at random from $\mathcal V_N$ (conditionally on $\mathcal G_N$). Under the assumption $\mathrm E[m] < \infty$, we shall show that the sequence $((\mathcal G_N, o_N))_N$ converges in distribution (with respect to the local topology) to a random rooted graph $(\mathcal G, o)$ (which we shall explicitly construct in terms of a sprinkling of $(\mathbb G, \mathbb o)$). As shown in \eqref{DOS-rewriting}, the expected empirical eigenvalue measure $\nu_N$ of $\mathcal G_N$ coincides with the expected spectral measure of $\mathcal G_N$ in the root $o_N$, i.e.
\begin{equation*}
\nu_N(I) = \mathbb E\big[\langle \delta_{o_N}, \1_I(A_N) \delta_{o_N} \rangle\big].
\end{equation*}	
The weak convergence of $(\mathcal G_N, o_N)$ to $(\mathcal G, o)$ in particular implies that $(\nu_N)_N$ converges to the density of states $\nu$ of $(\mathcal G, o)$ in a suitable sense; see Subsection \ref{Subsection: Results existence of the infinite volume limit} below.
	
	\section{Results}
	\label{Section: Results sprinkled graphs}
	\subsection{Existence of the infinite-volume limit}
	\label{Subsection: Results existence of the infinite volume limit}
	We start by explicitly constructing a random rooted graph $(\mathcal G, o)$ to which the sprinkled graphs $(\mathcal G_N, o_N)$ converge to as $N\to \infty$.
	
	\begin{thrm}
		\label{Theorem: Limit Theorem}
		Assume that $\mathrm E[m] < \infty$.
		The sequence of sprinkled graphs $((\mathcal G_N, o_N))_N$ converges in distribution with respect to the local topology to the random rooted graph $(\mathcal G, o)$ obtained in the following way: Conditionally on $(\mathbb G, \mathbb o)$:
		\begin{enumerate}
			\item Draw independently at any $x\in \mathbb V \setminus\{\mathbb o\}$ the random element $(\mathrm G^x, \varphi^x)$ from $\mathrm P$.
			\item Independently from this, draw $(\mathrm G^{\mathbb o}, \varphi^\mathbb o)$ from the biased distribution
			\begin{equation}
				\label{Equation: biased distribution}
				\widehat{\mathrm P}(\, \cdot \, )  \coloneqq \sum_{k=0}^\infty \frac{  (k+1) \mathrm P(m = k)}{1 + \mathrm E[m]} \mathrm P(\, \cdot \, | \,m = k).
			\end{equation}
			\item Define $\mathcal G$ to be the graph obtained by sprinkling $\mathbb G$ with $(\mathrm G^x, \varphi^x)_{x \in \mathbb V}$ as described in Subsection \ref{Subsection: notation spinkled graphs} above. Conditionally on $\mathcal G$, choose $o$ uniformly at random from the vertices $\{\mathbb o\} \cup \{(\mathbb o t): \, t \in \mathrm V^\mathbb o  \}$.
		\end{enumerate}
	\end{thrm}
	We point out that the bias of the sprinkling distribution in the root bears resemblance to the bias of the offspring distribution in the root for unimodular Galton--Watson trees, see e.g.\ \cite[Equation (1)]{AB23}. 
	Unsurprisingly, one can show that the graph $(\mathcal G, o)$ has conditionally on $\{o =  \mathbb o\}$ the same distribution as the rooted graph obtained by an iid sprinkling of $(\mathbb G, \mathbb o)$ with the unbiased distribution $\mathrm P$ and by keeping $\mathbb o$ as the root. Theorem \ref{Theorem: Limit Theorem} in particular implies that the limit $(\mathcal G, o)$ is, as the limit of a sequence of unimodular graphs, unimodular \cite[Lemma 2.1]{Bor16}. For our purposes it will be sufficient to notice that unimodularity implies that the adjacency operator $A$, acting on finitely supported $\psi \in \C^\mathbb V$, is almost surely essentially self-adjoint \cite[Proposition 2.2]{Bor16}, and that ``everything shows up at the root'', meaning that for every measurable subset $M \subseteq \mathscr G$
	\begin{equation}
		\label{Equation: Everything shows up at the root}
		\mathbb P ((\mathcal G, o) \in M) = 1 \implies \mathbb P (\forall v \in \mathcal V: \, (\mathcal G, v) \in M) = 1
	\end{equation}
	see \cite[Lemma 3]{Salez}. Let $\nu$ denote the density of states of $(\mathcal G, o)$, i.e.\ the expected spectral measure of $\mathcal G$ in the root $o$, defined by $\nu(I) = \mathbb E\big[\langle \delta_{o}, \1_I(A) \delta_{o} \rangle\big]$ for all $I \in \mathcal B(\R)$ (where we abuse notation and denote the self-adjoint closure of $A$ by $A$ as well) and let $F(\lambda) \coloneqq \nu((-\infty, \lambda])$ denote the integrated density of states (i.e.\ the distribution function of $\nu$). Let $F_N$ denote the distribution function of the expected empirical eigenvalue measure $\nu_{N}$ of $\mathcal G_N$.
	As a corollary of the weak local convergence given by Theorem \ref{Theorem: Limit Theorem} we obtain with\footnote{While \cite[Theorem 2.5]{Bor16} is formulated for a deterministic sequence of finite graphs, its proof via \cite[Corollary 2.4]{Bor16} and \cite[Proposition 2.7]{Bor16} directly generalizes to our setup.} \cite[Theorem 2.5]{Bor16} convergence of $\nu_N$ to $\nu$ in the Kolmogorov--Smirnov metric, i.e.\ uniform convergence of $F_N$ to $F$ as $N\to \infty$.
	\begin{cor}
		\label{Corollary: Convergence of DOS-measures}
		We have $F_{N} \to F$ uniformly as $N\to \infty$. In other words, $\nu_{N} \to \nu$ weakly and $\nu_{N}(\{\lambda\}) \to \nu(\{\lambda\})$ for all $\lambda \in \R$.
	\end{cor}
	In particular, any atom of $\nu$ necessarily has to be a totally real algebraic integer, i.e.\ the root of a real-rooted monic polynomial with integer coefficients.
	
	To ensure existence of the limit $(\mathcal G, o)$, we shall from now on always assume that $\mathrm E[m]<\infty$.
	We shall state all of our additional assumptions separately before the respective result.
	
	\subsection{Atoms and the Anderson-percolation representation}
	\label{Subsection: results atoms and the Anderson-percolation representation}
	We start our study of the density of states by looking at its atoms. We define
	\begin{equation*}
		\mathrm S^x \coloneqq \sigma(\mathrm T^{x}), \quad \mathrm D \coloneqq \big\{ \lambda \in \R:\, \mathrm P(\lambda \in \mathrm S^x  ) > 0 \big\}
	\end{equation*}
	to be the spectrum of the attached graph in $x$ and its support, where by convention $\mathrm S^x \coloneqq \emptyset$ if $\mathrm G^x = \emptyset$. One can easily see that $\nu$ might have atoms in $\mathrm D$ which come from eigenvectors that are supported solely on the attached graphs. We call eigenvectors of this form {\em tuning forks}.
	
	\begin{mydef}
		We call a non-zero vector $\psi \in l^2(\mathcal V)$ a {\em tuning fork} to the eigenvalue $\lambda\in \R$ if $A\psi = \lambda \psi$ and $\psi_x = 0$ for all $x\in \mathbb V$. Consequently, we call $\lambda\in \R$ a {\em tuning fork eigenvalue} if there exists a tuning fork to the eigenvalue $\lambda$. We denote by $\mathcal T$ the set of tuning fork eigenvalues. 
	\end{mydef}
	See Figure \ref{fig:example_TF} for examples.
	\begin{figure}
		\centering
		\includegraphics[width=0.7\linewidth]{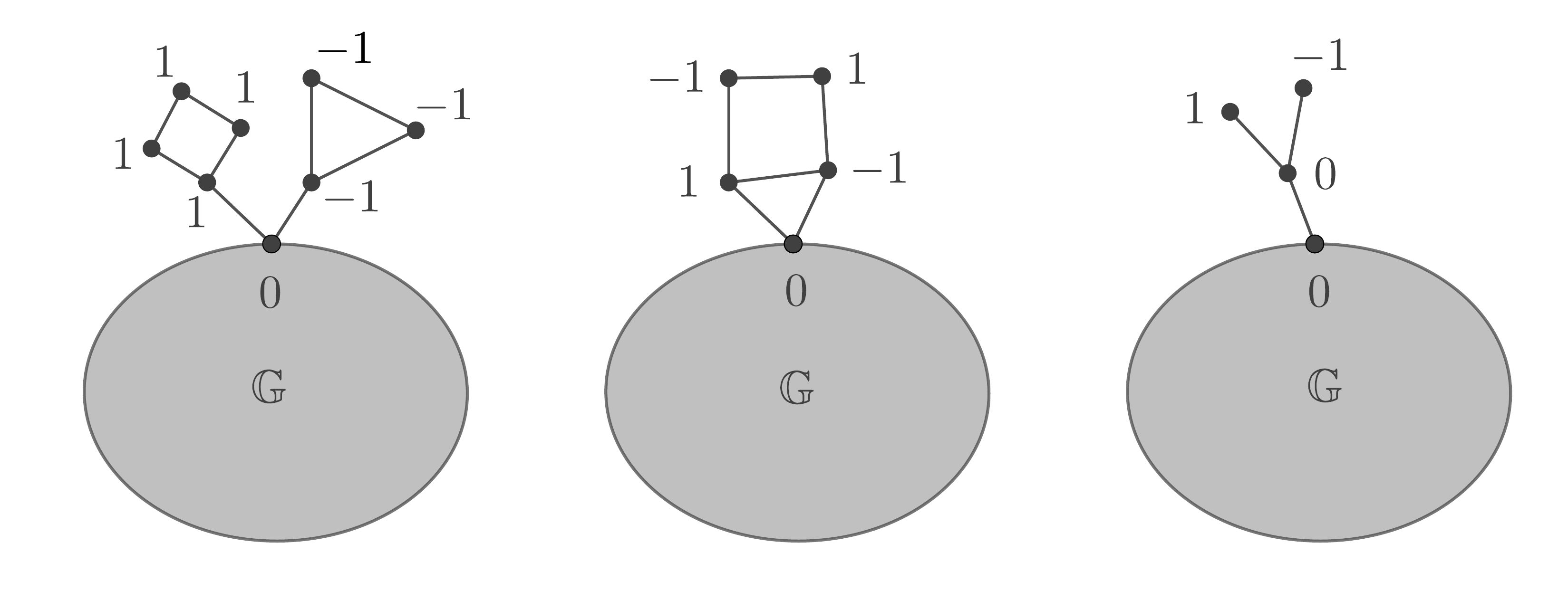}
		\caption{Example of tuning forks to the eigenvalues $2$, $-2$ and $0$ (from the left). 
		In the first case, $2$ is an admissible eigenvalue in the sense of Definition \ref{Definition: Admissible eigenvalue}, in the second and third case $-2$ and $0$ are not admissible eigenvalues.}
		\label{fig:example_TF}
	\end{figure}
	Notice that tuning forks do not depend on the interaction between $\mathbb G$ and the attached graphs:
	if $\psi^x  \in \mathbb C^{\mathrm V^x}$ is a non-zero vector such that
	\begin{equation*}
		\mathrm T^{x} \psi = \lambda \psi, \quad \langle \varphi^x, \psi^x \rangle = 0
	\end{equation*}
	then we obtain a tuning fork by extending $\psi^x$ by zero to all of $\mathcal V$, i.e.\ by setting $\psi_y \coloneqq 0$ and $\psi_{(yt)} \coloneqq \delta_{xy} \psi^x_t$ for all $y\in \mathbb V$ and $t\in \mathrm V_y$. On the other hand, every tuning fork is the direct sum of tuning forks of that form. In particular, we have
	\begin{equation}
		\label{Equation: tuning fork contained in sepctrum of AG}
		\mathcal T  \, \subseteq \,  \bigcup_{x\in \mathbb V} \mathrm S^x.
	\end{equation}
	Eigenvectors of a similar form have already been used in \cite{CCFST86} in order to show discontinuity of the density of states of the infinite cluster of site percolation on the integer lattice. 
	Notice that in the case where $(\mathbb G, \mathbb o)$ is almost surely infinite, both sets in \eqref{Equation: tuning fork contained in sepctrum of AG} are almost surely constant by the Borel--Cantelli lemma (recall that both are subsets of the totally real algebraic integers, which are countable). We introduce the tuning fork-contribution $\nu_{ \operatorname{TF}}$ to the density of states. 
	Let
	\begin{align*}
		\nu_{\operatorname{TF}} &\coloneqq \lim_{N\to \infty}\mathbb E\Big[\frac{1}{|\mathcal V_N|} \sum_{\lambda \in \mathrm D} \gamma_{\operatorname{TF}}(A_N, \lambda) \delta_\lambda \Big] \\
		\text{ where } \quad \gamma_{\operatorname{TF}}(A_N, \lambda) &\coloneqq \dim \big\{\psi\in \mathbb C^{\mathcal V_N}:\, A_N \psi = \lambda \psi\, \text{ and } \psi_x = 0 \text{ for all } x\in \mathbb V_N\big\}
	\end{align*}
	denotes the contribution of tuning forks to the multiplicity of the eigenvalue $\lambda$; we shall see in the proofs of Proposition \ref{Proposition: Tuning fork contribution} and Lemma \ref{Lemma: Total variation convergence of tuning fork contribution} that this limit indeed exists in total variation. Let us express the tuning fork-contribution in terms of the sprinkling distribution. We define
	\begin{equation*}
		\mathrm S^x_* \coloneqq \{\lambda \in \mathrm S^x:\, \exists \psi\in \C^{\mathrm V^x}:\, \mathrm T^{x}\psi = \lambda \psi \text{ and } \langle \varphi^x, \psi \rangle \neq 0 \}
	\end{equation*}
	to be the set of all eigenvalues of the attached graph in $x$ to which there exists an eigenvector that is non-orthogonal to $\varphi^x$. 
	As usual, we drop the sub- and superscript $x$ if we draw a generic element from the sprinkling distribution, independently of all other sources of randomness. As before, for a matrix $M$ and $\lambda \in \R$, we denote $\gamma(M, \lambda) \coloneqq \dim\operatorname{ker}(M-\lambda)$.
	\begin{prop}
		\label{Proposition: Tuning fork contribution}
		We have
		\begin{equation*}
			\nu_{\operatorname{TF}} = \frac{1}{1 + \mathrm E[m]} \mathrm E\Big[\sum_{\lambda \in \mathrm S}\big(\gamma(\mathrm T, \lambda) -\1_{\{\lambda \in  \mathrm S_*\}} \big) \delta_\lambda \Big].
		\end{equation*}
		In particular,
		\begin{equation}
			\label{Equation: total mass of tuning fork contribution}
			\nu_{\operatorname{TF}}(\R) = \frac{\mathrm E\big[ m  - | \mathrm S_*|\big]}{1 + \mathrm E[m]}. 
		\end{equation}
	\end{prop}
	
	We now represent the non-tuning fork-contribution $(\nu -\nu_{ \operatorname{TF}})(\{\lambda\})$ to an atom in terms of a family of Anderson models on a graph obtained by site percolation on $\mathbb G$.
	For $z\in \mathbb C\setminus \mathrm S^x$, we define
	\begin{equation*}
		r_x(z) \coloneqq \big\langle \varphi^x, (\mathrm  T^x-z)^{-1} \varphi^x \big\rangle
	\end{equation*}
	where $r_x \coloneqq 0$ in case that $\mathrm G^x = \emptyset$.
	Notice that we can continuously extend $r_x$ to a function $r_x: \hat{\mathbb C} \to \hat{\mathbb C}$ on the Riemann sphere $\hat{\mathbb C} = \mathbb C \cup \{\infty\}$. Then $r_x(\lambda) = \infty$ if and only if $\lambda \in \mathrm S^x_*$.
	We define the vertex set
	\begin{equation*}
		\mathbb V_{\lambda} \coloneqq \{x \in \mathbb V:\, \lambda \notin \mathrm S^x_*\}
	\end{equation*}
	and denote by $\mathbb G_{\lambda}$ the induced subgraph of $\mathbb G$ and by $B_\lambda$ its adjacency operator. Provided that $\mathrm P(r(\lambda) \neq \infty)>0$, one can show that $(\mathbb G_\lambda, \mathbb o)$ conditionally on $\{\mathbb o \in \mathbb V_\lambda\}$ is unimodular as well, as it is the Benjamini--Schramm limit of the graphs obtained from site percolation on $\mathbb G_N$, compare to the proof of Proposition \ref{Proposition: Weak convergence of DOS measures of Anderson Hamiltonian}. Hence, 
	$B_\lambda$ is almost surely essentially self-adjoint as well. On $l^2(\mathbb V_\lambda)$ we define the operator
	\begin{equation} \label{def_AH}
		H_{\lambda} \coloneqq B_\lambda + V_\lambda \quad \text{ where } \quad V_\lambda \coloneqq -\sum_{x\in \mathbb V_{\lambda}} r_x(\lambda) \langle \delta_x, \, \cdot \, \rangle \delta_x.
	\end{equation}
	In particular, for $\lambda \notin \mathrm D$ we have $\mathbb V_{\lambda} = \mathbb V$ and $H_\lambda$ is, conditionally on $o\in \mathbb V$ (which removes the bias of the distribution of $r_\mathbb o$), the Anderson Hamiltonian with potential $(-r_x(\lambda))_{x\in \mathbb V}$ on $l^2(\mathbb V)$. In order to guarantee self-adjointness of $H_\lambda$ by the Kato-Rellich Theorem,  we assume that at least one of the following two assumptions is satisfied:
	\begin{align}
		\text{We have }\operatorname{sup}_{x \in \mathbb V}\operatorname{deg}(x) < \infty\text{ almost surely.} \label{Assumption: bounded degrees in G} \tag {BoundedDeg}\\
		\text{The exists some }M>0\text{ s.t. }m\leq M \text{ almost surely.} \label{Assumption:: m a.s. bounded} \tag{BoundedAG}
	\end{align}
	For $\lambda\in \R$, we define $\rho_{\lambda}$ to be the density of states of $H_\lambda$ conditionally on $\{o \in \mathbb V_\lambda\}$ i.e.\
	\begin{equation*}
		\rho_{\lambda}(I) \coloneqq \mathbb E\big[ \big\langle \delta_o, \1_I(H_\lambda) \delta_o \big\rangle \big|o \in \mathbb V_\lambda\big], \quad I \in \mathcal B(\R).
	\end{equation*}
	As the next theorem shows, we can express the non-tuning fork contribution to the atoms of the density of states $\nu$ of our sprinkled graph in terms of the family $(\rho_{\lambda})_{\lambda \in \R}$ of density of states.
	\begin{thrm}
		\label{Theorem: Anderson representation of atoms}
		Assume that \eqref{Assumption: bounded degrees in G} or \eqref{Assumption:: m a.s. bounded} holds. Then we have for all $\lambda \in \R$
		\begin{equation*}
			\nu(\{\lambda\}) = \nu_{\operatorname{TF}}(\{\lambda\}) + \frac{\mathrm P(\lambda \notin \mathrm S_*)}{1+\mathrm E[m]} \rho_{\lambda} (\{\lambda\}).
		\end{equation*}
		If $I\subseteq \R$ is an open interval and $\lambda \in \overline{I}$ then
		\begin{equation*}
			\nu(I) \geq  \nu_{\operatorname{TF}}(I) + \frac{\mathrm P(\lambda \notin \mathrm S_*)}{1+\mathrm E[m]} \rho_{\lambda} (I). 
		\end{equation*}
	\end{thrm}
	Let us summarize some of the consequences of Theorem \ref{Theorem: Anderson representation of atoms}.
	\begin{cor}
		\label{Corollary: Atoms through subgraphs}
		Assume that $\mathrm P(\lambda \in  \mathrm S_*) > 0$ and that $p = \mathrm P(\mathrm G \neq \emptyset)<1$. Assume that $(\mathbb G, \mathbb o)$ has with positive probability a finite subgraph with eigenvalue $\lambda$. Then
		\begin{equation*}
			\nu(\{\lambda\}) > \nu_{ \operatorname{TF}}(\{\lambda\}).
		\end{equation*}
	\end{cor}
	\begin{ex}
		Assume that $\mathbb G  = \mathbb Z$ and that $\mathrm G$ is for any $n\in \N$ with probability $2^{-n-1}$ the line segment $\{1, \hdots, n\}$ with $\varphi = \delta_1$. Then $\nu_{\operatorname{TF}} = 0$ but the set of atoms of $\nu$ is dense in $[-2, 2] = \operatorname{supp}(\nu_0)$. 
	\end{ex}
	For the proof of Corollary \ref{Corollary: Atoms through subgraphs}, it is sufficient to notice that the following event happens with positive probability: the graph $\mathbb G$ has a finite subgraph which has an eigenfunction $\psi$ to the eigenvalue $\lambda$, there is no attached graph in any vertex $x\in \operatorname{supp}(\psi)$ while in every vertex $x$ with $\operatorname{dist}(x, \operatorname{supp}(\psi)) = 1$ we have $r_x(\lambda) = \infty$. Combining this with \eqref{Equation: Everything shows up at the root} (which guarantees that $\psi_\mathbb o \neq 0$ with positive probability) and Theorem \ref{Theorem: Anderson representation of atoms} yields the claim. See Figure \ref{fig:atomsdosfromsubgraph} for an illustration of the corresponding eigenvectors as given by the proof of Theorem \ref{Theorem: Anderson representation of atoms}.
	
	\begin{figure}
		\centering
		\includegraphics[width=0.6\linewidth]{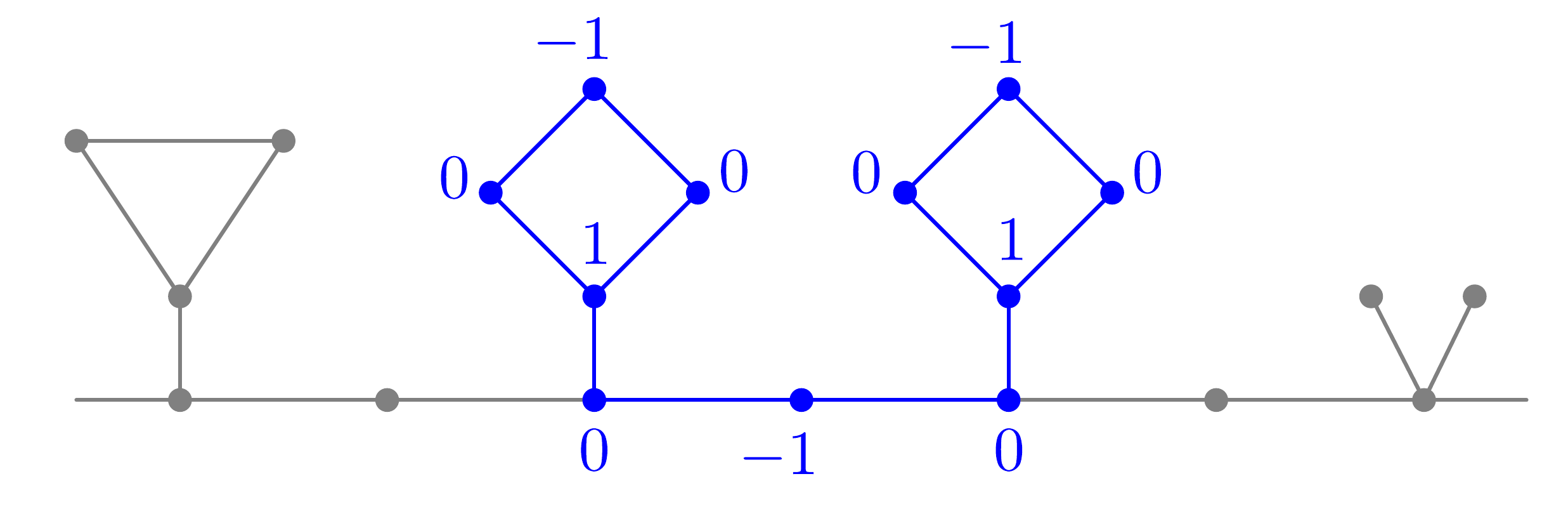}
		\caption{Illustration of Corollary \ref{Corollary: Atoms through subgraphs} for $\mathbb G = \mathbb Z$. Assuming that $\mathrm P(\mathrm G = \text{square},\, m^* = 1) > 0$ and $\mathrm P(\mathrm G = \emptyset) > 0$ there is an contribution to the atom in zero (which is the eigenvalue of a single isolated vertex) coming from compactly supported eigenvectors as depicted. Vertices in grey lie outside of the support of the eigenvector.}
		\label{fig:atomsdosfromsubgraph}
	\end{figure}
	
	A natural question might be how the operation of sprinkling relates to other graph operations such as the cartesian product or the free product of two graphs (whose spectral measures are obtained from taking the convolution and the free additive convolution of the spectral measures of the components, see \cite[Section 1.3]{Bor16}). As Corollary \ref{Corollary: Atoms through subgraphs} shows, the sprinkling of a graph might decrease the regularity of its density of states, even if we subtract the contribution of tuning forks. Hence, a potential connection at least does not seem to be trivial as both the convolution as well as the free additive convolution of two measures is continuous as soon as one of the two components is continuous. 
	
	The density of states of the Anderson-Hamiltonian on $\mathbb Z^d$ is known to be atomless \cite{DS84}. Hence, if $\mathbb G = \mathbb Z^d$ and $\lambda \notin \mathrm D$ (such that $\mathbb V_\lambda = \mathbb Z^d$) then $\rho_\lambda(\{\lambda\}) = 0$. We hence obtain as a corollary of Theorem \ref{Theorem: Anderson representation of atoms} that $\nu$ cannot have any atoms outside of $\mathrm D$.  
	
	\begin{cor}
		\label{Corollary: No atoms outside of D for integer lattice}
		Assume that $\mathbb G = \mathbb Z^d$ for some $d\geq 1$. Then $\nu(\{\lambda\}) = 0$ for all $\lambda \notin \mathrm D$.
	\end{cor}	
	
	Let us now consider the simplest case in which $(\mathrm G, \varphi)$ agrees with probability $p$ with the deterministic pair $(\mathrm G_*, \varphi_*)$ and is with probability $1-p$ empty, i.e.\ in which the sprinkling distribution is given by
	\begin{equation}
		\label{Assumption: Detetministic attached graphs}
		\mathrm P = p \delta_{(\mathrm G_*, \varphi_*)} + (1-p)\delta_{(\emptyset, \emptyset)}. \tag{BernoulliAG}
	\end{equation}
	We denote by $M$ the total number of vertices contained in $\mathrm G_*$ such that $\mathrm E[m] = pM$.
	
	\begin{mydef}
		\label{Definition: Admissible eigenvalue}
		Assuming $\mathrm P$ is of the form \eqref{Assumption: Detetministic attached graphs}, we call $\lambda \in \mathrm D$ admissible if $\lambda \in \mathrm S^x_*$ conditionally on $\mathrm G = \mathrm G_*$.
	\end{mydef}
	See Figure \ref{fig:example_TF} for an example.
	Assuming \eqref{Assumption: Detetministic attached graphs} and admissible eigenvalues, Theorem \ref{Theorem: Anderson representation of atoms} can be restated in terms of site percolation on $\mathbb G$. We denote by $\mathbb G_{\operatorname{perc}}$ the subgraph induced by 
	\begin{equation*}
		\mathbb V^+ \coloneqq \{x\in \mathbb V:\, m_x \neq 0\}
	\end{equation*}
	(which defines a site percolation on $\mathbb G$ with parameter $p$) and by $\nu_{\operatorname{perc}}$ the density of states of $(\mathbb G_{\operatorname{perc}}, \mathbb o)$ conditionally on $\{\mathbb o \in \mathbb V_{\operatorname{perc}}\}$.
	\begin{cor}
		\label{Proposition: Tuning fork contribution for deterministic AG}
		Assuming \eqref{Assumption: Detetministic attached graphs} and that $\lambda \in \mathrm D$ is admissible we have
		\begin{equation*}
			\nu(\{\lambda\}) = \nu_{ \operatorname{TF}}(\{\lambda\}) + \frac{1-p}{1 + pM} \nu_{\operatorname{perc}}(\{\lambda\}).
		\end{equation*}
		Assuming that all $\lambda \in \mathrm D$ are admissible, we have
		\begin{align}
			\nu(\mathrm D) &= \frac{p(M - |\mathrm D|)}{1 + pM} + \frac{1-p}{1 + pM} \nu_{\operatorname{perc}}(\mathrm D), \nonumber \\
			\nu(\mathbb R \setminus \mathrm D) &= \frac{p(|\mathrm D| + 1)}{1 + pM} + \frac{1-p}{1 + pM} \nu_{ \operatorname{perc}}(\mathbb R \setminus \mathrm D) \label{Equation: Mass continuous part}. 
		\end{align}
	\end{cor}
	For the proof of Corollary \ref{Proposition: Tuning fork contribution for deterministic AG}, it suffices to combine Theorem \ref{Theorem: Anderson representation of atoms} with the total mass of the tuning fork contribution given in \eqref{Equation: total mass of tuning fork contribution}. 
	
	Assume that $\mathbb G$ is deterministic and vertex transitive, such that the density of states $\nu_0$ of $\mathbb G$ coincides with the spectral measure in some arbitrary vertex of $\mathbb G$. If $\Sigma_0 \coloneqq \operatorname{supp}(\nu_0)$ is a symmetric\footnote{We point out that the spectrum of an bipartite graph is always symmetric.} interval then $\operatorname{supp}(\nu_0) = [-\|B\|, \|B\|]$. If $P_\lambda:l^2(\mathbb V)\to l^2(\mathbb V_\lambda)$ denotes the restriction to $\mathbb V_\lambda$ then $\|B_\lambda\| = \|P_\lambda B P_\lambda^*\| \leq \|B \|$. We hence have $\nu_{\operatorname{perc}}(\R \setminus \Sigma_0) = 0$.
	With Corollary \ref{Proposition: Tuning fork contribution for deterministic AG} we hence obtain that the mass to an atom in $\mathrm D \setminus \Sigma_0$ comes solely from tuning forks.
	\begin{cor}
		\label{Corollary: Only TF contribution outside Sigma0 for berAG}
		Assume that $\mathbb G$ is vertex transitive and deterministic and assume that $\Sigma_0 \coloneqq \operatorname{supp}(\nu_0)$ is a symmetric interval. Assume that \eqref{Assumption: Detetministic attached graphs} holds. If $\lambda \in \mathrm D\setminus \Sigma_0$ is admissible then
		\begin{equation*}
			\nu(\{\lambda\}) = \nu_{ \operatorname{TF}}(\{\lambda\}).
		\end{equation*}
		In particular, if $\mathbb G = \mathbb Z^d$, if \eqref{Assumption: Detetministic attached graphs} holds and if all $\lambda \in \mathrm D$ are admissible then the pure point part $\nu_{\operatorname{pp}}$ of $\nu$ satisfies $\nu_{\operatorname{pp}}(\R \setminus \Sigma_0) =  \nu_{ \operatorname{TF}}(\R \setminus \Sigma_0)$.
	\end{cor}
	For $\mathbb G= \mathbb Z^d$ we have seen in Corollary \ref{Corollary: No atoms outside of D for integer lattice} above that $\nu$ does not have any atoms outside of $\mathrm D$. Hence, Equation \eqref{Equation: Mass continuous part} gives us a lower bound on the total mass of the continuous part of $\nu$.
	\begin{cor}
		Assume that $\mathbb G = \mathbb Z^d$ for some $d\geq 1$. Assuming \eqref{Assumption: Detetministic attached graphs} and that all $\lambda \in \mathrm D$ are admissible then the continuous part $\nu_{\operatorname{cts}}$ of $\nu$ satisfies
		\begin{equation*}
			\nu_{\operatorname{cts}}(\R) = \nu(\R \setminus \mathrm D) \geq \frac{p(|\mathrm D| + 1)}{1 + pM}.
		\end{equation*}
	\end{cor}
	We point out that for the case of super-critical edge percolation and $d=2$, a lower bound on the continuous part of the density of states is known \cite{BSV17}. A corresponding lower bound for site percolation would in particular give a non-trivial lower bound on $\nu_{\operatorname{perc}}( \R \setminus \mathrm D)$ and hence, as a consequence of \eqref{Equation: Mass continuous part}, an improved lower bound on $\nu_{\operatorname{cts}}(\R)$.
	
	As an example, we explicitly calculate $\nu(\{0\})$ for $\mathbb G = \mathbb Z$ with \eqref{Assumption: Detetministic attached graphs} and 0 being admissible. We shall revisit this result in Proposition \ref{Proposition: Smoothening is necessary 2} below.

	\begin{prop}
		\label{Proposition: Mass of atom in zero for line}
		Assume that $\mathbb G = \mathbb Z$ and that \eqref{Assumption: Detetministic attached graphs} holds. If $0\in \mathrm D$ is admissible, we have
		\begin{equation*}
			\nu(\{0\}) - \nu_{\operatorname{TF}}(\{0\})  = \frac{p(1-p)}{(1+pM)(2-p)}.
		\end{equation*}
	\end{prop}
	
	\subsection{Support in the lattice case}
	\label{Subsection: results support}
	We now explicitly calculate the support of the density of states under assumptions that in particular include the case where $\mathbb G$ is the integer lattice $\mathbb Z^d$, the Bethe lattice $\mathbb B_d$ or the Honeycomb lattice (i.e.\ Graphene) $\mathbb H$.
	As an application of ergodic theory shows, in all of those cases the spectrum of the adjacency operator $A$ of $\mathcal G$ is almost surely constant and agrees with the support of the density of states, see Proposition \ref{Proposition: Spectrum agrees almost surely with support of DOS measure} later in the text. In \cite{KS80}, Kunz and Souillard calculated the spectrum of the Anderson Hamiltonian on $\mathbb Z^d$. In order to apply a similar argument to our model, we give a representation of the spectrum of the adjacency operator $A$ in terms of the Anderson-percolation Hamiltonians $H_\lambda$.
	\begin{thrm}
		\label{Theorem: Anderson representation of the spectrum}
		If both \eqref{Assumption: bounded degrees in G} and \eqref{Assumption:: m a.s. bounded} hold we have almost surely
		\begin{equation*}
			\sigma(A) = \big\{\lambda \in \mathbb R:\, \lambda \in \sigma(H_\lambda)  \big\} \cup \mathcal T.
		\end{equation*}
	\end{thrm}
	We shall apply Theorem \ref{Theorem: Anderson representation of the spectrum} and the Kunz-Souillard Theorem for the Anderson model in order to explicitly determine the support of $\nu$ under suitable assumptions of $\mathbb G$. From now on, we denote by $\nu_0$ the density of states of $(\mathbb G, \mathbb o)$. We define
	\begin{equation*}
		\Sigma \coloneqq \operatorname{supp}(\nu), \quad \Sigma_0 \coloneqq \operatorname{supp}(\nu_0).
	\end{equation*}

	\begin{figure}
		\centering
		\begin{minipage}[c]{0.5\textwidth}
			\centering
			\includegraphics[width=\textwidth]{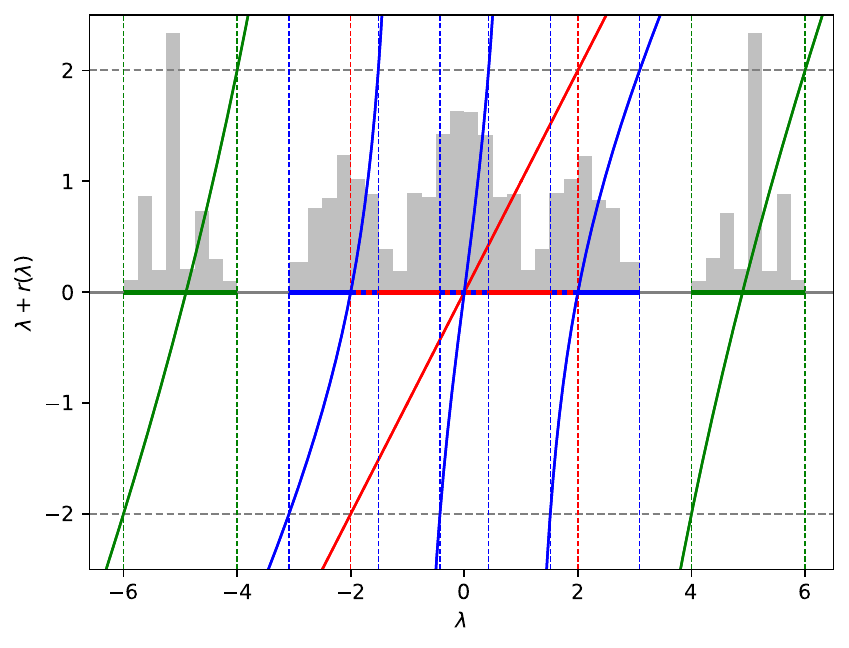}
		\end{minipage}
		\hfill
		\begin{minipage}[c]{0.47\textwidth}
			\centering
			\includegraphics[width=\textwidth]{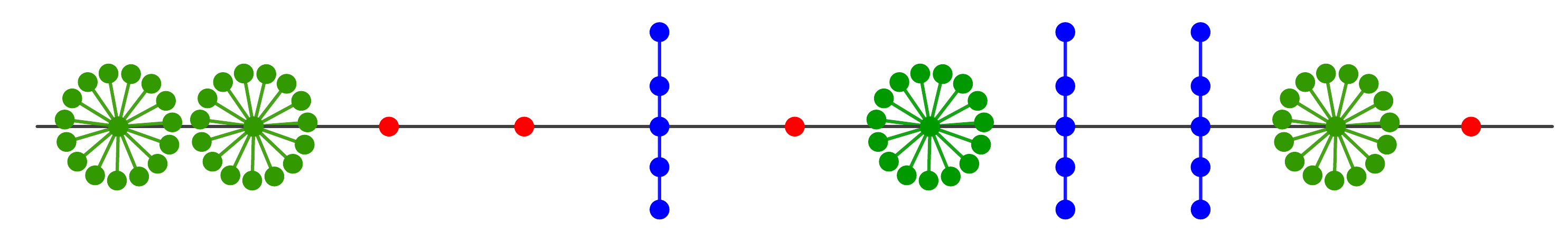}
		\end{minipage}
		\caption{Support of the density of states for $\mathbb G = \mathbb Z$ (for which $\Sigma_0 = [-2, 2]$) and where $\mathrm G$ is with probability $0.4$ empty (graph of $\lambda + r(\lambda)$ in red) with probability $0.3$ a collection of two line segments of length two (graph in blue) and with  probability $0.3$ a collection of 15 leaves (graph in green). By Theorem \ref{Theorem: Spectrum in the lattice case}, the spectrum $\Sigma$ is the union of the red, blue and green intervals. In grey, a histogram of the mass distribution of $\nu - \nu_{\operatorname{TF}}$ is depicted (simulation, not normalized).}
		\label{fig:support}
	\end{figure}

	\begin{thrm}
		\label{Theorem: Spectrum in the lattice case}
		Assume that $\mathbb G$ is deterministic, vertex transitive and assume that $\Sigma_0$ is a symmetric interval\footnote{Which is, for example, the case if $\Sigma_0$ is an interval and $\mathbb G$ is bipartite.}. Then $\sigma(A) = \Sigma$ almost surely and
		\begin{equation*}
			\Sigma =  \overline{\big\{\lambda \in \mathbb R:\, \mathbb P(\lambda + r(\lambda) \in \Sigma_0) > 0  \big\} \cup \mathcal T }.
		\end{equation*}	
	\end{thrm}
	See Figure \ref{fig:support} for an illustration of Theorem \ref{Theorem: Spectrum in the lattice case}. A family of examples for which the assumptions of Theorem \ref{Theorem: Spectrum in the lattice case} are satisfied can be constructed by using the following observation: if $\mathbb G$ is deterministic, vertex transitive and bipartite and if $\Sigma_0$ is an interval then the $d$-fold cartesian product $\mathbb G^d$ and the $d$-fold free product $\mathbb G^{\boxtimes d}$ (whose spectral measures are the $d$-fold convolution $\nu_0^{*d}$ and the $d$-fold free additive convolution $\nu_0^{\boxtimes d}$ respectively, see \cite{Bor16}) satisfy the assumptions of Theorem \ref{Theorem: Spectrum in the lattice case}. We refer the reader to \cite{BelinschiBercoviciHo.2024} for a proof that the free product of two probability measures with connected support has connected support.

	\subsection{Concentration of mass}
	\label{subsection: Concentration of mass}
	We show, roughly speaking, that the absolute number of eigenvalues tends to increase predominately in intervals in which many eigenvalues of the attached graphs lie. 
	In order to state our following results in a cleaner fashion, we introduce the measure
	\begin{equation}
		\label{Equation: Definition hat nu}
		\hat{\nu} \coloneqq (1+ \mathrm E[m])\big(\nu - \nu_{\operatorname{TF}}\big).
	\end{equation}
	Notice that Proposition \ref{Proposition: Tuning fork contribution} implies that the total mass of $\hat \nu$ is given by
	\begin{equation}
		\label{Equation: Total mass of hat nu}
		\hat\nu(\R) = 1 + \mathrm E[| \mathrm S_*|].
	\end{equation}
	One easily shows that the convergence of $\nu_N$ to $\nu$ with respect to the Kolmogorov-Smirnov metric (i.e.\ the uniform convergence of the cumulative distributions functions given by Corollary \ref{Corollary: Convergence of DOS-measures}) implies that 
	\begin{equation*}
		\lim_{N\to \infty} \mathbb E\Big[\frac{1}{|\mathbb V_N|}\sum_{\lambda \in \sigma(A_N)} \big(\gamma(A_N, \lambda) - \gamma_{\operatorname{TF}, N}(A_N, \lambda) \big) \delta_{\lambda} \Big] = \hat \nu
	\end{equation*}
	in the Kolmogorov--Smirnov metric, see Lemma \ref{Lemma: convergence of nu hat} later in the text.
	Hence, the measure $\hat{\nu}$ can be used in order to quantify the expected absolute (rather than relative) change of the number of eigenvalues in a given set, not considering the contribution of tuning forks. Let
	\begin{equation*}
		 \mathrm N \coloneqq \{\lambda \in \R:\, r(\lambda) = 0\}, \quad \zeta \coloneqq \mathrm E\Big[ \1_{\{\mathrm G \neq \emptyset\}} \sum_{\lambda \in \mathrm N} \delta_\lambda \Big].
	\end{equation*} 
	On the event $\{\mathrm G \neq \emptyset\}$, the function $r$ is strictly increasing on every connected component of $\mathbb R \setminus \mathrm S_*$, with $\lim_{t  \to \lambda} |r(t)| = \infty$ for all $\lambda \in \mathrm S_*$. Between two elements of $\mathrm S_*$ there hence lies exactly one element of $\mathrm N$ and we have $|\mathrm N| = |\mathrm S_*| - 1$. We obtain with \eqref{Equation: Total mass of hat nu}
	\begin{equation*}
		(\hat\nu - \nu_0)(\R) = \mathbb E[|\mathrm S_*|] = \zeta(\R) + \mathrm P(\mathrm G \neq \emptyset)
	\end{equation*}
	where $\nu_0$ denotes, as earlier, the density of states of $(\mathbb G, \mathbb o)$.
	We show with the Anderson-percolation representation that the measure $\zeta$ can even be used in order to estimate the influence of the sprinkling on the expected number of eigenvalues in some interval $I$.

	\begin{figure}
		\centering
		\begin{minipage}[c]{0.6\textwidth}
			\centering
			\includegraphics[width=\textwidth]{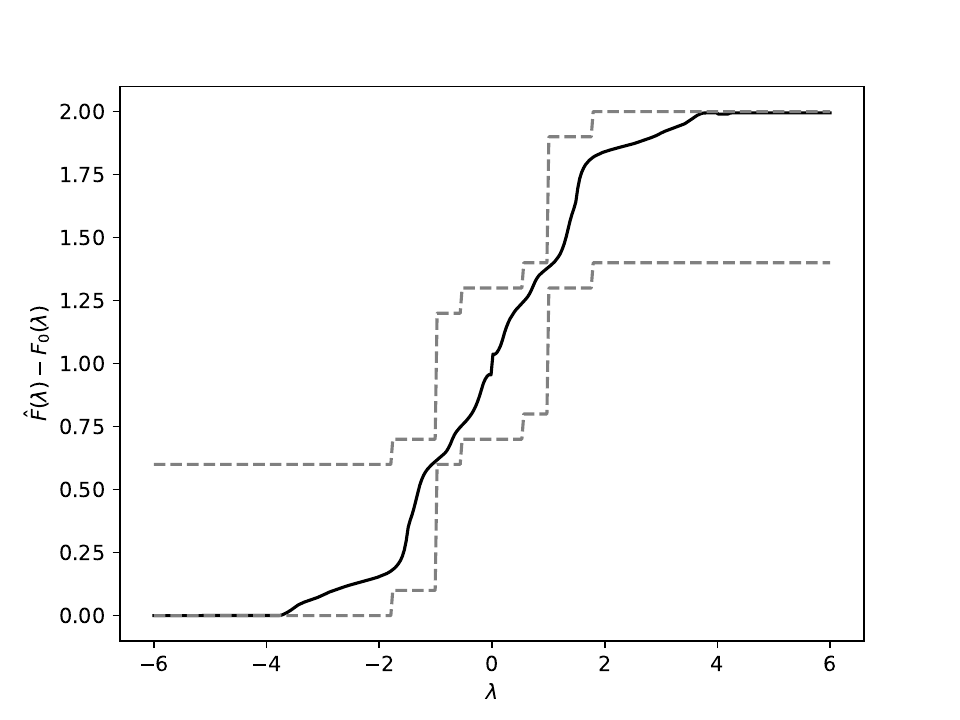}
		\end{minipage}
		\hfill
		\begin{minipage}[c]{0.39\textwidth}
			\centering
			\includegraphics[width=\textwidth]{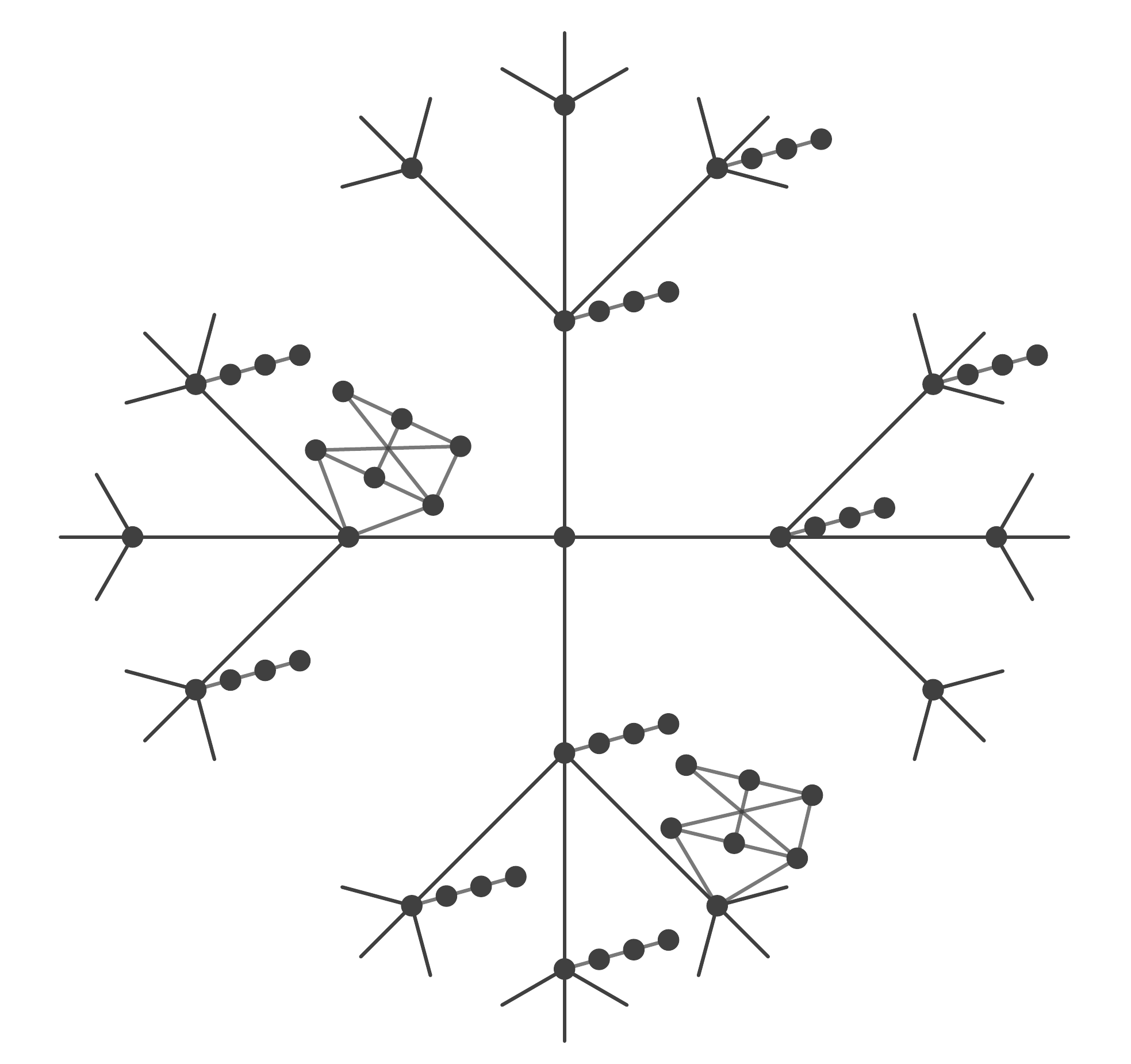}
		\end{minipage}
		\caption{Distribution function of $\hat\nu-\nu_0$ for the sprinkled Bethe lattice $\mathbb B_4$ (simulation) and the upper and lower bounds given by Theorem  \ref{Theorem: concentration of mass inside of Sigma_0}. Here $\mathrm G$ is with probability 0.4 empty, with probability 0.5 a line segment of length 3 and with probability 0.1 the graph containing six vertices depicted above. The total mass of $\hat\nu-\nu_0$ is given by $\mathrm E[|\mathrm S_*|] = 0.5\cdot 3 + 0.1\cdot 5 = 2$.}
		\label{fig:comparedfs}
	\end{figure}

	\begin{thrm}
		\label{Theorem: concentration of mass inside of Sigma_0}
		Let $\hat F$, $F_0$ and $F_\zeta$ denote the cumulative distribution functions of $\hat \nu$, $\nu_0$ and $\zeta$ respectively. Then we have for all $\lambda \in \R$
		\begin{equation*}
			F_\zeta(\lambda) \leq (\hat F - F_0)(\lambda) \leq F_\zeta(\lambda) + p \quad \text{ where } \quad p \coloneqq \mathrm P(\mathrm G \neq \emptyset).
		\end{equation*}
		In particular, if $I = (a, b]$ with $a<b$ then
		\begin{equation*}
			\zeta(I)-p \, \leq \, (\hat \nu - \nu_0)(I) \, \leq  \zeta(I)+p.
		\end{equation*}
	\end{thrm}
	In other words, the total number of eigenvalues tends to increase predominately in intervals $I$ in which many zeros of $r$ lie; see Figure \ref{fig:comparedfs} for an illustration. Theorem \ref{Theorem: concentration of mass inside of Sigma_0} in particular improves the trivial bound
	\begin{equation*}
		\| \hat F - F_0\|_\infty \leq 4 \mathrm E[m],
	\end{equation*}
	which can directly be obtained by finite rank perturbation theory; see the proof of Lemma \ref{Lemma: General Lipschitz estimates} later in the text.
	
	If the expected number of admissible eigenvalues in an interval $I$ is at least $2p$, i.e.\ if $\mathbb E\big[|\mathrm S_*\cap I||\mathrm G \neq \emptyset\big] \geq 2$, then $\zeta(I) \geq p$ and hence Theorem \ref{Theorem: concentration of mass inside of Sigma_0} implies that
	\begin{equation}
		\label{Equation: Increase of eigenvalues in I by sprinkling}
		\hat \nu(I) \geq \nu_0(I)\,.
	\end{equation}
If the expected number of admissible eigenvalues in $I$ is $p$, \eqref{Equation: Increase of eigenvalues in I by sprinkling} does not need to hold; see Corollary \ref{Corollary: mass can decrease around tree eigenvalue for small p} and Corollary \ref{Corollary: Total repulsion for p=1 and detAG} for examples where the expected number of eigenvalues in $I$ is $p$ (with $p$ being close to 0 and 1 respectively) and where \eqref{Equation: Increase of eigenvalues in I by sprinkling} fails.
	
	We now consider the case of sparse sprinkling, i.e.\ small $p$, and study where the mass of $\hat \nu$ is concentrated. 
	For simplicity, we only consider the case where $\Sigma_0$ is an interval and $\mathrm D\subseteq \Sigma_0$. Then $r$ does not have any zeros outside of $\Sigma_0$, i.e.\ $\zeta(\R \setminus \Sigma_0) = 0$, and Theorem \ref{Theorem: concentration of mass inside of Sigma_0} yields
	\begin{equation}
		\label{Equation: Estimate on mass outside of Sigma0}
		\hat{\nu}(\R \setminus \Sigma_0) \leq 2 p.
	\end{equation}
	Hence, for small $p$ most of the mass of $\hat \nu$ lies inside of $\Sigma_0$. In the upcoming Subsection \ref{Subsection: Results sparse srpinkling}, we present an expansion of $\hat \nu$ for sparse sprinkling that, under assumptions including the case where $\mathbb G$ is the integer or Bethe lattice and \eqref{Assumption:: m a.s. bounded} holds, implies that most mass outside of $\Sigma_0$ (meaning all mass up to order $o(p)$) is concentrated in a set of small Lebesgue measure. Let $g$ denote the Stieltjes transform of the spectral measure of the adjacency operator $B$ of $\mathbb G$ in the root $\mathbb o$ i.e.\ let
	\begin{equation}
		\label{Equation: Definition of g}
		g(z) \coloneqq \big\langle \delta_\mathbb o, (B-z)^{-1} \delta_\mathbb o \big\rangle, \quad z\in \C^+.
	\end{equation}
	If $\mathbb G$ is deterministic and vertex transitive then $g$ is deterministic as well and coincides with the Stieltjes transform $s_0$ of $\nu_0$. We define
	\begin{equation}
		\label{Equation: Definition of Z}
		\mathcal Z \coloneqq \big\{\lambda \in \mathbb R\setminus \Sigma_0:\, \mathbb P(g(\lambda)r(\lambda) = 1) > 0 \big\}
	\end{equation}	
	which is a finite set if $g$ is deterministic and \eqref{Assumption:: m a.s. bounded} holds. For $\varepsilon>0$ we define
	\begin{equation*}
		\mathcal Z_\varepsilon \coloneqq \{\lambda  \in \R: \, \operatorname{dist}(\lambda, \mathcal Z) \leq \varepsilon  \}, \quad  \mathcal I_\varepsilon \coloneqq \{\lambda  \in \R: \, \operatorname{dist}(\lambda, \Sigma_0 \cup \mathcal Z) \leq \varepsilon  \}.
	\end{equation*}

	\begin{figure}
		\centering
		\begin{minipage}[c]{0.6\textwidth}
			\centering
			\includegraphics[width=\textwidth]{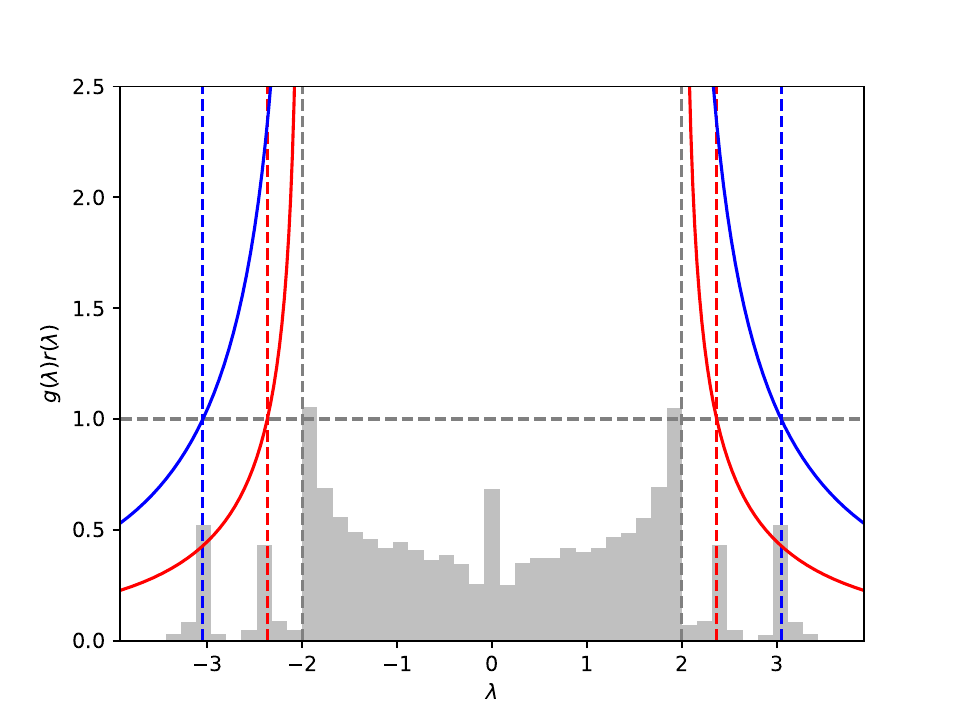}
		\end{minipage}
		\hfill
		\begin{minipage}[c]{0.39\textwidth}
			\centering
			\includegraphics[width=\textwidth]{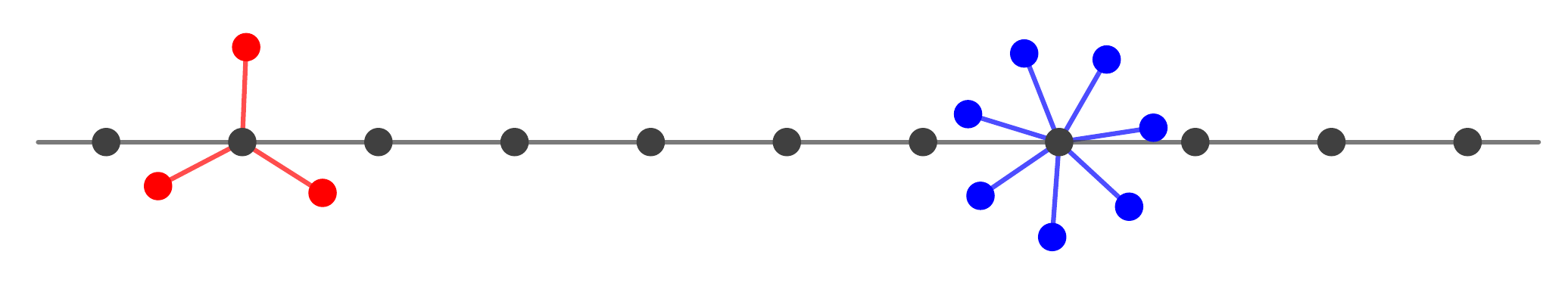}
		\end{minipage}
		\caption{Histogram of the mass distribution of $\hat \nu$ (simulation, non-normalized) for $\mathbb G = \mathbb Z$ and $\mathrm P(\mathrm G= \text{three leaves}) =  0.05$, $\mathrm P(\mathrm G= \text{seven leaves}) = 0.05$ and $\mathrm P(\mathrm G = \emptyset) = 0.9$. Graph of $\lambda \mapsto g(\lambda)r(\lambda)$ conditionally on $\{\mathrm G = n \text{ leaves}\}$ in red and blue for $n=3$ and $n=7$ respectively. Mass outside of $\Sigma_0 = [-2, 2]$ is located close to $\Sigma_0 \cup \mathcal Z$ where $\mathcal Z$ is indicated by the dashed vertical red and blue lines. Notice that $\hat \nu(\mathcal Z) = 0$ by Corollary \ref{Corollary: No atoms outside of D for integer lattice}, i.e.\ the mass close to $\mathcal Z$ does not come from atoms.}
		\label{fig:concentrationmassoutsidesigma0}
	\end{figure}

	\begin{prop}
		\label{Proposition: Concentration of mass outside of Sigma0}
		For every $M \in \mathbb N$ there exists a constant $ C = C(M)$ such that the following holds.
		If $\mathbb G$ is deterministic and vertex transitive with $\mathrm D \subseteq \Sigma_0 = [a, b]$ for some $a<b$ and if \eqref{Assumption:: m a.s. bounded} holds with $m\leq M$ a.s. then for every $\delta \in (0, 1/8)$
		\begin{equation}
			\label{Equation: Most mass close to Sigma_0 or Z}
			\hat \nu(\mathbb R \setminus \mathcal I_{p^{\delta}}) \leq C p^{2-8\delta}.
		\end{equation}
		Moreover, if $p^{\delta} \leq \operatorname{dist}(\Sigma_0, \mathcal Z)$ then
		\begin{equation}
			\label{Equation: Bounds on mass close to Z}
			c(\mathbb G, \mathrm P)p - C p^{2-8\delta} \leq \hat \nu(\mathcal Z_{p^{\delta}}) \leq 2p
		\end{equation}
		where
		\begin{equation}
			\label{Equation: Definiton of c(G)}
			c(\mathbb G, \mathrm P) \coloneqq \mathrm P\big(\lim_{\lambda \uparrow a} g(\lambda)r(\lambda)>1 |\mathrm G \neq \emptyset \big) +  \mathrm P\big(\lim_{\lambda \downarrow b} g(\lambda)r(\lambda)>1 |\mathrm G \neq \emptyset \big)
		\end{equation}
		is positive if and only if $\mathcal Z \neq \emptyset$.
	\end{prop}
	
	In other words, the mass of $\hat \nu$ is up to order $o(p)$ located in $\mathcal I_{p^{\delta}}$ i.e.\ close to $\Sigma_0 \cup \mathcal Z$ and, in case that $c(\mathbb G, \mathrm P) = 0$, even close to $\Sigma_0$. Notice that Theorem \ref{Theorem: Spectrum in the lattice case} implies that the support of $\hat \nu$ is typically not a subset of $\mathcal I_{p^\delta}$ and that Corollary \ref{Corollary: No atoms outside of D for integer lattice} implies for $\mathbb G = \mathbb Z^d$ that the mass inside of $\mathcal Z_{p^\delta}$ does not come from atoms of $\hat \nu$.
	
	Since we assumed $\mathrm D \subset [a, b]$, $r(a)$ and $r(b)$ are non-negative/non-positive and are increasing/decreasing under the addition of attached graphs to $\mathrm G$ (under the constraint $\mathrm D \subseteq \Sigma_0$).
	Hence, it might come to a transition from $c(\mathbb G, \mathrm P) = 0$ to $c(\mathbb G, \mathrm P) > 0$ as we increase $|r|$ on $\R \setminus [a, b]$  (e.g.\ by increasing the number connected components of $\mathrm G$), see Figure \ref{Figure: Example small p expansion} for an example.
	If $\sup(\Sigma_0) = \operatorname{deg}(\mathbb o)$ (which is equivalent to $0$ being contained in the spectrum of the discrete Laplacian on $\mathbb G$, a property sometimes called the absence of a spectral gap) then
	\begin{equation}
		\label{Equation: limit of g in terms of transience probability}
		-\lim_{\lambda \downarrow b} 1/g(\lambda) = \operatorname{deg}(\mathbb o) \operatorname{tp}(\mathbb G)
	\end{equation} 
	where  $\operatorname{tp}(\mathbb G)$ denotes the transience probability of a symmetric random walk on $\mathbb G$, i.e.\ the probability that the walk never returns to its initial position. If $\mathbb G$ is additionally bipartite, we have $a = -b$ as well as $\lim_{\lambda \uparrow a} 1/g(\lambda) = \operatorname{deg}(\mathbb o) \operatorname{tp}(\mathbb G)$ by symmetry of $\nu_0$ around the origin. In particular if $\sup(\Sigma_0) = \operatorname{deg}(\mathbb o)$ and if $\mathbb G$ is recurrent and bipartite, e.g.\ if $\mathbb G = \mathbb Z^d$ with $d\leq 2$, then $\lim_{\lambda \uparrow a}g(\lambda) = \lim_{\lambda \downarrow b}-g(\lambda) = \infty$ such that $c(\mathbb G, \mathrm P)=2$ and hence $\hat \nu(\mathcal Z_{p^\delta}) = 2p + o(p)$. On the other hand, for $\mathbb G = \mathbb Z^3$ we have \cite{MW40}
	\begin{equation*}
		0.65 \leq \operatorname{tp}(\Z^3) \leq 0.66.
	\end{equation*}
	If $\mathrm G$ is, for example, with probability $p$ a collection of $M$ leaves and with probability $1-p$ empty then we have with Equation \eqref{Equation: limit of g in terms of transience probability}
	\begin{equation}
		\label{Eqaution: critical number of leaves}
		c(\Z^3, \mathrm P) =
		\begin{cases}
			0 &\text{ if } M \leq \lfloor 36 \operatorname{tp}(\Z^3) \rfloor = 23 \\
			2 &\text{ if } M \geq 24
		\end{cases}
	\end{equation}
	(as $b=-a= \operatorname{deg}(\mathbb o) = 6$, $r(6) = -M/6$). Hence, if $M\leq 23$ then $\mathcal Z = \emptyset$ and all mass up to order $o(p)$ is located no further than $\mathcal O(p^\delta)$ from $\Sigma_0$, and if $M\geq 24$ then $\hat \nu(\mathcal Z_{p^\delta}) = 2p + o(p)$.
	
	\subsection{Expansion for sparse sprinkling}
	\label{Subsection: Results sparse srpinkling}
	Our remaining results will cover an approximation of the density of states in the limit of sparse sprinkling, i.e.\ for small $p \coloneqq \mathrm P(\mathrm G \neq \emptyset)$.  By treating $p$ as a continuous parameter while keeping the conditional distribution $\mathrm P_1 \coloneqq \mathrm P( \, \cdot \, | \mathrm G \neq \emptyset)$ constant and then differentiating the Stieltjes transform
	\begin{equation*}
		\hat s(z) \coloneqq \int_{\R} \frac{1}{\lambda-z} \, \hat \nu(\mathrm d\lambda)
	\end{equation*}
	of $\hat \nu$ with respect to $p$, we derive an expansion of the form
	\begin{equation}
		\label{Equation: Small p expansion ST}
		\hat s(z) = s_0(z) + ph(z) + p^2 S(z)
	\end{equation}
	where $s_0$ is the Stieltjes transform of the density of states $\nu_0$ of $(\mathbb G, \mathbb o)$ and where we bound the remainder $S$ in a suitable form.
	Using the Helffer-Sjöstrand formula, we then translate \eqref{Equation: Small p expansion ST} into its functional form
	\begin{equation}
		\label{Equation: general small p exansion}
		\hat \nu = \nu_0 + p \phi +  p^2 R
	\end{equation}
	for suitable distributions $\phi$ and $R$. To avoid pathological cases, we assume $p \in (0, 1)$. We additionally make the following assumption.
	\begin{assumption}
		\label{Assumption: Assumption for small p expansion}
		We have $\mathrm E[m^2] < \infty$ and for any $a<b$, there exists a constant $c(a, b)>0$ such that almost surely
		\begin{equation}
			\label{Eqaution: Lower bound on r}
			\operatorname{Im} r(z) \geq \1_{\{\mathrm G \neq \emptyset\}}c(a, b)^{-1} \operatorname{Im}(z) \text{ for all } z\in [a, b] \times (0, 1) \subseteq \C.
		\end{equation}
	\end{assumption}
	The lower bound on $\operatorname{Im}r$ in \eqref{Eqaution: Lower bound on r} is a rather weak assumption, for which we need, roughly speaking, a lower bound on the mass of a large compact interval under the spectral measure of $\mathrm T$ with respect to $\varphi$.
	Namely, one can easily show that for every $N>0$ such that $[a, b]\subseteq [-N, N]$ we have almost surely for all $z\in [a, b] \times (0, 1)$
	\begin{equation*}
		\operatorname{Im}r(z) \geq \1_{\{\mathrm G \neq \emptyset\}} \operatorname{Im}(z) C(a, b, N) \cdot \big\langle \varphi, \1_{[-N, N]}(\mathrm T) \varphi \big\rangle
	\end{equation*}
	and where $C(a, b, N) \coloneqq (1 + |b-N|^2\vee |a+N|^2)^{-1}$. For example, if there exists some $l\in \N$ such that $\mathrm G$ contains, conditionally on $\mathrm G \neq \emptyset$, almost surely a connected component with at most $l$ vertices (on which $\varphi$ does not vanish everywhere) then \eqref{Eqaution: Lower bound on r} holds almost surely. We shall bound the remainders $S$ and $R$ in \eqref{Equation: Small p expansion ST} and \eqref{Equation: general small p exansion} in terms of the constant
	\begin{equation*}
		\mathrm C(a, b, \mathrm P_1) \coloneqq 60 (c(a, b) + 1) \mathrm E\big[m^2\big|\mathrm G \neq \emptyset\big]\mathrm E\big[m\big|\mathrm G \neq \emptyset\big].
	\end{equation*}
	Recall that we defined $g$ in \eqref{Equation: Definition of g} to be the Stieltjes transform of the spectral measure of the adjacency operator $B$ of $\mathbb G$ in the root $\mathbb o$ such that in particular $s_0 = \mathbb E[g]$. The function $h$ in \eqref{Equation: Small p expansion ST} is given by
	\begin{equation}
		\label{Equation: Definition h}
		h \coloneqq \mathbb E\bigg[\frac{\partial_z (g r)}{1 - g r} + \sum_{\lambda \in \mathrm S_*}\frac{1}{\lambda-z} \, \Big| \, \mathrm G \neq \emptyset \bigg].
	\end{equation}
	We shall see in Section \ref{Section: sparse sprinkling} that $h$ is the derivative of $\hat s$ in $p = 0$ if we treat the latter as a continuous parameter while keeping $\mathrm P_1$ constant.
	
	\begin{figure}[htbp]
		\centering
		\begin{minipage}[c]{0.6\textwidth}
			\centering
			\includegraphics[width=\textwidth]{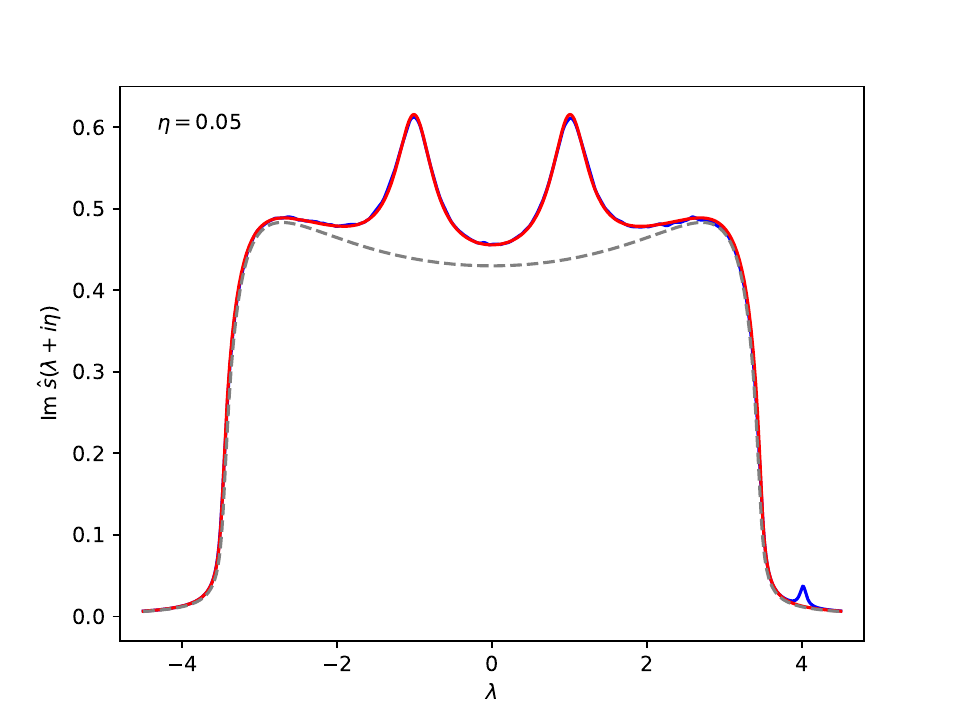}
		\end{minipage}
		\hfill
		\begin{minipage}[c]{0.39\textwidth}
			\centering
			\includegraphics[width=\textwidth]{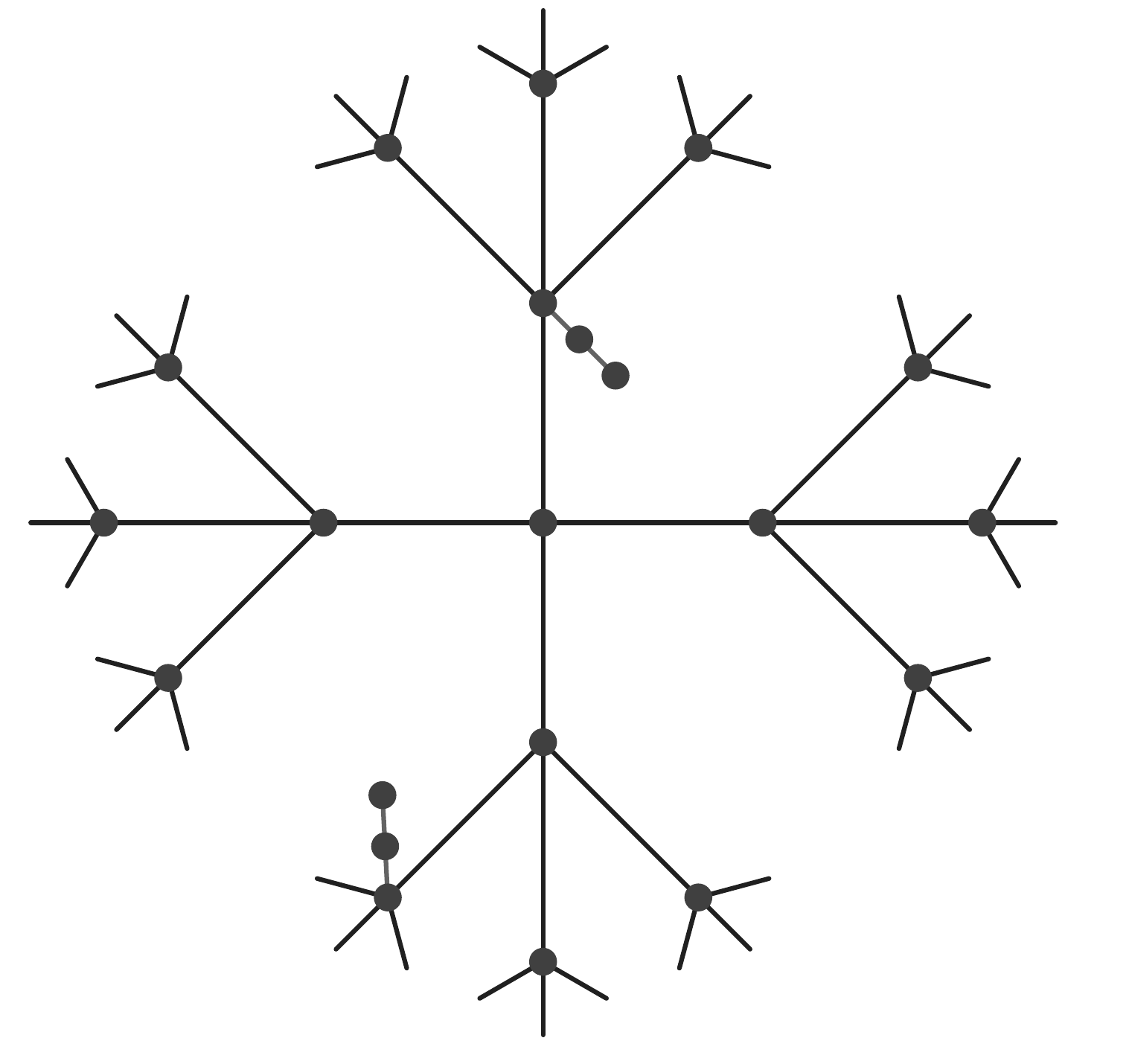}
		\end{minipage}
			\centering
		\begin{minipage}[c]{0.6\textwidth}
			\centering
			\includegraphics[width=\textwidth]{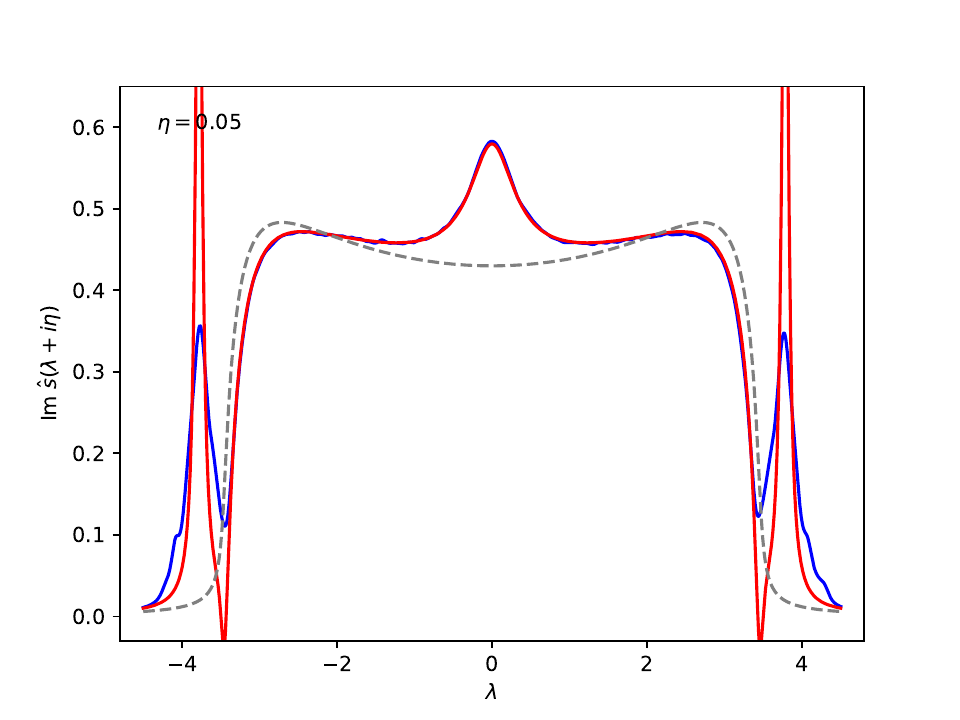}
		\end{minipage}
		\hfill
		\begin{minipage}[c]{0.39\textwidth}
			\centering
			\includegraphics[width=\textwidth]{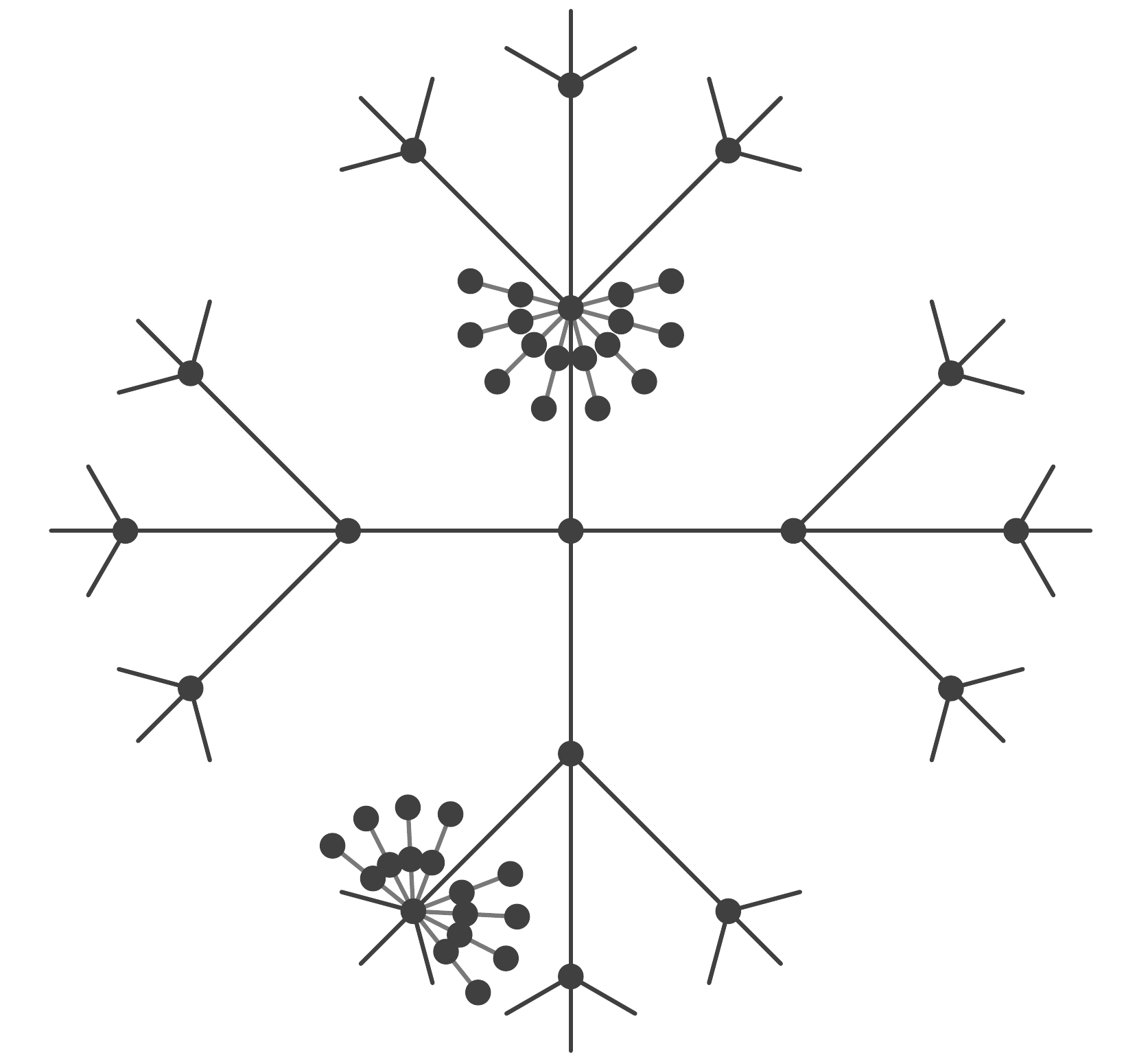}
		\end{minipage}
	
		\caption{
			Imaginary part of the Stieltjes transform of $\hat \nu$ for $\mathbb G = \mathbb B_4$ and $\mathrm P(\mathrm G = n\text{ lines of length 2})=0.05$, $\mathrm P(\mathrm G = \emptyset )=0.95$, where $n=1$ on the left hand side and $n=8$ on the right hand side, for a spectral resolution of $\eta = 0.05$. Simulation in blue, approximation by $s_0 + ph$ as given by Theorem \ref{Theorem: small p expansion ST} in red, Stieltjes transform of the Kesten--McKay law (i.e.\ of the density of states of $\mathbb B_4$, see the imaginary part of \eqref{Equation: ST Kesten-McKay} below) in grey. As we shall see in Section \ref{Section: Anderson-percolation representation},
			increasing $n$ pushes the mass of $\hat \nu$ towards $-\infty$ and $\infty$ and, in this case, towards 0 as it is a zero of $r$. We have $c(\mathbb G, \mathrm P) = 0$ for $n=1$ and $c(\mathbb G, \mathrm P)¨= 2$ for $n=8$; creating in the second case two peaks of mass $p + o(p)$ outside of $\Sigma_0$, see Proposition \ref{Proposition: Concentration of mass outside of Sigma0}. These peaks correspond to singularities of $h$ which are atoms of the distribution $\phi$ given in Theorem \ref{Theorem: small p expansion in dual of Ck}, compare to Proposition \ref{Proposition: Condition that phi locally agrees with measure}. The small peak at $\lambda = 4$ is a finite-volume artifact (coming from the Perron--Frobenius eigenvalue).
		}
		\label{Figure: Example small p expansion}
	\end{figure}	
	
	\begin{thrm}
		\label{Theorem: small p expansion ST}
		If $S:\mathbb C^+ \to \mathbb C$ is the unique function such that
		\begin{equation*}
			\hat s(z) = s_0(z) + ph(z) + p^2S(z)
		\end{equation*}
		for all $z\in \mathbb C^+$, then for all $a<b$ and $z\in (a, b) \times (0, 1) \subseteq\mathbb C$
		\begin{equation}
			\label{Equation: Estimates on h and S}
			|h(z)| \leq 5 \mathrm E[m|\mathrm G \neq \emptyset] \operatorname{Im}(z)^{-1}, \quad   |S(z)| \leq \mathrm C(a, b, \mathrm P_1) (1-p)^{-1} \operatorname{Im}(z)^{-7}.
		\end{equation}
	\end{thrm}
	See Figure \ref{Figure: Example small p expansion} for an illustration of Theorem \ref{Theorem: small p expansion ST}.
	Let us now translate \eqref{Equation: Small p expansion ST} into its functional form \eqref{Equation: general small p exansion}. As we shall see below, the remainder $pR^2$ in \eqref{Equation: general small p exansion} will in general only be small compared to $\nu_0 + p \phi$ if we test both against sufficiently smooth functions, and the quantity $\phi$ in our expansion will in general not be a signed measure but a distribution.
	On that note, let us introduce some terminology. A linear map $\phi: C^1_c(\R) \to \mathbb R$ is called a distribution of order at most one if for any $a<b$ its restriction to $C^1_c(a, b)$ is continuous (with respect to the usual norm $\|f\|_{C^1} \coloneqq \|f\|_\infty +  \|f'\|_\infty$), we refer the interested reader to \cite{Tr67} for an introduction into the topic. Any distribution $\phi$ of order at most one can be written in the form $\phi(f) = \mu_0(f) + \mu_1(f')$ with signed Radon measures\footnote{Here, a signed Radon measure $\mu$ is a linear map $\phi: C_c(\R) \to \R$ such that the restriction of $\mu$ to $C_c(a, b)$ is for any $a<b$ continuous. By the Riesz–Markov–Kakutani representation theorem, there exists Radon measures $\mu^+$, $\mu^-$ such that $\phi(f) = \int f(x) \,\mu^+(\mathrm dx) - \int f(x) \,\mu^-(\mathrm dx)$ for all $f \in C_c(\R)$. Notice that $\mu^+ - \mu^-$ is in general not a well defined signed measure in the measure theoretic sense as both $\mu_1$ and $\mu_2$ might have infinite total mass.} $\mu_0$ and $\mu_1$. As usual, we denote for $a<b$ and $k\in \N$
	\begin{equation*}
		\|\phi\|_{(C_c^{k}(a, b))'} \coloneqq \sup\big\{|\phi(f)|:\, f\in C^k_c(a, b),\, \|f\|_{C^k} \leq 1 \big\} \text{ where } \|f\|_{C^k}  \coloneqq \sum_{i=0}^k \|f^{(i)}\|_\infty \text{ for }f\in C^k_c(a, b).
	\end{equation*}
	Comparing  \eqref{Equation: Small p expansion ST} and \eqref{Equation: general small p exansion} seems to suggest that $h$ plays the role of a Stieltjes transform of $\phi$. This is in fact true if we assume that the support of $\nu$ is compact, in the sense that we then have
	\begin{equation}
		\label{Equation: ST of phi for compact support}
		h(z) = \phi\bigg(\frac{\varphi(\cdot )}{\, \cdot \, -z} \bigg) \quad \text{ for } z \in \C \setminus \R
	\end{equation}
	where $\varphi \in C_c^1(\R)$ is an arbitrary cutoff function with $\varphi = 1 $ on $\operatorname{supp}(\nu)$. In order to get a similar duality even for the case where the support of $\hat \nu$ is not compact (where we cannot define the Stieltjes transform of $\phi$ as the right hand side of \eqref{Equation: ST of phi for compact support}), we apply the Helffer-Sjöstrand formula. The latter states that for any $f\in C^1_c(\C)$ and any $\lambda \in \mathbb R$
	\begin{equation}
		\label{Equation: Helffer-Sjöstrand}
		f(\lambda) = \frac{1}{\pi}\int_\C \frac{\overline{\partial} f(z)}{\lambda - z} \, \mathrm dz 
	\end{equation}
	where $\overline{\partial} = \frac{1}{2}(\partial_\lambda + \mathrm i\partial_\eta)$ denotes the antiholomorphic derivative, see e.g.\ \cite[Proposition C1]{BK19} for a proof.
	Assuming that there exists some $C>0$ such that
	\begin{equation}
		\label{Equation: Bound on antiholomorphic derivative of f}
		|\overline{\partial} f(\lambda + \mathrm i \eta)| \leq C \eta
	\end{equation}
	for all $\lambda, \eta \in \R$, then \eqref{Equation: Helffer-Sjöstrand} and Fubini's theorem yield
	\begin{equation}
		\label{Equation: ST and Helffer-Sjöstrand}
		\hat \nu(f) = \frac{1}{\pi}\int_\C \hat s(z)\overline{\partial} f(z)\, \mathrm dz.
	\end{equation}
	The Helffer-Sjöstrand formula is commonly applied to sufficiently smooth functions on $\mathbb R$ by defining almost analytic extensions of the latter. For that purpose, let us fix some even cutoff function $\psi\in C_c^\infty(\R)$ such that
	\begin{equation}
		\label{Equation: Cutoff-function}
		\forall \eta \in [-1/2, 1/2]:\, \psi(\eta) = 1, \quad \operatorname{supp}(\psi) \subseteq [-1, 1], \quad  \|\psi\|_{\infty} \leq 1, \quad  \|\psi'\|_{\infty} \leq 2.1.
	\end{equation}
	Through
	\begin{equation}
		\label{Equation: almost analytic extension of f}
		f_{\psi}(\lambda + \mathrm i \eta) \coloneqq \psi(\eta)\cdot \big( f(\lambda) + \mathrm i \eta f'(\lambda) \big),
	\end{equation}
every $f\in C_c^2(\R)$ can
	be extended to a function $f_{\psi}\in C^1_c(\C)$ for which \eqref{Equation: Bound on antiholomorphic derivative of f} and consequently \eqref{Equation: ST and Helffer-Sjöstrand} holds. 
	\begin{thrm}
		\label{Theorem: small p expansion in dual of Ck}
		Under Assumption \ref{Assumption: Assumption for small p expansion} the functional
		\begin{equation}
			\label{Equation: Definition of phi through h}
			C^2_c(\R) \to \R, \quad f \mapsto \phi(f) \coloneqq \frac{1}{\pi}\int_{\C \setminus\R} h(z) \overline{\partial} f_{\psi}(z) \,  \mathrm dz
		\end{equation}
		is well defined, independent of the choice of $\psi$ and can be extended to a distribution of order at most one. If $R$ is the unique distribution such that
		\begin{equation*}
			\hat \nu = \nu_0 +  p \phi +   p^2 R,
		\end{equation*}
		then for any $a<b$
		\begin{equation*}
			\|R\|_{(C_c^{8}(a, b))'}  \leq 2^{6}|b-a|C(a, b, \mathrm P_1)(1-p)^{-1}.
		\end{equation*}
	\end{thrm}
	In other words, if $p^2\mathbb E[m^2\big|\mathrm G \neq \emptyset]$ and $p^2c(a, b)$ are small compared to $p$ then $\hat \nu \approx \nu_0 +  p \phi$ in a weak sense, namely when tested against functions in $C^8_c(a, b)$. Under the assumption \eqref{Assumption: Detetministic attached graphs}, a similar expansion is possible around $p=1$, see Subsection \ref{Subsection: Expansion around p=1} later in the text. As the next Proposition shows, the distribution $\phi$ given in Theorem \ref{Theorem: small p expansion in dual of Ck} is in general not a signed measure.
	\begin{prop}
		\label{Proposition: Smoothening is necessary 1}
		Assume that $\mathbb G$ is a deterministic vertex transitive graph and that $\nu_0$ has infinitely many atoms. Assume that \eqref{Assumption:: m a.s. bounded} holds. Then $\phi$ as given by Theorem \ref{Theorem: small p expansion in dual of Ck} is not a signed Radon measure.
	\end{prop}
	
	For example, it is know that the lamplighter group has a finite generating set $S$ such that its Cayley graph $\mathbb G$ with respect to $S$ has a purely atomic density of states, see \cite{GZ01}. Since the lamplighter group is amenable, it is in particular sofic, and hence its Cayley graph is the weak local limit of finite graphs $\mathbb G_N$ \cite[Example 4.4, Theorem 5.1]{Pes08}. That being said, under conditions that are in particular satisfied if $\mathbb G$ is the integer or Bethe lattice with \eqref{Assumption:: m a.s. bounded}, one can show that $\phi$ is (locally) a signed measure. Given two distributions $\phi_1$, $\phi_2$ of order at most one and an open set $U \subseteq \R$ we say that $\phi_1 = \phi_2$ on $U$ if $\phi_1(f) = \phi_2(f)$ for all $f \in C^1_c(U)$. Recall that we defined the set $\mathcal Z$ in \eqref{Equation: Definition of Z}.
	\begin{prop}
		\label{Proposition: Condition that phi locally agrees with measure}
		Assume that $\mathbb G$ is a deterministic vertex transitive graph and that \eqref{Assumption:: m a.s. bounded} holds. Assume there exists $a<b$ such that $\Sigma_0 = [a, b]$. Then the following holds.
		\begin{enumerate}
			\item If $\mathrm D \subseteq [a, b]$ then
			\begin{equation}
				\label{Equation: phi outside of support is purely atomic}
				\phi = \sum_{\lambda \in \mathcal Z}\mathbb P(g(\lambda)r(\lambda) = 1|\mathrm G \neq \emptyset)\cdot \delta_{\lambda} \quad \text{ on } \R \setminus [a, b].
			\end{equation}
			\item If $\nu_0$ has a density $f\in C^2(a, b)$ which is strictly positive on $(a, b)$ then the limit $h(t) \coloneqq \lim_{\eta \downarrow 0} h(t + \mathrm \eta \mathrm i)$ exists locally uniformly on $(a, b)$ and
			\begin{equation}
				\label{Equation: phi insinde of support is ac}
				\phi = \pi^{-1} \operatorname{Im}(h(t)) \, \mathrm dt \quad \text{ on }(a, b).
			\end{equation}	
		\end{enumerate}
	\end{prop}
	Notice that we do not make any statement about the behavior of $\phi$ at the boundary of $\Sigma_0$.
	Recall that for $\mathbb G = \mathbb Z^d$ and $\mathrm D\subseteq \Sigma_0$ we have $\hat \nu(\mathcal Z) = 0$ by Corollary \ref{Corollary: No atoms outside of D for integer lattice}. In other words, while $\phi$ has in this case atoms in $\mathcal Z$ (provided it is non-empty) $\hat \nu$ does not. As we summarize in the upcoming Proposition \ref{Proposition: Smoothening is necessary 2}, this shows that even under the assumptions of Proposition \ref{Proposition: Condition that phi locally agrees with measure}, which guarantee that $\phi$ is locally a measure, it is in general not possible to give an analogous bound (i.e.\ one that only depends on $a, b$ and $\mathrm P_1$) on the remainder $p^2R$ from Theorem \ref{Theorem: small p expansion in dual of Ck} in total variation norm, i.e.\ in operator norm in the dual of the continuous functions.
	\begin{figure}
		\begin{minipage}[t]{1.0\textwidth}
			\includegraphics[width=0.5\textwidth]{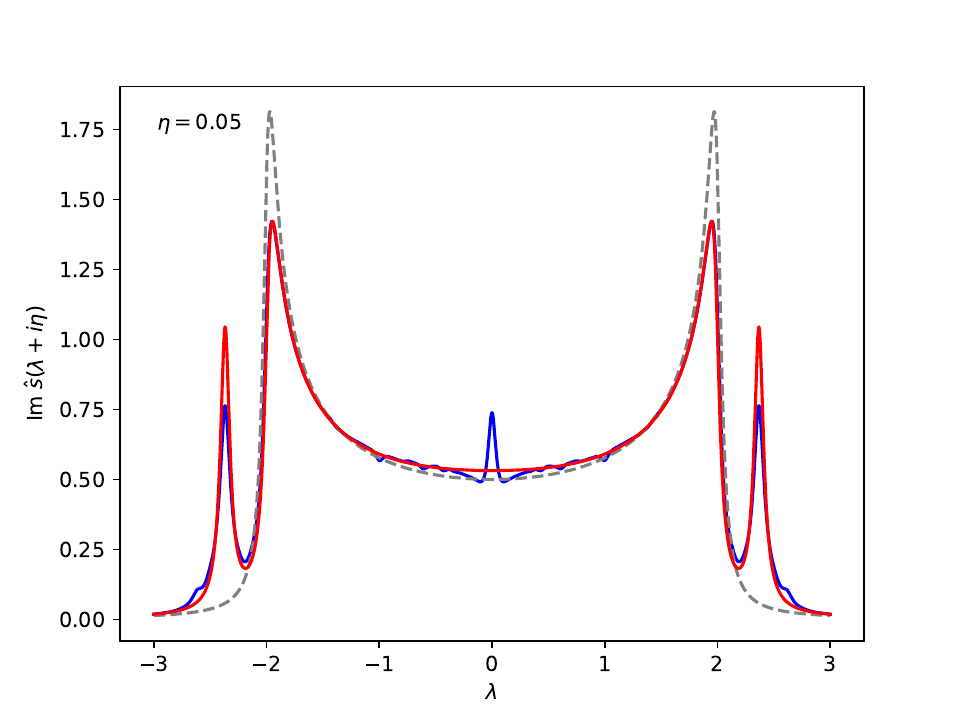}
			\includegraphics[width=0.5\textwidth]{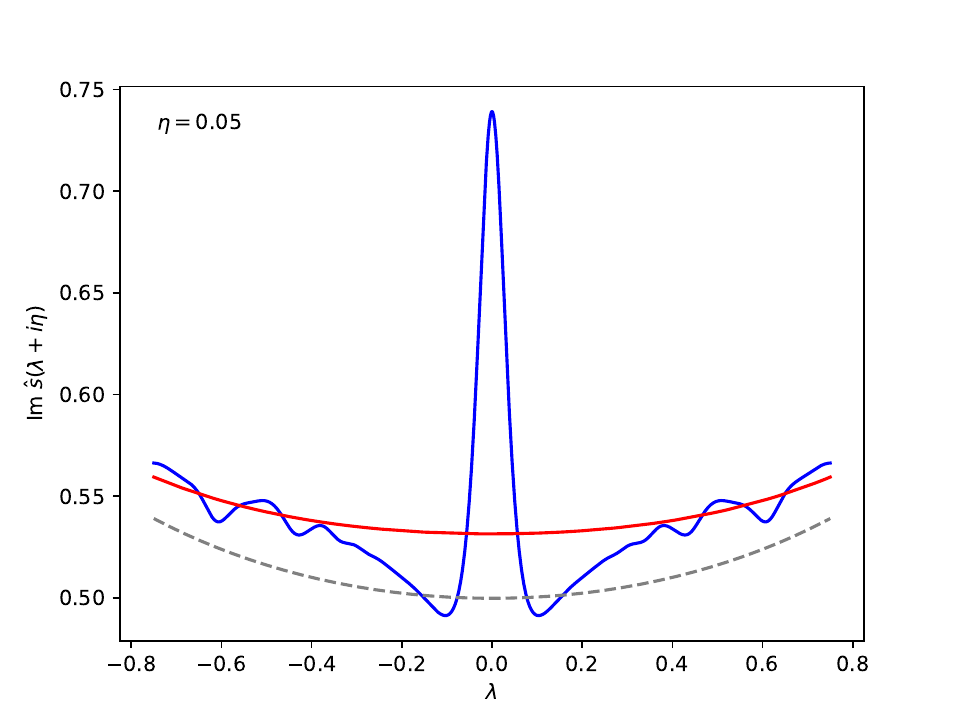}
		\end{minipage}
		\caption{Imaginary part of the Stieltjes transform of $\hat \nu$ for $\mathbb G = \mathbb Z$ and $\mathrm P(\mathrm G = \text{three leaves})=0.05$, $\mathrm P(\mathrm G = \emptyset )=0.95$ for a spectral resolution of $\eta = 0.05$. Simulation in blue, approximation via $s_0 + ph$ as given by Theorem \ref{Theorem: small p expansion ST} in red, Stieltjes transform of the $\operatorname{arcsin}$ law (i.e.\ of the density of states of $\mathbb Z$) in grey. Notice that $\hat \nu$ has, contrary to $\phi$, an atom in zero of mass $p + o(p)$, compare to Propositions \ref{Proposition: Mass of atom in zero for line} and \ref{Proposition: Condition that phi locally agrees with measure}. By recurrence of $\mathbb Z$, the sprinkling creates peaks of mass $p + o(p)$ outside of $\Sigma_0 = [-2, 2]$, compare to Propositions \ref{Proposition: Concentration of mass outside of Sigma0} and \eqref{Equation: limit of g in terms of transience probability}. These peaks correspond to atoms of $\phi$ by Proposition \ref{Proposition: Condition that phi locally agrees with measure} but do not come from atoms of $\hat \nu$ by Corollary \ref{Corollary: No atoms outside of D for integer lattice}.
		}
		\label{Figure: Smoothening is necessary}
	\end{figure}
	
	\begin{prop}
		\label{Proposition: Smoothening is necessary 2}
		Assume that $\mathbb G = \mathbb Z^d$ and that \eqref{Assumption: Detetministic attached graphs} holds with $\mathrm D \subset \Sigma_0 = [-2d, 2d]$. Let $R_p$ be the remainder defined in Theorem \ref{Theorem: small p expansion in dual of Ck}. If $c(\mathbb G, \mathrm P)$ defined in \eqref{Equation: Definiton of c(G)} is positive then\footnote{Notice that $R_p$ is by Proposition \ref{Proposition: Condition that phi locally agrees with measure} a measure on $\R \setminus \Sigma_0$ and on $\Sigma_0^\circ$.} $(R_p(\mathcal Z))_p$ is unbounded. In particular, by \eqref{Equation: limit of g in terms of transience probability} the family $(R_p(\mathcal Z))_p$ is for $d\leq 2$ never bounded. Moreover, if $d=1$ and if $0$ is admissible then $(R_p(\{0\}))_p$ is unbounded as well.  
	\end{prop}
	See Figure \ref{Figure: Smoothening is necessary} for an illustration of Proposition \ref{Proposition: Smoothening is necessary 2}. We finish the section by pointing out that the statement of Theorem \ref{Theorem: small p expansion ST} can be easily rephrased in terms of the Stieltjes transform of the  density of states $\nu$ by solving \eqref{Equation: Definition hat nu} for the latter,
	expanding $(1+\mathrm E[m])^{-1} = 1-p\mathrm E[m| \mathrm G \neq \emptyset] + \mathcal O(p^2)$ and then applying Lemma \ref{Lemma: ST of tuning fork contribution}.
	An equivalent of Theorem \ref{Theorem: small p expansion in dual of Ck} for $\nu$ then follows in the same manner.
	\begin{cor}
		\label{Corollary: small p expansion in dual of Ck for nu}
		The statements from Theorem \ref{Theorem: small p expansion ST} and Theorem \ref{Theorem: small p expansion in dual of Ck} remain true if we replace $\hat \nu$ by $\nu$, $\mathrm C(a, b, \mathrm P_1)$ by $2\mathrm C(a, b, \mathrm P_1)$ and $h$ by
		\begin{equation*}
			\mathbb E\bigg[q- mg + \frac{\partial_z (g r)}{1 - g r} \,  \Big| \,\mathrm G \neq \emptyset \bigg] \quad  \text{ where } \quad q(z) \coloneqq \operatorname{tr} \, (\mathrm T-z)^{-1}.
		\end{equation*}
	\end{cor}

	\subsection{Expansion for sparse sprinkling and large degrees in $\mathbb G$}
	\label{Subsection: results sparse sprinkling and large d}
	We now study the density of states in the limit of large degrees in $\mathbb G$ and sparse sprinkling. Namely, we sprinkle a sequence of graphs $(\mathbb G_d)_d$, under assumptions which in particular include the case where $\mathbb G_d= \mathbb Z^d$ or $\mathbb G_d= \mathbb B_d$, and study the density of states of the sprinkled graph in the limit of large $d$ and small $p$. Roughly speaking, we consider the case where the density of states $\nu_{0, d}$ of $\mathbb G_d$ has a density $f_d$ which scales as $f_d(x) \approx d^{-1/2}f(d^{-1/2}x)$ for some sufficiently smooth probability density $f$, see Assumption \ref{Assumption large d expansion} later in the text. This is in particular the case if $\mathbb G_d$ is either the integer lattice $\mathbb Z^d$ or the Bethe lattice $\mathbb B_d$. The density $f_d$ is in both cases of the form $f_d(x) \approx d^{-1/2}f(d^{-1/2}x)$, where $f$ is for $\mathbb G_d = \mathbb Z^d$ the density of $ \mathcal N(0, 2)$ and for $\mathbb G_d = \mathbb B_d$ the density of the semicircular law: $\mathbb Z^d$ is the $d$-fold cartesian product of the graph $\mathbb Z$, implying that\footnote{This can be shown by comparing the moments of both measures, see \cite[Section 1.3]{Bor16}.} $\nu_{0, d} = \nu_{0, 1}^{*d}$, and (at least if we restrict ourselves to even $d$) from the fact that $\mathbb B_{2d}$ is the $d$-fold free product of $\mathbb G$, implying that\footnote{This can be found in \cite[Section 1.3]{Bor16} as well.} $\nu_{0, 2d} = \nu_{0, 1}^{\boxplus d}$, combined with the local central limit theorem and its free counterpart. To state our result, let us denote for $a\in \R$ and $b>0$ by $\operatorname{Cau}(a, b)$ the Cauchy distribution with density
	\begin{equation*}
		x \mapsto \frac{1}{\pi} \frac{b}{(x-a)^2 + b^2}.
	\end{equation*}	
	Under Assumption \ref{Assumption large d expansion}, on any bounded interval $I$ the limit $g_d(\lambda) = \lim_{\eta \downarrow 0} g_d(\lambda + \mathrm i \eta)$ exists for sufficiently large $d$ uniformly in $\lambda\in I$, compare with Proposition \ref{Proposition: Condition that phi locally agrees with measure}. For $\lambda \in \R$, we define	
	\begin{equation*}
		\quad b_d(\lambda) \coloneqq w(\lambda) \operatorname{Im}(g_d(\lambda)) \quad \text{ where }w(\lambda) \coloneqq \big\langle \varphi, \1_{\{\lambda\}}(\mathrm T) \varphi \big\rangle
	\end{equation*}
	denotes mass of $\{\lambda\}$ under the spectral measure of $\mathrm T$ with respect to $\varphi$.
	
	\begin{figure}
		\centering
		\begin{minipage}[c]{0.6\textwidth}
			\centering
			\includegraphics[width=\textwidth]{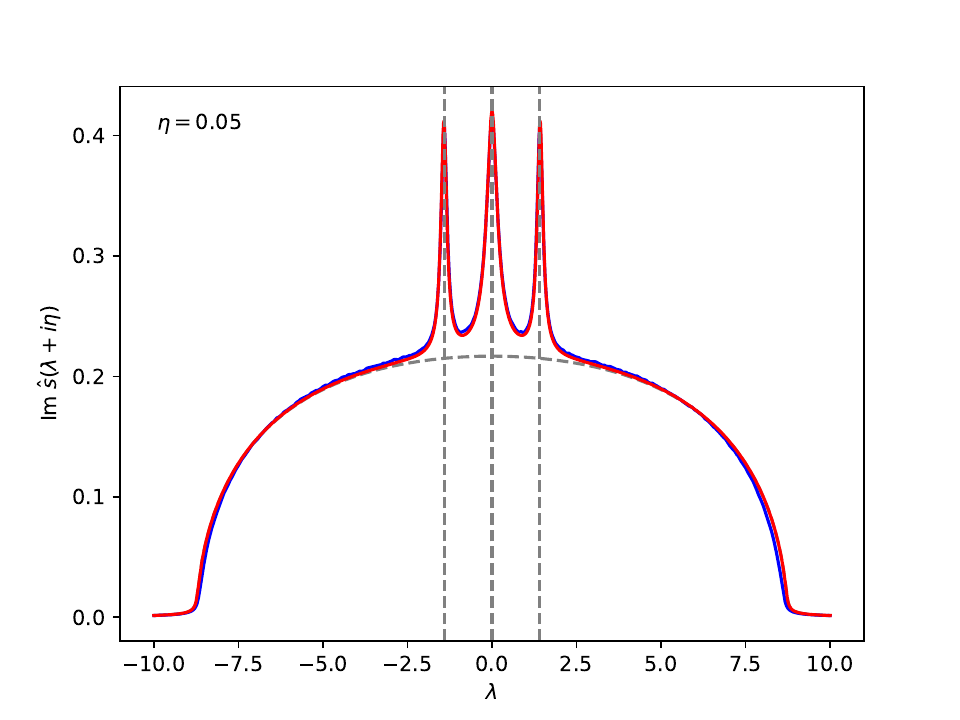}
		\end{minipage}
		\hfill
		\begin{minipage}[c]{0.39\textwidth}
			\centering
			\includegraphics[width=\textwidth]{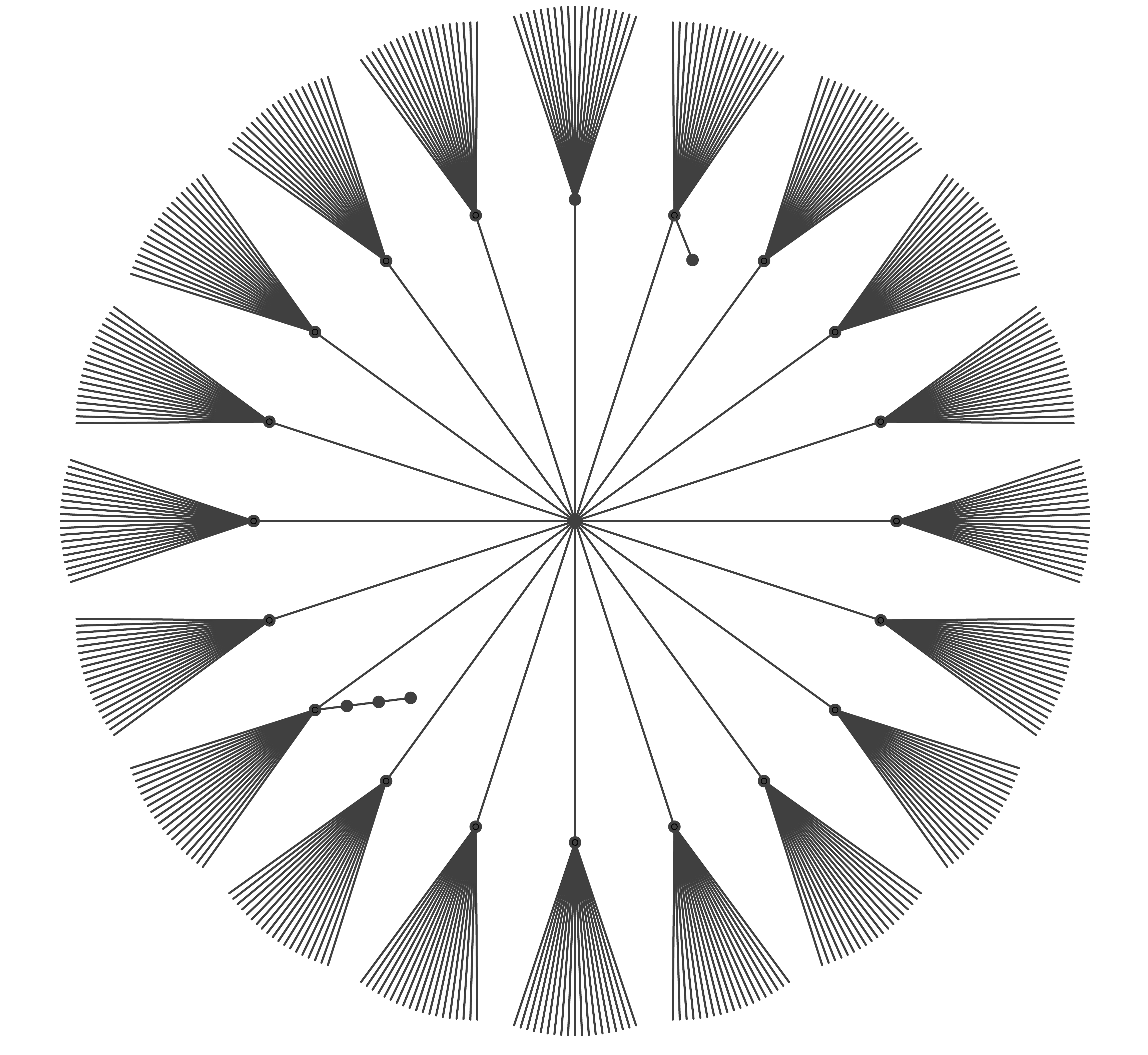}
		\end{minipage}
		\caption{Imaginary part of the Stieltjes transform of $\hat \nu$ for $\mathbb G = \mathbb B_{20}$ and $\mathrm P(\mathrm G = \text{line of length }3)=0.01$, $\mathrm P(\mathrm G = \text{leaf})=0.01$, $\mathrm P(\mathrm G = \emptyset )=0.98$ for a spectral resolution of $\eta = 0.05$. Simulation in blue, approximation via the Stieltjes transform of the two leading terms in \eqref{Equation: Small p large d} in red, Kesten--McKay law in grey. The vertical grey lines indicate the eigenvalues of the attached graphs. Notice that the peak in the middle is the superposition of two Cauchy peaks, as both the leaf as well as the line of length 3 have eigenvalue 0, with $w(\lambda)$ (and hence the width of the peak) being bigger for the single leaf.}
		\label{fig:smallplargedglobal}
	\end{figure}

	\begin{thrm}
		\label{Theorem: Large d-expansion of first order term}
		Assume that $\mathbb G = \mathbb Z^d$ or $\mathbb G = \mathbb B_d$ (or, more generally, that Assumption \ref{Assumption large d expansion} holds) and assume that \eqref{Assumption:: m a.s. bounded} holds. If $R_d$ is the remainder from Theorem \ref{Theorem: small p expansion in dual of Ck} then
		\begin{equation}
			\label{Equation: Small p large d}
			\hat \nu_d = \nu_{0,d} + p \mathbb E\Big[\sum_{\lambda \in \mathrm S_*} \operatorname{Cau}\big(\lambda, b_d(\lambda)\big) \big|\mathrm G \neq \emptyset \Big] + p \mathcal O(d^{-1/30}) + p^2 R_{d}
		\end{equation}
		where $\mathcal O(d^{-1/30})$ is with respect to total variation norm on any bounded and open interval.
	\end{thrm}
	See Figure \ref{fig:smallplargedglobal} for an illustration of Theorem \ref{Theorem: Large d-expansion of first order term}.
	Notice that the width of the Cauchy peaks is of order $w(\lambda)\operatorname{Im}(g_d(\lambda)) \asymp d^{-1/2}$. While the proof of Theorem \ref{Theorem: Large d-expansion of first order term} suggests that the Cauchy peaks are offset by $w(\lambda) \operatorname{Re}(g_d(\lambda)) \asymp d^{-1}$ from $\lambda$, our error estimate is in fact too rough to draw this conclusion, see Lemma \ref{Lemma: TV-distance of Cauchy distributions} later it the text. Notice that Theorem \ref{Theorem: Large d-expansion of first order term} applied to the Stieltjes transforms allows us to choose the spectral resolution $\eta$ smaller than the width of the Cauchy peaks provided that we choose $p$ small enough. However, the imaginary part of the Stieltjes transform of the error-term $p \mathcal O(d^{-1/30})$ might then not be pointwise small, but only be small in the $L^1$ sense. 
	
	The  heuristic leading to the proof of Theorem \ref{Theorem: Large d-expansion of first order term} suggests that even for small $d$ peaks in the spectrum might be generated by eigenvalues of the attached graphs $\lambda$ with small but positive $w(\lambda)$, see Remark \ref{Remark: Expansion around small w} and Figure \ref{fig: small w expansion} later in the text.
	
	\subsection{Sources and Sinks}
	\label{Subsection: results sources and sinks}
	As we shall see in the Proof of Theorem \ref{Theorem: concentration of mass inside of Sigma_0}, connecting the attached graphs to a vertex $x$ has two competing effects: existing eigenvalues close to an admissible eigenvalue $\lambda$ are repulsed from $\lambda$ towards the zeros of $r_x$ while new eigenvalues close to $\lambda$ might be generated. Assuming \eqref{Assumption: Detetministic attached graphs}, we study which of the two effects dominates for small $p$ (in a smoothened fashion). For $\lambda \in \mathbb R$ and $\delta>0$ we set
	\begin{equation*}
		\mathcal D(\lambda, \delta) \coloneqq \big\{f\in C_c^1(\lambda - \delta, \lambda + \delta):\,  f \neq 0 \text { and } f(x)\geq 0 \text { for all }x \big\}.
	\end{equation*}
	\begin{mydef}
		Assuming \eqref{Assumption: Detetministic attached graphs}, we call $\lambda \in \mathrm D$ a {\em source} if there exists a $\delta>0$ such that for all $f\in \mathcal D(\lambda, \delta)$ we have $\hat{\nu}_p(f) > \nu_0(f)$ as long as $p>0$ is sufficiently small. If the statement holds after we reverse the inequality, we call $\lambda$ a {\em sink}.
	\end{mydef}
	
	\begin{prop}
		\label{Proposition: Criterion for repulsive eigenvalues}
		Assume that $\mathbb G$ is a deterministic vertex transitive graph and that there exists $a<b$ such that $\Sigma_0 = [a, b]$. Assume that $\nu_0$ has a density $f \in C^2(a, b)$ which is strictly positive on $(a, b)$. Assume that \eqref{Assumption: Detetministic attached graphs} holds. Assume that $\lambda \in \mathrm D \cap (a, b)$ is admissible. Then
		\begin{align*}
			&w(\lambda)^{-1} \operatorname{Im}g(\lambda) > \operatorname{Im}g'(\lambda) \operatorname{Re}g(\lambda) - \operatorname{Im}g(\lambda) \operatorname{Re}g'(\lambda)\quad \implies \quad \text{ $\lambda$ is a source}\\
			&w(\lambda)^{-1} \operatorname{Im}g(\lambda) < \operatorname{Im}g'(\lambda) \operatorname{Re}g(\lambda) - \operatorname{Im}g(\lambda) \operatorname{Re}g'(\lambda) \quad \implies \quad \text{ $\lambda$ is a sink}.
		\end{align*}
	\end{prop}
	
	Whether one of the criteria above is satisfied depends only locally on the spectrum of the attached graphs (in the form of $w(\lambda)$) but globally on the spectrum of $\mathbb G$, as $\operatorname{Re}(g)$ depends globally on the latter, see Proposition \ref{Proposition: Limit of ST} in Appendix \ref{Section: Facts about Stieltjes transforms}. For bipartite $\mathbb G$, the spectral density $\operatorname{Im}(g)$ is symmetric and $\operatorname{Re}g(0) = 0$. In particular, provided that $\operatorname{Im} g(0)>0$, we then have that $\lambda = 0$ is a source if $w(0)^{-1} > - \operatorname{Re}g'(0)$
	and a sink if the reverse inequality holds.
	If $\operatorname{Im}(g)$ has a global maximum in $\lambda = 0$, then $\operatorname{Re}g'(0)<0$ holds, again compare to Proposition \ref{Proposition: Limit of ST}. In this case $\lambda$ transitions from being a source to a sink as $w(0)^{-1}$ increases (e.g.\ as we increase the number of graphs contained in $\mathrm G$ who have zero as an eigenvalue). A natural question might be under which conditions the same transition occurs for $\lambda \neq 0$. If we consider e.g.\ again the bipartite case and assume that $\operatorname{Im}(g)$ is increasing on $(-\infty, 0]$ and decreasing on $[0, \infty)$ (e.g.\ for $\mathbb G = \mathbb B_d$ with $d$ sufficiently large), then a sufficient condition for the occurrence of the aforementioned transition in any $\lambda \in \Sigma_0$ is that $\operatorname{Re}g$ is strictly decreasing in $\Sigma_0$ (where we at least formally have $ \operatorname{Re}g' = -\sqrt{-\partial_\lambda^2}\operatorname{Im} g > 0$, see Remark \ref{Remark: Derivative of real part of ST}). This is in particular the case for the Bethe lattice with $d\geq 3$.
	\begin{prop}
		\label{Proposition: Source-Sink transition Bethe lattice}
		Assume $\mathbb G = \mathbb B_d$ for some $d\geq 2$. Assume that \eqref{Assumption: Detetministic attached graphs} holds and that  $\lambda \in \mathrm D \cap (-2 \sqrt{d-1}, 2\sqrt{d-1})$ is admissible. If $d = 2$ then $\lambda$ is always a source. If $d\geq 3$ then $\lambda$ is a source if
		\begin{equation*}
			w(\lambda) < \frac{2 (d^2 -\lambda^2)}{d-2} \Big(1 - \frac{\lambda^2}{4(d-1)} \Big)
		\end{equation*}
		and a sink if the reverse inequality holds.
	\end{prop}
	Hence, as $d$ grows it becomes easier to be a source and if we are sufficiently close to the boundary $\pm 2\sqrt{d-1}$ of $\Sigma_0$ it becomes easier to be a sink.
	We finish this section by pointing out that Proposition \ref{Proposition: Criterion for repulsive eigenvalues} shows that \eqref{Equation: Increase of eigenvalues in I by sprinkling} does not need to hold for intervals $I$ for which the expected number of admissible eigenvalues in $I$ is only $p$.
	\begin{cor}
		\label{Corollary: mass can decrease around tree eigenvalue for small p}
		Under the assumptions of Proposition \ref{Proposition: Source-Sink transition Bethe lattice}, assume that 
		\begin{equation*}
			w(\lambda) > \frac{2 (d^2 -\lambda^2)}{d-2} \Big(1 - \frac{\lambda^2}{4(d-1)} \Big).
		\end{equation*}
		Then for all sufficiently small $p$ there exists some $c = c(p)>0$ such that $\lim_{p \downarrow 0}c(p) = 0$ and such that
		\begin{equation*}
			\hat \nu([\lambda-c, \lambda+c]) < \nu_0([\lambda-c, \lambda+c]).
		\end{equation*}
	\end{cor}

	\section{Existence of the infinite-volume limit: proof of Theorem \ref{Theorem: Limit Theorem}}
\label{Section: BS-Convergence}

In this section, we prove Theorem \ref{Theorem: Limit Theorem}, i.e.\ show convergence of the sprinkled graphs $\mathcal G_N$ as $N\to \infty$ in the sense of Benjamini-Schramm.  
We start by briefly summarizing our approach.  
Remember that we denote by $o_N$ an uniformly chosen vertex from $\mathcal V_N$. For $x\in \mathbb V_N$ let
\begin{equation*}
	\mathcal V_{N}^x \coloneqq \{x\} \cup \{(xt):\, t\in \mathrm V^x  \}.
\end{equation*}
We project $o_N$ to $\mathbb V_N$ by defining $\mathbb o_N \coloneqq x$ where $x \in \mathbb V_N$ is the unique vertex such that $o_N \in \mathcal V_{N}^x$. Notice that $\mathbb o_N$ is uniformly distributed in $\mathbb V_N$ and independent of $\mathbb G_N$. Hence, $(\mathbb G_N, \mathbb o_N)$ converges by assumption in distribution to $(\mathbb G, \mathbb o)$ as $N\to \infty$. However, conditionally on $\mathbb o_N$ the family $(\mathrm G^x, \varphi^x)_x$ is neither independent nor identically distributed and large $m_{\mathbb o_N}$ is favored. That being said, we can recover a sample of $(\mathcal G_N, o_N)$ in the following way: Denote $W_{x} = 1 + m_x$ and
\begin{enumerate}
	\item Draw a sample of $(\mathbb G_N, \mathbb o_N)$.
	\item Draw a sample of $(W_x)_{x}$ from the conditional distribution of $W$ given $\mathbb o_N$.
	\item Draw $(\mathrm G^x, \varphi^x)_{x\in \mathbb V_N}$ from the conditional distribution $\bigotimes_{x \in \mathbb V_N}\mathrm P( \, \cdot \, |W_x)$. 
	\item Draw $o_N$ uniformly from $\mathcal V_{N}^{\mathbb o_N}$.
\end{enumerate}
Our task of determining the local limit of $\mathcal G_N$ therefore becomes to determine the limiting distribution of the conditional distribution of a finite subfamily of $W$ (that contains $W_{\mathbb o_N}$) given $\mathbb o_N$. Notice that for any fixed $x\in \mathbb V_N$, the conditional probability that $\mathbb o_N$ is $x$ given $W$ is given by
\begin{equation*}
	\mathbb P(\mathbb o_N = x|W) = \frac{W_{x}}{\sum_{y\in \mathbb V_N} W_y}.
\end{equation*}

We now show the aforementioned convergence of a finite subfamily of $W$ conditionally on $\mathbb o_N$. While in the previous construction $W_x\geq 1$ is certainly satisfied, in order to allow for the treatment of the Anderson-percolation model in the next section, we consider the case in which $W_x$ might be zero with positive probability. We denote by $*$ a general placeholder element.

\begin{lemma}
	\label{Lemma: Convgerence of biased distribution}
	Let $W = (W_n)_{n \geq 1}$ be a sequence of iid $\mathbb N$ valued random variables such that $\mathbb E[W_1] < \infty$ and such that $\mathbb P(W_1=0) < 1$.
	For $N\in \mathbb N$ let $I_N$ be a $\{*, 1 , \hdots, N\}$ valued random variable satisfying
	\begin{equation*}
		\mathbb P(I_N = k|W) =
		\begin{cases}
			\dfrac{W_k}{\sum_{i=1}^N W_i} \quad &\text{ if }(W_1, \hdots, W_N) \neq (0, \hdots, 0) \\
			\1_{\{k=*\}} \quad &\text{ if }(W_1, \hdots, W_N) = (0, \hdots, 0)
		\end{cases}
	\end{equation*}
	for all $k\geq 1$. Let $\widehat{W}_1$ be $\mathbb N$ valued random variable independent of $W$ with
	\begin{equation*}
		\mathbb P(\widehat W_1 = k) = \frac{k\cdot \mathbb P(W_1 = k)}{\mathbb E[W_1]} \quad \text{ for all }k\in \mathbb N.
	\end{equation*}
	Then, for any fixed $M\in \mathbb N$ the distribution of $(W_1, \hdots, W_M)$ conditionally on $\{I_N = 1\}$ converges in total variation to the distribution of $(\widehat W_1, W_2, \hdots, W_M)$.
\end{lemma}		

\begin{proof}
	Since $\mathbb N^M$ is countable, it is sufficient to show that for any $k \in \mathbb N^M$
	\begin{equation}
		\label{Equation: Pointwise convegrence of pmf}
		\mathbb P((W_1, \hdots, W_M) = k |I_N = 1) \to \mathbb P((\widehat W_1, \hdots, W_M) = k) 
	\end{equation}
	as $N\to \infty$. If $k_1 = 0$ then
	\begin{equation*}
		\mathbb P((W_1, \hdots, W_M) = k |I_N = 1) = 0 = \mathbb P((\widehat W_1, \hdots, W_M) = k),
	\end{equation*}
	so let us assume that $k_1 \neq 0$. Let $p \coloneqq \mathbb  P(W_1 = 0)$. Notice that
	\begin{equation*}
		\mathbb P(I_N = 1) = \mathbb P\big(I_N = 1|(W_1, \hdots, W_N) \neq (0, \hdots, 0)\big)\cdot (1-p^N) = \frac{1-p^N}{N}.
	\end{equation*}
	We therefore have
	\begin{align*}
		&\mathbb P\big((W_1, \hdots, W_M) = k| I_N = 1 \big) \\
		&= \sum_{\tilde k \in \mathbb N^{N-M}} \mathbb P\big(I_N = 1|W = (k, \tilde k)\big) \mathbb P(W = (k, \tilde k)) \mathbb P(I_N = 1)^{-1} \nonumber \\
		&= \frac{N}{1-p^{N}} \cdot \mathbb P(W_1, \hdots, W_M = k) \sum_{\tilde k \in \mathbb N^{N-M}} \frac{k_1}{\sum_{i=1}^M k_i + \sum_{j=1}^{N-M} \tilde k_j  } \mathbb P\big((W_{M+1}, \hdots, W_N) = \tilde k\big) \\
		&= \frac{1}{1-p^{N}} \cdot k_1 \mathbb P(W_1 = k_1) \prod_{i=2}^M \mathbb P(W_i = k_i) \cdot \mathbb E\bigg[\frac{1}{N^{-1}\sum_{i=1}^M k_i +  N^{-1}\sum_{i=1}^{N-M} W_i} \bigg].  
	\end{align*}
	We estimate
	\begin{align*}
		&\big|\mathbb P((W_1, \hdots, W_M) = k |I_N = 1) - \mathbb P((\widehat W_1, \hdots, W_M) = k) \big| \\
		&\leq \frac{p^N}{1-p^N} \mathbb P((\widehat W_1, \hdots, W_M) = k) + \Big|\mathbb P((W_1, \hdots, W_M) = k |I_N = 1) - \frac{1}{1-p^N} \mathbb P((\widehat W_1, \hdots, W_M) = k) \Big|\\
		&\leq \frac{1}{1-p^N}\bigg(p^N \, + \,  k_1 \mathbb P(W_1 = k_1) \prod_{i=2}^M \mathbb P(W_i = k_i) \cdot \bigg|\mathbb E\bigg[\frac{1}{N^{-1}\sum_{i=1}^M k_i +  N^{-1}\sum_{i=1}^{N-M} W_i} \bigg] - \frac{1}{\mathbb E[W_1]} \bigg| \bigg).
	\end{align*}
	Let us define the random variable
	\begin{equation*}
		Y_N \coloneqq \frac{\Big|N^{-1}\sum_{i=1}^{N-M} W_i - \mathbb E[W_1]\Big| + C N^{-1}}{N^{-1} +  N^{-1}\sum_{i=1}^{N-M} W_i} \quad \text{ where } C \coloneqq \sum_{i=1}^M k_i.
	\end{equation*}
	Since $k_1 \geq 1$, we have
	\begin{equation}
		\label{Equation: Estimate ev 1/averages Wi}	
		\bigg| \mathbb E\bigg[\frac{1}{N^{-1}\sum_{i=1}^M k_i +  N^{-1}\sum_{i=1}^{N-M} W_i} \bigg] - \frac{1}{\mathbb E[W_1]}\bigg| \leq \frac{\mathbb E[Y_N]}{\mathbb E[W_1]}.
	\end{equation}
	We define the event
	\begin{equation*}
		\mathcal A_N \coloneqq \bigg\{\frac{1}{N}\sum_{i=1}^{N-M} \1_{\{W_i \geq 1\}} < \frac{1-p}{2}  \bigg\}.
	\end{equation*}
	We then have
	\begin{align*}
		\mathbb E[Y_N\1_{\mathcal A_N^c}] \leq \frac{2}{1-p} \mathbb E\Big[\Big|N^{-1}\sum_{i=1}^{N-M} W_i - \mathbb E[W_1]\Big| + C N^{-1} \Big] \to 0
	\end{align*}
	as $N\to \infty$ by the $L^1$-law of large numbers. On the other hand, we have
	\begin{align*}
		Y_N \leq \frac{N^{-1}\sum_{i=1}^{N-M} W_i}{N^{-1} +  N^{-1}\sum_{i=1}^{N-M} W_i} +  \frac{\mathbb E[W_1] + C N^{-1}}{N^{-1} +  N^{-1}\sum_{i=1}^{N-M} W_i} \leq 1 + N\mathbb E[W_1] + C
	\end{align*}
	leading to
	\begin{equation*}
		\mathbb E[Y_N \1_{\mathcal A_N}] \leq \big(1 + N\mathbb E[W_1] + C\big) \mathbb P(\mathcal A_N) \to 0
	\end{equation*}
	as $\mathbb P(\mathcal A_N) \to 0$ exponentially fast by large deviations theory (for example by Hoeffding's inequality). Hence, $\mathbb E\mathbb [Y_N] \to 0$ as $N\to \infty$ which implies with \eqref{Equation: Estimate ev 1/averages Wi} that \eqref{Equation: Pointwise convegrence of pmf} holds.
\end{proof}

\begin{proof}[Proof of Theorem \ref{Theorem: Limit Theorem}]
	As the space of locally finite unlabeled graphs is separable, we might define by Skorokhod's representation theorem $(\mathbb G_N, \mathbb o_N)$ and $(\mathbb G, \mathbb o)$ on a common probability space such that $(\mathbb G_N, \mathbb o_N) \to (\mathbb G, \mathbb o)$ almost surely, i.e.\ such that almost surely we have for every $R>0$ that $(\mathbb G_N, \mathbb o_N)_R \cong (\mathbb G, \mathbb o)_R$ eventually. After relabeling the vertices, we might assume without loss of generality that $\mathbb V_N, \mathbb V \subseteq \mathbb N^+$ and that $\mathbb o_N = \mathbb o = 1$ for every $N\in \N$; and that almost surely $(\mathbb G_N, 1)_R = (\mathbb G, 1)_R$ eventually for all $R>0$. We now define the sprinkled graphs $(\mathcal G_N, o_N)$ and $(\mathcal G, o)$ as described in the beginning of this section. Namely, let $(W^x)_{x\geq 1}$ a sequence of iid random variables such that the distribution of $W^x$ is for every $x \in \mathbb N^+$ given by the distribution of $m + 1$ under the sprinkling distribution $\mathrm P$. Let $I_N$ be for every $N\in \N$ be a $\{1, \hdots, |\mathbb V_N|\}$ valued random variable such that 
	\begin{equation*}
		\mathbb P(I_N = x|W) =
		\frac{W^x}{\sum_{y=1}^{|\mathbb V_N|} W^y} \quad \text{ for all }x\in \{1, \hdots, |\mathbb V_N|\}.
	\end{equation*}
	For every $N\in \mathbb N$ let $((\widehat W_N^x)_{x \geq 1}, ((\mathrm G_N^x, \varphi_N^x))_{x\geq 1}, o_N)$ be independent of $(\mathbb G_N, 1)$ such that
	\begin{enumerate}
		\item The distribution of $\widehat W_N = (\widehat W_N^1, \hdots, \widehat W_N^{|\mathbb V_N|})$ coincides with the distribution of $(W^1, \hdots, W^{|\mathbb V_N|})$ conditioned on $\{I_N = 1\}$
		\item The distribution of $((\mathrm G^1_N, \varphi^1_N), \hdots, (\mathrm G^{|\mathbb V_N|}_N, \varphi^{|\mathbb V_N|}_N) )$ conditionally on $\widehat W_N$ is $\bigotimes_{x=1}^{|\mathbb V_N|} \mathrm P(\, \cdot \, |\widehat W_N^x)$
		\item Conditionally on $((\widehat W_N^x)_x, (\mathrm G_N^x, \varphi^x)_x)$, the random vertex $o_N$ is chosen uniformly from the set $\{1\} \cup \{(1t):\, t\in \mathrm V^1\}$.
	\end{enumerate}
	Let $\mathcal G_N$ be the graph obtained from sprinkling $\mathbb G_N$ with $(\mathrm G_N^x, \varphi_N^x)_{1\leq x\leq |\mathbb V_N|}$. In the same manner we define the sprinkled graph $(\mathcal G, o)$ through $((\widehat W^x)_{x \geq1}, ((\mathrm G^x, \varphi^x))_{x\geq 1}, o)$ where the sequence $(\widehat W^x)_{x \geq 1}$ is independent with
	\begin{equation*}
		\mathbb P(\widehat W^1 = k) = \frac{k\cdot \mathbb P(W^1 = k)}{\mathbb E[W^1]} \quad \text{ for all }k\in \mathbb N, \quad (\widehat W^x)_{x \geq 2} \overset{d}{=} (W^x)_{x\geq 1},
	\end{equation*}
	and where $((\mathrm G^x, \varphi^x))_{x\geq 1}, o$ are defined accordingly. By Lemma \ref{Lemma: Convgerence of biased distribution} and another application of Skorokhod's representation theorem\footnote{Where we equip the state spaces of $(W^x)_{x\geq 1}$ and $((\mathrm G^x, \varphi^x))_{x\geq 1}$ with the respective product topologies.} we might assume that $((\widehat W_N^x)_{x \geq 1}, (\mathrm G_N^x)_{x\geq 1}, o_N)$ and $((\widehat W_x)_{x\geq 1}, (\mathrm G^x)_{x\geq 1}, o)$ are defined on the same probability space and that we have almost surely for every $R>0$ that
	\begin{equation*}
		(\widehat W_N^x)_{1\leq x \leq R} = (\widehat W^x)_{1\leq x \leq R}, \quad ((\mathrm G_N^x, \varphi^x_N))_{1\leq x \leq R} = ((\mathrm G^x, \varphi^x))_{1\leq x \leq R}, \quad o_N = o
	\end{equation*}
 eventually. We then have almost surely for every $R>0$ that $(\mathcal G_N, o_N)_R = (\mathcal G, o)_R$ eventually, which shows the claim. 
\end{proof}

\section{Anderson-percolation representation}
\label{Section: Anderson-percolation representation}
In this Section, we prove the Anderson-percolation representation of the atoms and the spectrum, as given in Theorem \ref{Theorem: Anderson representation of atoms} and Theorem \ref{Theorem: Anderson representation of the spectrum} respectively. We start our discussion in Subsection \ref{Subsection: Atom in zero for Z and BerAG}, and motivate our approach by calculating $\nu(\{0\})$ under the assumptions $\mathbb G= \mathbb Z$ and \eqref{Assumption: Detetministic attached graphs} with 0 being admissible. Generalizing this approach in Subsections \ref{Subsection: starting remarks Anderson-perolation representation} and \ref{Subsection: Anderson-percolation representation in finite volume} will then yield the equivalent of Theorem \ref{Theorem: Anderson representation of atoms} in finite volume. In Subsection \ref{Subsection: Anderson Percolation representation of atoms} we take the infinite-volume limit in the latter in order to show Theorem \ref{Theorem: Anderson representation of the spectrum}. In Subsection \ref{Subsection: Anderson-percolation representation of the spectrum}, we show a similar representation for the spectrum of the adjacency operator and thereby prove Theorem \ref{Theorem: Anderson representation of the spectrum}, which we apply in Section \ref{Section: support of the DOS} to determine the support of the density of states.

\subsection{Introductory example: proof of Proposition \ref{Proposition: Mass of atom in zero for line}}
\label{Subsection: Atom in zero for Z and BerAG}
To motivate the approach taken in the proof of Theorem \ref{Theorem: Anderson representation of atoms}, we start by showing Proposition \ref{Proposition: Mass of atom in zero for line} and calculate the atom in zero for the case $\mathbb G = \mathbb Z$, assuming \eqref{Assumption: Detetministic attached graphs} and that $0$ is an admissible eigenvalue.
Considering Theorem \ref{Theorem: Anderson representation of atoms}, it is sufficient to show the statement for the case where $\mathrm G_*$ is a single leaf, i.e.\ to show that in this case
\begin{equation}
	\label{Equation: Atom in zero for Z and single leaves}
	\nu(\{0\}) = \frac{p(1-p)}{(1+p)(2-p)}.
\end{equation}
As we shall see in the Proof of Theorem \ref{Theorem: Anderson representation of atoms}, we defined admissible eigenvalues exactly such  that the proof of \eqref{Equation: Atom in zero for Z and single leaves} for the case of single leaves generalizes to the full proof of Proposition \ref{Proposition: Mass of atom in zero for line}.
In the following, let $\mathbb G_N$ be the line segment of length $N$, i.e.\ the vertex set of $\mathbb G_N$ is given by $\mathbb V_N = [N]$ and its edge set is given by $\{\{i, i+1\}: \, 1 \leq i \leq N-1\}$. We assume that \eqref{Assumption: Detetministic attached graphs} holds, with $\mathrm G_*$ being a single leaf.
\begin{lemma}
	We have for all $N\in \N$
	\begin{equation*}
		\operatorname{dim}(\operatorname{ker}(B_N)) = \begin{cases}
			0  \text{ if $N$ is even} \\
			1 \text{ if $N$ is odd}.
		\end{cases}
	\end{equation*}
\end{lemma}
\begin{proof}
	One easily convinces oneself that $\operatorname{rank}(B_N) \geq N-1$ and hence $\operatorname{dim}(\operatorname{ker}(B_N)) \leq 1$. As $B_N$ is the adjacency matrix of a tree, its spectrum is symmetric and the statement follows.
\end{proof}
We now use an observation that was used in \cite{BLS11} in order to study the atom in zero for Galton--Watson trees: the dimension of the kernel of the adjacency matrix stays invariant under removing a leaf and its unique neighbour. We apply this and remove all $x\in \mathbb V^+$ as well as all previously attached leaves. We are left with the site percolation graph $\mathbb G_{N, \operatorname{perc}}$ which consists of a collection of line segments. The kernel of each of these line segments has either dimension zero (in case the length of the segment is even) or dimension one (in case the length of the segment is odd) and we obtain the following result.
\begin{lemma}
	\label{Lemma: Lemma 0}
	We have
	\begin{equation*}
		\nu(\{0\}) = \lim_{N\to \infty} \mathbb E\bigg[\frac{X_N^-}{|\mathcal G_{N}|}\bigg],
	\end{equation*}
	where $X_N^-$ denotes for $N\in \N$ the number of connected components of $\mathbb G_{N, \operatorname{perc}}$ of odd size.
\end{lemma}

Hence, the task becomes to determine the expected number of connected components of odd size. We do this by deriving several recursions. For that purpose, we define
\begin{align*}
	X_N &\coloneqq \text{number of connected components of }\mathbb G_{N, \operatorname{perc}} \\
	X_N^+ &\coloneqq \text{number of connected components of }\mathbb G_{N, \operatorname{perc}} \text{ of even size}
\end{align*}
as well as
\begin{equation*}
	\sigma_N \coloneqq \begin{cases}
		1 &\text{ if } m_N = 0 \text { and the connected component of $N$ in $\mathbb G_{N, \operatorname{perc}}$ has even size} \\
		- 1 &\text{ if } m_N = 0 \text { and the connected component of $N$ in $\mathbb G_{N, \operatorname{perc}}$ has odd size} \\
		0 &\text{ if } m_N = 1.
	\end{cases}
\end{equation*}
We couple the graphs $(\mathbb G_{N, \operatorname{perc}})_N$ in a natural manner by defining them through restriction of the corresponding site percolation on $\mathbb N$.

\begin{lemma}
	\label{Lemma: Lemma 1}
	For all $N\in \mathbb N$
	\begin{align*}
		\mathbb E[X_N] &= 1 + (N-2)p - (N-1)p^2\\
		\mathbb P(\sigma_N = 1) &= \frac{1 - (p-1)^{N+1}}{2-p} - p.
	\end{align*}
\end{lemma}

\begin{proof}
	For all $N \in \N$, we have
	\begin{equation*}
		X_{N+1} = X_N + \1_{\{m_N = 1, m_{N+1} = 0\}}
	\end{equation*}
	and therefore
	\begin{equation*}
		\mathbb E[X_{N+1}] = \mathbb E[X_N] + p(1-p).
	\end{equation*}
	Since $\mathbb E[X_{1}] = 1-p$, the first identity follows. For the second identity, notice that
	\begin{align*}
		\mathbb P(\sigma_{N+1} = 1) &= \mathbb P(m_{N+1} = 0, \sigma_{N} = -1) \\
		&=(1-p)\big(1 - \mathbb P(\sigma_{N} = 0) - \mathbb P(\sigma_{N} = 1) \big) \\
		& = (1-p)^2 - (1-p)\mathbb P(\sigma_{N} = 1).
	\end{align*}
	Since $\mathbb P(\sigma_{1} = 1) = 0$, we get for all $N\in \N$
	\begin{equation*}
		\mathbb P(\sigma_{N} = 1) = \sum_{k=2}^N (-1)^k(1-p)^k = \frac{1 - (p-1)^{N+1}}{2-p} - p. \qedhere
	\end{equation*}
\end{proof}

\begin{lemma}
	\label{Lemma: Lemma 2}
	For all $N\in \mathbb N$
	\begin{align*}
		\mathbb E[X_N^-] = (N-1) \frac{p(1-p)}{2-p} + 1-p + \frac{2(1-p)}{2-p}\Big(p - \frac{1 - (p-1)^{N+1}}{2-p}\Big).
	\end{align*}
\end{lemma}

\begin{proof}
	We have for all $N \in \mathbb N$
	\begin{equation*}
		X_{N+1}^+ = \1_{\{m_{N} = 1\}}X_N^+ +  \1_{\{m_N = 0\}}\Big(X_N^+ - \1_{\{m_{N+1} = 0\}} \sigma_N \Big) = X_N^+ - \1_{\{m_{N+1} = 0\}} \sigma_N
	\end{equation*}
	(since $\sigma_N = 0$ if $m_N = 1$). Taking the expected value yields
	\begin{equation*}
		\mathbb E[X_{N+1}^+] = \mathbb E[X_{N}^+] - (1-p)\mathbb E[\sigma_N] = \mathbb E[X_{N}^+] - 2(1-p)\mathbb P(\sigma_N = 1) + (1-p)^2.
	\end{equation*}
	Since $\mathbb E[X_{1}^+] = 0$, we get inductively
	\begin{equation*}
		\mathbb E[X_{N}^+] = (N-1)(1-p)^2 - 2(1-p) \sum_{k=1}^{N-1}\mathbb P(\sigma_k = 1)
	\end{equation*}
	for all $N \in \N$. Lemma \ref{Lemma: Lemma 1} and a straight forward calculation yield
	\begin{equation*}
		\mathbb E[X_{N}^+] = (N-1)\Big(\frac{p}{2-p} - p^2\Big) -  \frac{2(1-p)}{2-p}\Big(p - \frac{1 - (p-1)^{N+1}}{2-p}\Big).
	\end{equation*}
	Again with Lemma \ref{Lemma: Lemma 1}, we get
	\begin{equation*}
		\mathbb E[X_N^-] = \mathbb E[X_N] - \mathbb E[X_N^+] = 1 + (N-2)p - (N-1)p^2  - \mathbb E[X_N^+]
	\end{equation*}
	which yields the result. 
\end{proof}
\begin{proof}[Proof of \eqref{Equation: Atom in zero for Z and single leaves}]
	By the law of large numbers, the dominated convergence theorem and Lemmas \ref{Lemma: Lemma 0} and \ref{Lemma: Lemma 2} we have
	\begin{equation*}
		\nu_{N}(\{0\}) = \lim_{N\to \infty} \mathbb E\Big[\frac{X_N^-}{N + \sum_{i=1}^N m_i}\Big] = \lim_{N\to \infty} \frac{1}{1+p}  \mathbb E\big[X_N^-/N\big] = \frac{p(1-p)}{(1+p)(2-p)}. \qedhere
	\end{equation*}
\end{proof}

\subsection{Notation and preliminary remarks}
\label{Subsection: starting remarks Anderson-perolation representation}
In order to show Theorems \ref{Theorem: Anderson representation of atoms} and \ref{Theorem: Anderson representation of the spectrum}, we write our adjacency operator $A$ in the block form \eqref{Equation: blockform adjacency matrix}. We write vectors in $l^2(\mathcal V)$ in the form $(\phi, \psi)$ where $\phi\in l^2(\mathbb V)$ and $\psi \in \bigoplus_{x \in \mathbb V^+} l^2(\mathrm V^x)$. We additionally denote for $x\in \mathbb V^+$  by $\psi^x \in l^2(\mathrm V^x)$ the vector defined by $\psi^x_{t} \coloneqq \psi_{(xt)}$.
We then have
\begin{equation*}
	\forall x\in \mathbb V:\, (E \psi)_{x} = \1_{\{x\in \mathbb V^+\}}\langle \varphi^x, \psi^x \rangle, \quad \quad \forall x\in \mathbb V^+:\, (E^* \phi)^x = \phi_x \varphi^x.
\end{equation*}
In order to motivate the approach taken in the proofs of Theorem \ref{Theorem: Anderson representation of atoms} and Theorem \ref{Theorem: Anderson representation of the spectrum}, let us rewrite the eigenvalue equation $A(\phi, \psi) = \lambda (\phi, \psi)$ as 
\begin{align}
	\label{Equation: eigenvalue equation in block form}
	&(B-\lambda)\phi + E\psi = 0      &\iff &\forall x\in \mathbb V:\,  [(A-\lambda) \phi]_x + \1_{\{x\in \mathbb V^+\}}\langle \varphi^x, \psi^x \rangle = 0 \nonumber \\ 
	&(T-\lambda)\psi + E^*\phi = 0	   &     &\forall x\in \mathbb V^+:\, (\mathrm  T^x-\lambda) \psi^x + \phi_x\varphi^x = 0.									  
\end{align}
Now, if $\lambda \notin \mathrm D$, then 
\begin{equation*}
	(\mathrm  T^x-\lambda) \psi^x + \phi_x\varphi^x = 0 \implies \langle \varphi^x, \psi^x \rangle = \phi_x \big\langle \varphi^x, (\lambda-\mathrm  T^x)^{-1} \varphi^x \big\rangle = -\phi_x r_x(\lambda)
\end{equation*}
and hence \eqref{Equation: eigenvalue equation in block form} implies that $(H_\lambda - \lambda)\phi = 0$. For the proof of Theorem \ref{Theorem: Anderson representation of the spectrum} we now need to modify and extend this argument to take into account that $\lambda \in \sigma(A)$ might not be an eigenvalue of $A$ i.e.\ might lie in the continuous spectrum of $A$ and might lie in $\mathrm D$ i.e.\ might be an eigenvalue of one of the attached graphs. We start by proving three simple lemmas which will be central for our further analysis. 

\begin{lemma}
	\label{Lemma: r in terms of pseudoinverse}
	Let $x\in \mathbb V_\lambda$ and let $(\mathrm  T^x - \lambda)^+$ denote the Moore--Penrose pseudoinverse of $\mathrm  T^x - \lambda$. Then
	\begin{equation*}
		r_x(\lambda) = \langle \varphi^x, (\mathrm  T^x - \lambda)^+ \varphi^x \rangle.
	\end{equation*}
\end{lemma}

\begin{proof}
	Let $x\in \mathbb V_\lambda$ and let $\mu$ denote the spectral measure of $T^{x}$ with respect to $\varphi^x$ such that
	\begin{equation*}
		r_x(\lambda) = \int_{\R} \frac{1}{t-\lambda} \, \mu(\mathrm dt).
	\end{equation*}
	Using the representation
	\begin{equation*}
		(T^{x} - \lambda)^+ = \lim_{\varepsilon \downarrow 0} \big( (\mathrm  T^x-\lambda)^2 + \varepsilon\big)^{-1}(\mathrm  T^x-\lambda)
	\end{equation*}
	we obtain
	\begin{align*}
		\langle \varphi^x, (\mathrm  T^x - \lambda)^+ \varphi^x \rangle &= \lim_{\varepsilon \downarrow 0} \big\langle \varphi^x,\big( ( T^{x}-\lambda)^2 + \varepsilon\big)^{-1}(T^{x} - \lambda) \varphi^x \big\rangle  =  \lim_{\varepsilon \downarrow 0} \int_{\R} \frac{t-\lambda}{(t-\lambda)^2+\varepsilon} \, \mu(\mathrm dt) = r_x(\lambda). \qedhere
	\end{align*}
\end{proof}

\begin{lemma}
	\label{Lemma: when is varphi in image}
	For $x \in \mathbb V^+$ we have
	\begin{equation*}
		x\in \mathbb V_\lambda  \iff \varphi^x \in \operatorname{im}(\mathrm  T^x - \lambda).
	\end{equation*}
\end{lemma}
\begin{proof}
	For $\lambda \in \R$ and $x\in \mathbb V$ we have $x\notin \mathbb V_\lambda$ i.e.\ $r_x(\lambda) = \infty$ if any only if $\mathrm T^{x}$ has an eigenvector to the eigenvalue $\lambda$ which is non orthogonal to $\varphi^x$. Hence,
	\begin{equation*}
		x\notin \mathbb V_\lambda \iff \varphi^x \notin \operatorname{ker}(\mathrm  T^x - \lambda)^\perp = \operatorname{im}(\mathrm  T^x - \lambda). \qedhere
	\end{equation*}
\end{proof}

\begin{lemma}
	\label{Lemma: get Anderson potential from solution to eigenvalue equation}
	For all $x\in \mathbb V^+ \cap \mathbb V_\lambda$ we have
	\begin{equation*}
		(\mathrm  T^x-\lambda) \psi^x + \phi_x\varphi^x = 0 \implies \langle \varphi^x, \psi^x \rangle  = -\phi_x r_x(\lambda).
	\end{equation*}
\end{lemma}

\begin{proof}
	If $x\in \mathbb V^+ \cap \mathbb V_\lambda$, then $\varphi^x \in \operatorname{im}(\mathrm  T^x - \lambda)$ and hence $\hat \psi^x \coloneqq \phi_x(\lambda-\mathrm  T^x)^+ \varphi^x$ is the minimal norm solution to
	\begin{equation*}
		(\mathrm  T^x-\lambda) \hat \psi^x + \phi_x\varphi^x = 0.
	\end{equation*}
	Hence,
	\begin{equation*}
		\langle \varphi^x, \hat \psi^x \rangle = \phi_x \langle \varphi^x, (\lambda-\mathrm  T^x)^+ \varphi^x \rangle = -\phi_x r_x(\lambda).
	\end{equation*}
	If $\psi^x$ is a general solution to $(\mathrm  T^x-\lambda) \psi^x + \phi_x\varphi^x = 0$ then $\psi^x - \hat \psi^x \in \operatorname{ker}(\mathrm  T^x - \lambda)$ and hence
	\begin{equation*}
		\langle \varphi^x, \psi^x \rangle = \langle \varphi^x, \hat \psi^x\rangle + \langle \varphi^x, \psi^x - \hat \psi^x\rangle = \langle \varphi^x, \hat \psi^x\rangle = -\phi_x r_x(\lambda)
	\end{equation*}
	since $\varphi^x \in  \operatorname{ker}(\mathrm  T^x-\lambda)^\perp$.
\end{proof}

\subsection{Anderson-percolation representation in finite volume}
\label{Subsection: Anderson-percolation representation in finite volume}

In the following, we show a finite-volume version of Theorem \ref{Theorem: Anderson representation of atoms} and Theorem \ref{Theorem: Anderson representation of the spectrum}.
Let $\lambda \in \mathbb R$ and $\mathbb V_{\lambda, N} \coloneqq \{x\in \mathbb V_N:\, r_x(\lambda) \neq \infty\}$. We define our Anderson-percolation model in finite volume by
\begin{equation}
	H_{\lambda, N} \coloneqq B_{\lambda, N} + V_{\lambda, N}, \quad V_{\lambda, N} \coloneqq -\sum_{x\in \mathbb V_{\lambda, N}} r_x(\lambda) \ket{\delta_x}\bra{\delta_x}
\end{equation}
where $B_{\lambda, N}$ is the adjacency matrix of $\mathbb V_{\lambda, N}$. As before, we denote by
\begin{equation*}
	\gamma_{\operatorname{TF}}(A_N, \lambda)\coloneqq \dim \big\{\psi\in \mathbb C^{\mathcal V_N}:\, A_N \psi = \lambda \psi\, \text{ and } \psi_x = 0 \text{ for all } x\in \mathbb V_N\big\}
\end{equation*}
the tuning fork contribution to the multiplicity of an eigenvalue $\lambda$ of $\mathcal G_N$.

Our approach will rely mainly on the following observation: As in \eqref{Equation: eigenvalue equation in block form}, we can rewrite the eigenvalue equation $A_N(\phi, \psi) = \lambda(\phi, \psi)$ as
\begin{align*}
	&\forall x\in \mathbb V_N:\, [(A_N-\lambda) \phi]_x + \1_{\{x\in \mathbb V_N^+\}}\langle \varphi^x, \psi^x \rangle = 0 \nonumber \\ 
	&\forall x\in \mathbb V_N^+:\, (\mathrm  T^x-\lambda) \psi^x + \phi_x\varphi^x = 0.									  
\end{align*}
If $x\in \mathbb V_N \setminus \mathbb V_{\lambda, N}$, i.e.\ if $r_x(\lambda) = \infty$, then Lemma \ref{Lemma: when is varphi in image} implies that $\phi_x = 0$. On the other hand, if $x\in \mathbb V_N^+\setminus \mathbb V_{\lambda, N}$ then Lemma \ref{Lemma: get Anderson potential from solution to eigenvalue equation} yields $\langle \varphi^x, \psi^x \rangle = -\phi_x r_x(\lambda) = (V_{\lambda, N}\phi)_x$.

\begin{figure}
	\centering
	\includegraphics[width=0.6\linewidth]{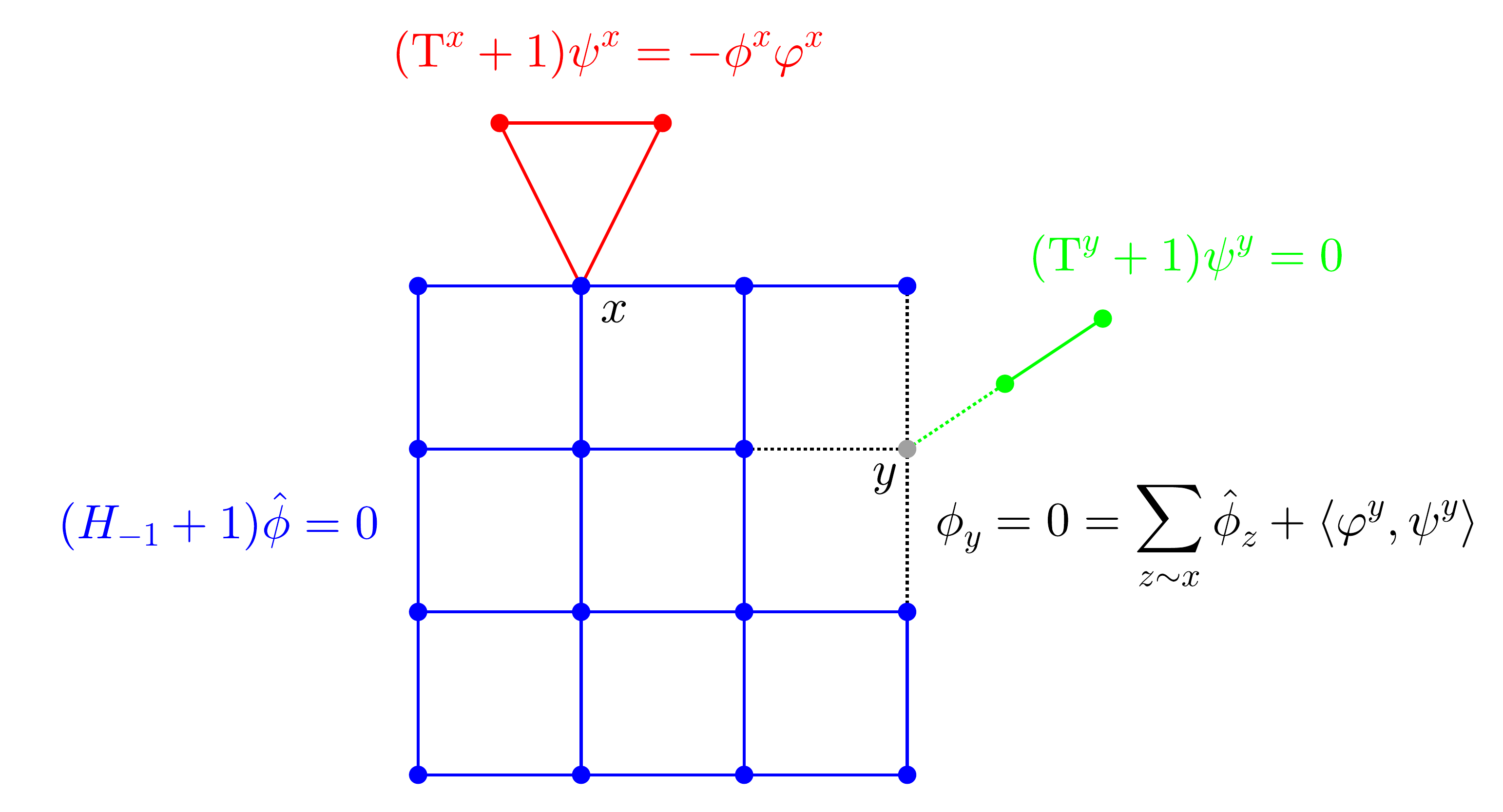}
	\caption{Illustration of the proof of Proposition \ref{Proposition: Anderson representation in finite volume} for $\lambda = -1$. Vertices in $\mathbb V_{-1}$ are in blue. Notice that $x\in \mathbb V_{-1}$ since $\varphi^x = (1,1)^T$ is orthogonal to the eigenvector $(-1, 1)^T$ of $\mathrm T^x = \mathrm T^y$ to the eigenvalue $-1$, while $y \notin \mathbb V_{-1}$ since $\varphi^y = (1,0)^T$ is non-orthogonal to $(-1, 1)^T$. By Lemma \ref{Lemma: when is varphi in image}, the equation $(\mathrm T^z + 1)\psi^z = -\phi^z \varphi^z$ has for $\phi^z \neq 0$ a solution if $z=x$ and no solution if $z=y$.
	}
	\label{fig:andersonpercolationeigenvectors}
\end{figure}

\begin{prop}
	\label{Proposition: Anderson representation in finite volume}
	We have for all $\lambda \in \R$
	\begin{equation*}
		\gamma(A_N, \lambda) = \gamma(H_{\lambda, N}, \lambda) + \gamma_{\operatorname{TF}}(A_N, \lambda)  
	\end{equation*}
	as well as
	\begin{equation*}
		\gamma_{\operatorname{TF}}(A_N, \lambda) = \sum_{x\in \mathbb V_N}\gamma(\mathrm T^x, \lambda) - \1_{\{\lambda \in \mathrm S^x_* \}}.
	\end{equation*}
\end{prop}

\begin{proof}
	We omit the $N$-dependency in the notation. Let $\hat \phi$ denote the restriction of $\phi$ to $\mathbb V_\lambda$.
	With Lemmas \ref{Lemma: when is varphi in image} and \ref{Lemma: get Anderson potential from solution to eigenvalue equation}, we can rewrite the eigenvalue equation $(A-\lambda)(\psi, \psi) = 0$ as
	\begin{align*}
		\begin{cases}
			\forall x\in \mathbb V:\,  &[(A-\lambda) \phi]_x + \langle \varphi^x, \psi^x \rangle = 0 \nonumber \\ 
			\forall x\in \mathbb V^+:\, &(\mathrm  T^x-\lambda) \psi^x + \phi_x\varphi^x = 0	
		\end{cases}
		\iff
		\begin{cases}
			&(H_\lambda - \lambda) \hat \phi = 0 \nonumber \\
			&\forall x\in \mathbb V \setminus \mathbb V_\lambda:\, \phi_x = 0  = \sum_{y\sim x}\phi_y + \langle \varphi^x, \psi^x \rangle \\
			&\forall x\in \mathbb V_\lambda\cap \mathbb V^+:\, (\mathrm  T^x-\lambda) \psi^x + \phi_x \varphi^x = 0 \\
			&\forall x\in \mathbb V\setminus \mathbb V_\lambda:\, (\mathrm  T^x-\lambda) \psi^x = 0;
		\end{cases}
	\end{align*}
	see Figure \ref{fig:andersonpercolationeigenvectors} for an illustration. 
	Let us denote by $T^{(1)}$ and $T^{(2)}$ the restriction of $T$ to $\mathbb V^+ \cap \mathbb V_\lambda$ and $\mathbb V\setminus \mathbb V_\lambda = \mathbb V^+\setminus \mathbb V_\lambda$ respectively. Then we can rewrite the above as	
	\begin{align*}
		\forall x\in \mathbb V\setminus \mathbb V_\lambda:\, \phi_x = 0, \quad
		\begin{pmatrix}
			H_\lambda - \lambda & 0 & 0\\
			C^{(0)} & T^{(1)} - \lambda & 0 & \\
			0 & 0 & T^{(2)} - \lambda \\
			C^{(1)}  & 0  &  C^{(2)}
		\end{pmatrix}
		\begin{pmatrix}
			\hat \phi \\
			\psi_1 \\
			\psi_2
		\end{pmatrix}
		=
		\begin{pmatrix}
			0 \\
			0 \\
			0
		\end{pmatrix},
	\end{align*}
	where
	\begin{align*}
		&\forall x \in \mathbb V_\lambda\cap \mathbb V_+,\, \forall t\in \mathrm V^x,\, \forall  y\in \mathbb V_\lambda:\ & C^{(0)}_{(xt)y} &\coloneqq \delta_{xy} \varphi^x_{t}\\
		&\forall x\in \mathbb V\setminus \mathbb V_\lambda, \, \forall y\in \mathbb V_\lambda:\, & C^{(1)}_{x y} &\coloneqq \delta_{x\sim y}\\
		&\forall x \in \mathbb V\setminus \mathbb V_\lambda,\, 	\forall y \in \mathbb V\setminus \mathbb V_\lambda, \, \forall t\in \mathrm V^y:\, & C^{(2)}_{x(yt)} &\coloneqq \delta_{xy} \varphi^x_{t}.
	\end{align*}
	In other words, if we enumerate the vertices of $\mathbb V_\lambda$ by $x_1, \hdots, x_n$ and the vertices of $\mathbb V \setminus \mathbb V_\lambda$ by $y_1, \hdots, y_m$, we have
	\setcounter{MaxMatrixCols}{20}
	\begin{equation*}
		\begin{pmatrix}
			C^{(0)} & T^{(1)} - \lambda
		\end{pmatrix}\,
		=
		\,
		\begin{pmatrix}
			&\varphi^{x_1} &\vline & &\vline & &\vline& \mathrm T^{x_1}-\lambda & \vline   &         &  \vline &         & \\ \hline
			&&\vline & \varphi^{x_2} &\vline & &\vline &  & \vline   &    \mathrm T^{x_2}-\lambda      &  \vline &         & \\ \hline
			& &\vline &  &\vline & \ddots &\vline &   &\vline   &          &  \vline &   \ddots      &
		\end{pmatrix}							
	\end{equation*}
	as well as
	\begin{equation*}
		\begin{pmatrix}
			T^{(2)} - \lambda \\ C^{(2)}
		\end{pmatrix}^*\,
		=
		\,
		\begin{pmatrix}
			&\mathrm T^{y_1}-\lambda & \vline &                  & \vline &         &  \vline&\varphi^{y_1} &\vline & &\vline & &    \\ \hline
			& & \vline & \mathrm T^{y_1}-\lambda                 & \vline &         &  \vline & &\vline &  \varphi^{y_2} &\vline &  &   \\ \hline
			& & \vline &                & \vline &    \ddots     &  \vline            & &\vline & &\vline & \ddots &
		\end{pmatrix}.							
	\end{equation*}
	Now, Lemma \ref{Lemma: when is varphi in image} implies that the columns of $C^{(0)}$ are contained in the column space of $T^{(1)}-\lambda$ while the rows of $C^{(2)}$ are not contained in the row space of $T^{(2)}-\lambda$ (which coincides by symmetry with the column space of $T^{(2)}-\lambda$). We hence obtain
	\begin{align*}
		\operatorname{rank}
		\begin{pmatrix}
			H_\lambda - \lambda & 0 & 0\\
			C^{(0)} & T^{(1)} - \lambda & 0 & \\
			0 & 0 & T^{(2)} - \lambda \\
			C^{(1)}  & 0  &  C^{(2)}
		\end{pmatrix}
		&= 
		\operatorname{rank}
		\begin{pmatrix}
			H_\lambda - \lambda & 0 & 0\\
			0 & T^{(1)} - \lambda & 0 & \\
			0 & 0 & T^{(2)} - \lambda \\
			C^{(1)}  & 0  &  C^{(2)}
		\end{pmatrix} \\
		&= 
		\operatorname{rank}
		\begin{pmatrix}
			H_\lambda - \lambda & 0 & 0\\
			0 & T^{(1)} - \lambda & 0 & \\
			0 & 0 & T^{(2)} - \lambda 
		\end{pmatrix}
		+  |\mathbb V \setminus \mathbb V_{\lambda}|
	\end{align*}
	implying by the dimension formula
	\begin{align*}
		\gamma(A, \lambda) &= \operatorname{dim}\operatorname{ker} \begin{pmatrix}
			H_\lambda - \lambda & 0 & 0\\
			0 & T^{(1)} - \lambda & 0 & \\
			0 & 0 & T^{(2)} - \lambda 
		\end{pmatrix} - |\mathbb V \setminus \mathbb V_{\lambda}|\\
		&= \gamma(H_\lambda, \lambda) + \sum_{x\in \mathbb V_N} \gamma(\mathrm T^{x}, \lambda) -  \1_{\{\lambda \in \mathrm S^x_*\}}.
	\end{align*}			
	On the other hand, we can apply a similar reasoning to determine $\gamma_{\operatorname{TF}}(A, \lambda)$. The vector $(\phi, \psi)$ is a tuning fork to the eigenvalue $\lambda$ i.e.\ satisfies $(A-\lambda)(\phi, \psi) = 0$ and $\phi=0$ if and only if 
	\begin{equation*}
		(T-\lambda)\psi = 0,\, \quad \forall x \in \mathbb V: \, \phi_x = 0 =\1_{\{x\in \mathbb V^+\}} \langle \varphi^x, \psi^x \rangle.
	\end{equation*}
	Hence, if we define
	\begin{equation*}
		\forall x\in \mathbb V^+,\, \forall y\in \mathbb V^+\cap \mathbb V_\lambda, \, \forall t\in \mathrm V^y:\, \hat C^{(1)}_{x(yt)} \coloneqq \delta_{xy} \varphi^x_t, \quad \forall x\in \mathbb V^+,\, \forall y\in \mathbb V \setminus \mathbb V_\lambda,\, \forall t\in \mathrm V^y: \, \hat C^{(2)}_{x(yt)} \coloneqq \delta_{xy} \varphi^x_t
	\end{equation*}
	we have
	\begin{equation*}
		\gamma_{\operatorname{TF}}(A,\lambda) = 
		\operatorname{dim} \operatorname{ker}
		\begin{pmatrix}
			T^{(1)}-\lambda & 0\\
			0 & T^{(2)}-\lambda\\
			\hat C^{(1)} & \hat C^{(2)}
		\end{pmatrix} = \sum_{x\in \mathbb V_N}  \gamma(\mathrm T^{x}, \lambda) - \1_{\{\lambda \in \mathrm S^x_*\}}. \qedhere
	\end{equation*}
\end{proof}

\begin{figure}
	\centering
	\includegraphics[width=0.6\linewidth]{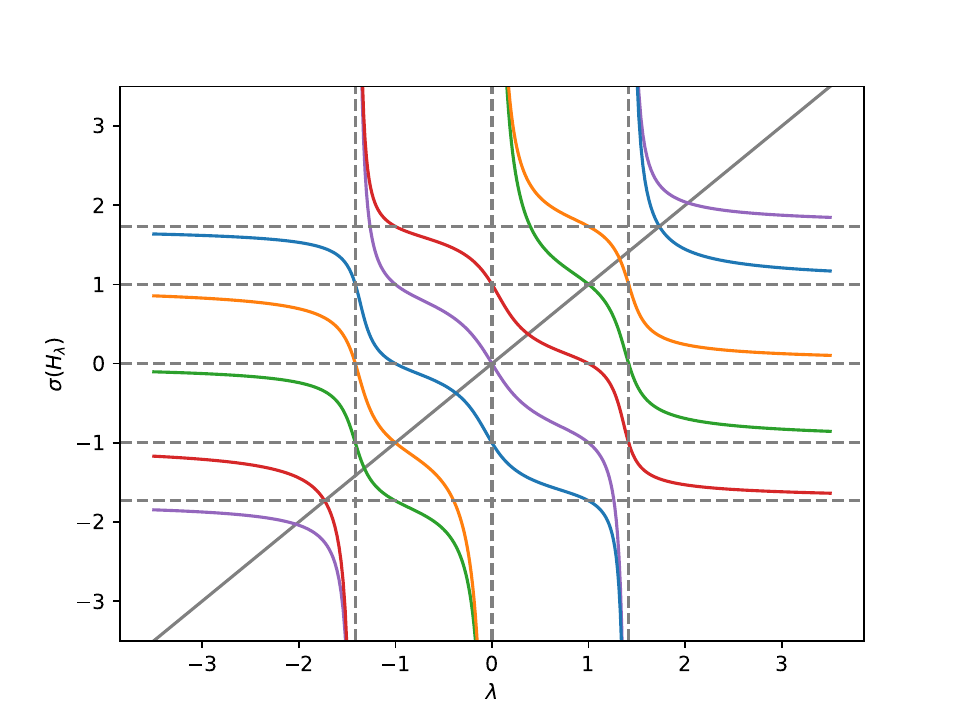}
	\caption{Illustration of Proposition \ref{Proposition: Anderson representation in finite volume} and Proposition \ref{Proposition: Parametrize eigenvalues of Anderson Hamiltonian in finite volume}. The eigenvalues of the core graph $\mathbb G_N$ are represented by the dashed horizontal lines, the eigenvalues of the attached graphs which are contained in $\mathrm S_*$ are represented by the dashed vertical lines. The eigenvalues of $H_{\lambda, N}$ as a function of $\lambda$ (on the horizontal axis) are represented coloured. Their points of intersection with the diagonal give the eigenvalues of the sprinkled graph $\mathrm G$.}
	\label{fig:andersonrepresentation1}
\end{figure}

\begin{prop}
	\label{Proposition: Parametrize eigenvalues of Anderson Hamiltonian in finite volume}
	Let $n \coloneqq |\mathbb V_N|$. There exist continuous functions $\alpha_1, \hdots, \alpha_n: \R \to \hat{\mathbb R}$ to the one point compactification $\hat{\mathbb R} = \mathbb \R\cup \{\infty\}\cong \mathbb S^1$ of $\mathbb R$
	such that $\alpha_i$ is for all $i \in [n]$ decreasing on every connected component of $\{\lambda \in \R: \, \alpha_i(\lambda) \neq \infty\}$ and such that for all $\alpha \in \R$
	\begin{equation}
		\label{Equation: continuous parametrization of spectrum of Anderson Hamiltonian in finite volume}
		\gamma(H_{\lambda, N},  \alpha) = \# \big\{i\in [n]:\, \alpha_i(\lambda) = \alpha\big\}.
	\end{equation}	
\end{prop}

\begin{proof}
	We once more omit the $N$-dependency in the notation.
	Let $\lambda_1< \hdots< \lambda_k$ be the eigenvalues of the attached graphs in increasing order, i.e.\
	\begin{equation*}
		\{\lambda_1, \hdots, \lambda_k\} = \bigcup_{x\in \mathbb V^+} \mathrm S^x.
	\end{equation*}
	Set $\lambda_0 \coloneqq -\infty$ and  $\lambda_{k+1} \coloneqq \infty$ as well as $I_i \coloneqq (\lambda_i, \lambda_{i+1})$ for all $i\in \{0, \hdots, k\}$.
	Notice that $\lambda \mapsto -r_x(\lambda)$ is on any $I_i$ a decreasing continuous function.
	As a consequence of Weyl's inequality, we can parametrize the eigenvalues of $H_{\lambda}$ on each interval $I_i$ as continuous, decreasing functions in $\lambda$. Let $\alpha_1, \hdots, \alpha_{n}: I_0 \to \R$, $\tilde \alpha_1, \hdots, \tilde \alpha_{n}: I_1 \to \R$ be continuous, decreasing functions with $\alpha_1 \leq \hdots \leq \alpha_n$, $\tilde \alpha_1 \leq \hdots \leq \tilde\alpha_n$ such that
	\begin{equation}
		\label{Equation: continuous parametrization of eigenvalues of Anderson Hamiltonian in finite volume}
		\forall \alpha \in I_0:\, \gamma(H_{\lambda},  \alpha) = \# \{i:\, \alpha_i(\lambda) = \alpha\}, \quad \forall \alpha \in I_1:\, \gamma(H_{\lambda},  \alpha) = \# \{i:\, \tilde \alpha_i(\lambda) = \alpha\}.
	\end{equation}	
	Let $n_1 \coloneqq |\mathbb V_{\lambda_1}|$ be the number of vertices $x\in \mathbb V_N$ for which $r_x(\lambda_1) \neq \infty$. We define for all $i \in [n]$
	\begin{equation*}
		\alpha_i(\lambda_1) \coloneqq \lim_{\lambda \uparrow \lambda_1} \alpha_i(\lambda), \quad  \tilde \alpha_i(\lambda_1) \coloneqq \lim_{\lambda \downarrow \lambda_1} \tilde \alpha_i(\lambda), 
	\end{equation*}
	where the limit is taken in $\hat{\mathbb R}$. We show that for all $i\in [n]$
	\begin{equation}
		\label{Equation: limit of Anderson eigenvalues as we approach tree eigenvalues 1}
		\alpha_i(\lambda_1) \neq \infty \iff i> n - n_1, \quad \tilde \alpha_i(\lambda_1) \neq \infty \iff i\leq n_1, \quad
	\end{equation}
	as well as
	\begin{equation}
		\label{Equation: limit of Anderson eigenvalues as we approach tree eigenvalues 2}
		\gamma(H_{\lambda_1},  \alpha) = \# \{i:\, \alpha_i(\lambda_1) = \alpha\} =  \# \{i:\, \tilde \alpha_i(\lambda_1) = \alpha\}
	\end{equation}
	for all $\alpha \in \R$. Hence, if we extend $\alpha_1, \hdots, \alpha_n$ to $I_1 \cup \{\lambda_1\} \cup I_2$ by
	\begin{equation*}
		\alpha_i(\lambda) \coloneqq \tilde \alpha_j(\lambda) \text{ with }j=i-(n-n_1) \text{ mod } n
	\end{equation*}
	for $\lambda \in I_2$, then $\alpha_1, \hdots, \alpha_n$ are decreasing on $I_1 \cup I_2$ and continuous on $I_1 \cup \{\lambda_1\} \cup I_2$ with \eqref{Equation: continuous parametrization of spectrum of Anderson Hamiltonian in finite volume} holding for all $\alpha \in  I_1 \cup \{\lambda_1\} \cup I_2$. Inductively applying this argument then yields the claim. In order to show \eqref{Equation: limit of Anderson eigenvalues as we approach tree eigenvalues 1} and \eqref{Equation: limit of Anderson eigenvalues as we approach tree eigenvalues 2}, let $\lambda \in I_1 \cup I_2$. We decompose $\mathbb V = \mathbb V_{\lambda_1}\times (\mathbb V \setminus \mathbb V_{\lambda_1})$ and write $H_{\lambda}$ in the block form
	\begin{equation*}
		H_{\lambda} = \begin{pmatrix}
			B^{(1)} + V^{(1)}_\lambda & B^{(12)} \\
			B^{(21)} & B^{(2)} + V^{(2)}_\lambda
		\end{pmatrix}.
	\end{equation*}	
	Notice that every diagonal element of the diagonal matrix $V^{(2)}_\lambda$ converges to $\infty \in \hat{\mathbb R}$ as $\lambda \to \lambda_1$. By the Schur complement formula, $(H_\lambda - z)^{-1}$ is for any $z\in \mathbb C^+$ of the form
	\begin{equation*}
		\big[(H_\lambda-z)^{-1}\big]_{xx}
		=
		\begin{cases}
			\Big[\Big(B^{(1)} + V^{(1)}_\lambda - z - B^{(12)}\big(B^{(2)} + V^{(2)}_\lambda  - z\big)^{-1} B^{(21)}\Big)^{-1} \Big]_{xx} &\text{ for }x\in \mathbb V_{\lambda_1} \\
			\Big[\Big(B^{(2)} + V^{(2)}_\lambda - z - B^{(21)}\big(B^{(1)} + V^{(1)}_\lambda  - z\big)^{-1} B^{(12)}\Big)^{-1} \Big]_{xx} &\text{ for }x\in \mathbb V \setminus \mathbb V_{\lambda_1}
		\end{cases}
	\end{equation*}	
	Now,
	\begin{equation*}
		(B^{(2)} + V^{(2)}_\lambda - z)^{-1} = \big(V^{(2)}_\lambda\big)^{-1} \Big((B^{(2)} - z) \big(V^{(2)}_\lambda\big)^{-1} + \1\Big)^{-1} \to 0
	\end{equation*}
	as $\lambda \to \lambda_1$, implying that
	\begin{equation*}
		\big[(H_\lambda-z)^{-1}\big]_{xx} \to \big[\big(B^{(1)} + V^{(1)}_{\lambda_0} - z\big)^{-1}\big]_{xx} = 	\big[(H_{\lambda_1}-z)^{-1}\big]_{xx} \quad \text{ for all }x\in \mathbb V_{\lambda_1} 
	\end{equation*}		
	as $\lambda \to  \lambda_1$. Similarly, one can see that
	\begin{equation*}
		\big[(H_\lambda-z)^{-1}\big]_{xx} \to 0 \text{ for all }x\in \mathbb V\setminus \mathbb V_{\lambda_1}
	\end{equation*}	
	as $\lambda \to  \lambda_1$. Hence, if
	\begin{equation*}
		\mu_\lambda \coloneqq \sum_{\alpha \in \sigma(H_\lambda)} \gamma(\alpha, H_\lambda) \cdot \delta_\alpha
	\end{equation*}
	denotes the unnormed empirical eigenvalue distribution of $H_\lambda$ we have for every $z\in \mathbb C^+$
	\begin{equation*}
		\int_{\R} \frac{1}{t-z} \mu_\lambda(\mathrm dt) = \operatorname{tr}(H_\lambda - z)^{-1} \to \operatorname{tr}(H_{\lambda_1} - z)^{-1} = \int_{\R} \frac{1}{t-z} \mu_{\lambda_1}(\mathrm dt)
	\end{equation*} 
	i.e.\ $\mu_{\lambda} \to \mu_{\lambda_1}$ vaguely as $\lambda \to \lambda_1$. Hence, $n_1$ eigenvalues of $H_\lambda$ converge to the eigenvalues of $H_{\lambda_1}$ as $\lambda \to \lambda_1$ while $n-n_1$ eigenvalues diverge to $\infty\in \hat{\R}$. Hence, \eqref{Equation: limit of Anderson eigenvalues as we approach tree eigenvalues 1} and \eqref{Equation: limit of Anderson eigenvalues as we approach tree eigenvalues 2} hold.
\end{proof}

\begin{figure}
	\centering
	\includegraphics[width=0.6\linewidth]{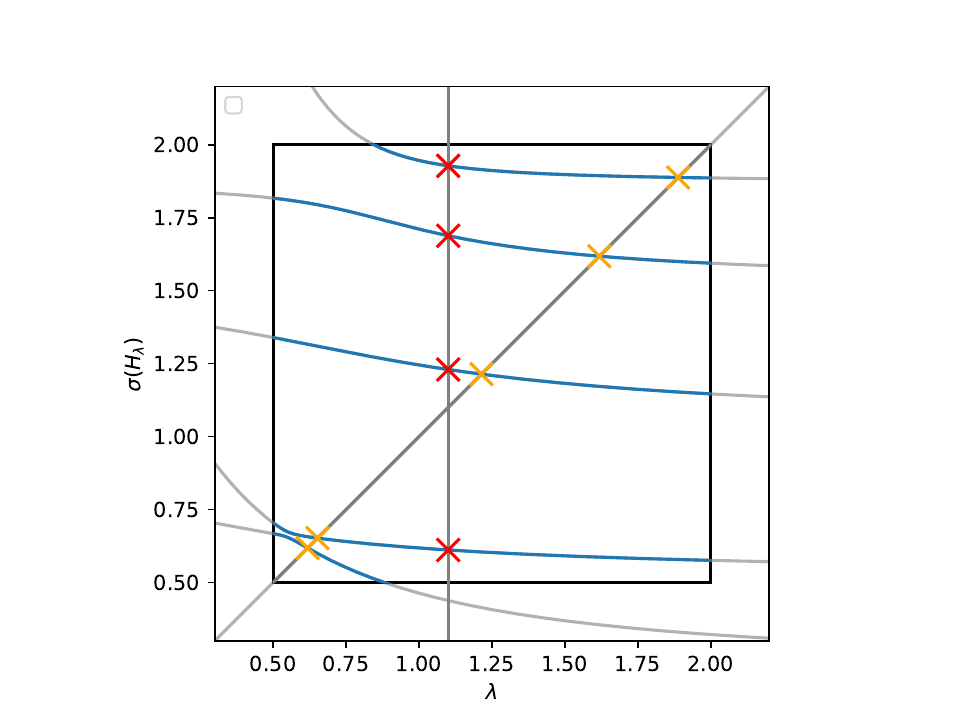}
	\caption{Illustration of the proof of Proposition \ref{Proposition: Estimate mass in terms of Anderson Hamiltonian in finite volume}. Every eigenvalue of $H_\lambda$ that lies in $I$ (depicted in red) leads to a solution $\alpha_i(\lambda_i) = \lambda_i$ that lies in $I$ (depicted in orange).}
	\label{fig:andersonrepresentation2}
\end{figure}

\begin{prop}
	\label{Proposition: Estimate mass in terms of Anderson Hamiltonian in finite volume} 
	Let $I$ be an open interval and $\lambda \in \overline{I}$. Then
	\begin{equation*}
		\sum_{\alpha \in I} \gamma(A_{N}, \alpha) - \gamma_{\operatorname{TF}}(A_{N}, \alpha)  \geq \sum_{\alpha \in I} \gamma(H_{\lambda, N}, \alpha).
	\end{equation*}
\end{prop}

\begin{proof}
	We omit the $N$-dependency in the notation.
	Let us assume that $\lambda \in I = (a, b)$, the case where $\lambda \in \{a, b\}$ can be conducted in the same manner. Let $\alpha_1, \hdots, \alpha_n$ be as in Proposition \ref{Proposition: Parametrize eigenvalues of Anderson Hamiltonian in finite volume}. We define $I, J \subseteq N$ by
	\begin{equation*}
		I \coloneqq \{i\in [n]:\, \alpha_i(\lambda) \in (a, \lambda)\}, \quad 	J \coloneqq \{j\in[n]:\, \alpha_j(\lambda) \in [\lambda, b)\}
	\end{equation*}
	such that in particular
	\begin{equation*}
		\sum_{\alpha \in I} \gamma(H_{\lambda}, \alpha) = |I| + |J|.
	\end{equation*}
	By the intermediate value theorem, there exist for all $i\in I$ some $\lambda_i \in [\alpha_i(\lambda), \lambda)$ such that $\alpha_i(\lambda_i) = \lambda_i$ and for all $j\in J$ some $\lambda_j \in [\lambda, \alpha_j(\lambda)]$ such that $\alpha_j(\lambda_j) = \lambda_j$, see Figure \ref{fig:andersonrepresentation2} for an illustration. Hence, with Proposition \ref{Proposition: Anderson representation in finite volume} we obtain
	\begin{equation*}
		\sum_{\alpha \in I}  \gamma(A, \alpha) - \gamma_{\operatorname{TF}}(A, \alpha)  = \sum_{\alpha \in I} \gamma(H_{\alpha}, \alpha) = \sum_{\alpha \in I} |\{i: \alpha_i(\alpha) = \alpha\}| \geq |I|+|J| =	\sum_{\alpha \in I} \gamma(H_{\lambda}, \alpha). \qedhere
	\end{equation*}
\end{proof}

\subsection{Anderson-percolation representation of the atoms: proof of Theorem \ref{Theorem: Anderson representation of atoms} and Proposition \ref{Proposition: Tuning fork contribution}}
\label{Subsection: Anderson Percolation representation of atoms}
In order to prove Theorem \ref{Theorem: Anderson representation of atoms}, we first take the expected value in Proposition \ref{Proposition: Anderson representation in finite volume} and then calculate its limit as $N\to\infty$. We denote by
\begin{equation*}
	\nu_{\operatorname{TF}, N} \coloneqq \mathbb E\bigg[ \frac{1}{|\mathcal V_N|}\sum_{\lambda \in \mathrm D} \gamma_{\operatorname{TF}}(A_N, \lambda) \delta_{\lambda} \bigg]
\end{equation*}
the tuning fork contribution to $\nu_{N}$, which has support in $\mathrm D$. In accordance with Subsection \ref{Subsection: results atoms and the Anderson-percolation representation}, we define 
\begin{equation*}
	\forall \lambda \in \mathrm D:\, \nu_{\operatorname{TF}}(\{\lambda\}) \coloneqq \lim_{N\to \infty} \nu_{\operatorname{TF}, N}(\{\lambda\}), \quad \nu_{\operatorname{TF}} \coloneqq \sum_{\lambda \in \mathrm D} \nu_{\operatorname{TF}}(\{\lambda\}) \delta_\lambda;
\end{equation*}
we shall see in the proof of Proposition \ref{Proposition: Tuning fork contribution} that these limits indeed exist, and in Lemma \ref{Lemma: Total variation convergence of tuning fork contribution} that we have $\nu_{\operatorname{TF}, N} \to \nu_{\operatorname{TF}}$ in total variation as $N\to \infty$.

We denote by $\mu_N$ and $\mu_{\lambda, N}$ the empirical eigenvalue distributions of $A_N$ and $H_{\lambda,N}$, i.e.\ 
\begin{equation*}
	\mu_N \coloneqq \frac{1}{|\mathcal V_N|}\sum_{\lambda \in \sigma(A_N)} \gamma(A_N, \lambda) \delta_\lambda, \quad 	\mu_{\lambda, N} \coloneqq \frac{1}{|\mathbb V_{\lambda, N}|}\sum_{\lambda \in \sigma(H_{\lambda})} \gamma(H_{\lambda}, \lambda) \delta_\lambda
\end{equation*}
(with $\mu_{\lambda, N} \coloneqq 0$ if $\mathbb V_{\lambda, N} = \emptyset$). Then Proposition \ref{Proposition: Anderson representation in finite volume} states
\begin{equation*}
	\nu_{N}(\{\lambda\}) = \nu_{\operatorname{TF}, N}(\{\lambda\}) + \mathbb E\Big[ \frac{|\mathbb V_{\lambda, N}|}{|\mathcal V_N|}\mu_{\lambda, N}(\{\lambda\})\Big]
\end{equation*}
for all $\lambda \in \R$. Since $\frac{|\mathbb V_{\lambda, N}|}{|\mathcal V_N|} \to \frac{\mathbb P(r(\lambda) \neq \infty)}{1+ \mathbb E[m]}$
almost surely as $N\to \infty$, the main task for the proof of Theorem \ref{Theorem: Anderson representation of atoms} is to show that for every $\lambda \in \R$ the density of states $\rho_{\lambda, N} = \mathbb E[\mu_{\lambda, N}]$ of $H_{\lambda, N}$ satisfies
\begin{equation}
	\label{Equation: convergence of atom of AH in lambda}
	\lim_{N\to \infty} \rho_{ \lambda, N}(\{\lambda\}) = \rho_{\lambda}(\{\lambda\}). 
\end{equation}
While the weak convergence $\rho_{\lambda, N} \to \rho_{\lambda}$ can be shown similarly to the weak convergence of the density of states under Benjamini-Schramm convergence, the same is not directly clear for the pointwise convergence of the atoms of the density of states (for which, in the graph case, one uses that the adjacency matrix has integer coefficients). That being said, we shall see that we can apply Proposition \ref{Proposition: Estimate mass in terms of Anderson Hamiltonian in finite volume} in order to conclude from the pointwise convergence of the atoms of $\nu_N$ that at least \eqref{Equation: convergence of atom of AH in lambda} holds, without making any statement about the convergence of $\rho_{\lambda, N}(\{\alpha\})$ for $\alpha \neq \lambda$. We start the proof of Theorem \ref{Theorem: Anderson representation of atoms} by showing Proposition \ref{Proposition: Tuning fork contribution} which characterizes the tuning fork contribution in the limit $N\to \infty$.

\begin{proof}[Proof of Proposition \ref{Proposition: Tuning fork contribution}]
	By Proposition \ref{Proposition: Anderson representation in finite volume}, the law of large numbers and the dominated convergence theorem, we have
	\begin{align*}
		\nu_{\operatorname{TF}}(\{\lambda\}) &= \lim_{N\to \infty} \mathbb E\Big[\frac{1}{\frac{1}{|\mathbb V_N|}\sum_{x\in \mathbb V_N}(1 + m_x)} \frac{1}{|\mathbb V_N|}\sum_{x\in \mathbb V_N} \gamma(\mathrm T^{x}, \lambda) - \1_{\{\lambda \in \mathrm S^x_*\} } \Big] \\
		&=\frac{1}{1 + \mathrm E[m]} \mathrm E\Big[\gamma(\mathrm T, \lambda ) - \1_{\{\lambda \in \mathrm S_*\}} \Big]. \qedhere
	\end{align*}
\end{proof}

\begin{lemma}
	\label{Lemma: Total variation convergence of tuning fork contribution}
	We have $\nu_{\operatorname{TF}, N}\to \nu_{\operatorname{TF}}$ in total variation as $N\to \infty$.
\end{lemma}
\begin{proof}
	We assume w.l.o.g.\ that $|\mathrm D| = \infty$ and enumerate $\mathrm D = \{\lambda_1, \lambda_2, \hdots\}$. We define for all $k\in \mathbb N$ and $x\in \mathbb V_N$
	\begin{equation*}
		Y_{x,k} \coloneqq \gamma(\mathrm T^{x}, \lambda_k) -\1_{\{r_x(\lambda_k) = \infty \} } 
	\end{equation*}
	such that $\gamma_{\operatorname{TF}, N}(\{\lambda_k\}) = \sum_{x\in \mathbb V_N} Y_{x, k}$. Notice that for all $x\in \mathbb V_N$
	\begin{equation}
		\label{Equation: upper bound on sum of Y_{xk}}
		\sum_{k=1}^\infty Y_{x,k} \leq m_{x}.
	\end{equation}
	For $M\in \mathbb N$, let $\mathrm D_M \coloneqq \{\lambda_M, \lambda_{M+1}, \hdots \}$. We obtain with \eqref{Equation: upper bound on sum of Y_{xk}}, the law of large numbers and the dominated convergence theorem
	\begin{align}
		\label{Equation: Estimate tuning fork tails}
		\nu_{\operatorname{TF}, N}(\mathrm D_M) &= \mathbb E\Big[\sum_{k\geq M} \frac{1}{\sum_{x\in \mathbb V_N}(1+m_x)} \sum_{x\in \mathbb V_N} Y_{x, k} \Big] \nonumber \\
		&= \mathbb E\Big[\frac{1}{|\mathbb V_N|^{-1}\sum_{x\in \mathbb V_N}(1+m_x)} \frac{1}{|\mathbb V_N|}\sum_{x\in \mathbb V_N} \sum_{k\geq M} Y_{x, k} \Big] \to \frac{\mathbb E\Big[\sum_{k\geq M} Y_{k}\Big]}{1+ \mathrm E[m]}
	\end{align}
	as $N\to \infty$, where, as usual, we drop the subscript $x$ in order to denote a generic random variable $Y_k$ such that $Y_k \overset{d}{=} Y_{x, k}$ for all $x$. Let $\varepsilon>0$. By \eqref{Equation: upper bound on sum of Y_{xk}} and \eqref{Equation: Estimate tuning fork tails}, there exists some $M\in \mathbb N$ and some $N_0 \in \mathbb N$ such that $\nu_{\operatorname{TF}, N}(\mathrm D_M) \leq \varepsilon$ for all $N\geq N_0$ and such that $\nu_{\operatorname{TF}}(\mathrm D_M) \leq \varepsilon$. Let $N_1 \geq N_0$ be such that
	\begin{equation*}
		\big|\nu_{\operatorname{TF}, N}(\{\lambda_k\}) - \nu_{\operatorname{TF}}(\{\lambda_k\})\big| \leq \varepsilon M^{-1}
	\end{equation*}
	for all $N\geq N_1$ and all $k\in \{1, \hdots, M\}$. Then for every $I\in \mathcal B(\R)$ and all $N\geq N_1$
	\begin{equation*}
		|\nu_{\operatorname{TF}, N}(I) - \nu_{\operatorname{TF}}(I)| \leq \nu_{\operatorname{TF}, N}(\mathrm D_M \cap I) + \nu_{\operatorname{TF}}(\mathrm D_M \cap I) + \sum_{k=1}^{M-1} 	|\nu_{\operatorname{TF}, N}(\{\lambda_k\}) - \nu_{\operatorname{TF}}(\{\lambda_k\})| \leq 3\varepsilon. \qedhere
	\end{equation*}
	
\end{proof}

\begin{prop}
	\label{Proposition: Weak convergence of DOS measures of Anderson Hamiltonian}
	Let $\lambda \in \mathbb R$ be such that $\mathbb P(r(\lambda) = \infty) < 1$. Then $\rho_{ \lambda, N} \to \rho_{\lambda}$ weakly as $N \to \infty$.
\end{prop}

\begin{proof}
	As in Section \ref{Section: BS-Convergence}, we denote by $\mathbb o$ the projection of $o$ to $\mathbb V$. 
	As in the proof of Theorem \ref{Theorem: Limit Theorem}, we apply Lemma \ref{Lemma: Convgerence of biased distribution}; this time for the choice $W_x \coloneqq \1_{\{x\in \mathbb V_\lambda\}}$. We rewrite
	\begin{equation*}
		\rho _{\lambda} = \mathbb E[\mu_{\lambda, o}(\, \cdot \,)|o \in \mathbb V_\lambda] = \mathbb E[\mu_{\lambda, \mathbb o}(\, \cdot \,)|o \in \mathbb V, W_{\mathbb o} = 1]
	\end{equation*}
	where $\mu_{\lambda, x}$ denotes the spectral measure of $H_\lambda$ in $x\in \mathbb V_\lambda$. Conditioning on $\{o \in \mathbb V\}$ removes the bias of $(\mathrm G^\mathbb o, \varphi^\mathbb o)$ such that the subsequent conditioning on $W_{\mathbb o} = 1$ exactly biases the distribution of $W$ according to Lemma \ref{Lemma: Convgerence of biased distribution}.
	Let us again enumerate the vertices of $\mathbb G_N$ and $\mathbb G$ such that both have vertex sets in $\mathbb N$ with $\mathbb o = \mathbb o_N = 0$. We say that a unitary map $U:l^2(\N)\to l^2(\N)$ is $(N, R)$-unitary if $U\delta_0 = \delta_0$ and if 
	\begin{equation*}
		U H_{\lambda, N}U^* \psi = H_{\lambda} \psi \text{ for all }\psi \in l^2(\mathbb N) \text{ with }\psi_x = 0 \text{ for all }x > R.
	\end{equation*}
	As in the proof of Theorem \ref{Theorem: Limit Theorem}, Lemma \ref{Lemma: Convgerence of biased distribution} and Skorokhod's representation theorem yield the existence of a coupling of $H_\lambda$, $H_{\lambda, 1}$, $H_{\lambda, 2}, \hdots$ such that\footnote{All of the following statements should be understood pointwise on the common probability space on which $H_\lambda$, $H_{\lambda, 1}$, $H_{\lambda, 2}, \hdots$ are defined} for all $R>0$ there eventually exists a $(N, R)$ unitary map $U$
	(which is of the form $(U\psi)_x = \psi_{\varphi(x)}$ for some bijection $\varphi: \N \to \N$ obtained from extending a graph isomorphism between $(\mathbb G_N, 0)_R$ and  $(\mathbb G, 0)_R$).
	Hence, there exists a strictly increasing sequence $(N_k)_k$ of natural numbers and a sequence $(U_k)_k$ of unitary maps with $U_k \delta_0 = \delta_0$ such that for all compactly supported $\psi \in l^2(\mathbb N)$ we have $U_k H_{\lambda, N_k}U_k^* \psi = H_{\lambda} \psi$ eventually. Since the subspace of compactly supported $\psi \in l^2(\mathbb N)$ is a common core for $\tilde H_{\lambda, N_k} \coloneqq U_k H_{\lambda, N_k}U_k^*$ and $H_{\lambda}$, we have $\tilde H_{\lambda, N_k} \to H_{\lambda}$ in the strong resolvent sense. By definition of strong resolvent convergence, this in particular implies pointwise convergence of the Stieltjes transforms of the spectral measure of $\tilde H_{\lambda, N_k}$ in $\delta_0$ as to $\mu_{\lambda, 0}$ as $k\to \infty$. As $U_k\delta_0 = \delta_0$, the spectral measure of $\tilde H_{\lambda, N_k}$ in $\delta_0$ agrees with the spectral measure of $H_{\lambda, N_k}$ in $\delta_0$. Since the same argument applies to every subsequence of $(H_{\lambda, N})_N$, we obtain weak convergence of the spectral measure of $H_{\lambda, N}$ in $\delta_0$ to $\mu_{\lambda, 0}$ as $N\to \infty$.
	Taking the expected value hence yields $\rho _{\lambda, N} \to \rho _{\lambda}$ weakly as $k \to \infty$.
\end{proof}

\begin{lemma}
	\label{Lemma: Estimate mass in terms of AH for convergence of atoms}
	Let $I$ be an open interval and $\lambda \in \overline I$ such that $\mathbb P(r(\lambda) = \infty) < 1$. Then
	\begin{equation*}
		\limsup_{N \to \infty} \rho_{\lambda, N}(I) \leq \frac{1+ \mathbb E[m]}{\mathbb P(r(\lambda) \neq \infty)} (\nu-\nu_{\operatorname{TF}})(I). 
	\end{equation*}
\end{lemma}

\begin{proof}
	By the law of large numbers
	\begin{equation}
		\label{Equation: Convergence of fraction of vertices}
		\frac{|\mathbb V_{\lambda, N}|}{|\mathcal V_N|} = \frac{|\mathbb V_N|^{-1} \sum_{x}\1_{\{r_x(\lambda) \neq \infty\}}}{ |\mathbb V_N|^{-1} \sum_{x}1+m_x} \to  \frac{\mathbb P(r(\lambda) \neq \infty)}{1+ \mathbb E[m]}
	\end{equation}
	almost surely as $N\to \infty$. Hence, by Proposition \ref{Proposition: Estimate mass in terms of Anderson Hamiltonian in finite volume}, Proposition \ref{Proposition: Weak convergence of DOS measures of Anderson Hamiltonian}, Lemma \ref{Lemma: Total variation convergence of tuning fork contribution}, the uniform convergence of the cumulative distribution functions of $\nu_N$ and dominated convergence
	\begin{align*}
		(\nu-\nu_{\operatorname{TF}})(I) &= \lim_{N \to \infty}(\nu_N-\nu_{\operatorname{TF, N}})(I) \\
		&\geq \limsup_{N\to \infty} \mathbb E\Big[\frac{|\mathbb V_{\lambda, N}|}{|\mathcal V_N|} \mu_{\lambda, N}(I)\Big]
		=  \limsup_{N\to \infty} \frac{\mathbb P(r(\lambda) \neq \infty)}{1+ \mathbb E[m]} \mathbb \rho_{\lambda, N}(I).\qedhere
	\end{align*}
\end{proof}

\begin{prop}
	\label{Proposition: convergence of atoms for Anderson DOS}
	For every $\lambda \in \R$ with $\mathbb P(r(\lambda) = \infty)<1$ we have $\lim_{N \to \infty} \rho_{\lambda, N}(\{\lambda\}) = \rho_{\lambda}(\{\lambda\})$.
\end{prop}

\begin{proof}
	Let $\varepsilon>0$. We write
	\begin{equation*}
		\rho_{\lambda, N}(\{  \lambda\}) = \rho_{\lambda, N}((  \lambda- \varepsilon,   \lambda + \varepsilon)) - \rho_{\lambda, N}((  \lambda- \varepsilon,   \lambda)) - \rho_{\lambda, N}((  \lambda,   \lambda + \varepsilon)).
	\end{equation*}
	Let 
	\begin{equation*}
		c \coloneqq  \frac{1+ \mathbb E[m]}{\mathbb P(r(\lambda) \neq \infty)}.
	\end{equation*}
	By Lemma \ref{Lemma: Estimate mass in terms of AH for convergence of atoms}, we have
	\begin{equation*}
		\limsup_{N \to \infty}\rho_{\lambda, N}((  \lambda- \varepsilon,   \lambda)) + \rho_{\lambda, N}((  \lambda,   \lambda + \varepsilon)) \leq  c \nu((  \lambda- \varepsilon,   \lambda)) + c \nu((  \lambda,   \lambda + \varepsilon))
	\end{equation*}
	implying with the Portmanteau theorem
	\begin{align*}
		\liminf_{N\to \infty}\rho_{\lambda, N}(\{ \lambda\})&\geq \liminf_{N\to \infty}\rho_{\lambda, N}((  \lambda- \varepsilon,   \lambda + \varepsilon)) - c\big[\nu((  \lambda- \varepsilon,   \lambda)) + \nu((  \lambda,   \lambda + \varepsilon))\big] \\
		&\geq \rho_{\lambda}((  \lambda- \varepsilon,   \lambda + \varepsilon)) - c\big[\nu((  \lambda- \varepsilon,   \lambda)) + \nu((  \lambda,   \lambda + \varepsilon))\big].
	\end{align*}
	Letting $\varepsilon\downarrow 0$ and the Portmanteau theorem yield
	\begin{equation*}
		\rho_{\lambda}(\{ \lambda\}) \leq \liminf_{N\to \infty}\rho_{\lambda, N}(\{ \lambda\}) \leq \limsup_{N\to \infty}\rho_{\lambda, N}(\{ \lambda\}) \leq \rho_{\lambda}(\{ \lambda\}).  \qedhere
	\end{equation*}
\end{proof}

\begin{proof}[Proof of Theorem \ref{Theorem: Anderson representation of atoms}]
	If $\mathbb P(r(\lambda) = \infty)=1$ the statement follows directly from Proposition \ref{Proposition: Anderson representation in finite volume}, so let us assume that $\mathbb P(r(\lambda) = \infty)<1$.
	By Proposition \ref{Proposition: Anderson representation in finite volume}, Proposition \ref{Proposition: convergence of atoms for Anderson DOS} and \eqref{Equation: Convergence of fraction of vertices} we have
	\begin{align*}
		(\nu-\nu_{\operatorname{TF}})(\{\lambda\}) &= \lim_{N\to \infty} \mathbb E\Big[\frac{|\mathbb V_{\lambda, N}|}{|\mathcal V_N|} \mu_{\lambda, N}(\{\lambda\})\Big]\\
		&=  \frac{\mathbb P(r(\lambda) \neq \infty)}{1+p\mathbb E[m]} \lim_{N\to \infty}\rho_{\lambda, N}(\{\lambda\}) =  \frac{\mathbb P(r(\lambda) \neq \infty)}{1+p\mathbb E[m]} \rho_{\lambda}(\{\lambda\}).
	\end{align*}
	If $I\subseteq \R$ is an open interval and $\lambda \in \overline I$, we obtain with Lemma \ref{Lemma: Estimate mass in terms of AH for convergence of atoms} and the Portmanteau theorem
	\begin{equation*}
		(\nu-\nu_{\operatorname{TF}})(I) \geq \liminf_{N\to \infty} \frac{\mathbb P(r(\lambda) \neq \infty)}{1+p\mathbb E[m]} \rho_{\lambda, N}(I) \geq  \frac{\mathbb P(r(\lambda) \neq \infty)}{1+p\mathbb E[m]} \rho_{\lambda}(I). \qedhere
	\end{equation*}
	
\end{proof}

\subsection{Anderson-percolation representation of the spectrum: proof of Theorem \ref{Theorem: Anderson representation of the spectrum}}	
\label{Subsection: Anderson-percolation representation of the spectrum}
We now show the representation of the spectrum of the adjacency operator as given in Theorem \ref{Theorem: Anderson representation of the spectrum}. While it might be tempting to use a similar approach as taken in the proof of Theorem \ref{Theorem: Anderson representation of atoms} and take the limit in the finite dimensional representation given in Proposition \ref{Proposition: Anderson representation in finite volume}, Benjamini--Schramm convergence does in general not imply convergence of the spectra: we might, for example, add an arbitrary connected component containing $o(|\mathbb V_N|)$ vertices to $\mathbb G_N$ without changing the convergence. We shall hence prove the statement directly in infinite volume. 

\begin{lemma}
	\label{Lemma: distance to finite sets}
	Let $I, J \subseteq \R$ be two finite, non-empty sets. Then for every $\delta>0$ there exists some $C>0$ such that for all $\lambda \in \R$
	\begin{equation*}
		C \cdot \operatorname{dist}(\lambda, J) \leq \operatorname{dist}(\lambda, I)  \implies \operatorname{dist}(\lambda, J\setminus I) < \delta \text { and }\operatorname{dist}(\lambda, I) \leq C.
	\end{equation*}
\end{lemma}

\begin{proof}
	Let us first assume that $I\cap J \neq \emptyset$.
	We define the function $f:\R \to \R$ and the set $M \subseteq \R$ by
	\begin{equation*}
		f(x) \coloneqq
		\begin{cases}
			\frac{\operatorname{dist}(x, J)}{\operatorname{dist}(x, I\cap J)} &\text{ if }x \notin I\cap J\\
			1 &\text{ else,}
		\end{cases}
		\quad \quad M \coloneqq \{x\in \R:\, \operatorname{dist}(x, J\setminus I)\geq \delta\}.
	\end{equation*}
	Notice that $f$ is continuous. Since $f(x) \to 1$ as $|x|\to \infty$ and since
	\begin{equation*}
		f(x) = 0 \iff x \in J \setminus I
	\end{equation*}
	we have $\inf_{x\in M} f(x) > 0$. Let $C_0 > (\inf_{x\in M} f(x))^{-1} \vee 1$. Since $f(x) \to 1$ as $|x|\to \infty$ there exists some $C>C_0$ such that $f(x)>1/C_0$ for all $x\in \R$ with $\operatorname{dist}(x, I)>C$. Hence, if $\operatorname{dist}(x, J\setminus I) \geq \delta$ or $\operatorname{dist}(x, I)> C$ then 
	\begin{equation*}
		f(x) > C^{-1} \quad   \text{ i.e.\ } \quad C\cdot \operatorname{dist}(x, J) > \operatorname{dist}(x, I\cap J) \geq \operatorname{dist}(x, I).
	\end{equation*}
	The proof for $I\cap J = \emptyset$ follows similarly, using that $\lim_{\lambda \to \infty} \operatorname{dist}(\lambda, J)/\operatorname{dist}(\lambda, I) = 1$.
\end{proof}

Recall that we denote by $\mathcal T$ the set of tuning for eigenvalues.

\begin{lemma}
	\label{Lemma: Estimate distance to tuning fork eigenvalues}
	Assume that \eqref{Assumption:: m a.s. bounded} is satisfied. Then there exists some constant $C>0$ such that almost surely for all $\lambda \in \R$
	\begin{equation*}
		\1_{\{\mathcal T \neq \emptyset\}}\operatorname{dist}(\lambda, \mathcal T)^2 + \1_{\{\mathcal T = \emptyset\}} \leq C \inf_{\|\psi\| = 1} \|(T-\lambda) \psi\|^2 + \|E \psi\|^2.
	\end{equation*}
\end{lemma}

\begin{proof}
	Let us first assume that $\mathcal T \neq \emptyset$. We first fix some $x\in \mathbb V^+$ and some realization $\mathrm G^x$ and restrict ourselves to $\psi$ supported on $\mathrm V^x$. Notice that $\|\varphi^x\| \leq m_x \leq M$ where $M$ is as in our Assumption \eqref{Assumption:: m a.s. bounded}. We now choose constants $\delta_1, \delta_2, C_1, C_2, C_3$ in the following way.
	\begin{enumerate}[leftmargin=*]
		\item By definition of $\mathcal T$, we have for all $\lambda \in \sigma(\mathrm  T^x)\setminus \mathcal T$ and all $\psi\in \ker(\mathrm  T^x - \lambda)$ with $\psi \neq 0$ that $\langle \varphi^x, \psi \rangle \neq 0$. 
		A simple compactness argument yields the existence of some $\delta_1 > 0$ such that 
		\begin{equation*}
			C_1 \coloneqq (1-\delta_1)\inf \Big(\Big\{ \big|\langle \varphi^x, \psi^x \rangle\big|:\, \lambda \in \sigma(\mathrm  T^x)\setminus \mathcal T, \, \psi^x\in \ker(\mathrm  T^x - \lambda) \text{ with }\|\psi^x\|= 1 \Big\}\Big) - M \sqrt{\delta_1} > 0.
		\end{equation*}
		\item
		We define $\delta \coloneqq \min\{|\lambda_1 - \lambda_2|:\, \lambda_1, \lambda_2 \in \sigma(\mathrm  T^x), \, \lambda_1 \neq \lambda_2\}$.
		Let $\lambda \in \R$ and $\alpha \in \sigma(\mathrm  T^x)$. By spectral decomposition, if $|\lambda-\alpha|< \delta/2$ then
		\begin{equation}
			\label{Equation: estimate norm of (T-lambda)psi in terms of projection}
			\|(\mathrm  T^x - \lambda) \psi^x\|^2 \geq \|(1-P_\alpha) \psi^x\|^2 \delta^2/4
		\end{equation}
		where $P_{\alpha}$ is the projection onto $\operatorname{ker}(\mathrm  T^x - \alpha)$. For $\lambda \in \mathbb R$ let $\alpha(\lambda) \coloneqq \arg \min \{|\alpha- \lambda|:\, \alpha \in \sigma(\mathrm  T^x)\}$.
		Let $\delta_1$ be as chosen in the first step.
		By \eqref{Equation: estimate norm of (T-lambda)psi in terms of projection} we can choose some $\delta_2 \in (0, \delta/2)$ such that for all $\psi^x \in l^2(\mathrm V^x)$ with $\|\psi^x\|= 1$
		\begin{equation*}
			\|(\mathrm  T^x - \lambda) \psi^x\| < \delta_2 \implies \| P_{\alpha(\lambda)} \psi^x\|^2 \geq 1 - \delta_1.
		\end{equation*}
		\item
		By Lemma \ref{Lemma: distance to finite sets} we can choose some $C_2>0$ such that
		\begin{equation*}
			C_2 \operatorname{dist}(\lambda, \sigma(\mathrm  T^x)) \leq \operatorname{dist}(\lambda, \mathcal T) \implies \operatorname{dist}(\lambda, \sigma(\mathrm  T^x) \setminus \mathcal T) <\delta_2, \quad \operatorname{dist}(\lambda, \mathcal T) \leq C_2.
		\end{equation*}
		\item We choose $C_3\geq C_2$ such that $C_3 \cdot C_1 \geq C_2$ and $C_3\delta_2 > C_2$.
	\end{enumerate}
	So if $\|\psi^x \| = 1$ then: If $\lambda \in \R$ is such that $C_2 \operatorname{dist}(\lambda, \sigma(\mathrm  T^x)) > \operatorname{dist}(\lambda, \mathcal T)$ then
	\begin{equation*}
		C_3 \|(\mathrm  T^x-\lambda) \psi^x \| \geq C_3\operatorname{dist}(\lambda, \sigma(\mathrm  T^x)) \geq \operatorname{dist}(\lambda, \mathcal T).
	\end{equation*}
	If, on the other hand, $C_2 \operatorname{dist}(\lambda, \sigma(\mathrm  T^x)) \leq \operatorname{dist}(\lambda, \mathcal T)$ then $\operatorname{dist}(\lambda, \sigma(\mathrm  T^x) \setminus \mathcal T) <\delta_2$ and $\operatorname{dist}(\lambda, \mathcal T) \leq C_2$. If $\|(\mathrm  T^x - \lambda)\psi^x\|\geq \delta_2$ then
	\begin{equation*}
		C_3 \|(\mathrm  T^x - \lambda)\psi^x\| \geq C_3 \delta_2 \geq C_2 \geq \operatorname{dist}(\lambda, \mathcal T).
	\end{equation*}
	So let us assume $\|(\mathrm  T^x - \lambda)\psi^x\|< \delta_2$. Then $\| P_{\alpha(\lambda)} \psi^x\|  \geq  \| P_{\alpha(\lambda)} \psi^x\|^2 \geq 1-\delta_1$ and $\|(1-P_{\alpha(\lambda)})\psi^x\big\| \leq \sqrt{\delta_1}$ and hence 
	\begin{align}
		\label{Equation: Estimate overlap of psi^x and phi^x}
		|\langle \varphi^x, \psi^x\rangle| &= \Big| \| P_{\alpha(\lambda)} \psi^x\| \Big\langle \varphi^x, \frac{P_{\alpha(\lambda)}\psi^x}{\| P_{\alpha(\lambda)} \psi^x\|} \Big\rangle + \big\langle \varphi^x, (1-P_{\alpha(\lambda)})\psi^x \big\rangle \Big| \nonumber  \\
		&\geq (1-\delta_1) \Big|\Big\langle \varphi^x, \frac{P_{\alpha(\lambda)}\psi^x}{\| P_{\alpha(\lambda)} \psi^x\|} \Big\rangle\Big| - \big\|\varphi^x\|\cdot \big\|(1-P_{\alpha(\lambda)})\psi^x\big\| \geq C_1
	\end{align}
	since $\alpha(\lambda)\notin \mathcal T$ as $\delta_2 < \delta/2$ and $\operatorname{dist}(\lambda, \sigma(\mathrm  T^x) \setminus \mathcal T) <\delta_2$. Hence,
	\begin{equation*}
		C_3|\langle \varphi^x, \psi^x\rangle| \geq C_3 C_1 \geq C_2 \geq \operatorname{dist}(\lambda, \mathcal T).
	\end{equation*}
	Hence, in any case we have for all $\psi^x\in l^2(\mathrm V^x)$ with $\|\psi^x\| = 1$ and all $\lambda \in \R$
	\begin{equation*}
		\operatorname{dist}(\lambda, \mathcal T)^2 \leq C_3^2 \|(\mathrm  T^x-\lambda) \psi^x\|^2 + C_3^2  |\langle \varphi^x, \psi^x \rangle |^2.
	\end{equation*}
	Let us now reintroduce the randomness in the sprinkling.
	As $m$ is almost surely bounded by $M$, the support of $\mathrm P$ is finite. Hence, there exists some constant $\tilde C>0$ such that almost surely for all $\lambda \in \R$
	\begin{equation*}
		\operatorname{dist}(\lambda, \mathcal T)^2 \leq \tilde C \cdot \inf \big\{\|(\mathrm  T^x-\lambda) \psi^x\|^2 +   |\langle \varphi^x, \psi^x \rangle |^2:\, x \in \mathbb V, \,\psi^x\in l^2(\mathrm V^x),\, \|\psi^x\|= 1 \big\}.
	\end{equation*}
	Now, let $\psi\in \bigoplus_{x\in \mathbb V^+} l^2(\mathrm V^x) $ with $\|\psi\| = 1$. As before, we denote for $x\in \mathbb V$ by $\psi^x$  the restriction of $\psi$ to $\mathrm V^x$. Then
	\begin{align*}
		\operatorname{dist}(\lambda, \mathcal T)^2 = \sum_{x\in \mathbb V} 	\operatorname{dist}(\lambda, \mathcal T) ^2  \|\psi^x\|^2 &\leq \tilde C \sum_{x\in \mathbb V} \|(\mathrm  T^x - \lambda)\psi^x\|^2 + |\langle \varphi^x, \psi^x \rangle |^2 \\
		&= \tilde  C\|(T - \lambda)\psi\|^2 + \tilde C\|E\psi\|^2.
	\end{align*}
	To finish the proof, it is sufficient to notice that the same argument (via \eqref{Equation: Estimate overlap of psi^x and phi^x}) shows that there exists some $C \geq \tilde C$ such that if $\mathcal T = \emptyset$ then for all $\psi\in \bigoplus_{x\in \mathbb V^+} l^2(\mathrm V^x) $ with $\|\psi\| = 1$
	\begin{equation*}
		1 \leq C\|(T - \lambda)\psi\|^2 + C\|E\psi\|^2. \qedhere
	\end{equation*}
\end{proof}

\begin{proof}[Proof of Theorem \ref{Theorem: Anderson representation of the spectrum}]
	We start by showing that 
	\begin{equation*}
		\big\{\lambda \in \mathbb R:\, \lambda \in \sigma(H_\lambda)  \big\} \cup \mathcal T \subseteq \sigma(A).
	\end{equation*}
	Since $\mathcal T \subseteq \sigma(A)$ holds by definition, let $\lambda \in \R$ be such that $\lambda \in \sigma(H_\lambda)$. Let $\varepsilon>0$. By Weyl's criterion, there exists some $\tilde \phi\in l^2_c(\mathbb V_\lambda)$ with $\|\tilde \phi\| = 1$ such that $\|(H_\lambda - \lambda) \tilde \phi\| < \varepsilon$. For every $x\in \mathbb V\setminus \mathbb V_\lambda$, let $\Psi^{x}$ be an eigenvector of $\mathrm G^x$ such that $\langle \varphi^x, \Psi^{x} \rangle \neq 0$. We define $(\phi, \psi) \in l^2_c(\mathcal V)$ by
	\begin{align*}
		\phi_x &\coloneqq
		\begin{cases}
			\tilde \phi_x &\text{ for } x\in \mathbb V_\lambda \\
			0 &\text{ for } x\in \mathbb V \setminus \mathbb V_\lambda 
		\end{cases} \\
		\psi^x &\coloneqq
		\begin{cases}
			\tilde \phi_x (\lambda-\mathrm  T^x)^+ \varphi^x  &\text{ for } x\in \mathbb V^+ \cap \mathbb V_\lambda \\
			-\frac{1}{\langle \varphi^x, \Psi^{x} \rangle} \sum_{y \sim x, \, y\in \mathbb V_\lambda}\tilde \phi_y \cdot \Psi^{x}  &\text{ for } x\in  \mathbb V \setminus \mathbb V_\lambda
		\end{cases} 
	\end{align*}
	where $\psi^x$ again denotes the projection of $\psi$ onto $\mathrm V^x$. 
	By Lemma \ref{Lemma: when is varphi in image}, we have for $x\in \mathbb V^+ \cap \mathbb V_\lambda$
	\begin{equation*}
		\varphi^x \in \operatorname{ker}(\mathrm  T^x - \lambda)^\perp = \operatorname{im}(\mathrm  T^x - \lambda).
	\end{equation*}
	Since  $(\lambda-\mathrm  T^x)(\lambda - \mathrm  T^x)^+$ is the orthogonal projection onto $\operatorname{im}(\mathrm  T^x -\lambda)$, we hence have
	\begin{equation}
		\label{Equation: can't come up with a name right now}
		[(A-\lambda)(\phi, \psi)]^x = (\mathrm  T^x - \lambda) \psi^x + \tilde \phi_x \varphi^x = (\mathrm  T^x - \lambda)(\lambda - \mathrm  T^x)^+ \tilde \phi^x \varphi^x +\tilde \phi^x \varphi^x  = 0
	\end{equation}
	for all $x\in \mathbb V^+ \cap \mathbb V_\lambda$. Therefore, by Lemma \ref{Lemma: get Anderson potential from solution to eigenvalue equation}, $(E \psi)_x = \langle \varphi^x, \psi^x \rangle = -r_x(\lambda) \tilde \phi_x$
	for all $x\in \mathbb V^+ \cap \mathbb V_\lambda$, which implies that
	\begin{equation*}
		[(A-\lambda)(\phi, \psi)]_x = \big[(B-\lambda) \phi + E\psi\big]_x =  [(H_\lambda-\lambda) \tilde \phi]_x
	\end{equation*}
	for all $x\in \mathbb V_\lambda$ (as $\phi_x = r_x(\lambda)= 0$ for all $x\in \mathbb V\setminus \mathbb V^+$). Notice that we defined $(\phi, \psi)$ exactly such that
	\begin{equation*}
		[(A-\lambda)(\phi, \psi)]_v = 0 \text{ for all } v\in \mathcal V \setminus \mathbb V_\lambda
	\end{equation*}
	(if $v \in \mathcal V_x$ for some $x\in \mathbb V^+ \cap \mathbb V_\lambda$ this follows from \eqref{Equation: can't come up with a name right now}, for all other $v\in \mathcal V \setminus \mathbb V_\lambda$ this follows directly from the definition). As $\| (\phi, \psi) \| \geq \|\tilde \phi\| =1$, we hence have
	\begin{equation*}
		\Big\|(A-\lambda) \frac{(\phi, \psi)}{\| (\phi, \psi) \|} \Big\| \leq  \|(A-\lambda)(\phi, \psi)\| = \|(H_\lambda -\lambda)  \tilde \phi\| < \varepsilon.
	\end{equation*}
	As $\varepsilon>0$ was arbitrary, we obtain $\lambda \in \sigma(A)$. Let us now show that also the inclusion
	\begin{equation*}
		\sigma(A) \subseteq \big\{\lambda \in \mathbb R:\, \lambda \in \sigma(H_\lambda)  \big\} \cup \mathcal T
	\end{equation*}
	holds. For every $x\in \mathbb V \setminus \mathbb V_\lambda$, there exists some $ \Psi^{x} \in \operatorname{ker}(T^{x} -\lambda)$ such that $\|\Psi^{x}\|= 1$ and such that $\langle \varphi^x, \Psi^x\rangle \neq 0$. Since Assumption \eqref{Assumption:: m a.s. bounded} holds, we can choose the $\Psi^x$ such that
	\begin{equation*}
		\delta \coloneqq \inf_{x \in \mathbb V \setminus \mathbb V_\lambda} |\langle \varphi^x, \Psi^x\rangle| > 0, \quad C \coloneqq \sup_{x \in \mathbb V^+} \|(\mathrm  T^x - \lambda)^+ \| < \infty.
	\end{equation*}
	Let $\lambda \in \sigma(A)$ and let $(\phi, \psi) \in l^2_c(\mathbb V)$.
	Let $\tilde \phi$ denote the restriction of $\phi$ to $\mathbb V_\lambda$.
	We have for all $x\in \mathbb V_\lambda$
	\begin{align}
		\label{Equation: Estimate approximate eigenvector of Anderson Hamiltonian}
		\big|\big[(H_\lambda - \lambda) \tilde \phi \big]_x\big| &= \big|\big[(B_\lambda - \lambda + V_\lambda)\tilde \phi  \big]_x\big| \nonumber  \\
		&\leq \big|\big[(B - \lambda) \phi +E\psi]_x\big| + \big|(B_\lambda \tilde\phi)_x -  (B\phi)_x \big|+ \big|(V_\lambda \tilde \phi)_x - (E\psi)_x\big|\nonumber \\
		&\leq \big|\big[(B - \lambda) \phi +E\psi]_x\big| + \sum_{y\sim x, y\notin \mathbb V_\lambda} |\phi_y| + \big|(V_\lambda \tilde \phi)_x - (E\psi)_x\big|.
	\end{align}
	Let $y\in \mathbb V \setminus \mathbb V_\lambda$ and let $P^y_\lambda$ denote the orthogonal projection onto $\operatorname{ker}(T^y - \lambda)$. Since $y\in \mathbb V \setminus \mathbb V_\lambda$
	\begin{align}
		\label{Equation: Estimate phi_y for y outside of percolation graph}
		|\phi_y| |\langle \Psi^y, \varphi^y \rangle|&= |\langle\Psi^y, \phi_y \varphi^y \rangle| \nonumber \\ 
		&\leq \|P_\lambda^y \phi_y \varphi^y \|
		= \operatorname{dist}(\phi_y \varphi^y, \operatorname{ker}(T^y - \lambda)^\perp)
		= \operatorname{dist}(\phi_y \varphi^y, \operatorname{im}(T^y - \lambda)) \nonumber \\
		&\leq \|(T^y - \lambda)\psi^y + \phi_y \varphi^y \|
	\end{align}
	such that
	\begin{align}
		\label{Equation: Estimate sum of phi_y for y outside of percolation graph}
		\sum_{x \in \mathbb V_\lambda} \Big|\sum_{y\sim x, y\notin \mathbb V_\lambda}  |\phi_y| \Big|^2 &\leq \sum_{x \in \mathbb V_\lambda} d \sum_{y\sim x, y\notin \mathbb V_\lambda}  |\phi_y|^2 
		\leq \sum_{y \notin \mathbb V_\lambda} d^2 |\phi_y|^2 \nonumber\\
		&\leq \frac{d^2}{\delta^2} \sum_{y \notin \mathbb V_\lambda}  \|(T^y - \lambda)\psi^y + \phi_y \varphi^y \|^2
		\leq \frac{d^2}{\delta^2} \|(T - \lambda)\psi + E^* \phi\|^2
	\end{align}
	with $d \coloneqq \sup_{x \in \mathbb V} \operatorname{deg}(x)$.
	On the other hand, for $x\in \mathbb V_\lambda\cap \mathbb V^+$ we have with Lemma \ref{Lemma: r in terms of pseudoinverse}
	\begin{align*}
		\big|(V_\lambda \tilde \phi)_x - (E\psi)_x\big| &= \big|\langle \varphi^x, (\lambda-\mathrm  T^x)^+ \varphi^x \rangle \phi_x - \langle \varphi^x, \psi^x \rangle \big| \\
		&\leq \big|\big\langle (\lambda-\mathrm  T^x)^+\varphi^x, \phi_x \varphi^x - (\lambda-\mathrm  T^x)\psi^x \big\rangle \big| + \big|\big\langle \varphi^x, (\lambda - \mathrm  T^x)^+(\lambda - \mathrm T^x) \psi^x - \psi^x \big\rangle \big|.
	\end{align*}
	Now, by symmetry of $\mathrm  T^x$, we have that $(\lambda - \mathrm  T^x)^+(\lambda - \mathrm  T^x)$ is the orthogonal projection onto $\operatorname{im}(\lambda - \mathrm  T^x)$ and hence
	\begin{equation*}
		\langle \varphi^x, (\lambda - \mathrm  T^x)^+(\lambda - \mathrm  T^x) \psi^x - \psi^x \big\rangle = \langle \varphi^x, P^x_\lambda \psi \rangle = 0
	\end{equation*}
	as $\varphi^x \in \operatorname{ker}(\lambda - \mathrm  T^x)^\perp$ for $x\in \mathbb V_\lambda \cap \mathbb V^+$. As we additionally have
	\begin{equation*}
		\big|\big\langle (\lambda-\mathrm  T^x)^+\varphi^x, \phi_x \varphi^x - (\lambda-\mathrm  T^x)\psi^x \big\rangle \big| \leq C M \|\phi_x \varphi^x - (\lambda-\mathrm  T^x)\psi^x\| 
	\end{equation*}
	we get
	\begin{equation*}
		\big|(V_\lambda \tilde \phi)_x - (E\psi)_x\big| \leq  CM \|(\mathrm  T^x - \lambda)\psi^x + \phi_x \varphi^x\|.
	\end{equation*}
	for all $x\in \mathbb V_\lambda \cap \mathbb V^+$. Therefore
	\begin{equation}
		\label{Equation: Estimate difference of diagonal terms}
		\sum_{x\in \mathbb V_\lambda} \big|(V_\lambda \tilde \phi)_x - (E\psi)_x\big|^2 \leq C^2 M^2 \sum_{x\in \mathbb V_\lambda} \|(\mathrm  T^x - \lambda)\psi^x + \phi_x \varphi^x\|^2 =  C^2 M^2 \|(T - \lambda)\psi + E^*\phi \|^2.
	\end{equation}
	Combining \eqref{Equation: Estimate sum of phi_y for y outside of percolation graph} and \eqref{Equation: Estimate difference of diagonal terms}, we can further estimate in \eqref{Equation: Estimate approximate eigenvector of Anderson Hamiltonian} to obtain
	\begin{align}
		\label{Equation: Estimate norm approximate eigenfunction Anderson Hamiltonian}
		\frac{1}{3}\|(H_\lambda - \lambda) \tilde \phi \|^2 &\le \sum_{x\in \mathbb V_\lambda} \big|\big[(B - \lambda) \phi +E\psi]_x\big|^2 + \sum_{x\in \mathbb V_\lambda} \Big|\sum_{y\sim x, y\notin \mathbb V_\lambda}  |\phi_y| \Big|^2 +  \sum_{x\in \mathbb V_\lambda} \big|(V_\lambda \tilde \phi)_x - (E\psi)_x\big|^2 \nonumber \\
		&\leq \|(B-\lambda)\phi + E\psi\|^2  + (d^2\delta^{-2} + C^2 M^2)\|\|(T - \lambda)\psi + E^*\phi \|^2.
	\end{align}
	Since $\lambda \in \sigma(A)$ there exists for every $\varepsilon>0$ some $(\phi_\varepsilon, \psi_\varepsilon)\in l^2(\mathcal V)$ such that $\|(\phi_\varepsilon, \psi_\varepsilon)\| = 1$ and 
	\begin{equation}
		\label{Equation: approximative eigenvector for A}
		\|(A-\lambda)(\phi_\varepsilon, \psi_\varepsilon) \|^2 =  \|(B-\lambda)\phi_\varepsilon + E\psi_\varepsilon\|^2 + \|(T-\lambda)\psi_\varepsilon + E^*\phi_\varepsilon\|^2 < \varepsilon.
	\end{equation}
	Again, let $\tilde \phi_\varepsilon$ be the restriction of $\phi_\varepsilon$ to $\mathbb V_\lambda$. By \eqref{Equation: Estimate norm approximate eigenfunction Anderson Hamiltonian}, we have
	\begin{equation*}
		\lim_{\varepsilon \downarrow 0}\|(H_\lambda - \lambda) \tilde \phi_\varepsilon \|^2 = 0.
	\end{equation*}
	Hence, as long as $\tilde \phi_\varepsilon$ does not converge to zero as $\varepsilon \downarrow 0$ then we obtain (after switching to a subsequence and normalizing)  $\lambda \in \sigma(H_\lambda)$ by Weyl's criterion. So assume that $\tilde \phi_\varepsilon \to 0$ as $\varepsilon \downarrow 0$. Then we even have $\phi_\varepsilon \to 0$ as $\varepsilon \downarrow 0$ since \eqref{Equation: Estimate sum of phi_y for y outside of percolation graph} and \eqref{Equation: approximative eigenvector for A} imply
	\begin{equation*}
		\sum_{y \in \mathbb V \setminus \mathbb V_\lambda} |\phi_y|^2 \leq \frac{1}{\delta^2} \|(T - \lambda)\psi_\varepsilon + E^* \phi_\varepsilon\|^2 \to 0.
	\end{equation*}
	Hence, we obtain with \eqref{Equation: approximative eigenvector for A}
	\begin{align*}
		\|E\psi_\varepsilon\|&\leq  \|(B-\lambda)\phi_\varepsilon + E\psi_\varepsilon\| +  \|(B-\lambda)\phi_\varepsilon\| \to 0 \\
		\|(T-\lambda) \psi_\varepsilon\| &\leq \|(T-\lambda)\psi_\varepsilon + E^*\phi_\varepsilon\| + \|E^* \phi_\varepsilon\|\leq \|(T-\lambda)\psi_\varepsilon + E^*\phi_\varepsilon\| + M\|\phi_\varepsilon\| \to 0
	\end{align*}
	as $\varepsilon \downarrow 0$. Since $\|\psi_\varepsilon\|^2 = 1- \|\phi_\varepsilon\|^2 \to 1$ as $\varepsilon\downarrow0$ we obtain with Lemma \ref{Lemma: Estimate distance to tuning fork eigenvalues} that $\operatorname{dist}(\lambda, \mathcal T) = 0$ i.e.\ that $\lambda \in \mathcal T$.
\end{proof}

\section{Support of the density of states: proof of Theorem \ref{Theorem: Spectrum in the lattice case}}
\label{Section: support of the DOS}

We now apply Theorem \ref{Theorem: Anderson representation of the spectrum} in order to prove Theorem \ref{Theorem: Spectrum in the lattice case} and determine the support of the density of states under the given assumptions. We start by showing that, under the assumptions of Theorem \ref{Theorem: Spectrum in the lattice case}, the spectrum of the operators $A$, $H_\lambda$, are almost surely constant by ergodic theory. A random operator\footnote{While the usual definition of a random operator entails that $H(\omega)$ acts for all $\omega\in \Omega$ on the same Hilbert space, for simplicity of notation, we ignore this fact and simply note that  $\ell^2(\mathcal V) \cong \ell^2(\mathbb N)$ under the assumptions of Theorem \ref{Theorem: Spectrum in the lattice case}.} $H$ defined on a probability space $(\Omega, \mathcal F, \mathbb P)$ is called ergodic if there exists a family of measure preserving transformations $(\Phi_i)_{i\in I}$ on $\Omega$ which is ergodic in the sense that
\begin{equation*}
	\forall C\in \mathcal F:\, \big((\forall i \in I:\, \Phi_i^{-1}(C) = C) \implies \mathbb P(C) \in \{0, 1\}\big)
\end{equation*}	
such that $H(\Phi_i\omega)$ is for every $\omega \in \Omega$ and $i\in I$ unitarily equivalent to $H(\omega)$. If $H$ is an ergodic operator, then its spectrum is almost surely constant \cite{KM82}. While our operators $A$ and $H_\lambda$ are not ergodic themselves (as the distribution of $(\mathrm G^\mathbb o, \varphi^\mathbb o)$ is biased), their distributions are absolutely continuous with respect to ergodic operators, namely the one obtained by removing the bias of $\mathrm G^\mathbb o$ (which can, for example, be done by conditioning on $o \in \mathbb V$). Hence, they have almost surely constant spectrum.

\begin{prop}
	\label{Proposition: Ergodicity of operators}
	Assume that $\mathbb G$ is deterministic, vertex transitive and infinite. Consider the operators $\hat A$, $\hat H_\lambda$, $\lambda \in \R$ which are obtained by removing the bias of $(\mathrm G^{\mathbb o}, \varphi^\mathbb o)$. Then $\hat A$, $\hat H_\lambda$ can be defined as ergodic operators.
\end{prop}

\begin{proof}
	We only show the statement for $\hat A$, the statement for $\hat H_\lambda$ can be shown in the same manner.
	We equip
	\begin{equation*}
		\Omega_0 \coloneqq \Big\{ (\mathrm G, \varphi):\, \mathrm G \text{ finite graph with vertex set }\mathrm V \subset \N, \, \varphi \in \{0, 1\}^{\mathrm V}   \Big\}
	\end{equation*}
	with the power set $\mathcal F_0 = \mathrm P(\Omega_0)$ as a $\sigma$-algebra and define our iid sequence $(\mathrm G^x, \varphi^x)_{x\in \mathbb V}$ as the identity on $(\Omega_0^{\mathbb V}, \mathcal F_0^{\otimes \mathbb V}, \mathrm P^{\otimes \mathbb V})$. For all $x, y\in \mathbb G$ there exists by assumption a graph automorphism $\phi_{xy}:\mathbb V \to \mathbb V$ with $\phi_{xy}(x) = y$. The map $\phi = \phi_{xy}$ naturally defines a measure preserving transformation $\Phi = \Phi_{xy}$ on $(\Omega_0^{\mathbb V}, \mathcal F_0^{\otimes \mathbb V}, \mathrm P^{\otimes \mathbb V})$ by
	\begin{equation*}
		\Phi(\mathrm G^z, \varphi^z)_{z\in \mathbb V} \coloneqq (\mathrm G^{\phi^{-1}(z)}, \varphi^{\phi^{-1}(z)} )_{z\in \mathbb V}.
	\end{equation*}
	For all $\omega \in  \Omega_0^{\mathbb V}$, our adjacency operator $\hat A$ satisfies $\hat A(\Phi\omega) = U_\omega^* \hat A(\omega) U_\omega$ where the unitary map $U_\omega:l^2(\mathcal V(\Phi\omega)) \to l^2(\mathcal V(\omega) )$ is defined by
	\begin{equation*}
		\forall z\in \mathbb V:\, (U_\omega\psi)_z \coloneqq 	\psi_{\phi(z)}, \quad \forall z\in \mathbb V^+, \forall t\in \mathrm V^z(\omega):\, (U_\omega\psi)_{(zt)} \coloneqq \psi_{(\phi(z)t)}
	\end{equation*}
	for all $\psi \in l^2(\mathcal V(\Phi\omega))$. It is left to show that the family $(\Phi_{xy})_{x, y \in \mathbb V}$ is ergodic. We proceed similarly to the proof of the Hewitt–Savage zero–one law. 
	Let $C \in \mathcal F$ be such that $\Phi^{xy} C = C$ for all $x, y\in \mathbb V$. We denote by $\Delta$ the symmetric set difference. Let $\varepsilon>0$. Then there exists some $R>0$ and some $\tilde C \in \sigma\big((\mathrm G^{x}, \varphi^x)_{x \in (\mathbb G, \mathbb o)_R}\big)$ such that $\mathbb P(C \Delta \tilde C) < \varepsilon$. Let $y\in \mathbb V$ be such that $(\mathbb G, y)_R$ and $(\mathbb G, \mathbb o)_R$ do not have any vertices in common and let $\Phi = \Phi_{\mathbb oy}$. For $D \coloneqq \Phi \tilde C$ we have
	\begin{equation*}
		\mathbb P(C \Delta D) = \mathbb P(\Phi C \Delta \Phi \tilde C) = \mathbb P( C \Delta \tilde C)< \varepsilon
	\end{equation*}
	and hence
	\begin{equation*}
		\mathbb P\big(C \Delta (\tilde C \cap D)\big) \leq \mathbb P(C \Delta \tilde C) + \mathbb P(C \Delta D) < 2 \varepsilon.
	\end{equation*}
	Now $D$ (which only depends on vertices in $(\mathbb G, y)_R$) and $\tilde C$ are independent and hence
	\begin{equation*}
		\mathbb P(C) \leq \mathbb P( \tilde C \cap D) + \mathbb P\big(C \Delta (\tilde C \cap D)\big) \leq \mathbb P(\tilde C)^2 + 2 \varepsilon \leq (\mathbb P(C) + \varepsilon)^2 + 2 \varepsilon.
	\end{equation*}
	Since $\varepsilon>0$ was arbitrary, we obtain $\mathbb P(C) \leq \mathbb P(C)^2$ and hence $\mathbb P(C) \in \{0, 1\}$.
\end{proof}

\begin{prop}
	\label{Proposition: Spectrum agrees almost surely with support of DOS measure}
	Assume that $\mathbb G$ is deterministic, vertex transitive and infinite. Then $\sigma(A) = \Sigma \coloneqq \operatorname{supp}(\nu)$ almost surely.
\end{prop}

\begin{proof}
	By Proposition \ref{Proposition: Ergodicity of operators}, there exist some set $\tilde \Sigma \subseteq \R$ such that $\sigma(A) = \tilde \Sigma$ almost surely.
	For a rooted locally finite graph $(G, v)$, we denote by $\mu_{(G, v)}$ the spectral measure of $G$ in the root $v$.
	For the inclusion $\Sigma \subseteq \tilde \Sigma$ it is sufficient to notice that for all $\lambda \in \R \setminus \tilde \Sigma$ there exists some $\varepsilon>0$ such that $I \coloneqq (\lambda -\varepsilon, \lambda + \varepsilon) \subseteq \R\setminus \tilde \Sigma$ and hence $\mu_{(\mathcal G, o)}(I) = 0$ almost surely i.e.\ $\nu(I) = \mathbb E[\mu_{(\mathcal G, o)}(I)] = 0$ and thus $\lambda \notin \Sigma$.  On the other hand, applying \eqref{Equation: Everything shows up at the root} for the choice $M \coloneqq \{(G, v):\, \mu_{(G, v)}(\R \setminus \Sigma) = 0\}$ yields for the spectral projection $P = \1_{\R \setminus \Sigma}(A)$ that almost surely $\| P \delta_v\|^2 =  \mu_{(\mathcal G, v)}(\R \setminus \Sigma) = 0$
	for all $v\in \mathcal V$, i.e.\ $P = 0$ almost surely and hence $\tilde \Sigma \subseteq \Sigma$. 
\end{proof}

\begin{prop}
	\label{Proposition: Kunz-Souillard-type theorem}  
	Under the assumptions of Theorem \ref{Theorem: Spectrum in the lattice case}, additionally assume that \eqref{Assumption:: m a.s. bounded} holds. We then have for all $\lambda \in \R$
	\begin{equation*}
		\mathbb P(\sigma(H_\lambda) = \Sigma_\lambda) = 1 \quad \text{ where }\quad \Sigma_\lambda \coloneqq \{\alpha \in \R:\, \mathbb P(\alpha + r(\lambda) \in \Sigma_0) > 0\}.
	\end{equation*}
\end{prop}

\begin{proof}
	While this follows as the Kunz--Souillard Theorem \cite{KS80} for the spectrum of the usual Anderson Hamiltonian, we give a proof for the convenience of the reader. By Proposition \ref{Proposition: Ergodicity of operators} there exists some deterministic set $\tilde \Sigma_\lambda \subset \R$ such that $\sigma(H_\lambda) = \tilde \Sigma_\lambda$ almost surely. Since $H_\lambda = B_\lambda + V_\lambda$, and as deleting vertices does not increase the spectral radius, we have
	\begin{equation}
		\label{Equation: Distance of spectra}
		\operatorname{dist}(\sigma(H_\lambda), \sigma(V_\lambda)) \leq \|B_\lambda\| \leq \|B\|.
	\end{equation}
	Since \eqref{Assumption:: m a.s. bounded} holds, the spectrum $\sigma(V_\lambda)$ is a finite set with
	\begin{equation*}
		\sigma(V_\lambda) = \{\beta \in \R: \, \mathbb P(-r(\lambda) = \beta) > 0\} \quad \text{a.s.}
	\end{equation*}
	As the spectrum of $B$ is a symmetric interval, we have $\Sigma_0 = \big[-\|B\|, \|B\|\big]$. Hence, if $\alpha \in \tilde \Sigma_\lambda$ then by \eqref{Equation: Distance of spectra}
	\begin{equation*}
		\mathbb P\big(\alpha + r(\lambda) \in \Sigma_0\big) = \mathbb P(|\alpha + r(\lambda)|\leq \|B\|) > 0.
	\end{equation*}
	On the other hand, assume that $\alpha\in \R$ is such that $\mathbb P(\alpha + r(\lambda) \in \Sigma_0) > 0$,
	i.e.\ that there almost surely exist some $\beta \in \sigma(V_\lambda)$ such that $\alpha - \beta \in \Sigma_0$. Let $\varepsilon>0$. Then there exists some compactly supported unit vector $\psi\in l^2_c(\mathbb V)$ such that $\|(B - \alpha + \beta)\psi\|< \varepsilon$. Now, with positive probability $r_x(\lambda) = -\beta$ for all $x\in \operatorname{supp}(\psi)$ (which in particular implies $x\in \mathbb V_\lambda$ for all $x\in \operatorname{supp}(\psi)$) and hence
	\begin{equation*}
		\mathbb P(\|(H_\lambda - \alpha)\psi\| < \varepsilon) \geq \mathbb P\big(\|(H_\lambda - \alpha)\psi\| = \|(B - \alpha + \beta)\psi\| \big) \geq \mathbb P(\forall x\in \operatorname{supp}(\psi):\, r_x(\lambda)= -\beta) > 0
	\end{equation*}
	and hence $\mathbb P(\operatorname{dist}(\alpha, \sigma(H_\lambda)) < \varepsilon) > 0$.
	Since $\sigma(H_\lambda) = \tilde \Sigma_\lambda$ almost surely and since $\tilde\Sigma_\lambda$ is closed we obtain $\alpha \in \tilde \Sigma_\lambda$.
\end{proof}

\begin{prop}
	\label{Proposition: Kunz-Souillard-type theorem simultaneously in all lambda}  
	Under the assumptions of Theorem \ref{Theorem: Spectrum in the lattice case}, additionally assume that \eqref{Assumption:: m a.s. bounded} holds. Then $\mathbb P(\forall \lambda \in \R:\, \sigma(H_\lambda) = \Sigma_\lambda) = 1$.
\end{prop}

\begin{proof}
	For a sequence $(M_n)_n$ of subsets of $\mathbb R$ we write $\lim_{n\to \infty} M_n$ for the set of all $\lambda\in \R$ for which there exists a sequence $(\lambda_n)_n$ with $\lambda_n\in M_n$ for all $n\in \N$ such that $\lim_{n\to \infty} \lambda_n = \lambda$. Let $\lambda\in \mathbb \R\setminus \mathrm D$. Since \eqref{Assumption:: m a.s. bounded} holds, $\mathrm D$ is finite. There hence exists some $(\lambda_n)_n \subseteq \Q \setminus \mathrm D$ such that $\lambda_n \to \lambda$ as $n\to \infty$. Then $\mathbb V_{\lambda_n} = \mathbb V_\lambda = \mathbb V$ for all $n\in \N$ and $\|V_{\lambda_n} - V_{\lambda}\| \to 0$ (again as a consequence of \eqref{Assumption:: m a.s. bounded}). Hence,  $\sigma(H_{\lambda}) = \lim_{n \to \infty} \sigma(H_{\lambda_n})$.
	The statement now follows from the fact that by Proposition \ref{Proposition: Kunz-Souillard-type theorem}, $\mathbb P(\forall \lambda \in \Q \cup \mathrm D:\, \sigma(H_\lambda) = \Sigma_\lambda) = 1$.
\end{proof}

\begin{cor}
	\label{Corollary: Support under additonal assumtpion of bounded m}
	Under the assumptions of Theorem \ref{Theorem: Spectrum in the lattice case}, additionally assume that \eqref{Assumption:: m a.s. bounded} holds. Then
	\begin{equation*}
		\Sigma = \{\lambda \in \R: \,  \mathbb P(\lambda + r(\lambda) \in \Sigma_0) > 0\} \cup \mathcal T.
	\end{equation*}
\end{cor}

\begin{proof}
	By Theorem \ref{Theorem: Anderson representation of the spectrum}, Proposition \ref{Proposition: Spectrum agrees almost surely with support of DOS measure}, Proposition \ref{Proposition: Kunz-Souillard-type theorem} and Proposition \ref{Proposition: Kunz-Souillard-type theorem simultaneously in all lambda}, we have almost surely
	\begin{equation*}
		\Sigma = \sigma(A) = \big\{\lambda \in \mathbb R:\, \lambda \in \sigma(H_\lambda)  \big\} \cup \mathcal T =  \{\lambda \in \R:\, \mathbb P(\lambda + r(\lambda) \in \Sigma_0) > 0\} \cup \mathcal T. \qedhere
	\end{equation*}
\end{proof}

\begin{proof}[Proof of Theorem \ref{Theorem: Spectrum in the lattice case}]
	We approximate $\mathcal G$ by a sequence of sprinkled graphs with bounded degrees. For sufficiently large $n\in \mathbb N$, we define the sprinkling distribution $\mathrm P_n(\, \cdot\, ) \coloneqq \mathrm P(\, \cdot \, |m\leq n)$,
	and denote by $(\tilde{\mathcal G}_{n},\tilde{o}_{n}) $ the $N\to \infty$ limit of the sprinkling of $\mathbb G_N$ with $\mathrm P_n$ as given by Theorem \ref{Theorem: Limit Theorem}. We denote by $\Sigma^n$ the support of the density of states of $(\tilde{\mathcal G}_{n},\tilde{o}_{n})$, by $A_n$ the adjacency operator of $\tilde{\mathcal G}_{n}$ and by $\mathcal T_n$ the respective tuning fork eigenvalues. With Skorokhod's representation theorem, we might assume without loss of generality that for every $R>0$ we have $(\tilde{\mathcal G}_{n},\tilde{o}_{n})_R = (\mathcal G, o)_R$ for sufficiently large $n$. Then we have for all $\psi \in l^2_c(\mathcal V)$ that $A_n \psi \to A \psi$ almost surely. Since $l^2_c(\mathcal V)$ is a common kernel, we have \cite[Theorem VIII.25]{RS80} almost surely $A_n \to A$ in the strong resolvent sense. This implies \cite[Theorem VIII.24]{RS80} that almost surely $\sigma(A) \subseteq \lim_{n\to \infty} \sigma(A_n)$ and henceforth $\Sigma \subseteq \lim_{n\to \infty} \Sigma^n$ by Proposition \ref{Proposition: Spectrum agrees almost surely with support of DOS measure}. On the other hand, by a similar argument as in the proof of Proposition \ref{Proposition: Kunz-Souillard-type theorem} one sees that $\Sigma^n \subseteq \Sigma$ for all $n\in \N$. As $\Sigma$ is closed, we hence have $\lim_{n\to \infty} \Sigma^n \subseteq \Sigma$ as well. We therefore obtain with Corollary \ref{Corollary: Support under additonal assumtpion of bounded m} that
	\begin{equation*}
		\Sigma = \lim_{n\to \infty}\Sigma^n = \lim_{n\to \infty} \{\lambda \in \R:\, \mathbb P(\lambda + r_n(\lambda) \in \Sigma_0) > 0\} \cup \mathcal T_n.
	\end{equation*}	
	Since
	\begin{equation*}
		\{\lambda \in \R:\, \mathbb P(\lambda + r_n(\lambda) \in \Sigma_0) > 0\} \cup \mathcal T_n \subseteq \{\lambda \in \R:\, \mathbb P(\lambda + r(\lambda) \in \Sigma_0) > 0\} \cup \mathcal T 
	\end{equation*}
	for all $n\in \N$, we directly obtain
	\begin{equation*}
		\lim_{n\to \infty}\{\lambda \in \R:\, \mathbb P(\lambda + r_n(\lambda) \in \Sigma_0) > 0\} \cup \mathcal T_n \subseteq \overline{\{\lambda \in \R:\, \mathbb P(\lambda + r(\lambda) \in \Sigma_0) > 0\} \cup \mathcal T }.
	\end{equation*}
	On the other hand, if $\mathbb P(\lambda + r(\lambda) \in \Sigma_0) > 0$ or $\lambda \in \mathcal T$ then $\mathbb P(\lambda + r_n(\lambda) \in \Sigma_0) > 0$ or $\lambda \in \mathcal T_n$ for all sufficiently large $n$ leading to 
	\begin{equation}
		\label{Equation: Set inclusion for spectra}
		\{\lambda \in \R:\, \mathbb P(\lambda + r(\lambda) \in \Sigma_0) > 0\} \cup \mathcal T \subseteq \lim_{n\to \infty} \{\lambda \in \R:\, \mathbb P(\lambda + r_n(\lambda) \in \Sigma_0) > 0\} \cup \mathcal T_n = \Sigma.
	\end{equation}
	Taking the closure in \eqref{Equation: Set inclusion for spectra} finishes the proof.
\end{proof}

\section{Concentration of mass: proof of Theorem \ref{Theorem: concentration of mass inside of Sigma_0}}
\label{Section: concentration of mass}
We now apply the Anderson-percolation representation in order to prove Theorem \ref{Theorem: concentration of mass inside of Sigma_0} which states that the sprinkling increases the expected number of eigenvalues (not considering tuning forks) predominately in intervals in which many eigenvalues of the attached graphs lie. To simplify notation, let us define
\begin{equation*}
	\mu_{N, 0} \coloneqq \frac{1}{|\mathbb V_N|}\sum_{\lambda \in \sigma(B_N)} \gamma(B_N, \lambda) \delta_{\lambda}, \quad
	\hat \mu_{N, \operatorname{TF}} \coloneqq \frac{1}{|\mathbb V_N|}\sum_{\lambda \in \sigma(A_N)} \gamma_{\operatorname{TF}, N}(A_N, \lambda) \cdot \delta_{\lambda}
\end{equation*}
and let
\begin{equation} \label{Equation: Definition hatnu_N}
	\hat \mu_{N} \coloneqq \frac{1}{|\mathbb V_N|}\sum_{\lambda \in \sigma(A_N)} \big(\gamma(A_N, \lambda) - \gamma_{\operatorname{TF}, N}(A_N, \lambda) \big) \delta_{\lambda} = \frac{|\mathcal V_N|}{|\mathbb V_N|} \cdot \mu_N - \hat \mu_{N, \operatorname{TF}}, \qquad
	\hat \nu_N \coloneqq \mathbb E[\hat \mu_N],
\end{equation}
where $\mu_N$ denotes, as previously, the empirical eigenvalue distribution of $B_N$. We start by showing convergence of $\hat \nu_N$ to $\nu$ as $N\to \infty$.

\begin{lemma}
	\label{Lemma: convergence of nu hat} 
	We have $\lim_{N\to \infty}\hat \nu_{N} = \hat \nu$ in the Kolmogorov--Smirnov metric.
\end{lemma}

\begin{proof}
	The statement readily follows from the convergence of $\nu_N$ to $\nu$ in the Kolmogorov--Smirnov metric (given by Corollary \ref{Corollary: Convergence of DOS-measures}), convergence of  $\nu_{N, \operatorname{TF}}$ to $\nu_{\operatorname{TF}}$ in total variation (given by Lemma \ref{Lemma: Total variation convergence of tuning fork contribution}) and the $L^1$ law of large numbers after writing
	\begin{align*}
		\hat \nu_N - \hat\nu &= \mathbb E\Big[\frac{\sum_{x\in \mathbb V_N} 1+m_x}{|\mathbb V_N|}(\mu_{N} - \mu_{N, \operatorname{TF}})  \Big] - (1+ \mathbb E[m])(\nu - \nu_{\operatorname{TF}}) \\
		&= \mathbb E\Big[ \Big(\frac{\sum_{x\in \mathbb V_N} 1+m_x}{|\mathbb V_N|}- (1+ \mathbb E[m]) \Big) \mu_{N}  \Big] + (1+ \mathbb E[m])(\nu_N - \nu) \\
		&\quad-\mathbb E\Big[ \Big(\frac{\sum_{x\in \mathbb V_N} 1+m_x}{|\mathbb V_N|}- (1+ \mathbb E[m]) \Big) \mu_{N, \operatorname{TF}}  \Big] - (1+ \mathbb E[m])(\mathbb E[\mu_{N, \operatorname{TF}} ] - \nu_{\operatorname{TF}}). \qedhere
	\end{align*}
\end{proof}
Using the Anderson-percolation representation in finite volume, i.e.\ Proposition \ref{Proposition: Anderson representation in finite volume}, we now show a finite-volume version of Theorem \ref{Theorem: concentration of mass inside of Sigma_0}. For $x\in \mathbb V_N$, we define
\begin{equation*}
	\mathrm N^x \coloneqq \{\lambda \in \mathbb R:\, r_x(\lambda) = 0 \}, \quad 	\xi_x \coloneqq \1_{\{\mathrm G^x \neq \emptyset\}}\sum_{\lambda \in \mathrm N^x} \delta_\lambda
\end{equation*}
such that in particular $\zeta = \mathbb E[\xi_x]$.

\begin{figure}
	\centering
	\includegraphics[width=0.6\linewidth]{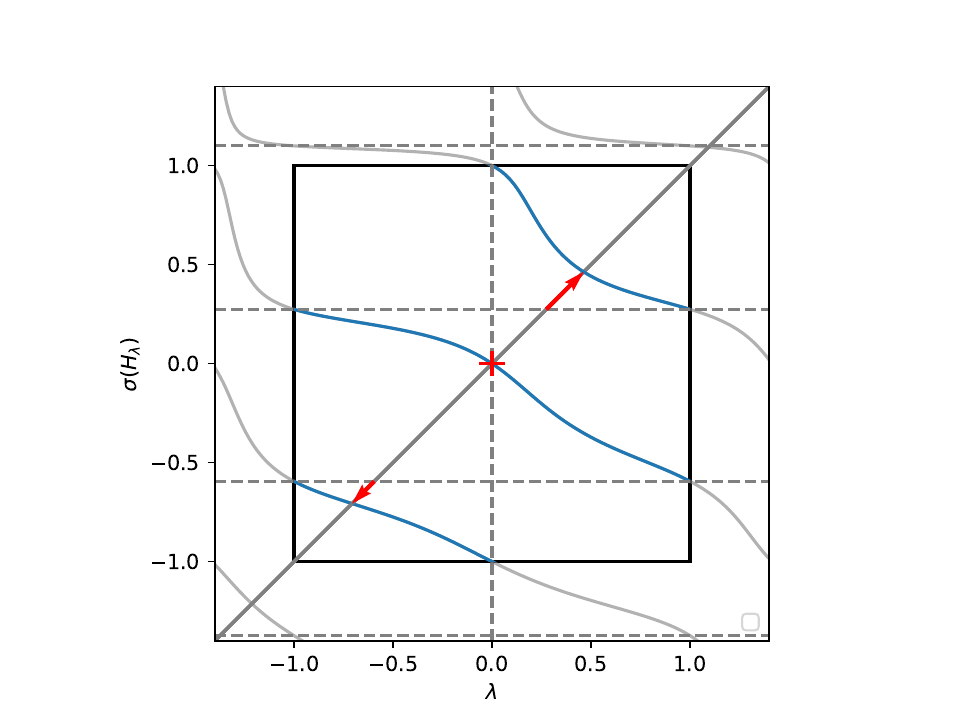}
	\caption{Illustration of the proof of Proposition \ref{Proposition: Concentration of mass in finite volume}. The horizontal dashed lines represent the eigenvalues of $\mathbb G_N$, the vertical dashed line the eigenvalue $0\in \mathrm S_x$. By attaching the attached graphs at $x$, the existing eigenvalues a repulsed from zero towards the zeros $\pm 1$ of $r_x$ while a new eigenvalue is created (in red).}
	\label{fig:andersonrepresentation3}
\end{figure}

\begin{prop}
	\label{Proposition: Concentration of mass in finite volume}
	Let $I = (-\infty, b]$ for some $b\in \mathbb R$. Then
	\begin{equation}
		\label{Equation: Concentration of mass in finite volume}
		\frac{1}{|\mathbb V_N|}\sum_{x\in\mathbb V} \xi_x(I) \leq (\hat \mu_N - \mu_{N, 0})(I) \leq	\frac{1}{|\mathbb V_N|}\sum_{x\in\mathbb V} \xi_x(I) + \1_{\{\mathrm G^x \neq \emptyset\}}.
	\end{equation}
\end{prop}

\begin{proof}
	Let $n\coloneqq |\mathbb V_N|$.
	We only show the statement for the case where $\mathbb V^+_N = \{x\}$ for some $x\in \mathbb V$, the general statement follows by inductively applying the same argument. Let $t_1< \hdots< t_k \in \mathbb R$ be such that $\mathrm N^x =  \{t_1, \hdots, t_k\}$.
	We define $t_0 \coloneqq -\infty$ and $t_{k+1} \coloneqq \infty$. There exist a unique $\beta \in (-t_0, t_{1}]$ such that $\beta \in \mathrm S^x_*$.
	Let $\alpha_1, \hdots, \alpha_n:[-\infty, \infty]\to \hat{\mathbb R}$ be parametrizations of the eigenvalues of  $H_{\lambda, N} = B_N - r_x(\lambda) \langle \delta_x, \, \cdot \, \rangle \delta_x$,
	as given by Proposition \ref{Proposition: Parametrize eigenvalues of Anderson Hamiltonian in finite volume}, where $r_x(-\infty) = r_x(\infty) \coloneqq 0$, and where we choose $\alpha_1, \hdots, \alpha_n$ such that $\alpha_1(t_0) \leq \hdots \leq \alpha_n(t_0)$.
	Since $r_x(t_0) = r_x(t_1) = 0$, we have that $(\alpha_i(t_0))_{1\leq i \leq n}$ and $(\alpha_i(t_1))_{1\leq i \leq n}$ are the eigenvalues of $B_N$.
	By eigenvalue interlacing, the following holds:
	The functions $\alpha_2, \hdots, \alpha_n$ are decreasing on $(t_0, t_1]$ while $\alpha_1$ is decreasing on $(-t_0, \beta)$ and on $(\beta, t_{1}]$ with $\lim_{\lambda \uparrow \beta} \alpha_1(\lambda) = -\infty$, $\lim_{\lambda \downarrow \beta} \alpha_1(\lambda) = \infty$. We have $\alpha_1(t_1) = \alpha_n(t_0)$, $\alpha_i(t_1) = \alpha_{i-1}(t_0)$ for $i\in \{2, \hdots, n\}$. Hence, if $l\in [n]$ is such that $\{j\in [n]:\, \alpha_j(t_1)\in (t_0, t_{1}]  \} = \{1, \hdots, l\}$,
	then, by the intermediate value theorem, there are $l$ solutions $\lambda_i \in (t_0, t_1]$ of $\lambda_i = \alpha_i(\lambda_i)$ corresponding to $i=1, \hdots, l$, and there is an additional solution in $(t_0, t_{1}]$ corresponding to $i=l+1 \, \operatorname{mod}(n)$, see Figure \ref{fig:andersonrepresentation3} for an illustration. By Proposition \ref{Proposition: Anderson representation in finite volume}, we have that $|\mathbb V_N|\hat \mu_N((t_0, t_{1}])$ is exactly the total number of solutions to one of the equations $\lambda = \alpha_i(\lambda)$, $\lambda \in (t_0, t_{1}]$.
	Hence, we obtain $(\hat \mu_N - \mu_{N, 0})((t_0, t_{1}]) = \frac{1}{|\mathbb V_N|}$.
	The same argument shows that for all $b\in (t_0, t_{1})$ we have $0 \leq (\hat \mu_N - \mu_{N, 0})((t_0, b]) \leq \frac{1}{|\mathbb V_N|}$.
	Repeating the same argument on all intervals $(t_i, t_{i+1}]$ with $i=2, \hdots, k$ and summing over all $i$ such that $t_i < b$ yields
	\begin{equation*}
		\frac{1}{|\mathbb V_N|} \xi_x(I) \leq (\hat \mu_N - \mu_{N, 0})((-\infty, b]) \leq \frac{1}{|\mathbb V_N|}(\xi_x(I)+1). \qedhere
	\end{equation*}
\end{proof}

\begin{proof}[Proof of Theorem \ref{Theorem: concentration of mass inside of Sigma_0}]
	After taking the expected value in \eqref{Equation: Concentration of mass in finite volume}, this follows from Lemma \ref{Lemma: convergence of nu hat}.
\end{proof}

\section{Expansion for sparse sprinkling}
\label{Section: sparse sprinkling}
In this section, we prove the expansion of the density of states in the limit of sparse sprinkling as given by Theorem \ref{Theorem: small p expansion in dual of Ck} and study its properties. We start in Subsection \ref{Subsection: Small p expansion in finite volume} by proving the equivalent of Theorem \ref{Theorem: small p expansion in dual of Ck} in finite volume before then taking in Subsection \ref{Subsection: Small p expansion in infinite volume} the infinite-volume limit. In Subsection \ref{Subsection: Regularity of first order term}, we prove Proposition \ref{Proposition: Condition that phi locally agrees with measure} in which we give conditions under which the first-order term $p\phi$ in our expansion is (locally) a signed measure. In Subsection \ref{Subsection: Necessity of Smoothening}, we prove Propositions \ref{Proposition: Smoothening is necessary 1} and \ref{Proposition: Smoothening is necessary 2} which show that our expansion in general only holds in a weak sense. Finally, we discuss in Subsection \ref{Subsection: Expansion around p=1} how one can derive under the Assumption \eqref{Assumption: Detetministic attached graphs} a similar expansion around $p=1$. 

\subsection{Expansion for sparse sprinkling in finite volume}
\label{Subsection: Small p expansion in finite volume}
We start the proof of Theorem \ref{Theorem: small p expansion in dual of Ck} by showing its equivalent in finite volume, i.e.\ for the measures $\hat{\nu}_N$ defined in \eqref{Equation: Definition hatnu_N}, which converge to $\hat{\nu}$ as $N\to \infty$ by Lemma \ref{Lemma: convergence of nu hat}. Namely, we show
\begin{prop}
	\label{Propositio: Small p expansion in finite volume}
	Under Assumption \ref{Assumption: Assumption for small p expansion}, the function
	\begin{equation}
		h_N \coloneqq \mathbb E\bigg[\frac{\partial_z (g_N r)}{1 - g_N r} + \sum_{\lambda \in \mathrm S_*}\frac{1}{\lambda-z}\Big| \mathrm G\neq \emptyset \bigg] .
	\end{equation}
	is the Stieltjes transform of a signed measure $\phi_N$. If $R_N$ is the unique signed measure such that
	\begin{equation*}
		\hat \nu_N = \nu_{0, N} +  p \phi_N +   p^2 R_N,
	\end{equation*}
then $\|R_N\|_{(C^{8}(a, b))'}  \leq 2^{6} |b-a| \mathrm C(a, b, \mathrm P_1)(1-p)^{-1}$ for any $a<b$. Moreover, we have
	\begin{align}
		\|\phi_N\|_{C^1_c(a, b)} &\leq |b-a| \mathbb E[m|\mathrm G \neq \emptyset],  \label{Equation: Estimate on phi in finite volume}  \\ 
		\sup_{z\in \C \setminus \R}  |\operatorname{Im}(z) h_N(z)| &\leq 5 \cdot \mathbb E[m|\mathrm G \neq \emptyset] \label{Equation: Estimate on h in finite volume}.
	\end{align}
\end{prop}	
For the rest of Subsection \ref{Subsection: Small p expansion in finite volume}, we fix $N$ and omit the $N$-dependency in the notation. 
For the proof of Proposition \ref{Propositio: Small p expansion in finite volume}, we treat $p = \mathrm P(\mathrm G \neq \emptyset)$ as a continuous parameter and sprinkle $\mathbb G$ with the sprinkling distribution
\begin{equation} \label{def_Pp}
	\mathrm P_p \coloneqq p \mathrm P(\, \cdot \, | \mathrm G \neq \emptyset) + (1-p) \delta_\emptyset
\end{equation}
in order to interpolate between the core graph $\mathbb G$ (corresponding to the choice $p = 0$), and our sprinkled graph $\mathcal G$ (corresponding to the choice $p = \mathrm P(\mathrm G \neq \emptyset)$). We then differentiate the accordingly defined measure $\hat \nu_p$ in $p$. In order to do so, we define the following coupling: Assume that $\mathbb G$ is defined on the probability space $(\Omega, \mathcal A, \mathbb P)$.
\begin{enumerate}
	\item Let $\mathrm P_1 \coloneqq \mathrm P(\, \cdot \, |\mathrm G \neq \emptyset)$
	and let $(\tilde{\mathrm G}^x, \tilde \varphi^x)_{x\in \mathbb V}$ be a family of iid random elements with distribution $\mathrm P_1$, defined on some probability space $(\Sigma, \mathcal F, \mathbf P)$.
	\item Let $\chi$ be the identity on $\{0, 1\}^\mathbb V$. We define $\mathcal G$ on $\Omega\times \Sigma \times \{0, 1\}^{\mathbb V}$ by sprinkling $\mathbb G$ with $(\mathrm G^x, \varphi^x)_{x\in \mathbb V}$ defined by
	\begin{equation*}
		(\mathrm G^x, \varphi^x) \coloneqq \begin{cases}
			(\tilde{\mathrm G}^x, \tilde \varphi^x) \quad &\text{ if }\chi_x = 1\\
			\, (\emptyset, \emptyset) \quad &\text{ if }\chi_x = 0
		\end{cases}
	\end{equation*}
	\item For $p\in [0, 1]$ we equip $\Omega\times \Sigma \times \{0, 1\}^{\mathbb V}$ with $\mathbb P_p \coloneqq \mathbb P \otimes \mathbf P \otimes \operatorname{Ber}(p)^{\otimes \mathbb V}$.
\end{enumerate}
For $p=0$ the graph $\mathcal G$ coincides in distribution with $\mathbb G$ while for $p=\mathrm P(\mathrm G \neq \emptyset)$ we obtain the usual iid sprinkling of $\mathbb G$ with sprinkling distribution $\mathrm P$.
We denote expected values taken with respect to $\mathbb P_p$ by $\mathbb E_p$. For $x\in \mathbb V$, we define $\tilde r_x$, $\tilde m_x$, $\tilde m_x^*$, $ \tilde{\mathrm  T}^x$, $\tilde{\mathrm S}^x$, $\tilde{\mathrm S}^x_*$ accordingly, i.e.\ with $(\mathrm G^x, \varphi^x)$ being replaced by $(\tilde{\mathrm G}^x, \tilde \varphi^x)$. We omit, as usual, the super/subscript $x$ to denote generic versions of these random variables drawn from $\mathrm P_1$ (independently from all other sources of randomness).

\subsubsection{Lipschitz estimates}
We start by showing local Lipschitz continuity at $p = 0$ of the cumulative distribution function $\hat{F}_p$ and the Stieltjes transform $\hat{s}_p$ of $\hat{\nu}_p$. This will later yield \eqref{Equation: Estimate on phi in finite volume} and \eqref{Equation: Estimate on h in finite volume}. In Subsection \ref{Subsection: Expansion around p=1}, we discuss how one can show global Lipschitz estimates.

\begin{lemma}
	\label{Proposition: C1 distance estimate}
	We have
	\begin{equation}
		\label{Equation: Estimate C1 distance}
		\|\hat \nu_p - \nu_0\|_{C^1_c(a, b)} \leq  p |b-a|\mathbb E[\tilde m].
	\end{equation}
	
\end{lemma}

\begin{proof}
With the finite-volume version of Theorem \ref{Theorem: concentration of mass inside of Sigma_0} (see Proposition \ref{Proposition: Concentration of mass in finite volume}), we obtain $\|\hat F_p - F_0 \|_\infty \leq p\mathbb E[\tilde m]$,
since the total number of zeros of $r$ on the event $\{\mathrm G\neq \emptyset\}$ is $|\mathrm S_{*}|-1$ and hence does not exceed $m-1$. For every $a<b$ and $f\in C^1_c(a, b)$ we obtain by integration by parts,
	\begin{equation*}
		\int_\R (\hat \nu_p - \nu_0)(\mathrm dx) \, f(x) = -\int_\R (\hat \nu_p - \nu_0)(\mathrm dx) \int_{x}^b \mathrm dy \, f'(y) = -\int_a^b \mathrm dy \, f'(y) (\hat F_p(y) - F_0(y)),
	\end{equation*}
	and \eqref{Equation: Estimate C1 distance} follows. 		
\end{proof}

\begin{lemma}
	\label{Lemma: ST of tuning fork contribution}
	The Stieltjes transform of
	\begin{equation*}
		\hat \nu_{p, \operatorname{TF}} \coloneqq \frac{1}{|\mathbb V|} \mathbb E_p\Big[\sum_{\lambda \in \sigma(A)} \gamma_{\operatorname{TF}}(A, \lambda) \delta_\lambda \Big]
	\end{equation*}
	is given by
	\begin{equation*}
		\hat s_{p, \operatorname{TF}} = p\mathbb E\Big[\tilde q - \sum_{\lambda \in \tilde{\mathrm S}_*}\frac{1}{\lambda - z} \Big] \quad \text{ where } \quad \tilde q \coloneqq \operatorname{tr}\, (\tilde{\mathrm T} - z)^{-1}.
	\end{equation*}
\end{lemma}

\begin{proof}
	By Proposition \ref{Proposition: Anderson representation in finite volume} we have 
	\begin{equation*}
		\hat \nu_{p, \operatorname{TF}} = \mathbb E_p\Big[ \sum_{\lambda \in \mathrm S} \big(\gamma(T, \lambda) -\1_{\{\lambda \in \mathrm S_*\} } \big) \delta_\lambda\Big]
	\end{equation*}
	such that
	\begin{equation*}
		\hat s_{p, \operatorname{TF}} = \mathbb E_p\Big[\sum_{\lambda \in \mathrm S} \frac{\gamma(\mathrm T, \lambda)}{\lambda - z} - \sum_{\lambda \in \mathrm S_*}\frac{1}{\lambda - z}\Big]= \mathbb E_p\Big[\operatorname{tr} \, (\mathrm T - z)^{-1}  - \sum_{\lambda \in \mathrm S_*}\frac{1}{\lambda - z}\Big].
	\end{equation*}
	Since the integrand vanishes if $\mathrm G = \emptyset$, the statement follows by conditioning on $\mathrm G\neq \emptyset$.
\end{proof}

\begin{lemma}
	\label{Proposition: Estimate on difference of ST}
For all $z \in \C^+$ we have $|\hat s_p(z) - s_0(z)| \leq 5 p \mathbb E[\tilde m] \operatorname{Im}(z)^{-1}$.
\end{lemma}

\begin{proof}
	Let $\mathcal G'$ denote the graph obtained by deleting the edges
	\begin{equation*}
		\big\{  \{x, (xt) \}:\, x \in \mathbb V^+,\, t\in \mathrm V^x \text{ with } \varphi^x_t = 1 \big\}
	\end{equation*}
	from the graph $\mathcal G$ and let $A'$ be its adjacency matrix. Notice that
	\begin{equation*}
		\frac{1}{|\mathbb V|}\mathbb E_p\big[ \operatorname{tr} \, (A'-z)^{-1} \big] = s_0 + \mathbb E_p \Big[\frac{1}{|\mathbb V|} \sum_{x\in \mathbb V^+} \operatorname{tr}\, (\mathrm T^{x} -z)^{-1} \Big]  = s_0 + p \mathbb E[\tilde q]
	\end{equation*}
	such that we obtain with Lemma \ref{Lemma: ST of tuning fork contribution}
	\begin{align*}
		\hat s_p - s_0 &= |\mathbb V|^{-1} \mathbb E_p\big[ \operatorname{tr} \, (A-z)^{-1} \big] - \hat s_{p, \operatorname{TF}} - s_0 \\
		&= |\mathbb V|^{-1} \mathbb E_p\big[ \operatorname{tr} \, (A-z)^{-1} - \operatorname{tr} \, (A'-z)^{-1} \big]  + p \mathbb E\Big[\sum_{\lambda \in \tilde{\mathrm S}_*}\frac{1}{\lambda - z} \Big].
	\end{align*}
	We use the estimate \cite[Lemma F.5]{BK19}
	\begin{align*}
		\Big|\operatorname{tr}\big((A-z)^{-1}  - (A'-z)^{-1} \big) \Big| &\leq \operatorname{Rank}\big((A-z)^{-1}  - (A'-z)^{-1} \big) \big\|(A-z)^{-1}  - (A'-z)^{-1} \big\| \\
		&\leq 2 \operatorname{Rank}(A - A') \operatorname{Im}(z)^{-1}
	\end{align*}	
	where first inequality follows from singular value decomposition and the second inequality follows from the resolvent identity
	\begin{equation*}
		(A-z)^{-1}  - (A'-z)^{-1} = (A-z)^{-1}(A' - A)(A'-z)^{-1}.
	\end{equation*}
	Since $\operatorname{Rank}(A - A') = 2\sum_{x \in \mathbb V} m_x^*$, we obtain
	\begin{align*}
		|\hat s_p - s_0|  &\leq 2|\mathbb V|^{-1} \operatorname{Im}(z)^{-1} \mathbb E_p\big[ \operatorname{Rank}(A' - A) \big] + p \mathbb E\Big[\sum_{\lambda \in \tilde{\mathrm S}_*}\frac{1}{|\lambda - z|} \Big] \leq 5 p \mathbb E[\tilde m] \operatorname{Im}(z)^{-1}. \qedhere
	\end{align*}

\end{proof}	  

\subsubsection{Differentiating $\hat{\nu}_p$ in $p$}
As the next step towards the proof of Proposition \ref{Propositio: Small p expansion in finite volume}, we differentiate $\hat{\nu}_p$ in $p$.
The distribution of $\mathcal G$ conditionally on $\chi$ no longer depends on $p$ and by integrating out the randomness in $\chi$, one convinces oneself that $\hat \nu_p$ is a polynomial of degree $|\mathbb V|$ in $p$ (whose coefficients are signed measures) and hence in particular differentiable in $p$, even with respect to the total variation norm. To explicitly calculate this derivative, we  will differentiate the Stieltjes transform
\begin{equation*}
	\hat s_p(z) = \int_\R \frac{1}{t-z} \, \hat \nu_p(\mathrm dt) =  \int_\R \frac{t-\lambda + \mathrm i \eta}{(t-\lambda)^2 + \eta^2} \, \hat \nu_p(\mathrm dt), \quad z= \lambda + \mathrm i \eta \in \C \setminus\R
\end{equation*}
of $\hat \nu_p$ in $p$. After integrating out the randomness in $\mathbb G$ and $(\tilde{\mathrm G^x})_{x}$, the task becomes to differentiate an expected value taken with respect to the distribution $\operatorname{Ber}(p)^{\otimes \mathbb V}$ of $\chi$ in $p$. We define for $x\in \mathbb V$
\begin{align*}
	&\theta^x: \{0, 1\}^{\mathbb V} \to \{0, 1\}^{\mathbb V},\quad \theta^x(\omega) \coloneqq (\omega_y \vee \delta_{xy})_{y} \\
	&\theta_x: \{0, 1\}^{\mathbb V} \to \{0, 1\}^{\mathbb V},\quad \theta_x(\omega) \coloneqq (\omega_y \wedge \delta_{xy})_{y}.
\end{align*}
For a function $f:\Omega\times \Sigma \times \{0, 1\}^{\mathbb V} \to M$ and $x\in \mathbb V$ we define $\theta^xf$ and $\theta_x f$ by 
\begin{equation*}
	(\theta_x f)(\omega_1, \omega_2, \omega_3) \coloneqq f(\omega_1, \omega_2, \theta^x \omega_3), \quad  (\theta_x f)(\omega_1, \omega_2, \omega_3) \coloneqq f(\omega_1, \omega_2, \theta_x \omega_3).
\end{equation*}
Moreover, for $x, y\in \mathbb V$ with $x\neq y$, we set $\theta^x_y f \coloneqq \theta^x \theta_y f$, $\theta_{xy} f  \coloneqq \theta_x \theta_y f$, $\theta^{xy} f \coloneqq \theta^x \theta^y f$.
\begin{lemma}
	\label{Lemma: Derivative Bernoulli expectation}
	Let $f:\Omega \times \Sigma \times \{0, 1\}^{\mathbb V} \to \C$ be a bounded measurable function. Then for $p\in [0, 1]$
	\begin{equation*}
		\partial_p \mathbb E_p[f] =  \sum_{x \in \mathbb V} \mathbb E_p[\theta^x f - \theta_x f].
	\end{equation*}
\end{lemma}

\begin{proof}
	By Fubini's theorem, we might assume without loss of generality that $f$ is of the form $f(\omega_1, \omega_2, \omega_3) = g(\omega_3)$ for some function $g:\{0, 1\}^{\mathbb V} \to \mathbb C$.
	If $g$ is of the form $g(\omega) = \prod_{x} g_x(\omega_x)$ with $g_x: \{0, 1\} \to \mathbb C$ for $x\in \mathbb V$ then the equality simply follows from the product rule. For general $g$, it suffices to notice that $g$ can be written as a linear combination of functions of the form $\prod_{x\in I} \1_{\{ \chi_x = 1 \}}$ with $I \subseteq \mathbb V$, see e.g.\ \cite[Lemma 3.19]{FV17}.
\end{proof}

In order to differentiate $\hat s_p$ (which we express in terms of the expected trace of the resolvent), we hence need to study the change of the diagonal elements of the resolvent under attaching the attached graphs at a single vertex $x \in \mathbb V$. 
For $z\in\C^+$, we denote
\begin{equation*}
	G(z) = (A-z)^{-1}, \quad \tilde {\mathrm R}^{x}(z) \coloneqq (\tilde{\mathrm T}^{x}-z)^{-1}.
\end{equation*}

\begin{lemma}
	\label{Lemma: attach n-blocks to one vertex}
	We have for all $x\in \mathbb V$
	\begin{align}
		&\forall s,t\in \mathrm V^x: &\theta^x G_{(xs)(xt)} = \tilde{\mathrm R}^{x}_{st} +\frac{\theta_x G_{xx} ( \tilde{\mathrm R}^x \varphi^x)_s  (\tilde{\mathrm R}^x \varphi^x)_t }{1 - \theta_x G_{xx} \tilde r_x} \label{eq: Schur complement I} \\
		&\forall v, w \in \theta_x\mathcal V:  &\theta^x G_{vw} =  \theta_x G_{vw} + \frac{(\theta_x G_{xv})(\theta_x G_{xw}) \tilde r_x }{1 - \theta_x G_{xx} \tilde r_x}. \label{eq: Schur complement II}
	\end{align}
\end{lemma}

\begin{proof}
	To simplify notation, we assume without loss of generality that $\chi_x = 0$.
	We write $\theta^x A$ in the block form
	\begin{equation*}
		\theta^x A \coloneqq 
		\begin{pmatrix}
			A & V \\
			V^* &  \tilde{\mathrm  T}^x
		\end{pmatrix}
		\quad \text{ where }
		V = \delta_x \cdot (\varphi^x)^T
	\end{equation*}
	By the Schur-complement formula, $\theta^x G$ is of the form
	\begin{align*}
		\label{Equation: Schur complement attach one graph}
		\tag{SC}
		&\theta^x G = 
		\begin{pmatrix}
			G + G V(T-z-V^* G V)^{-1} V^* G &  * \\
			*					  &\tilde{\mathrm R}^x + \tilde{\mathrm R}^x V^*(A-z-V \tilde{\mathrm R}^x V^*)^{-1} V \tilde{\mathrm R}^x \\	
		\end{pmatrix} \\
		&= \begin{pmatrix}
			G + G \delta_x \big \langle \varphi^x, \big[T-z- G_{xx }\varphi^x (\varphi^x)^T\big]^{-1} \varphi^x \big \rangle \delta_x^T G &  * \\
			*					  &\tilde{\mathrm R}^x + \tilde{\mathrm R}^x \varphi^x \big \langle \delta_x, \big[A-z- \tilde r_x \delta_x \delta_x^T\big]^{-1}\delta_x \big \rangle  (\varphi^x)^T \tilde{\mathrm R}^x \\	
		\end{pmatrix}.
	\end{align*}
	With the Sherman–Morrison formula
	\begin{equation*}
		(A-z - \tilde r_x \delta_x \delta_x^T)^{-1} = G + \frac{\tilde r_x}{1 - G_{xx}\tilde r_x}G\delta_x^T \delta_x G
	\end{equation*}
	such that we obtain with \eqref{Equation: Schur complement attach one graph}
	\begin{equation*}
	(\theta^x G)_{(xs)(xt)} = \tilde{\mathrm R}^x_{st}  + (R^x\varphi^x)_s  (R^x \varphi^x)_t \Big(G_{xx} + \frac{G_{xx}^2 \tilde r_x}{1-G_{xx} \tilde r_x} \Big)
	\end{equation*}
	which is is \eqref{eq: Schur complement I}. The equality in \eqref{eq: Schur complement II} follows in the same manner.
\end{proof}

\begin{prop}
	\label{Proposition: Differentiation of ST in finite volume}
	Let $\mathbb o$ be drawn uniformly at random from $\mathbb V$ and let $g = \theta_{\mathbb o} G_{\mathbb o \mathbb o}$. Then for all $p\in [0, 1]$
	\begin{equation*}
		h_p \coloneqq \partial_p \hat s_p = \mathbb E_p\Big[\frac{\partial_z (g \tilde r)}{1 - g \tilde r} + \sum_{\lambda \in \tilde{ \mathrm S}_*}\frac{1}{\lambda-z} \Big] .
	\end{equation*}

\end{prop}

\begin{proof}
	Since 
	\begin{equation*}
		\hat s_p = \mathbb E_p\Big[\frac{1}{|\mathbb V|}\sum_{\lambda \in \sigma(A)} \frac{\gamma(\lambda, A)}{\lambda - z} \Big] - \hat s_{p, \operatorname{TF}}(z) = \mathbb E_p\Big[\frac{1}{|\mathbb V|} \operatorname{tr}(G) \Big] - \hat s_{p, \operatorname{TF}}(z)
	\end{equation*}
	we can write with Lemma \ref{Lemma: ST of tuning fork contribution}
	\begin{equation}
		\label{Equation: Derivative of ST without TF}
		\partial_p \hat s_p = \frac{1}{|\mathbb V|} \partial_p \mathbb E_p\Big[\sum_{v\in \mathcal V} G_{vv}\Big] - \mathbb E[\tilde q] + \mathbb E\Big[\sum_{\lambda \in \tilde{\mathrm S}_*}\frac{1}{\lambda - z}\Big].
	\end{equation}
	We split the sum in the first expected value into contributions from vertices in $\mathbb V$ and from vertices in the attached graphs. From the former, we have by Lemma \ref{Lemma: attach n-blocks to one vertex} the contribution 
	\begin{align}
		\label{Equation: eq for term 1}
		T_1 \coloneqq \partial_p \mathbb E_p\Big[\sum_{x\in \mathbb V} G_{xx} \Big]
		&= \sum_{y\in \mathbb V} \sum_{x\in \mathbb V} \mathbb E_p\Big[\theta^y G_{xx} - \theta_y G_{xx}  \Big] \nonumber \\
		&=\sum_{y\in \mathbb V} \sum_{x\in \mathbb V} \mathbb E_p\Big[\frac{ (\theta_yG_{xy})^2 \tilde r_y }{1- (\theta_yG_{yy}) \tilde r_y} \Big] = |\mathbb V| \mathbb E\Big[\sum_{x\in \mathbb V} \frac{ (\theta_\mathbb o G_{\mathbb o x})^2  \tilde r_\mathbb o }{1- (\theta_\mathbb oG_{\mathbb o \mathbb o}) \tilde r_\mathbb o}\Big].
	\end{align}
	On the other hand
	\begin{align*}
		&\partial_p \mathbb E_p\Big[\sum_{x \in \mathbb V}\1_{\{\chi_x = 1\}}\sum_{t \in \mathrm V^x } G_{(xt)(xt)} \Big]\\
		&= \sum_{x \in \mathbb V} \sum_{y\in \mathbb V}\mathbb E_p\Big[\theta^y \1_{\{\chi_x = 1\}}\sum_{t \in \mathrm V^x } G_{(xt)(xt)}- \theta_y \1_{\{\chi_x = 1\}}\sum_{t\in \mathrm V^x } G_{(xt)(xt)}\Big] = T_2 + T_3
	\end{align*}
	where 
	\begin{equation*}
		T_2 \coloneqq  \sum_{x \in \mathbb V} \sum_{y\neq x}\mathbb E_p\Big[\1_{\{\chi_x = 1\}}\Big(\sum_{t \in \mathrm V^x } \theta^y G_{(xt)(xt)} - \sum_{t \in \mathrm V^x } \theta_y G_{(xt)(xt)} \Big)\Big], \quad 
		T_3 \coloneqq  \sum_{x \in \mathbb V} \mathbb E_p\Big[\theta^x \sum_{t \in \mathrm V^x } G_{(xt)(xt)}\Big].
	\end{equation*}
	With Lemma \ref{Lemma: attach n-blocks to one vertex} we have	
	\begin{equation}
		\label{Equation: eq for term 2}
		\frac{1}{|\mathbb V|} T_2 = \frac{1}{|\mathbb V|} \sum_{x\in \mathbb V} \sum_{y \neq x} \mathbb E_p\Big[\1_{\{\chi_x = 1 \}}\sum_{t\in \mathrm V^x } \frac{(\theta_yG_{(xt)y})^2 \tilde r_y}{1-(\theta_yG_{yy}) \tilde r_y } \Big] = \mathbb E_p\Big[ \sum_{x \neq \mathbb o} \1_{\{\chi_x = 1 \}}\sum_{t\in \mathrm V^x } \frac{(\theta_\mathbb o G_{(xt)\mathbb o})^2 \tilde r_\mathbb o}{1-(\theta_\mathbb oG_{\mathbb o\mathbb o}) \tilde r_\mathbb o }\Big].
	\end{equation}
	Combining \eqref{Equation: eq for term 1} and \eqref{Equation: eq for term 2} yields
	\begin{align}
		\label{Equation: eq for sum of term 1 and term 2}
		\frac{1}{|\mathbb V|}\big(T_1 + T_2\big) &=  \mathbb E_p\Big[ \sum_{x \in \mathbb V} \frac{ (\theta_\mathbb o G_{\mathbb o x})^2 \tilde r_\mathbb o }{1- (\theta_\mathbb oG_{\mathbb o \mathbb o}) \tilde r_\mathbb o} +    \sum_{x \neq \mathbb o} \1_{\{\chi_x = 1 \}}\sum_{t\in \tilde{\mathrm V}_x} \frac{(\theta_\mathbb o G_{(xt)\mathbb o})^2 \tilde r_\mathbb o}{1-(\theta_\mathbb oG_{\mathbb o\mathbb o}) \tilde r_\mathbb o } \Big] \nonumber \\
		&= \mathbb E_p \Big[ \frac{(\theta_\mathbb o G^2)_{\mathbb o \mathbb o} \tilde r_\mathbb o}{1-(\theta_\mathbb oG_{\mathbb o\mathbb o}) \tilde r_\mathbb o } \Big] = \mathbb E_p \Big[ \frac{ (\partial_z g) \tilde r_\mathbb o}{1-g \tilde r_\mathbb o } \Big]
	\end{align}
	were we used that there are no attached graphs in $\mathbb o$ in the graph $\theta_\mathbb o \mathcal G$. 
	
	For the term $T_3$, we have
	\begin{align}
		\label{Equation: eq for third term}
		\frac{1}{|\mathbb V|} T_3 
		&= \frac{1}{|\mathbb V|} \sum_{x\in \mathbb V} \mathbb E_p \Big[ \tilde{\mathrm R}^x_{tt} + \frac{\theta_xG_{xx} \sum_{t\in \tilde{\mathrm V}_x} \big[(\tilde {\mathrm R}^x \varphi^x)_t \big]^2 } {1-\theta_xG_{xx}\tilde r_x}  \Big] \nonumber \\
		&= \mathbb E[\tilde q] + \frac{1}{|\mathbb V|} \sum_{x\in \mathbb V} \mathbb E_p \Big[ \frac{\theta_xG_{xx} \sum_{t\in \tilde{\mathrm V}_x} \big[(\tilde {\mathrm R}^x \varphi^x)_t \big]^2 } {1-\theta_xG_{xx}\tilde r_x}  \Big].
	\end{align}
	As $\tilde{\mathrm R}^x$ is symmetric and $\varphi^x$ real, we have
	$
		(\tilde{\mathrm R}^x \varphi)_t = \overline{\big((\tilde{\mathrm R}^x)^* \varphi \big)_t}
	$
	for any $x\in \mathbb V$ and $t\in \tilde{\mathrm V}_x$ such that
	\begin{equation*}
		\sum_{t\in \tilde{\mathrm V}_x} \big[(\tilde {\mathrm R}^x \varphi^x)_t \big]^2 = \big \langle (\tilde{\mathrm R}^x)^* \varphi^x, \tilde{\mathrm R}^x \varphi^x \big \rangle = \big \langle \varphi^x, (\tilde{\mathrm R}^x)^2 \varphi^x \big \rangle = \partial_z \tilde r_x
	\end{equation*}
	for any $x\in \mathbb V$. We hence obtain with \eqref{Equation: eq for third term}
	\begin{equation}
		\label{Equation: eq for term 3}
		\frac{1}{|\mathbb V|} T_3 = \mathbb E[\tilde q] + \mathbb E_p \Big[ \frac{g \cdot \partial_z \tilde r} {1-g\tilde r}\Big].
	\end{equation}
	Combining \eqref{Equation: eq for sum of term 1 and term 2} and \eqref{Equation: eq for term 3} yields
	\begin{equation*}
		\partial_p \mathbb E_p\Big[ \frac{1}{| \mathbb V|} \sum_{v\in \mathcal V} G_{vv} \Big] =  \frac{1}{|\mathbb V|}(T_1 + T_2 + T_3) =  \mathbb E[\tilde q] + \mathbb E_p \Big[ \frac{g \cdot \partial_z \tilde r} {1-g\tilde r}\Big] + \mathbb E_p \Big[ \frac{ (\partial_z g) \tilde r}{1-g \tilde r }\Big],
	\end{equation*}
	and the claim follows from \eqref{Equation: Derivative of ST without TF}.
\end{proof}

Proposition \ref{Proposition: Differentiation of ST in finite volume} in particular yields for $p=0$ that the function $h$ is the Stieltjes transform of the signed measure $\partial_p \hat \nu_p|_{p=0}$. Lemma \ref{Proposition: C1 distance estimate} and Lemma \ref{Proposition: Estimate on difference of ST} yield \eqref{Equation: Estimate on phi in finite volume} and \eqref{Equation: Estimate on h in finite volume}. In order to finish the proof of Proposition \ref{Proposition: Differentiation of ST in finite volume} it is left to bound the second derivative $\partial_p^2 \hat \nu_p$ with respect to $\|\cdot \|_{(C^8)'}$.

\subsubsection{Bounding the second derivative}
In order bound the second derivative $\partial_p^2 \hat \nu_p$ with respect to $\|\cdot \|_{(C^8)'}$, we show that
\begin{equation}
	\label{Equation: Desired bound on second derivative of ST}
	|\partial_p^2 \hat s_p(z)| \leq \mathrm C(a, b, \mathrm P_1) (1-p)^{-1} \eta^{-7}, \quad  \forall z = \lambda + \mathrm i \eta \in (a, b) \times (0, 1).
\end{equation}
Applying the Helffer-Sjöstrand formula will then yield the desired bound on $\partial_p^2 \hat \nu_p$ in  $\|\cdot \|_{(C^8_c(a, b))'}$. In order to show \eqref{Equation: Desired bound on second derivative of ST}, we follow the following strategy: We apply Lemma \ref{Lemma: Derivative Bernoulli expectation} and Proposition \ref{Proposition: Differentiation of ST in finite volume} in order to express the second derivative  $\partial_p^2 \hat s_p$ as
\begin{equation*}
	\partial_p^2 \hat s_p = \partial_p \mathbb E_p\Big[\frac{\partial_z(gr)}{1-gr} \Big] = \mathbb E_p \bigg[\frac{1}{|\mathbb V|}\sum_{y}\sum_{x\neq y} \frac{\partial_z(\theta^x_y G_{yy} \tilde r_y)}{1-\theta^x_y G_{yy}\tilde r_y} -  \frac{\partial_z(\theta_{xy}G_{yy} \tilde r_y)}{1-\theta_{xy}G_{yy}\tilde r_y}  \bigg].
\end{equation*}
After applying Lemma \ref{Lemma: attach n-blocks to one vertex} in order to calculate $\theta_x^y G_{yy} - \theta_{xy} G_xx$ we then need to estimate expressions of the form
\begin{equation*}
	\mathbb E\Big[\frac{1}{|\mathbb V|} \sum_{y\in \mathbb V} \sum_{x\neq y} C_{xy} \theta_{xy}G_{xy}^2  \Big]
\end{equation*}
where we can bound $|C_{xy}| \leq C m_x^2 \tilde m_y^2 \eta^{-5}$ using the upcoming Lemma \ref{Lemma: Estimates on diagonal factors}. Since $\theta_{xy} G$ is independent of $(\tilde m_x, \tilde m_y)$ we hence need to bound
\begin{equation}
	\label{Equation: Double sum of Greens function}
	\mathbb E\Big[\frac{1}{|\mathbb V|} \sum_{y\in \mathbb V} \sum_{x\neq y} |\theta_{xy}G_{xy}^2|  \Big].
\end{equation}
In order to do so, we estimate the inner sum in \eqref{Equation: Double sum of Greens function} in terms of the diagonal elements of a matrix product. We write
\begin{equation*}
	\mathbb E[|\theta_{xy}G_{xy}^2|] = \mathbb E\big[|\theta_{y}G_{xy}^2| \big| \chi_x = 0 \big]  \leq (1-p)^{-1}\mathbb E\big[|\theta_{y}G_{xy}|^2 \big]
\end{equation*}
in order to obtain
\begin{align*}
	\mathbb E\Big[\frac{1}{|\mathbb V|} \sum_{y\in \mathbb V} \sum_{x\neq y} |\theta_{xy}G_{xy}^2|  \Big] &\leq (1-p)^{-1} \mathbb E\Big[\frac{1}{|\mathbb V|} \sum_{y\in \mathbb V} \sum_{x \in \mathbb V} |\theta_{y}G_{xy}^2|  \Big] \\
	&\leq (1-p)^{-1} \mathbb E\Big[\frac{1}{|\mathbb V|} \sum_{y\in \mathbb V}  \big[(\theta_y G^*)(\theta_y G)\big]_{yy}  \Big] \\
	&\leq (1-p)^{-1} \mathbb E\Big[\frac{1}{|\mathbb V|} \sum_{y\in \mathbb V} \|\theta_yG^* \theta_yG\|  \Big] \leq (1-p)^{-1} \eta^{-2}.
\end{align*}
We start by showing the aforementioned  Lemma which will yield the estimate on $|C_{xy}|$.

\begin{lemma}
	\label{Lemma: Estimates on diagonal factors}
	For every $x\in \mathbb V$, $w\in \mathbb C^+$ and $z \in \lambda + \mathrm i \eta \in  \mathbb C^+$
	\begin{equation*}
		\bigg| \frac{\tilde r_x(z)}{1-w\tilde r_x(z)}  \bigg| \leq 2\tilde m_x \eta^{-1}, \quad \bigg|\frac{\partial_z \tilde r_x(z)}{\tilde r_x(z)}\bigg|\leq \eta^{-1}
	\end{equation*}
\end{lemma}

\begin{proof}
	We have
	\begin{equation*}
		\bigg| \frac{\tilde r_x}{1-w\tilde r_x}  \bigg| = \bigg|  \frac{1}{1/\tilde r_x-w} \bigg| \leq  \frac{1}{|\operatorname{Im}(1/\tilde r_x-w)|}  \leq  \frac{1}{|\operatorname{Im}(1/\tilde r_x)|} = \frac{|\tilde r_x|^2}{\operatorname{Im} (\tilde r_x)} = \frac{\operatorname{Re}(\tilde r_x)^2}{\operatorname{Im}(\tilde r_x)} + \operatorname{Im}(\tilde r_x)
	\end{equation*}
	where we used that $1/\tilde r_x$ and $-w$ have both negative imaginary part. Notice that $(\tilde m_x^*)^{-1} \tilde r_x$ is the Stieltjes transform of a probability measure on $\mathbb R$ which we call $\mu$. Using the inequality
	\begin{align*}
		|(\tilde m_x^*)^{-1} \operatorname{Re}(\tilde r_x(z))| &\leq \int_\R \frac{|t-\lambda|}{|t-\lambda|^2 + \eta^2} \, \mu(\mathrm dt) \\
		&\leq \int_\R \frac{1}{\sqrt{|t-\lambda|^2 + \eta^2}} \, \mu(\mathrm dt) \leq   \Big(\int_\R \frac{1}{|t-\lambda|^2 + \eta^2} \, \mu(\mathrm dt) \Big)^{1/2} = \eta^{-1/2} \sqrt{(m_x^*)^{-1} \operatorname{Im}( \tilde r_x(z))}
	\end{align*}
	(which we took from \cite[Theorem B11]{BS16}) yields
	\begin{equation*}
		\frac{\operatorname{Re}(\tilde r_x)^2}{\operatorname{Im}(\tilde r_x)} \leq \tilde m_x \eta^{-1}
	\end{equation*}
	which yields the first inequality. For the second inequality, it is sufficient to notice that
	\begin{equation*}
		|(\tilde m_x^*)^{-1}\partial_z \tilde r_x(z)| \leq \int_\R \frac{1}{|x-z|^2} \, \mu(\mathrm dx) = \int_\R \frac{1}{(x-\lambda)^2+\eta^2} \, \mu(\mathrm dx) = \eta^{-1} (\tilde m_x^*)^{-1} \operatorname{Im}(\tilde r_x(z)). \qedhere
	\end{equation*}
\end{proof}

\begin{prop}
	\label{Proposition: Estimate on second derivative of ST}
	We have for every $p\in [0, 1)$ and $z = \lambda + \mathrm i \eta \in (a, b) \times (0, 1) \subset \C$
	\begin{equation*}
		|\partial_p^2 \hat s_p(z)| \leq \mathrm C(a, b, \mathrm P_1) (1-p)^{-1} \eta^{-7}.
	\end{equation*}
\end{prop}

\begin{proof}
	Using that $\theta^y \theta_y G_{yy} =  \theta_y G_{yy} = \theta_y \theta_y G_{yy}$ for all $y \in \mathbb V$, we have with Lemma \ref{Lemma: Derivative Bernoulli expectation} and Proposition \ref{Proposition: Differentiation of ST in finite volume} 
	\begin{align*}
		\partial_p^2 \hat s_p = \partial_p \mathbb E_p\Big[\frac{\partial_z(gr)}{1-gr} \Big] &= \mathbb E_p \bigg[\frac{1}{|\mathbb V|}\sum_{y}\sum_{x\neq y} \frac{\partial_z(\theta^x_y G_{yy} \tilde r_y)}{1-\theta^x_y G_{yy}\tilde r_y} -  \frac{\partial_z(\theta_{xy}G_{yy} \tilde r_y)}{1-\theta_{xy}G_{yy}\tilde r_y}  \bigg] =  T_1 + T_2
	\end{align*}
	where 
	\begin{align*}
		T_1 &\coloneqq \mathbb E_p\bigg[\frac{1}{|\mathbb V|}\sum_{y}\sum_{x\neq y}  \frac{(\partial_z(\theta^x_y G_{yy}\tilde r_y) - \partial_z(\theta_{xy} G_{yy}\tilde r_y)}{1- \theta^x_y G_{yy}\tilde r_y} \bigg] \\
		T_2 &\coloneqq \mathbb E_p\bigg[\frac{1}{|\mathbb V|}\sum_{y}\sum_{x\neq y}  \big(\partial_z(\theta_{xy} G_{yy}\tilde r_y )\big) \Big(\frac{1}{1-\theta^x_y G_{yy}\tilde r_y } - \frac{1}{1-\theta_{xy} G_{yy}\tilde r_y } \Big)\bigg].
	\end{align*}
	We further write $T_1 = T_{11} + T_{12}$ where
	\begin{equation*}
		T_{11} = \mathbb E_p\bigg[ \frac{1}{|\mathbb V|}\sum_{y}\sum_{x\neq y}  \frac{(\theta^{x}_y G_{yy} - \theta_{xy} G_{yy})\cdot \partial_z \tilde r_y }{1-\theta^x_y G_{yy}\tilde r_y } \bigg]
		,\quad T_{12} \coloneqq \mathbb E_p\bigg[  \frac{1}{|\mathbb V|}\sum_{y}\sum_{x\neq y}  \frac{\tilde r_y \cdot \partial_z (\theta^x_y G_{yy} - \theta_{xy} G_{yy}) }{1-\theta_{y}^x G_{yy} \tilde r_y} \bigg].
	\end{equation*}
	By Lemma \ref{Lemma: attach n-blocks to one vertex}, we have for every $x\neq y$
	\begin{equation*}
		\theta^{x}_y G_{yy} - \theta_{xy} G_{yy} = \frac{\theta_{xy}G_{xy}^2\tilde r_x}{1-\theta_{xy} G_{xx}\tilde r_x }
	\end{equation*}
	so that
	\begin{align*}
		T_{11} &=  \mathbb E\bigg[\frac{1}{|\mathbb V|}\sum_{y}\sum_{x\neq y} \frac{\partial_z \tilde r_y}{1-\theta^x_y G_{yy}\tilde r_y} \frac{\theta_{xy}G_{xy}^2\tilde r_x}{1-\theta_{xy} G_{xx}\tilde r_x } \bigg]= \frac{1}{|\mathbb V|} \sum_{y} \sum_{x\neq y}\mathbb E\bigg[\frac{\partial_z \tilde r_y}{1-\theta^x_y G_{yy}\tilde r_y}  \frac{\theta_{y}G_{ yx}^2\tilde r_x}{1-\theta_{y} G_{xx}\tilde r_x }\Big|\chi_x = 0 \bigg].	
	\end{align*}
	Now, we have with Lemma \ref{Lemma: Estimates on diagonal factors} for every $x \neq y$
	\begin{equation*}
		\Big|\frac{\partial_z \tilde r_y}{1-\theta^x_y G_{yy }\tilde r_y} \frac{\theta_{y}G_{ yx}^2\tilde r_x}{1-\theta_{y} G_{xx}\tilde r_x }\Big| = \Big| \frac{\partial_z \tilde r_y}{\tilde r_y} \frac{\tilde r_y}{1-\theta^x_y G_{yy }\tilde r_y}\frac{\tilde r_x}{1-\theta_{y} G_{xx}\tilde r_x } \theta_{y}G_{ yx}^2 \Big| \leq 4\eta^{-3} \tilde m_x\tilde m_y \big|\theta_y G_{yx}^2\big|
	\end{equation*}
	and hence, since $\theta_y G$ is conditionally on $\chi_x = 0$ independent of $(\tilde m_x, \tilde m_y)$,
	\begin{align*}
		\bigg|\mathbb E\bigg[\frac{\partial_z \tilde r_y}{1-\theta^x_y G_{yy}\tilde r_y}  \frac{\theta_{y}G_{ yx}^2\tilde r_x}{1-\theta_{y} G_{xx}\tilde r_x }\Big|\chi_x = 0 \bigg]\bigg| &\leq 4\eta^{-3} \mathbb E[\tilde m]^2 \mathbb E\big[  |\theta_y G_{yx}^2| \big| \chi_x = 0  \big] \\
		&\leq 4\eta^{-3} \mathbb E[\tilde m]^2 (1-p)^{-1} \mathbb E\big[  |\theta_y G_{yx}^2|\big].
	\end{align*}
	Using that
	\begin{equation*}
		\sum_{x\neq y} \mathbb E_p \big[|\theta_y G_{xy}|^2 \big] \leq  \mathbb E_p\Big[\big[\theta_y G^* \theta_y G\big]_{yy}\Big] \leq \|\theta_y G^*\theta_y G\| \leq \eta^{-2}
	\end{equation*}
	hence yields
	\begin{align*}
		|T_{11}| \leq 4 \mathbb E[\tilde m]^2 (1-p)^{-1} \eta^{-5}.
	\end{align*}
	For the term $T_{12}$, we start by calculating
	\begin{equation*}
		\partial_z(\theta^x_y G_{yy} - \theta_{xy} G_{yy}) = \partial_z \frac{\theta_{xy}G_{yx}^2\tilde r_x}{1-\theta_{xy} G_{xx}\tilde r_x } = \frac{ (\partial_z \theta_{xy} G_{xy}^2) \tilde r_x}{1- \theta_{xy} G_{xx} \tilde r_x} + \frac{\theta_{xy}G_{yx}^2 \big( \partial_z \tilde r_x + \tilde r_x^2 \partial_z \theta_{xy} G_{xx} \big)}{(1- \theta_{xy} G_{xx}\tilde r_x)^2}.
	\end{equation*}
	and obtain $T_{12} = T_{121} + T_{122}$ with
	\begin{align*}
		T_{121} &\coloneqq \frac{1}{|\mathbb V|}\sum_{y} \sum_{x\neq y} \frac{\tilde r_y}{1-\theta^x_y G_{yy} \tilde r_y }\frac{\tilde r_x}{1- \theta_{xy} G_{xx} \tilde r_x} 	\partial_z \theta_{xy} G_{xy}^2\\
		T_{122} &\coloneqq \frac{1}{|\mathbb V|}\sum_{y} \sum_{x\neq y}  \frac{\tilde r_y}{1-\theta^x_y G_{yy} \tilde r_y }\frac{\tilde r_x^2}{(1- \theta_{xy} G_{xx}\tilde r_x)^2} \Big(\frac{\partial_z \tilde r_x}{\tilde r_x^2} + \partial_z \theta_{xy} G_{xx}\Big) \theta_{xy}G_{yx}^2 .
	\end{align*}
	Let $\gamma$ be a suitable parametrization of the boundary of the rectangle $[\lambda-\eta/2, \lambda+\eta/2] \times [\eta/2, 3\eta/2]$. Then by Cauchy's integral formula
	\begin{equation*}
		T_{121} \coloneqq \frac{1}{|\mathbb V|}\mathbb E_p \bigg[\sum_{y} \sum_{x\neq y} \frac{\tilde r_y}{1-\theta^x_y G_{yy} \tilde r_y }\frac{\tilde r_x}{1- \theta_{xy} G_{xx} \tilde r_x} 	\frac{1}{2\pi \mathrm i} \int_\gamma \mathrm d \xi \,  \frac{\theta_{xy}G_{xy}^2(\xi)}{(\xi - z)^2} \bigg]
	\end{equation*}
	such that we obtain with Lemma \ref{Lemma: Estimates on diagonal factors} and Fubini's theorem
	\begin{align*}
		|T_{121}| &\leq \int_\gamma\mathrm d \xi\,\frac{1}{|\mathbb V|} \sum_{y}\sum_{x\neq y} \mathbb E_p \Big[ 8 \tilde m_y \tilde m_x \eta^{-4} |\theta_{xy}G_{xy}(\xi)|^2  \Big] \\
		&=8\cdot   \mathbb E[\tilde m]^2   \eta^{-4}   \int_\gamma\mathrm d \xi\, \frac{1}{|\mathbb V|} \sum_{y}\sum_{x\neq y} \mathbb E_p \big[|\theta_yG_{xy}(\xi)|^2 \big| \chi_x = 0\big] \\
		&\leq 8\cdot  \mathbb E[\tilde m]^2 (1-p)^{-1}   \eta^{-4} \int_\gamma \mathrm d \xi\, \mathbb E_p\bigg[ \frac{1}{|\mathbb V|}\sum_{y}\big(\theta_y G^*(\xi) \theta_y G(\xi)\big)_{yy} \bigg]\\
		&\leq 32 \cdot \mathbb E[\tilde m]^2 (1-p)^{-1}  \eta^{-5}
	\end{align*}
	where we used in the last inequality that the length of $\gamma$ is $4\eta$.
	For the term $T_{122}$ we have
	\begin{align*}
		|T_{122}| &\leq \frac{1}{|\mathbb V|}\mathbb E \bigg[\sum_{y}\sum_{x\neq y} 8 \tilde m_x^2 \tilde m_y \eta^{-5} (c(a, b) + 1)  \big|\theta_{xy}G_{yx}^2\big|  \bigg] \\
		&\leq 8\cdot \mathbb E[\tilde m]\mathbb E[\tilde m^2] (c(a, b) + 1) (1-p)^{-1} \eta^{-7}.
	\end{align*}
	It is left to estimate $T_2$. We have
	\begin{align*}
		T_2 &= \mathbb E_p \bigg[\frac{1}{|\mathbb V|}\sum_y \sum_{x\neq y} \big(\partial_z(\tilde r_y \theta_{xy} G_{yy}) \big) \frac{\tilde r_y(\theta^x_y G_{yy} - \theta_{xy} G_{yy}  ) }{(1 - \theta^x_yG_{yy}\tilde r_y )(1 - \theta_{xy} G_{yy}\tilde r_y) }   \bigg] \\
		&=  \mathbb E_p \bigg[\frac{1}{|\mathbb V|}\sum_y \sum_{x\neq y} \Big(\frac{\partial_z\tilde r_y}{\tilde r_y}\theta_{xy} G_{yy}  + \partial_z \theta_{xy} G_{yy}  \Big) \frac{\tilde r_y^2}{(1 - \theta^x_yG_{yy}\tilde r_y )(1 - \theta_{xy} G_{yy}\tilde r_y)} \frac{\tilde r_x }{1 -  \theta_{xy} G_{xx} \tilde r_x} \theta_{xy} G_{xy}^2   \bigg] 
	\end{align*}
	such that by the same arguments
	\begin{equation*}
		|T_2| \leq 16 \cdot  \mathbb E[\tilde m^2] \mathbb E[\tilde m] (1-p)^{-1}\eta^{-7}.
	\end{equation*}
	Combining the estimates above concludes the proof.
\end{proof}

\subsubsection{Proof of Proposition \ref{Propositio: Small p expansion in finite volume}}
	By integrating out the probability in $\chi$ one can write $\hat \nu_p$ as a polynomial of order $|\mathbb V|$ in $p$ whose coefficients are finite signed measures. In particular, the first derivative $\phi = \partial_p \hat{\nu}_p|_{p=0}$ exists in total variation and has Stieltjes transform $h = h_0 = \partial_p \hat{s}_p|_{p=0}$. By Lemma \ref{Proposition: C1 distance estimate} and Lemma \ref{Proposition: Estimate on difference of ST} the estimates in \eqref{Equation: Estimate C1 distance} and \eqref{Equation: Estimate on h in finite volume} hold. By Proposition \ref{Proposition: Estimate on second derivative of ST}, we have for every $z = \lambda + \mathrm i \eta$ with $\lambda \in (a, b)$ and\footnote{Notice that $s_p(\overline z) = \overline{s_p(z)}$, $h(\overline z) = \overline{h(z)}$.} $\eta \in (-1, 1) \setminus\{0\}$ and all $p\in (0, 1)$
	\begin{equation}
		\label{Equation: Estimate on ST of R}
		p^{-2}\big|\hat s_p(z) - s_0(z) - p h(z)\big| \leq \frac{1}{2}\mathrm C(a, b, \mathrm P_1) (1-p)^{-1} |\eta|^{-7}.
	\end{equation}
	It is left to show that \eqref{Equation: Estimate on ST of R} implies that $R(p) \coloneqq p^{-2}(\hat \nu_p - \nu_0 - p \phi)$ satisfies for any $p\in (0, 1/2]$ and $a<b$ the estimate $\|R(p)\|_{(C^{8}(a, b))'}  \leq 2^{6}C(a, b, \mathrm P_1)(1-p)^{-1}$. We apply the Helffer-Sjöstrand formula in the form \eqref{Equation: ST and Helffer-Sjöstrand}. We extend $f\in C^8_c(a, b)$ via its almost analytic extension
	\begin{equation*}
		\tilde f(\lambda+ \mathrm i\eta) \coloneqq \psi(\eta) \sum_{k=0}^{7} \frac{(i\eta)^k}{k!} f^{(k)}(\lambda).
	\end{equation*}
	to a function $\tilde f\in C_c^1(\mathbb C)$. Then
	\begin{equation*}
		2\cdot \overline{\partial} f(\lambda + \mathrm i \eta) =  \psi(\eta) \cdot \frac{(\mathrm i \eta)^7}{7!} f^{(8)}(\lambda) + i \psi'(\eta) \sum_{k=0}^{7} \frac{(i\eta)^k}{k!} f^{(k)}(\lambda).
	\end{equation*}
	We define sets $M_1, M_2 \subset \C$ by
	\begin{equation*}
		M_1 = (a, b) \times [-1/2, 1/2], \quad M_2 = \big((a, b) \times [-1, 1])\setminus M_1.
	\end{equation*}
	We then have
	\begin{equation*}
		|\overline{\partial}  \tilde f(z)| \leq \frac{\eta^7}{2 \cdot 7!}\|f^{(8)}\|_\infty + \1_{\{z\in M_2\}} \frac{1}{2}  \|\psi' \|_\infty \|f\|_{C^7} 
	\end{equation*}
	since $\psi'(\eta) = 0$ for $|\eta| \leq 1/2$ or $|\eta|\geq 1$. Hence \eqref{Equation: Estimate on ST of R} yields with $|M_1| = |M_2| = |b-a|$ and $\operatorname{supp}(\overline{\partial}  \tilde f) \subseteq M_1\cup M_2$
	\begin{align*}
		\pi |R_p(f)| &= \Big|\int_\C p^{-2}(\hat s_p(z) - s_0(z) - p h(z))\overline{\partial} \tilde f(z) \, \mathrm dz   \Big| \\
		&\leq \frac{1}{2\cdot 7!} |b-a| \mathrm C(a, b, \mathrm P_1) (1-p)^{-1}  \|f^{(8)}\|_\infty  + 2^{5} |b-a|  \mathrm C(a, b, \mathrm P_1) (1-p)^{-1}\|\psi' \|_\infty \|f\|_{C^7}.
	\end{align*}
	Choosing $p = \mathrm P(\mathrm G \neq \emptyset)$ concludes the proof. \qed

\subsection{Expansion for sparse sprinkling in infinite volume: proofs of Theorems \ref{Theorem: small p expansion ST} and \ref{Theorem: small p expansion in dual of Ck}}
\label{Subsection: Small p expansion in infinite volume}

\begin{proof}[Proof of Theorem \ref{Theorem: small p expansion ST}]
	We reintroduce the $N$-dependency in the notation. By Lemma \ref{Lemma: convergence of nu hat} we have $\hat{\nu}_{N} \to \hat{\nu}$ in the Kolmogorov--Smirnov metric, and hence in particular weakly, such that $\hat s_N \to \hat s$ as $N\to \infty$. 
	We show that
	\begin{equation}
		\label{Equation: Convergence of h_N}
		h_N \coloneqq \mathbb E_0\Big[\frac{\partial_z (g_N r)}{1- g_N r} \big| \mathrm G\neq \emptyset \Big] \to \mathbb E_0\Big[\frac{\partial_z(g r)}{1 - g r} \big| \mathrm G\neq \emptyset \Big] \eqqcolon h
	\end{equation}
	pointwise on $\mathbb C \setminus \R$. After potentially changing the probability space (with Skorokhod's representation theorem) and relabelling the vertices, we might assume that almost surely for all $R>0$ we have $(\mathbb G_N, \mathbb o_N)_R = (\mathbb G, \mathbb o)_R$ eventually and hence almost surely $B_N\to B$ in the strong resolvent sense. This implies \cite[Theorem VIII20]{RS80} that $f(B_N) \to f(B)$ strongly for all bounded continuous $f:\mathbb R \to \mathbb C$. In particular, we have almost surely for any $z\in \C^+$
	\begin{align*}
		&g_N(z) = \langle \delta_{\mathbb o}, (B_N-z)^{-1} \delta_{\mathbb o} \rangle \to \langle \delta_{\mathbb o}, (B-z)^{-1} \delta_{\mathbb o} \rangle = g(z) \\
		&g_N'(z) =\langle \delta_{\mathbb o}, (B_N-z)^{-2} \delta_{\mathbb o} \rangle \to \langle \delta_{\mathbb o}, (B-z)^{-2} \delta_{\mathbb o} \rangle = g'(z)
	\end{align*}
	as $N\to \infty$. Using Lemma \ref{Lemma: Estimates on diagonal factors}, we have for all $a<b$ and every $z= \lambda + \mathrm i \eta \in (a, b) \times (0, 1)$
	\begin{equation*}
		\bigg|\frac{\partial_z (g_N(z) r(z))}{1- g_N(z) r(z)} \bigg| \leq 2m \eta^{-3} \bigg| \frac{1}{1- g_N(z) r(z)} \bigg| \leq 4m^2 c(a, b) \eta^{-5}
	\end{equation*}
	where we used that by assumption $|1/r(z)| \leq c(a, b) \eta^{-1}$ almost surely (conditionally on $\mathrm G\neq \emptyset$). Hence \eqref{Equation: Convergence of h_N} holds by dominated convergence. The estimates in \eqref{Equation: Estimates on h and S} follow directly from \eqref{Equation: Estimate on h in finite volume} and Proposition \ref{Proposition: Estimate on second derivative of ST}.
\end{proof}

\begin{proof}[Proof of Theorem \ref{Theorem: small p expansion in dual of Ck}]
	Using the Estimate \eqref{Equation: Estimate on h in finite volume} we obtain by dominated convergence for every $f\in C_c^2(\R)$
	\begin{equation*}
		\lim_{N\to \infty} \phi_N(f) =\lim_{N\to \infty}  \frac{1}{\pi}\int_{\C \setminus\R} h_N(z) \overline{\partial} f_{\psi}(z) \,  \mathrm dz =  \frac{1}{\pi}\int_{\C \setminus\R} h(z) \overline{\partial} f_{\psi}(z) \,  \mathrm dz \eqqcolon \phi(f).
	\end{equation*}
	Since
	\begin{equation*}
		R_{N} \coloneqq p^{-2}(\hat \nu_N - \nu_{0, N} - p \phi_N) \to p^{-2}(\hat \nu - \nu_{0} - p \phi) \eqqcolon R
	\end{equation*}			
pointwise on $C_c^2(\R)$, Proposition \ref{Propositio: Small p expansion in finite volume} implies $\|R \|_{(C^{8}(a, b))'}  \leq 2^6 |b-a| C(a, b, \mathrm P) (1-p)^{-1}$ for any $a<b$. By using \eqref{Equation: Estimate on phi in finite volume} and a diagonal argument\footnote{Namely, let $\mathcal F$ be a countable subset of $C_c^1(\R)$ such that for every $f\in C^1_c(\R)$ there exists a compact set $K\subset \R$ and a sequence $(f_n)_n \subset \mathcal F$ such that $\operatorname{supp}(f_n) \subseteq K$ for all $n\in \N$ and $f_n \to f$ with respect to $\|\cdot \|_{C^1}$. Then for all $f\in \mathcal F$ the sequence $(\phi_N(f))_N$ has a subsequence which converges. 
		By a diagonal argument, there exists a subsequence $(\phi_{N_k})_k$ such that $(\phi_{N_k}(f))_k$ converges for all $f \in \mathcal F$. One easily checks that for all $f\in C_c^1(\R)$ the sequence $(\phi_{N_k}(f))_k$ is Cauchy, i.e.\ converges, and that $f\mapsto \tilde \phi(f) \coloneqq \lim_{k \to \infty} \phi_{n_k}(f)$ defines a distribution of order at most one.} (as, for example, in the proof of the Banach--Alaoglu Theorem), one convinces oneself that there is a subsequence $(\phi_{N_k})_k$ and a distribution $\tilde \phi$ of order at most one such that $\phi_{N_k}(f) \to \tilde \phi(f)$ as $k\to \infty$ for all $f\in C_c^1(\R)$. Then $\tilde \phi$ is the unique extension of $\phi$ to a distribution of order at most 1.
\end{proof}

\subsection{Sufficient conditions for regularity: proof of Proposition \ref{Proposition: Condition that phi locally agrees with measure}}	
\label{Subsection: Regularity of first order term}

\begin{lemma}
	\label{Lemma: Expansion of h around admissible eigenvalue}
	We have on the event $\{\lambda_0 \in \mathrm S_*\}$
	\begin{equation*}
		\frac{\partial_z(gr)}{1-gr} + \sum_{\lambda \in \mathrm S_*} \frac{1}{\lambda - z} =\frac{(\partial_z r)/r^2  + \partial_zg}{1/r - g} + \sum_{\lambda\in \mathrm S_* \setminus \{\lambda_0 \}}  \frac{1}{\lambda - z} + \frac{\sum_{\lambda\in \mathrm S \setminus \{\lambda_0 \}}\Big( \frac{w(\lambda)}{\lambda - z} - (\lambda_0 - z)\frac{w(\lambda)}{(\lambda - z)^2}\Big)} {w(\lambda_0) + (\lambda_0 - z) \sum_{\lambda\in \mathrm S \setminus \{\lambda_0 \}}  \frac{w(\lambda)}{\lambda - z}}.
	\end{equation*}
\end{lemma}

\begin{proof}
	This follows by a simple computation.
\end{proof}

\begin{proof}[Proof of Proposition \ref{Proposition: Condition that phi locally agrees with measure}]
	We start with the proof of \eqref{Equation: phi outside of support is purely atomic}, so let us assume that $\mathrm D \subseteq \Sigma_0 = [a, b]$. Let $\lambda \in \R\setminus \Sigma_0$. Let us abbreviate $u = gr$. It is not hard to see that $u$ is strictly decreasing on $(b, \infty)$ and strictly increasing on $(-\infty, a)$. As \eqref{Assumption:: m a.s. bounded} holds, the support of $\mathrm P$ is finite, i.e.\ there exists $(\mathrm G_1, \varphi_1), \hdots, (\mathrm G_l, \varphi_l)$  such that
	\begin{equation*}
		\mathbb P\big((\mathrm G, \varphi) \in \big\{(\mathrm G_1, \varphi_1), \hdots, (\mathrm G_l, \varphi_l)\big\}\big) = 1.
	\end{equation*}
	We first assume that we are on the event $\{(\mathrm G, \varphi) = (\mathrm G_i, \varphi_i)\}$ for some $i\in \{1, \hdots, l\}$.
	If $u(\lambda) \neq 1$ then there exists some $\delta>0$ and some $\varepsilon_i \in (0, \operatorname{dist}(\lambda, \Sigma_0))$ such that $|1-u(z)|>\delta$ for all $z \in B_{\varepsilon_i}(\lambda)$ which implies
	\begin{equation*}
		\lim_{\eta \downarrow 0} \operatorname{Im} \frac{ \partial_zu(t + \mathrm i \eta) }{1- u(t + \mathrm i \eta)} = 0
	\end{equation*}
	uniformly\footnote{To show that the convergence is uniform, one can e.g.\ differentiate with respect to $z$, bound the derivative and apply the Arzelà--Ascoli theorem.} in $t \in (\lambda-\varepsilon_i, \lambda + \varepsilon_i)$. If, one the other hand, $u(\lambda) = 1$ then write
	\begin{equation*}
		\frac{\partial_zu(z)}{1- u(z)} =\frac{1}{\lambda- z} + \partial_zu(z) \bigg( \frac{1}{1-u(z)} - \frac{1}{\partial_zu(\lambda)(\lambda - z)}  \bigg) + \frac{\partial_zu(z) - \partial_zu(\lambda)}{\partial_zu(\lambda)(\lambda - z)}.
	\end{equation*}
	Expanding $u$ into a Taylor series around $\lambda$ yields the existence of holomorphic functions $u_1, u_2, u_3:B_{\varepsilon_i}(\lambda) \to \C$ on an open ball $B_{\varepsilon_i}(\lambda)$ around $\lambda$ with $u_j(B_{\varepsilon_i}(\lambda) \cap \R) \subset \R$ for $j = 1, 2, 3$ and \footnote{Notice that $u_2(\lambda) = \partial_zu(\lambda) \neq 0$.} $u_2(z) \neq 0$ for all $z\in B_{\varepsilon_i}(\lambda)$ such that
	\begin{align*}
		\frac{1}{1- u(z)} - \frac{1}{\partial_zu(\lambda)(\lambda - z)} = \frac{\partial_zu(\lambda)(\lambda - z) - (1 - u(z))}{(1- u(z))\partial_zu(\lambda)(\lambda-z) } = \frac{u_1(z)}{u_2(z) \partial_z u(\lambda)}
	\end{align*}
	as well as
	\begin{equation*}
		\frac{\partial_zu(z) - \partial_zu(\lambda)}{\partial_zu(\lambda)(\lambda-z)} = \frac{u_3(z)}{\partial_zu(\lambda)}
	\end{equation*}
	for all $z\in B_{\varepsilon_i}(\lambda)$. Hence, we obtain
	\begin{equation*}
		\lim_{\eta \downarrow 0} \operatorname{Im} \Big(\frac{ \partial_zu(t + \mathrm i \eta) }{1- u(t + \mathrm i \eta)} - \frac{1}{z-\lambda} \Big) = 0
	\end{equation*}
	uniformly in $t\in (\lambda-\varepsilon_i, \lambda+\varepsilon_i)$.
	Hence, if we set $\varepsilon \coloneqq \min_{i} \varepsilon_i$ we obtain by conditioning on $(\mathrm G, \varphi)$
	\begin{equation*}
		\lim_{\eta \downarrow 0} \operatorname{Im}\Big(h(t + \mathrm i \eta) - \frac{\mathbb P(g(\lambda)r(\lambda) = 1)}{\lambda - t - \mathrm i \eta} \Big)= \lim_{\eta \downarrow 0} \operatorname{Im}\Big(\mathbb E\Big[\frac{\partial_zu(t+ \mathrm i \eta)}{1- u(t + \mathrm i \eta)} + \frac{\1_{\{g(\lambda)r(\lambda) = 1\}}  }{\lambda - t - \mathrm i \eta} \Big] \Big)=0
	\end{equation*}
	uniformly in $t\in (\lambda - \varepsilon, \lambda + \varepsilon)$. This implies
	\begin{equation*}
		\phi - \sum_{\lambda \in \mathcal Z}\mathbb P(g(\lambda)r(\lambda) = 1)\cdot \delta_{\lambda} = 0  \quad \text{ on } \quad  \R \setminus \Sigma_0.
	\end{equation*}
	(compare to the proof of Lemma \ref{Lemma: Local uniform convergence of Im ST} in Appendix \ref{Section: Facts about Stieltjes transforms}).
	Let us now assume that $\nu_0$ has a density $f\in C[a, b] \cap C^2(a, b)$ which is strictly positive on $(a, b)$ and show that \eqref{Equation: phi insinde of support is ac} holds. Let us again assume that we are on the event $\{(\mathrm G, \varphi) = (\mathrm G_i, \varphi_i)\}$ for some $i\in \{1, \hdots, l\}$.
	Let $\lambda \in (a, b)$. If $\lambda \notin \mathrm S_*$ then can convince oneself with Proposition \ref{Proposition: Limit of ST} (using that $f$ is strictly positive on $(a, b)$) that there exists some $\varepsilon_i>0$ such that $1-u$ does not have any zeros in $(\lambda - \varepsilon_i, \lambda + \varepsilon_i)$ and such that the limit
	\begin{equation}
		\label{Equation: convergence of ST of first order term}
		\lim_{\eta \downarrow 0} \frac{\partial_z u(t + \mathrm i \eta)}{1 - u(t + \mathrm i \eta)} + \sum_{\lambda \in \mathrm S_*} \frac{1}{\lambda - t - \mathrm i \eta}
	\end{equation}
	exists uniformly in $t\in (\lambda - \varepsilon_i, \lambda + \varepsilon_i)$. If, on the other hand, $\lambda \in \mathrm S_*$ then $z \mapsto 1/r(z)$ and $z \mapsto \partial_z r(z)/r(z)^2$ have, by Riemann's theorem on removable singularities, a removable singularity in $\lambda$. Applying Proposition \ref{Proposition: Limit of ST},  Lemma \ref{Lemma: Expansion of h around admissible eigenvalue} and the estimate $\big|\frac{1}{1/r-g}\big| \leq \frac{1}{\operatorname{Im}{g}}$
	(which follows as both $1/r$ as well as $-g$ have negative imaginary part on $\mathbb C^+$)
	one convinces oneself that there exists some $\varepsilon_i>0$ such that the limit \eqref{Equation: convergence of ST of first order term} exists uniformly in $t\in (\lambda - \varepsilon_i, \lambda + \varepsilon_i)$. Setting $\varepsilon \coloneqq \min_i \varepsilon_i$ yields by conditioning on $(\mathrm G, \varphi)$ that $ h(t + \mathrm i \eta)$ converges uniformly in $t\in (\lambda - \varepsilon, \lambda + \varepsilon)$ as $\eta \downarrow 0$. The statement follows from Lemma \ref{Lemma: Local uniform convergence of Im ST}.
\end{proof}

\subsection{Necessity of Smoothing: proof of Proposition \ref{Proposition: Smoothening is necessary 1} and Proposition \ref{Proposition: Smoothening is necessary 2}}
\label{Subsection: Necessity of Smoothening}

\begin{proof}[Proof of Proposition \ref{Proposition: Smoothening is necessary 1}]
	We start by noticing that for every $\lambda \in \R$
	\begin{align*}
		&\lim_{\eta \downarrow 0} \eta g(\lambda + \mathrm i \eta) = \lim_{\eta \downarrow 0} \int_\R \frac{\eta(t-\lambda) }{(t-\lambda)^2 + \eta^2}\, \nu_0(\mathrm dt)  + \mathrm i \int_\R \frac{\eta^2}{(t-\lambda)^2 + \eta^2} \, \nu_0(\mathrm dt) = \mathrm i \nu_0(\{\lambda\})  \\
		&\lim_{\eta \downarrow 0} \eta^2 g'(\lambda + \mathrm i \eta) = \lim_{\eta \downarrow 0}\int_\R \frac{\eta^2(t - \lambda)^2 - \eta^4}{\big[(t-\lambda)^2 + \eta^2\big]^2} \, \nu_0(\mathrm dt)  + \mathrm i \int_\R  \frac{2\eta^3(t - \lambda)}{\big[(t-\lambda)^2 + \eta^2\big]^2} \, \nu_0(\mathrm dt) = -\nu_0(\{\lambda\}) 
	\end{align*}
	where we used dominated convergence and the bounds
	\begin{equation*}
		\sup_{t \in \R} \Big|\frac{\eta t}{t^2 + \eta^2}\Big| = \frac{1}{2}, \quad \sup_{t \in \R} \Big|\frac{\eta^2 t^2 - \eta^4}{(t^2 + \eta^2)^2}\Big| = 1, \quad  \sup_{t \in \R} \Big|\frac{2\eta^3 t}{(t^2 + \eta^2)^2}\Big| = \frac{3\sqrt{3}}{8},
	\end{equation*}
	which hold for every $\eta > 0$. Now, let us assume that $\phi$ would be a Radon measure. By the assumptions on $\mathbb G$ and $\mathrm P$, the support of $\hat \nu$ is compact such that $\phi$ would be a finite signed measure with compact support.
	Let $\mathbb A$ be the set of atoms of $\nu_0$  and let $\mathrm M = \{\lambda \in \R:\, \mathrm P(r(\lambda) = 0) > 0 \}$. For every $\lambda \in \mathbb A \setminus(\mathrm D\cup \mathbb M)$ we have
	\begin{align*}
		\phi(\{\lambda\}) = \lim_{\eta \downarrow 0} -i \eta h(\lambda + \mathrm i \eta) &=	\mathbb E\bigg[\lim_{\eta \downarrow 0} \frac{\big(-\mathrm i \eta^2 \partial_z g(\lambda  + \mathrm i \eta) \big) r(\lambda  + \mathrm i \eta) - \mathrm i \eta^2 g(\lambda  + \mathrm i \eta) \partial_z r(\lambda  + \mathrm i \eta)  }{\eta - \eta g(\lambda  + \mathrm i \eta) r(\lambda  + \mathrm i \eta)} \Big| \mathrm G \neq \emptyset \bigg] \\
		&= \mathbb E\bigg[ \frac{ \mathrm i\cdot \nu_0(\{\lambda\}) r(\lambda)}{-\mathrm i \cdot   \nu_0(\{\lambda\}) r(\lambda)} \Big| \mathrm G \neq \emptyset  \bigg] = -1
	\end{align*}
	where we might exchange the order of the limit and expected value as the support of $\mathrm P$ is finite. The assumptions imply that the set $\mathbb A \setminus(\mathrm D\cup \mathrm M)$ is infinite, yielding $\phi( \mathbb A \setminus(\mathrm D\cup \mathrm M)) = -\infty$ and hence a contradiction.
\end{proof}

\begin{proof}[Proof of Proposition \ref{Proposition: Smoothening is necessary 2}]
	Notice that the set $\mathcal Z$ does not depend on $p\in (0, 1)$. By Corollary \ref{Corollary: No atoms outside of D for integer lattice} we have $\nu_p(\mathcal Z) = 0$. We hence get with Proposition \ref{Proposition: Condition that phi locally agrees with measure}
	\begin{equation*}
		R_p(\mathcal Z) = p^{-2}\big(\nu_p(\mathcal Z) - \nu_0(\mathcal Z)  - p\phi(\mathcal Z) \big) = -p^{-1} |\mathcal Z|.
	\end{equation*}
	(for the last equality, notice that we assumed \eqref{Assumption: Detetministic attached graphs}).
	If $c(\mathbb Z^d, \mathrm P)>0$ then $\mathcal Z \neq \emptyset$ and hence
	$(R_p(\mathcal Z))_p$ is unbounded. In particular, for $d\leq 2$ this is always the case, see the discussion after Proposition \ref{Proposition: Concentration of mass outside of Sigma0}. Moreover, for $d=1$ with 0 being an admissible eigenvalue we have with Proposition \ref{Proposition: Mass of atom in zero for line} and Proposition \ref{Proposition: Condition that phi locally agrees with measure}
	\begin{equation*}
		R_p(\{0\}) = \frac{p^{-1}(1-p)}{2-p} \quad \text{ as }\quad  \hat{\nu}_p(\{0\}) = \frac{p(1-p)}{2-p}  \,\text{ and }\, \phi(\{0\}) = 0,
	\end{equation*}
	and hence $(R_p(\{0\}))_p$ is unbounded.
\end{proof}

\section{Applications of the expansion}
We now discuss two applications of the expansion given in Theorem \ref{Theorem: small p expansion in dual of Ck}:
In Subsection \ref{Subsection: Concentration of mass outside of Sigma0}, we show Proposition \ref{Proposition: Concentration of mass outside of Sigma0} which states, under the given assumptions, that the mass of $\hat \nu$ outside of $\Sigma_0$ is for small $p$ concentrated in a set of small Lebesgue measure. In Subsection \ref{Subsection: Sources and Sinks} we prove Proposition \ref{Proposition: Criterion for repulsive eigenvalues} which gives criteria whether sprinkling with \eqref{Assumption: Detetministic attached graphs} yields for small $p$ a local increase or decrease of the expected number of eigenvalues close to an admissible eigenvalue of the attached graphs. We then apply the latter to the case of the Bethe lattice and prove Proposition \ref{Proposition: Source-Sink transition Bethe lattice}. Finally, we show Corollary \ref{Corollary: mass can decrease around tree eigenvalue for small p} which gives a scenario in which sparse sprinkling yields a local reduction of mass around an eigenvalue of one of the attached graphs. 

\subsection{Concentration of mass outside of $\Sigma_0$: proof of Proposition \ref{Proposition: Concentration of mass outside of Sigma0}}
\label{Subsection: Concentration of mass outside of Sigma0}

\begin{proof}[Proof of Proposition \ref{Proposition: Concentration of mass outside of Sigma0}]
	To simplify notation, let
	\begin{equation*}
		\mathcal J_\varepsilon \coloneqq \mathbb R \setminus \mathcal I_\varepsilon = \{\lambda  \in \R: \, \operatorname{dist}(\lambda, \Sigma_0 \cup \mathcal Z) > \varepsilon  \}.
	\end{equation*}
	We define a suitable mollifier $(\varphi_\varepsilon)_{\varepsilon>0}$:  Let $\varphi \in C_c^\infty(\R)$ satisfy
	$\varphi \geq 0$, $\int_\R \varphi(x) \, \mathrm dx  = 1$, and $\operatorname{supp}(\varphi) \subseteq [-1, 1]$.
	For $\varepsilon \in (0, 1)$ define $\varphi_\varepsilon(x) \coloneqq  \varepsilon^{-1} \varphi(x/\varepsilon)$. Define $f_\varepsilon$ by taking the convolution $f_\varepsilon \coloneqq \1_{\mathcal J_{\varepsilon/2}}*\varphi_{\varepsilon/2}$. We then have for all $k \in \{0. \hdots, 8\}$
	\begin{equation*}
		\|f_{2\varepsilon}^{(k)}\|_\infty = \|\varphi_{\varepsilon}^{(k)}*\1_{\mathcal J_{\varepsilon}}\|_\infty \leq \int_\R \varepsilon^{-k-1} |\varphi^{(k)}(x/\varepsilon)| \, \mathrm dx \leq \int_{-1}^1 \varepsilon^{-k} |\varphi^{(k)}(x)| \, \mathrm dx \leq 2\varepsilon^{-k} \|\varphi^{(k)}\|_\infty
	\end{equation*}
	such that
	\begin{equation}
		\label{Equation: Upper bound on C8 norm of f_eps}
		\|f_{2\varepsilon}\|_{C^8} \leq 2\varepsilon^{-8}  \|\varphi \|_{C^8}
	\end{equation}
	for all $\varepsilon \in (0, 1)$. Notice that $f_\varepsilon$ satisfies
	\begin{equation*}
		f_\varepsilon(x)=1 \text{ for all } x\in \mathcal J_\varepsilon, \quad f_\varepsilon(x)=0 \text{ for all } x\in \Sigma_0 \cup \mathcal Z.
	\end{equation*}
	In particular, $\nu_0(f_\varepsilon) = 0$. By Proposition \ref{Proposition: Condition that phi locally agrees with measure} we have $\phi(f_\varepsilon)= 0$ as well such that we obtain with Theorem \ref{Theorem: small p expansion in dual of Ck} and \eqref{Equation: Upper bound on C8 norm of f_eps}
	\begin{equation*}
		\hat \nu(\mathcal J_\varepsilon) \leq  \hat \nu(f_\varepsilon) = p^2 R(f_\varepsilon) \leq C p^2 \|f_\varepsilon\|_{C^8} \leq C p^2 \varepsilon^{-8}   \|\varphi \|_{C^8}
	\end{equation*}
	for a suitable constant $C>0$ that only depends on $M$.
	The choice $\varepsilon = p^\delta$ yields \eqref{Equation: Most mass close to Sigma_0 or Z}. The lower bound in \eqref{Equation: Bounds on mass close to Z} follows in a similar manner: Assume $\varepsilon \leq \operatorname{dist}(\Sigma_0, \mathcal Z)$ and define $g_\varepsilon \coloneqq \1_{\mathcal Z_{\varepsilon/2}  }*\varphi_{\varepsilon/2}$.
	Then $g_\varepsilon \leq \1_{\mathcal Z_\varepsilon}$ and $g_\varepsilon(x)=1 \text{ for all } x\in \mathcal Z$, so that
 by Theorem \ref{Theorem: small p expansion in dual of Ck} and Proposition \ref{Proposition: Condition that phi locally agrees with measure}
	\begin{equation}
		\label{Equation: Lower bound on mass of J_varepsilon}
		\hat \nu(\mathcal Z_\varepsilon) \geq \hat \nu(g_\varepsilon)  = p \phi(g_\varepsilon) + p^2 R(g_\varepsilon) = p\phi(\mathcal Z) + p^2R(g_\varepsilon).
	\end{equation}
	Using that $\lambda  \mapsto r(\lambda) g(\lambda)$ is conditionally on $\mathrm G \neq \emptyset$ strictly increasing on $(-\infty, a)$ and strictly decreasing on $(b, \infty)$ we hence obtain
	\begin{align*}
		\phi(\mathcal Z) &= \sum_{\lambda \in \mathcal Z} \mathbb P(g(\lambda)r(\lambda) = 1 |\mathrm G \neq \emptyset) \\
		&= \mathbb E\Big[\sum_{\lambda \in \mathcal Z}  \1_{\{g(\lambda)r(\lambda) = 1 \} }|\mathrm G \neq \emptyset\Big] \\
		&= \mathbb P\big( \exists \lambda \in (-\infty, a):\, r(\lambda)g(\lambda) = 1\big)  + \mathbb P\big( \exists \lambda \in (b, \infty):\, r(\lambda)g(\lambda) = 1) = c(\mathbb G, \mathrm P \big).
	\end{align*}
	With \eqref{Equation: Lower bound on mass of J_varepsilon} and the estimate $R(g_\varepsilon) \leq C p^2 \varepsilon^{-8}   \|\varphi \|_{C^8}$ we obtain the lower bound in \eqref{Equation: Bounds on mass close to Z} for the choice $\varepsilon = p^{\delta}$. The upper bound in \eqref{Equation: Bounds on mass close to Z} directly follows from \eqref{Equation: Estimate on mass outside of Sigma0}.
\end{proof}

\subsection{Sources and sinks: proof of Proposition \ref{Proposition: Criterion for repulsive eigenvalues}, Proposition \ref{Proposition: Source-Sink transition Bethe lattice} and Corollary \ref{Corollary: mass can decrease around tree eigenvalue for small p}}
\label{Subsection: Sources and Sinks}

\begin{proof}[Proof of Proposition \ref{Proposition: Criterion for repulsive eigenvalues}]
	By Lemma \ref{Lemma: Expansion of h around admissible eigenvalue}, we have
	\begin{align*}
		\operatorname{Im} h(\lambda) \coloneqq \lim_{\eta \downarrow 0} \operatorname{Im} h(\lambda+ \mathrm i \eta) &= \operatorname{Im} \Big( \frac{-1/w(\lambda) - g'(\lambda)}{g(\lambda)}  \Big)\\
		&= \frac{1}{w(\lambda)} \frac{\operatorname{Im}(g(\lambda))}{|g(\lambda)|^2} - \frac{\operatorname{Im} g'(\lambda) \operatorname{Re} g(\lambda) - \operatorname{Re} g'(\lambda) \operatorname{Im} g(\lambda) }{|g(\lambda)|^2}
	\end{align*}
	so that
	\begin{equation*}
		\operatorname{Im} h(\lambda) > 0 \iff w(\lambda)^{-1} \operatorname{Im}g(\lambda) > \operatorname{Im}g'(\lambda) \operatorname{Re}g(\lambda) - \operatorname{Im}g(\lambda) \operatorname{Re}g'(\lambda).
	\end{equation*}
	Recall that $\phi$ is exactly defined such that under the assumption \eqref{Assumption: Detetministic attached graphs} we have $\phi(f) = \partial_p \hat \nu_p(f)|_{p=0}$ for all $f\in C^1_c(\R)$.
	By Proposition \ref{Proposition: Condition that phi locally agrees with measure}, $\phi$ agrees on $(a, b)$ the measure $\operatorname{Im}h(x) \, \mathrm dx$. As $h$ is continuous on $(a, b)$ (as a locally uniform limit of continuous functions) the statement follows. 
\end{proof}

\begin{proof}[Proof of Proposition \ref{Proposition: Source-Sink transition Bethe lattice}]
By a standard application of Schur's complement formula, one finds that the Stieltjes transform $g$ of $\mathbb B_d$ is 
	\begin{equation} \label{Equation: ST Kesten-McKay}
		g(\lambda) = -\frac{1}{2} \frac{(d-2)\lambda}{d^2-\lambda^2} + \mathrm i \cdot \frac{d}{2(d^2-\lambda^2)} \sqrt{4(d-1) - \lambda^2}.
	\end{equation}
	A routine calculation yields
	\begin{equation*}
		\frac{\operatorname{Im}(g'(\lambda))}{\operatorname{Im}(g(\lambda))} \operatorname{Re}(g(\lambda)) - \operatorname{Re}(g'(\lambda)) = \frac{d-2}{2} \frac{1}{d^2 - \lambda^2} \frac{4(d-1)}{4(d-1) - \lambda^2}.
	\end{equation*}
	The claim now follows from Proposition \ref{Proposition: Criterion for repulsive eigenvalues}.
\end{proof}

\begin{proof}[Proof of Corollary \ref{Corollary: mass can decrease around tree eigenvalue for small p}]
	As in the proof of Proposition \ref{Proposition: Concentration of mass outside of Sigma0}, let $(\varphi_\varepsilon)_{\varepsilon>0}$ be a Mollifier, where we now additionally choose $\varphi$ to be symmetric around the origin. Let $f_\varepsilon \coloneqq \varphi_\varepsilon(\, \cdot \, -\lambda)$.
	For every $s\in \R$ the set $I_s^\varepsilon \coloneqq \{t \in \R:\, s \leq f_\varepsilon(t)  \}$
	is either empty (if $s > \varepsilon^{-1}\varphi(0)$) or a symmetric interval around $\lambda$ with a length of at most $2\varepsilon$. We have
	\begin{equation}
		\label{Equation: cant think of a name right now}
		(\hat\nu - \nu_0)(f_\varepsilon) = \int_\R (\hat\nu - \nu_0)(\mathrm dt) \int_{0}^\infty \mathrm ds \, \1_{ \{s\leq f_\varepsilon(t) \} } = \int_{0}^{\infty} \mathrm ds \, (\hat\nu - \nu_0)(I^\varepsilon_{s}).
	\end{equation}
	As seen in the proof of Proposition \ref{Proposition: Criterion for repulsive eigenvalues}, there exists some $\varepsilon_0>0$ such that $\operatorname{Im}h(t) < 0$ for all $t\in [\lambda- \varepsilon_0, \lambda + \varepsilon_0]$ which shows by Proposition \ref{Proposition: Condition that phi locally agrees with measure} that $\phi(f_\varepsilon)< 0$ for all $\varepsilon \in (\lambda - \varepsilon_0, \lambda + \varepsilon_0)$. We fix some $\delta \in (0, 1/8)$. As seen in the proof of Proposition \ref{Proposition: Concentration of mass outside of Sigma0}, we have $|R_p(f_\varepsilon)| \leq C \varepsilon^{-8}$ yielding for the choice $\varepsilon = p^{\delta}$
	\begin{equation*}
		(\hat\nu - \nu_0)(f_\varepsilon) = p \phi(f_\varepsilon) + p^2 R_p(f_\varepsilon) < 0.
	\end{equation*}
	if $p$ is sufficiently small. Then Equation \eqref{Equation: cant think of a name right now} yields $(\hat\nu - \nu_0)(I^\varepsilon_s)<0$ for at least one $s\geq 0$.
\end{proof}

\section{Expansion for sparse sprinkling and large degrees in $\mathbb G$}
\label{Section: small p large d}
In this Section we study the expansion for sparse sprinkling for the case where degrees in $\mathbb G$ are large. Roughly speaking, we consider the case where the density of states $\nu_{0, d}$ of $\mathbb G_d$ has a density $f_d$ which scales as $f_d(x) \approx d^{-1/2}f(d^{-1/2}x)$ for some sufficiently smooth probability density $f$. More precisely, we make the the following assumptions.

\begin{assumption}
	\label{Assumption large d expansion}
	There exist constants $c, C, \gamma >0$ such that the following holds.
	\begin{enumerate}
		\item For any $d\geq 1$ the graph $\mathbb G_d$ is deterministic, vertex transitive and bipartite and there exists some $a_d >0$ such that $\Sigma_{0, d} = \operatorname{supp}(\nu_{0, d}) = [-a_d, a_d]$.
		\item The measure $\nu_{0}$ has a density $f_d \in C^2(-a_d, a_d)$ which is strictly positive on $(-a_d, a_d)$.
		\item For all sufficiently large $d$ we have $a_d> \gamma d^{1/2}$ as well as
		\begin{align*}
			\operatorname{sup} \big \{f_d(x): \, |x|\leq a_d \big\} \leq C d^{-1/2}, \quad \operatorname{sup} \big \{|f_d''(x)|: \, |x|\leq \gamma d^{1/2} \big\} \leq C d^{-3/2}.
		\end{align*}
		\item
		For every $K>0$ we have for all sufficiently large $d$
		\begin{align*}
			\operatorname{inf} \big \{f_d(x): \, |x|\leq K \big\} \geq c d^{-1/2}.
		\end{align*}
	\end{enumerate}
\end{assumption}

In Subsection \ref{Subsection: Verifying the large d small p assumptions}, we show that Assumption \ref{Assumption large d expansion} is satisfied for the integer and the Bethe lattice. In Subsection \ref{Subsetion: proof of large d small p expansion}, we then prove Theorem \ref{Theorem: Large d-expansion of first order term}.  		

\subsection{Verifying Assumption \ref{Assumption large d expansion} for the integer and Bethe lattice}
\label{Subsection: Verifying the large d small p assumptions}

We use the following local central limit theorem, whose proof is standard; see e.g.\ the argument in \cite[Theorem 4.5.1]{IL71}.

\begin{prop} 
	\label{Proposition: C2 CLT}
	Let $(X_n)_{n\in \mathbb N}$ be an iid sequence of centered random variables with variance $\sigma^2>0$ and density $f_1$ with compact support. Assume that there exists some $C, \alpha > 0$ such that the characteristic function $\varphi$ of $X_1$ satisfies $|\varphi(t)| \leq C(1+|t|)^{-\alpha}$ for all $t\in \R$. 
	Then the density $f_n$ of $n^{-1/2} \sum_{i=1}^n X_i$ is for sufficiently large $n$ two times continuously differentiable and 
	\begin{equation*}
		\|f_n - f\|_{C_0^2} = \mathcal O(n^{-1/2}) \text{ as }n \to \infty \text{ where }f(x) = \frac{1}{\sqrt{ 2\pi\sigma^2}} e^{-x^2/2\sigma^2}.
	\end{equation*}
\end{prop}

\begin{prop}
	\label{Proposition: Assumtptions large d expansion satisfied for integer lattice}
	Assumption \ref{Assumption large d expansion} is satisfied for $\mathbb G = \mathbb Z^d$.
\end{prop}

\begin{proof}

	Since $\mathbb Z^d$ is the $d$-fold cartesian product of the graph $\mathbb Z$, we have $\nu_{0, d} = \nu_0^{*d}$. In other words, if $X_1, \hdots, X_d$ are iid $\operatorname{arcsin}$ distributed and $Y_d \coloneqq d^{-1/2} \sum_{i=1}^{d} X_i$ has density $\rho_d$, then $\nu_{0, d}$ is the distribution of $d^{1/2} Y_d$, i.e.\ has density $f_d(x) = d^{-1/2} \rho_d(d^{-1/2}x)$. The characteristic function of the semicircle law is given by $\varphi(t) = e^{2\mathrm i t} J_0(2t)$,
	where $J_0$ is the zero order Bessel function of the first kind. Since $|J_0(t)| = \mathcal O(t^{-1/2})$ as $t\to \infty$, we have by Proposition \ref{Proposition: C2 CLT} that $\rho_d \to \rho$ in $C^2_0(\R)$ where $\rho$ is the density of $\mathrm N(0, 2)$. The statement readily follows.
\end{proof}

\begin{prop}
	\label{Proposition: Assumtptions large d expansion satisfied for Bethe lattice}
	Assumption \ref{Assumption large d expansion} is satisfied for $\mathbb G = \mathbb B_d$.
\end{prop}

\begin{proof}
	This can be easily shown by using the explicit formula for the density of the Kesten--McKay law $\nu_{0, d}$, namely
	\begin{equation*}
		f_d(x) = \frac{d}{d^2-x^2} \frac{1}{2\pi}\sqrt{(4(d-1)-x^2)\1_{\{|x|\leq 2\sqrt{d-1} \}}}
	\end{equation*}
	(see \cite{Ke84} or the proof of Proposition \ref{Proposition: Source-Sink transition Bethe lattice}). However, if we restrict ourselves to even $d$, a similar reasoning as in the proof of Proposition \ref{Proposition: Assumtptions large d expansion satisfied for integer lattice} can be applied: Since $\mathbb B_{2d}$ is the $d$-fold free product of the graph $\mathbb Z$, we have $\nu_{0, 2d} = \nu_0^{\boxplus d}$, where $\boxplus$ denotes the free additive convolution. In other words, if $X_1, \hdots, X_d$ are freely independent free random variables with $\operatorname{arcsin}$ distribution and $Y_d \coloneqq d^{-1/2} \sum_{i=1}^{d} X_i$ has density $\rho_d$, then $\nu_{0, 2d}$ is the distribution of $d^{1/2} Y_d$, i.e.\ has density $f_d(x) = d^{-1/2} \rho_d(d^{-1/2}x)$. If $\rho$ is the density of the semicircle law, we have by the local free central limit theorem \cite{BV95} local uniform convergence of the densities $\rho_d$ and its derivatives to $\rho$. Hence, up to the global estimate $\|\rho_d\|_\infty = \mathcal O(d^{-1/2})$, the proof  follows (at least for even $d$) as the proof of Proposition \ref{Proposition: Assumtptions large d expansion satisfied for integer lattice}.
\end{proof}

\subsection{Proof of Theorem \ref{Theorem: Large d-expansion of first order term}}
\label{Subsetion: proof of large d small p expansion}
Before coming to the proof of Theorem \ref{Theorem: Large d-expansion of first order term}, let us briefly summarize its main idea. For the rest of the section, we omit the $d$ dependency in the notation. As we shall show in Lemma \ref{Lemma: Estimate on g for large d}, Assumption \ref{Assumption large d expansion} implies that for every bounded interval $I$
\begin{equation*}
	\sup_{t \in I}| \operatorname{Im} g(t)| = \mathcal O(d^{-1/2}), \quad \sup_{t \in I}| \operatorname{Re} g(t)| = \mathcal O(d^{-1}), \quad \sup_{t \in I}| g'(t)| = O(d^{-1})
\end{equation*}
(where we extended $g$ to the real axis via Proposition \ref{Proposition: Limit of ST}). If $\lambda \notin \mathrm S_*$, close to $\lambda$ we have $\big|\frac{(gr)'}{1-gr} \big| \leq C d^{-1/2}$.
With Proposition \ref{Proposition: Condition that phi locally agrees with measure} and Lemma \ref{Lemma: Expansion of h around admissible eigenvalue} we hence have close to $\lambda$
\begin{equation*}
	\pi \phi \approx \mathbb E\Big[\operatorname{Im}\Big(\frac{r'(t)/r^2(t) + g'(t)}{1/r(t) - g(t)}\Big) \1_{\{\lambda \in \mathrm S_*\}} \Big] \mathrm dt
\end{equation*}
for large $d$. For $t$ close to $\lambda \in \mathrm S_*$
\begin{equation*}
	r'(t)/r^2(t) \approx w(\lambda)^{-1}, \quad 1/r(t) \approx w(\lambda)^{-1}(\lambda-t), \quad |g'(t)| \leq Cd^{-1}
\end{equation*}
which we show to imply
\begin{equation}
	\label{Equation: heuristic large d  expansion}
	\pi \phi \approx \mathbb E\Big[\operatorname{Im}\Big( \frac{w(\lambda)^{-1}}{w(\lambda)^{-1}(\lambda-t) - g(\lambda)} \Big) \1_{\{\lambda \in \mathrm S_*\}} \Big]\mathrm dt  \approx \mathbb E\Big[ \operatorname{Im}\Big( \frac{w(\lambda)^{-1}}{w(\lambda)^{-1}(\lambda-t) - \mathrm i \operatorname{Im}g(\lambda)}\Big)  \1_{\{\lambda \in \mathrm S_*\}} \Big]  \mathrm dt
\end{equation}
regardless of the fact that the integrand on the right hand side is diverging in $t = \lambda$ as $d\to \infty$. Using that
\begin{equation*}
	\pi^{-1}\operatorname{Im}\Big( \frac{w(\lambda)^{-1}}{w(\lambda)^{-1}(\lambda-t) - \mathrm i \operatorname{Im}g(\lambda)} \Big) \mathrm dt = \operatorname{Cau}(\lambda, b(\lambda))
\end{equation*}
will then yield the proof. We start by giving two Lemmas which we shall need in the proof of Theorem \ref{Theorem: Large d-expansion of first order term}.
\begin{lemma}
	\label{Lemma: Estimate on g for large d}
	Assume that Assumption \eqref{Assumption large d expansion} holds. Then for any bounded interval $I\subseteq \mathbb \R$ there exists constants $c, C>0$ such that for all sufficiently large $d$ and all $\lambda \in I$
	\begin{equation*}
		|\operatorname{Re}(g(\lambda))| \leq C d^{-1}, \quad  cd^{-1/2} \leq \operatorname{Im}(g(\lambda)) \leq C d^{-1/2}, \quad |g'(\lambda))| \leq Cd^{-1}.
	\end{equation*}
\end{lemma}

\begin{proof}
	The estimates on $\operatorname{Im}(g) = \pi f$ hold by assumption. Since $\mathbb G$ is bipartite, $f$ is symmetric and hence $f'(0) = 0$. Hence 
	\begin{equation*}
		\operatorname{Im}(g'(\lambda)) = \pi f'(\lambda) \leq \pi \int_0^\lambda |f''(t)|\, \mathrm dt = \mathcal O(d^{-3/2})
	\end{equation*}
	uniformly in $\lambda \in I$. For the derivative of the real part, we use that for all $\lambda, x\in\R$ with $|\lambda|, |x| \leq \tfrac{1}{2}\gamma d^{1/2}$ by Taylor expansion with rest-term of order two
	\begin{equation*}
		\Big|\frac{f(\lambda + x) + f(\lambda -x) - 2f(\lambda)}{x^2}\Big| \leq \operatorname{sup} \big \{|f''(x)|: \, |x|\leq \gamma d^{1/2} \big\} \leq C d^{-3/2}
	\end{equation*}
	such that			
	\begin{align*}
		|\operatorname{Re}(g'(\lambda))| &= \bigg|\int_0^\infty \frac{f(\lambda + x) + f(\lambda -x) - 2f(\lambda)}{x^2} \, \mathrm dx \bigg|\\
		&\leq \tfrac{1}{2} C \gamma d^{-1} + \mathcal O(d^{-1/2})\int_{\tfrac{1}{2}\gamma d^{1/2}}^\infty \frac{1}{x^2} \, \mathrm dx = \mathcal O(d^{-1})
	\end{align*}
	uniformly in $\lambda \in I$. It is left to bound the real part of $g$. For this it is sufficient to notice that $\operatorname{Re}g(0) = 0$ by symmetry of $f$ and we hence have uniformly in $\lambda \in I$
	\begin{equation*}
		|\operatorname{Re}(g(\lambda))| \leq |\lambda| \sup_{|x|\leq \lambda} |\operatorname{Re}g'(x)| = \mathcal O(d^{-1}). \qedhere
	\end{equation*}
\end{proof}
For two signed measures $\mu_1,  \mu_2$  we denote their distance in total variation norm by $\|\mu_1 -  \mu_2\|_{\operatorname{TV}}$.
For some interval $I\subseteq \R$, we denote by $\|\mu_1 -  \mu_2\|_{\operatorname{TV}, I}$ the total variation distance of the restrictions of $\mu_1$ and $\mu_2$ to $I$.

\begin{lemma}
	\label{Lemma: TV-distance of Cauchy distributions}
	There exists a constant $C>0$ such that for all $a_1, a_2 \in \mathbb R$ and $b>0$ with $|a_1 - a_2|<b$
	\begin{equation*}
		\| \operatorname{Cau}(a_1, b) - \operatorname{Cau}(a_2, b) \|_{\operatorname{TV}} \leq C |a_1 - a_2|\cdot b^{-1}.
	\end{equation*}
\end{lemma}

\begin{proof}
	We have
	\begin{align*}
		\pi  \| \operatorname{Cau}(a_1, b) - \operatorname{Cau}(a_2, b) \|_{\operatorname{TV}} &= \int_{-\infty}^\infty\bigg|\frac{b}{(x-a_1)^2 + b^2} - \frac{b}{(x-a_2)^2 + b^2}\bigg| \, \mathrm dx \\
		&= \int_{-\infty}^\infty\bigg|\frac{1}{x^2 + 1} - \frac{1}{(x-(a_2 - a_1)/b)^2 + 1}\bigg| \, \mathrm dx.
	\end{align*}
	For $a \coloneqq (a_2 - a_1)/b$ we have
	\begin{align*}
		\bigg|\frac{1}{x^2 + 1} - \frac{1}{(x-a)^2 + 1}\bigg| &= 	\bigg|\frac{a^2 - 2a x }{(x^2 + 1)((x-a)^2 + 1)}\bigg| \\
		&\leq  \1_{\{|x|<2|a| \}} 5a^2 +  \1_{\{|x|\geq 2|a| \}} \frac{3|ax|}{(x^2+1)(x^2/4 + 1)}.
	\end{align*}
	for all $x\in \R$, which yields the claim after integration over $x$.
\end{proof}

\begin{prop}
	\label{Proposition: distance of first order term to sum of Cauchy kernels}
	Assume that Assumption \eqref{Assumption large d expansion} holds.
	Then for any bounded and open interval $I$ there exists a constant $C>0$ such that for sufficiently large $d$
	\begin{equation*}
		\Big \| \phi -  \mathbb E_0\Big[\sum_{\lambda \in \mathrm S_* \cap I} \operatorname{Cau}\big(\lambda, b(\lambda)\big)  \Big] \Big\|_{\operatorname{TV}, I} \leq C d^{-1/30}.
	\end{equation*}
\end{prop}

\begin{proof}
	Let $a(\lambda) \coloneqq \lambda - w(\lambda) \operatorname{Re}(g(\lambda))$.
	By Lemma \ref{Lemma: Estimate on g for large d} and Lemma \ref{Lemma: TV-distance of Cauchy distributions} we have
	\begin{equation*}
		\| \operatorname{Cau}\big(a(\lambda), b(\lambda)\big) -  \operatorname{Cau}\big(\lambda, b(\lambda)\big)\|_{\operatorname{TV}} \leq Cd^{-1/2},
	\end{equation*}
	so that is sufficient to show that
	\begin{equation*}
		\Big \|  \phi-  \mathbb E_0\Big[\sum_{\lambda \in \mathrm S_* \cap I} \operatorname{Cau}\big(a(\lambda), b(\lambda)\big)  \Big] \Big\|_{\operatorname{TV}, I} = \mathcal O( d^{-1/30}).
	\end{equation*}
	By assumption, the distribution of $\mathrm G$ is supported on the set of graphs with overall less then $m$ vertices, i.e.\ a finite set. By conditioning on $(\mathrm G, \varphi)$ we hence might assume without loss of generality that $(\mathrm G, \varphi)$ is deterministic. Additionally, we might assume without loss of generality that $\mathrm D \cap I = \{\lambda_0\}$ for some $\lambda_0 \in \mathrm D$. By Proposition \ref{Proposition: Condition that phi locally agrees with measure}, $ \phi$ agrees (for sufficiently large $d$) with $\pi^{-1}\operatorname{Im} h(t)\, \mathrm dt$ on $I$. We define the map $h_4:\mathbb R\to\mathbb C^+ $ by 
	\begin{equation*}
		h_4(t) \coloneqq \frac{w(\lambda_0)^{-1}}{w(\lambda_0)^{-1}(\lambda_0-t) - g(\lambda_0)}.
	\end{equation*}
	Then $\pi^{-1}\operatorname{Im}(h_4)$ is the density of $\operatorname{Cau}(a(\lambda_0), b(\lambda_0))$.
	In order to show the statement, it is therefore sufficient to show that
	\begin{equation*}
		\int_{I} \big| \operatorname{Im}h(t) - \operatorname{Im}h_4(t)\big| \, \mathrm dt = \mathcal O(d^{-1/30}).
	\end{equation*}
	By Lemma \ref{Lemma: Expansion of h around admissible eigenvalue} we may write $h(t) = h_0(t) + h_1(t)$ if $|t-\lambda_0|$ is sufficiently small where
	\begin{equation*}
		h_0(t) \coloneqq \frac{g'(t)}{1/r(t) - g(t)} +  \frac{\sum_{\lambda\neq \lambda_0} \frac{w(\lambda)}{\lambda - t} -  (\lambda_0 - t)\frac{w(\lambda)}{(\lambda - t)^2}}{w(\lambda_0) + (\lambda_0 - t) \sum_{\lambda\neq \lambda_0} \frac{w(\lambda)}{\lambda - t}}, \quad h_1(t) \coloneqq \frac{r'(t)/r^2(t)}{1/r(t) - g(t)}
	\end{equation*}
	and where $r/r^2$ and $1/r$ are continuously extended to $\lambda_0$.
	We define $\alpha = \tfrac{7}{20}$ and $\beta = \tfrac{7}{30}$ and set $\varepsilon = d^{-\alpha}$, $\delta = d^{-\beta}$. We further decompose 
	\begin{align*}
		h(t) - h_4(t) &= h(t) \1_{\{|t-\lambda_0| \geq \delta\}} +  (h_0(t) + h_1(t))\1_{\{|t-\lambda_0| < \delta\}} - h_4(t)\\
		&=  h(t) \1_{\{|t-\lambda_0| \geq \delta\}} +  h_0(t)\1_{\{|t-\lambda_0| < \delta\}} + (h_1(t) - h_2(t)) \1_{\{|t-\lambda_0| < \delta\}} + h_2(t) \1_{\{\varepsilon < |t-\lambda_0| < \delta\}} \\
		& \quad + (h_2(t) - h_3(t)) \1_{\{|t-\lambda_0| < \varepsilon \}} + (h_3(t) -h_4(t)) \1_{\{|t-\lambda_0| < \varepsilon \}} - h_4(t)\1_{\{|t-\lambda_0| \geq \varepsilon \}}
	\end{align*}
	where
	\begin{equation*}
		h_2(t) \coloneqq \frac{r'(t)/r(t)^2}{1/r(t) - g(\lambda_0)}, \quad h_3(t) \coloneqq \frac{r'(t)/r(t)^2}{w(\lambda_0)^{-1}(\lambda_0 - t)- g(\lambda_0)}.
	\end{equation*}
	We now estimate each term in the expansion above separately. We might choose $C>0$ sufficiently large and $c>0$ sufficiently small such that the following inequalities hold. For all $t\in I$ with $|t-\lambda_0|\geq \delta$
	\begin{equation*}
		\operatorname{Im} h(t) \leq \frac{|g(t)r'(t) + g'(t)r(t)|}{1 - |g(t)r(t)|} \leq \frac{C d^{-1/2} \delta^{-2} + C d^{-1} \delta^{-1}}{1 - Cd^{-1/2} \delta^{-1}} = \mathcal O(d^{-1/30}),
	\end{equation*} 
	for all $t\in I$ with $|t-\lambda_0|< \delta$
	\begin{equation*}
		|\operatorname{Im} h_0(t)| \leq \Big|\frac{g'(t)}{1/r(t) - g(t)}\Big| \leq \frac{|g'|}{\operatorname{Im}(g(\lambda_0))} = \mathcal O(d^{-1/2}),
	\end{equation*}
	for all $t\in I$ with $|t-\lambda_0|< \delta$
	\begin{equation*}
		|\operatorname{Im} h_1(t) - \operatorname{Im} h_2(t)| \leq \frac{C}{|1/r(t) - \xi_1(t)|^2} |g'(\xi_2(t))||t-\lambda_0| \leq \frac{C^2}{\operatorname{Im}(\xi_1(t))^2} \cdot d^{-1} \cdot \delta = \mathcal O(d^{-7/30}),
	\end{equation*}
	for some suitable $\xi_1(t)$ which lies on the line segment in $\C^+$ that connects $g(\lambda_0)$ and $g(t)$ and some suitable $\xi_2(t)$ between $t$ and $\lambda_0$. 
	
	For all $t \in I$ with $|t-\lambda_0|> \varepsilon$ we have
	\begin{align*}
		|\operatorname{Im}h_2(t)| &\leq \operatorname{Im}\Big(\frac{C}{1/r(t) - g(\lambda_0)} \Big) \\ 
		&= \frac{C\operatorname{Im}(g(\lambda_0))}{(\operatorname{Re}(1/r(t)) - \operatorname{Re}(g(\lambda_0)))^2 + \operatorname{Im}(g(\lambda_0))^2}
		\leq C^2 \frac{d^{-1/2}}{(c \varepsilon - Cd^{-1})^2} = \mathcal O(d^{4/20})
	\end{align*}
	such that
	\begin{equation*}
		\int_I \1_{\{ \varepsilon \leq |t-\lambda_0|\leq \delta\}} \operatorname{Im}(h_2(t)) \, \mathrm dt = \mathcal O(d^{-1/30}).
	\end{equation*}
	We have for all $t \in I$ with $|t-\lambda_0|< \varepsilon$
	\begin{equation*}
		|\operatorname{Im}h_2(t) - \operatorname{Im}h_3(t)| \leq \frac{C}{|\xi_3(t) - g(\lambda_0)|^2}|1/r(t) - w(\lambda_0)^{-1}(\lambda_0 - t)| \leq C^2 d |\lambda_0 - t|^2 = \mathcal O(d^{3/10})
	\end{equation*}
	for some suitable $\xi_3(t)\in \R$ leading to
	\begin{equation*}
		\int_I \big|\operatorname{Im}h_2(t) - \operatorname{Im}h_3(t)\big| \1_{\{|t-\lambda_0| < \varepsilon\}} \, \mathrm dt =  \mathcal O(d^{-1/20}).
	\end{equation*}
	Since for all $t \in I$ with $|t-\lambda_0|< \varepsilon$
	\begin{equation*}
		|\operatorname{Im}h_3(t) - \operatorname{Im}h_4(t)| \leq \frac{C |t - \lambda_0|}{|w(\lambda_0)^{-1}(\lambda_0 - t) - g(\lambda_0)|} \leq C^2 \varepsilon d^{1/2} = \mathcal O(d^{3/20})
	\end{equation*}
	we have
	\begin{equation*}
		\int_I  \big|\operatorname{Im}h_3(t) - \operatorname{Im}h_4(t)\big| \1_{\{|t-\lambda_0| < \varepsilon\}} \, \mathrm dt =  \mathcal O(d^{-1/5}).
	\end{equation*}
	Using the identity $\pi/2-\arctan(s) = \arctan(1/s)$ for $s\neq 0$ we finally obtain
	\begin{align*}
		\int_I \operatorname{Im}h_4 (t)\1_{\{|t-\lambda_0| \geq \varepsilon\}} \, \mathrm dt  &\leq 2 \int_{\varepsilon - |\lambda_0 -a(\lambda_0)|}^\infty \frac{b(\lambda_0)}{t^2 + b(\lambda_0)^2} \, \mathrm dt \\
		&= \pi - 2\arctan\big((\varepsilon-|\lambda_0-a(\lambda_0)|)b(\lambda_0)^{-1}\big) \\
		&= 2\arctan\big((\varepsilon-|\lambda_0-a(\lambda_0)|)^{-1} b(\lambda_0) \big) \leq \frac{2b(\lambda_0)}{\varepsilon-|\lambda_0-a(\lambda_0)|} = \mathcal O(d^{-3/20})
	\end{align*}
	which concludes the proof.
\end{proof}

\begin{proof}[Proof of Theorem \ref{Theorem: Large d-expansion of first order term}]
	This directly follows from Proposition \ref{Proposition: distance of first order term to sum of Cauchy kernels} and Theorem \ref{Theorem: small p expansion in dual of Ck}.
\end{proof}

\begin{figure}
	\begin{minipage}{0.97\textwidth}
		\centering
		\includegraphics[width=0.48\linewidth]{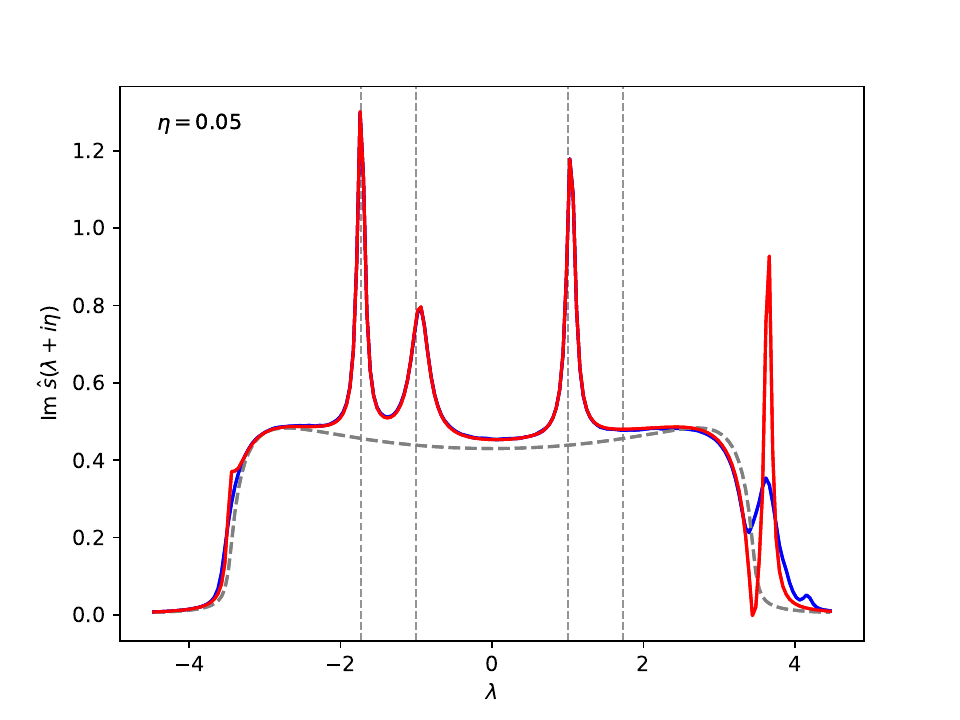}
		\includegraphics[width=0.48\linewidth]{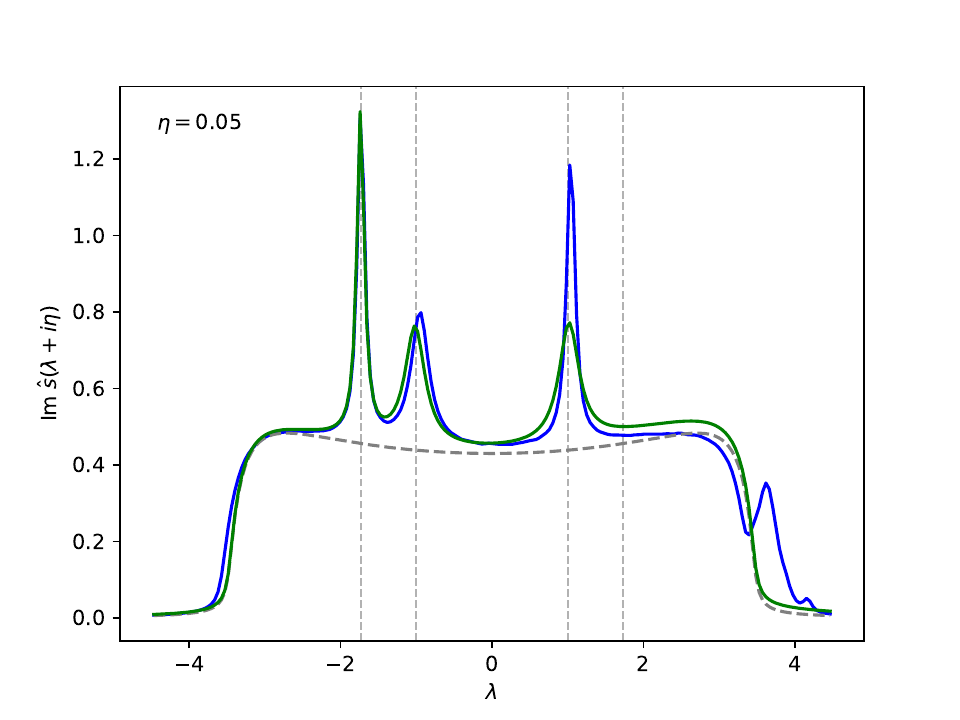}
		\includegraphics[width=0.3\linewidth]{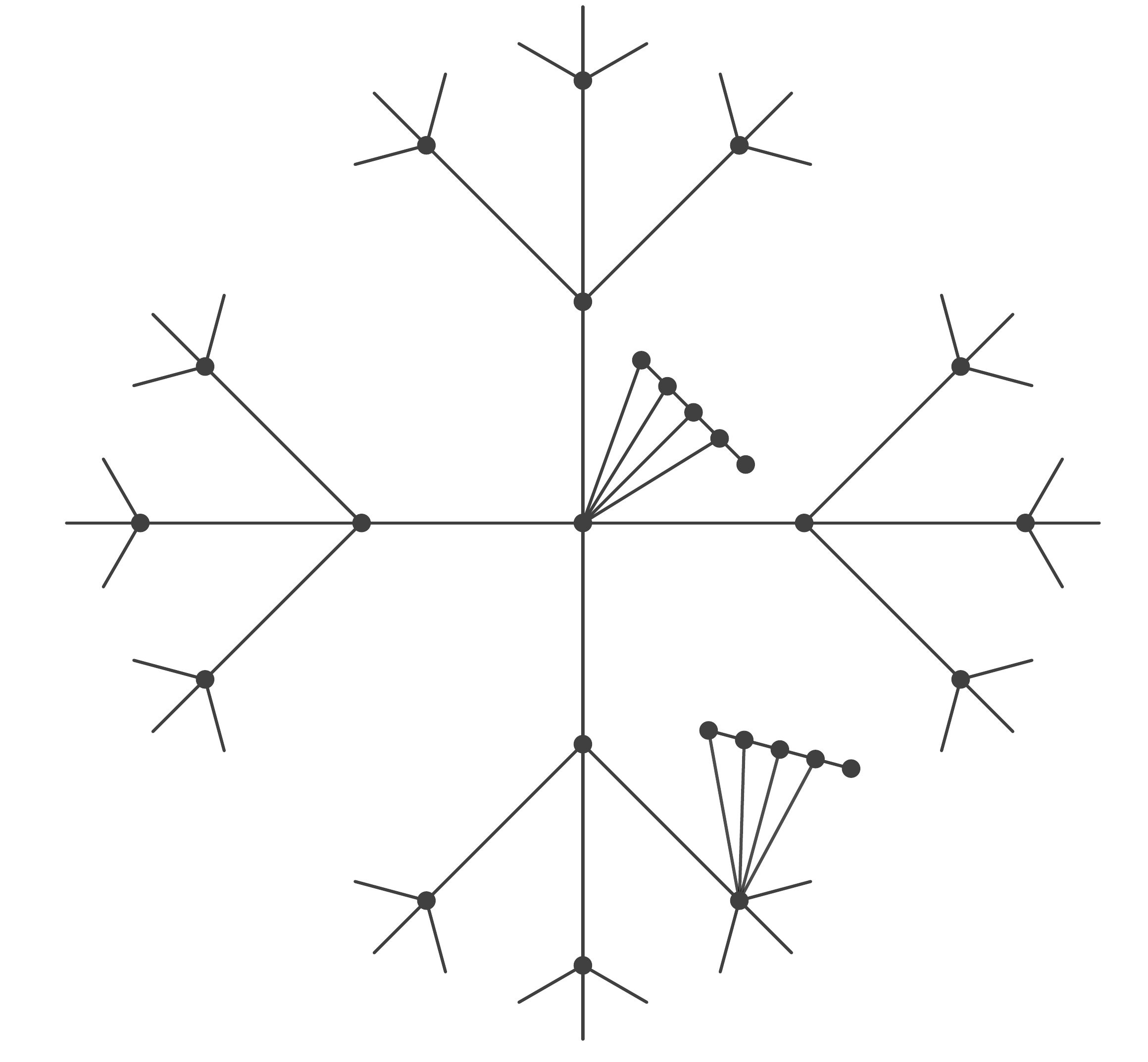}
	\end{minipage}
	\caption{Illustration of Remark \ref{Remark: Expansion around small w} for $\mathbb G = \mathbb B_4$
		and $\mathrm P(\mathrm G = \text{line of length 5})=0.05$, $\mathrm P(\mathrm G = \emptyset )=0.95$.
		Imaginary parts of the Stieltjes transforms of $\nu_0$ (grey), $\hat \nu$ from simulation (blue), and $\nu_0 + p\phi$ (red). In green, a further approximation of $\nu_0 + p\phi$ via \eqref{Equation: small w approx for h} is depicted.
		The grey vertical lines indicate the eigenvalues $(-\sqrt{3}, -1, 1, \sqrt{3})$ of the line segment that lie in $\mathrm S_*$. The corresponding weights are given by $(w(-\sqrt{3}), w(-1), w(1), w(\sqrt{3})) \approx (0.081, 0.25, 0.25, 3.482)$. Close to the eigenvalue $-\sqrt{3}$ with small weight $w(-\sqrt{3})\approx 0.081$, the approximation in \eqref{Equation: small w approx for h} matches the simulation well.}
	\label{fig: small w expansion}
\end{figure}

\begin{remark}
	\label{Remark: Expansion around small w}
	The heuristic leading to the proof of Theorem \ref{Theorem: Large d-expansion of first order term} suggests a second mechanism through which peaks in the spectrum are generated even at small $d$: assume that we are on the event $\{(\mathrm G, \varphi) = (\mathrm G_1, \varphi_1) \}$ with $(\mathrm G_1, \varphi_1)$ deterministic. Assume that we have $\lambda \in \mathrm S_*$ with $w(\lambda)$ small but positive and assume that $I$ is a sufficiently small interval containing $\lambda$ such that $r$ does not have any zeros in $I\setminus \{\lambda\}$. One might then expect with Lemma \ref{Lemma: Expansion of h around admissible eigenvalue} that
	\begin{equation}
		\label{Equation: small w approx for h}
		\pi \operatorname{Im} h(t) \mathrm dt \approx \operatorname{Cau}(a(\lambda)), b(\lambda)) \quad \text{ on }I
	\end{equation}
in a suitable sense, where as before $a(\lambda) \coloneqq \lambda - w(\lambda) \operatorname{Re}(g(\lambda))$.
	See Figure \ref{fig: small w expansion} for an illustration of this phenomenon. 
\end{remark}

\appendix

\section{Some facts about the Stieltjes transform}
\label{Section: Facts about Stieltjes transforms}

In this appendix, we collect some useful facts about the Stieltjes transform which we were unable to find in the literature.

\begin{lemma}
	\label{Lemma: Local uniform convergence of Im ST}
	Assume that $\phi$ is a distribution of order at most one with Stieltjes transform $h$, in the sense that $h$ is holomorphic\footnote{Notice that is in particular the case for $h$ defined in \eqref{Equation: Definition h}: as we saw in the proof of Theorem \ref{Theorem: small p expansion ST}, $h$ is the locally uniform limit of the Stieltjes transforms of sequence of signed measures.} on $\R \setminus \C$ with $h(\overline{z}) = \overline{h(z)}$ for all $z\in \C^+$ and
	\begin{equation*}
		\phi(f) = \frac{1}{\pi}\int_{\C} h(z) \overline \partial f_\psi(z) \, \mathrm dz
	\end{equation*}
	for all $f \in C_c^2(\R)$, where $\psi \in C_c^\infty(\R)$ is an even cutoff function satisfying \eqref{Equation: Cutoff-function} and where $f_\psi$ is the almost analytic extension defined in \eqref{Equation: almost analytic extension of f}.
	Assume that the limit $\lambda \mapsto \operatorname{Im} h(\lambda) \coloneqq \lim_{\eta \downarrow 0} \operatorname{Im} h(\lambda + \mathrm i \eta)$ exists uniformly on $[a, b]$. Then $\phi = \pi^{-1} \operatorname{Im}(h(x)) \, \mathrm dx$ on $(a, b)$.
\end{lemma}

\begin{proof}
	For $\eta\in (0, 1)$, let $K_\eta \coloneqq [a, b] \times [-\eta, \eta]$ and $M_\eta \coloneqq K_1 \setminus K_\eta$. We extend $f\in C^2_c(a, b)$ via \eqref{Equation: almost analytic extension of f} to a function $f_\psi \in C^1_c(\C)$.
	Since $h$ is holomorphic on $\C\setminus \R$, we have $\overline{\partial}h = 0$ on $\C\setminus \R$. Hence, we obtain via Greens Theorem
	\begin{align*}
		\int_{M_\eta} h(z) \overline{\partial} f_\psi(z) \, \mathrm dz &= \int_{M_\eta} \overline{\partial}\big( h(z) f_\psi(z) \big) \, \mathrm dz \\
		&= \frac{1}{2} \oint_{\partial M_\eta} \big(h(\lambda+ \mathrm i \eta)f_\psi(\lambda+ \mathrm i \eta) \mathrm d\eta - \mathrm i h(\lambda+ \mathrm i \eta)f_\psi(\lambda+ \mathrm i \eta)  \mathrm d\lambda\big) \\
		&= -\frac{\mathrm i}{2} \int_a^b \big(h(\lambda+ \mathrm i \eta)f_\psi(\lambda+ \mathrm i \eta) - h(\lambda- \mathrm i \eta)f_\psi(\lambda - \mathrm i \eta)\big) \, \mathrm d\lambda \\
		&= \int_a^b \operatorname{Im}\big(f_\psi(\lambda+ \mathrm i \eta)h(\lambda + \mathrm i \eta)\big) \, \mathrm d\lambda
	\end{align*}
	since $f_\psi$ vanishes on $\partial M_\eta \setminus \big([a, b]\times \{ \pm \eta\}\big)$ and since $h(\overline{z})f_\psi(\overline{z}) = \overline{h(z)f_\psi(z)}$. We hence obtain 
	\begin{equation*}
		\phi(f) = \lim_{\eta\downarrow 0} \frac{1}{\pi} \int_{M_\eta} h(z) \overline{\partial} f_\psi(z) \, \mathrm dz = \frac{1}{\pi} \int_a^b \operatorname{Im}(h(\lambda)) f(\lambda) \, \mathrm d\lambda. 
	\end{equation*}
	Since $C_c^2(a, b)$ is dense in $C_c^1(a, b)$, the statement follows.
\end{proof}

\begin{prop}
	\label{Proposition: Limit of ST}
	Let $f\in C^2(a, b)$ be a probability density and let $g$ be the Stieltjes transform of $f(x) \mathrm dx$. Then the limit $g(\lambda) \coloneqq \lim_{\eta \downarrow 0} g(\lambda + \mathrm i \eta)$ exists locally uniformly in $\lambda \in (a, b)$. Moreover, $\lambda \mapsto g(\lambda)$ is differentiable with $g'(\lambda) = \lim_{\eta \downarrow 0} g'(\lambda + \mathrm i \eta)$ locally uniformly in $\lambda \in (a, b)$ and
	\begin{align*}
		g(\lambda) &= \mathrm i \pi f(\lambda) + \int_{-c}^{c} \frac{f(\lambda + t) - f(\lambda)}{t} \mathrm dt \\
		g'(\lambda) &= \mathrm i \pi f'(\lambda) +  \int_0^\infty \frac{f(\lambda + t) + f(\lambda -t) - 2f(\lambda)}{t^2} \, \mathrm dt
	\end{align*}
	for some arbitrary $c \geq 2 \operatorname{max}(|a|, |b|)$ and where we extend $f$ by zero to all of $\mathbb R$.
\end{prop}

\begin{proof}
	We show that the limits
	\begin{equation*}
		\lim_{\eta \downarrow 0} \operatorname{Re}(g(\lambda + \mathrm i \eta)), \quad \lim_{\eta \downarrow 0} \operatorname{Im}(g(\lambda + \mathrm i \eta)), \quad \lim_{\eta \downarrow 0} \operatorname{Re}(g'(\lambda + \mathrm i \eta)), \quad \lim_{\eta \downarrow 0} \operatorname{Im}(g'(\lambda + \mathrm i \eta))
	\end{equation*}
	exit locally uniformly in $\lambda\in (a, b)$ and that their limit is as given in the claim. In particular, this yields differentiability of $\lambda \mapsto g(\lambda)$ with $g'(\lambda) = \lim_{\eta \downarrow 0} g'(\lambda + \mathrm i \eta)$.
	We start by noticing that for all $\lambda \in \R$ and $\eta>0$
	\begin{align*}
		g(\lambda + \mathrm i \eta) &= \int \frac{t-\lambda}{(t-\lambda)^2 + \eta^2} f(t)\, \mathrm dt  + \mathrm i \int \frac{\eta}{(t-\lambda)^2 + \eta^2} f(t)\, \mathrm dt  \\
		g'(\lambda + \mathrm i \eta) &= \int \frac{(t - \lambda)^2 - \eta^2}{\big[(t-\lambda)^2 + \eta^2\big]^2} f(t)\, \mathrm dt  + \mathrm i \int \frac{2(t - \lambda)\eta}{\big[(t-\lambda)^2 + \eta^2\big]^2} f(t)\, \mathrm dt.
	\end{align*}
	We set for $\delta>0$ and a function $\tilde f:(a, b) \to \R$
	\begin{equation*}
		K_\delta \coloneqq [a+\delta, b-\delta], \quad \|\tilde f\|_\delta \coloneqq \sup_{t\in K_\delta} |\tilde f(t)|.
	\end{equation*}
	In the following, we fix some $\delta>0$ and show uniform convergence on $K_{2\delta}$.
	\paragraph{Convergence of $\operatorname{Im}(g)$}
	Since $\pi^{-1}\operatorname{Im}(g)$ is the convolution of $f$ with the Cauchy kernel, we have $\operatorname{Im}(g) \to \pi f$ locally uniformly. While this is well known, we give the proof for convenience of the reader. Let $\varepsilon>0$ and $\lambda \in K_{2\delta}$. Since $f$ is uniformly continuous on $K_\delta$, there exists some $\eta_0\in (0, \delta^3)$ such that $|f(s) - f(t)| < \varepsilon$ for all $s, t \in K_\delta$ with $|s-t|<\eta_0^{1/3}$. Let $\eta\in (0, \eta_0)$. Since $\int_\R \frac{\eta}{(t-\lambda)^2 + \eta^2} \, \mathrm dt = \pi$ 
	and since $f$ is a probability density, we obtain for all $\lambda \in K_{2\delta}$
	\begin{align*}
		&|\operatorname{Im}g(\lambda + \mathrm i \eta) - \pi f(\lambda) |\\
		&=  \Big|\int_\R \frac{\eta}{(t-\lambda)^2 + \eta^2} (f(t) - f(\lambda))\, \mathrm dt \Big| \\
		&\leq \int_{|t-\lambda|\leq \eta^{1/3} } \frac{\eta \varepsilon}{(t-\lambda)^2 + \eta^2}\, \mathrm dt + \int_{|t-\lambda|> \eta^{1/3} } \eta^{1/3} f(t)\, \mathrm dt +  f(\lambda)\int_{|t-\lambda|> \eta^{1/3} } \frac{\eta}{(t-\lambda)^2 + \eta^2}\, \mathrm dt \\
		&\leq \varepsilon \pi + \eta^{1/3} +  \big(\pi - 2\operatorname{arctan}(\eta^{-2/3}) \big) \|f \|_{\delta}
	\end{align*}
	which implies the uniform convergence of $\lambda \mapsto \operatorname{Im}g(\lambda + \mathrm i \eta)$ to $\pi f$ on $K_\delta$ as $\eta \downarrow 0$.
	
	\paragraph{Convergence of $\operatorname{Re}(g)$} Let $\lambda \in K_{2\delta}$. Since $\int_{- c}^{c} \frac{t}{t^2 + \eta^2} \mathrm dt = 0$
	and as $f(t+\lambda) = 0$ for all $t\in [-c, c] \setminus [a-\lambda, b-\lambda]$, we have
	\begin{equation*}
		\operatorname{Re}(g(\lambda + \mathrm i \eta)) = \int_{a-\lambda}^{b-\lambda} \frac{t}{t^2 + \eta^2}f(\lambda + t) \, \mathrm dt  = \int_{-c}^{c} \frac{t^2}{t^2 + \eta^2} \frac{f(\lambda + t) - f(\lambda)}{t} \, \mathrm dt
	\end{equation*}
	yielding
	\begin{equation}
		\label{Equation: Difference of Re g to its limit}
		\operatorname{Re}(g(\lambda + \mathrm i \eta)) - \int_{-c}^{c} \frac{f(\lambda + t) - f(\lambda)}{t} \mathrm dt = \int_{-c}^{c} \frac{-\eta^2}{t^2 + \eta^2} \frac{f(\lambda + t) - f(\lambda)}{t} \, \mathrm dt.
	\end{equation}
	For $\lambda \in K_{2\delta}$ and $|t|\leq \delta$ we have $t+\lambda \in K_\delta$ and hence
	\begin{equation}
		\label{Equation: Estimate in difference quotient}
		\Big|\frac{f(\lambda + t) - f(\lambda)}{t}\Big| \leq \1_{\{|t|< \delta  \} }\|f'\|_{ K_{\delta}} +  \1_{\{|t|> \delta  \} } \delta^{-1}\big(f(\lambda + t) + \|f\|_{2\delta}\big).
	\end{equation}
	Since $f$ is a probability density, we have
	\begin{equation*}
		\int_{-c}^{c} \frac{\eta^2}{t^2 + \eta^2} f(\lambda + t) \1_{\{|t|> \delta  \}} \, \mathrm dt \leq \eta^{2} \delta^{-2}
	\end{equation*}
	and hence Equation \eqref{Equation: Difference of Re g to its limit} yields with the Estimate \eqref{Equation: Estimate in difference quotient}
	\begin{align*}
		\Big|\operatorname{Re}(g(\lambda + \mathrm i \eta)) - \int_{-c}^{c} \frac{f(\lambda + t) - f(\lambda)}{t} \mathrm dt \Big| &\leq  \Big(\|f'\|_{ K_{\delta}} + \delta^{-1}\|f\|_{2\delta}\Big)\int_{-c}^{c} \frac{\eta^2}{t^2 + \eta^2}\, \mathrm dt + \eta^{2} \delta^{-2}
	\end{align*}
	which in turn implies the uniform convergence of $\lambda \mapsto \operatorname{Re}g(\lambda + \mathrm i \eta)$ on $K_{2\delta}$ as $\eta \downarrow 0$.
	
	\paragraph{Convergence of $\operatorname{Im}(g')$} Let $\lambda \in K_{2\delta}$. Using that $\int_{-\infty}^\infty \frac{2\eta t}{(t^2 + \eta^2)^2} \, \mathrm dt = 0$ and $\int_{-\infty}^\infty \frac{2\eta t^2}{(t^2 + \eta^2)^2} \, \mathrm dt  = \pi$,
	we obtain 
	\begin{equation*}
		\operatorname{Im}(g'(\lambda + \mathrm i \eta)) - \pi f'(\lambda)
		= \int_{-\infty}^\infty \frac{2\eta t^2}{(t^2 + \eta^2)^2} \bigg(\frac{f(\lambda + t)- f(\lambda)}{t} - f'(\lambda)  \bigg) \, \mathrm dt.
	\end{equation*}
	Estimating
	\begin{equation*}
		\Big|\frac{f(\lambda + t)- f(\lambda)}{t} - f'(\lambda)\Big| \leq \1_{\{|t|\leq \delta \}} \|f''\|_{\delta} |t|  + \1_{\{|t|> \delta \}} \big( \delta^{-1} f(\lambda+t) + \delta^{-1} \|f\|_{2\delta}  +  \|f'\|_{2\delta} \big)
	\end{equation*}
	yields with
	\begin{equation*}
		\int_{-\infty}^\infty \frac{2\eta t^2}{(t^2 + \eta^2)^2} f(\lambda+t) \1_{\{|t|> \delta \}} \, \mathrm dt \leq \int_{-\infty}^\infty 2\eta \delta^{-2} f(\lambda+t) \, \mathrm dt = 2\eta \delta^{-2}
	\end{equation*}
	the estimate
	\begin{align*}
		&|\operatorname{Im}(g'(\lambda + \mathrm i \eta)) - \pi f'(\lambda)| \\
		&\leq  \|f''\|_{\delta} \int_{-\infty}^\infty \frac{2\eta |t|^3}{(t^2 + \eta^2)^2}  \1_{\{|t|\leq \delta \}} \, \mathrm dt + (\delta^{-1} \|f\|_{2\delta}  +  \|f'\|_{2\delta}) \int_{-\infty}^\infty \frac{2\eta t^2}{(t^2 + \eta^2)^2} \1_{\{|t|> \delta \}}\, \mathrm dt + 2\delta^{-3} \eta \\
		&\leq  \|f''\|_{\delta} \int_{-\infty}^\infty \frac{\eta|t|^3}{(t^2 + 1)^2} \1_{\{|t| \leq \delta/\eta \}} \, \mathrm dt + (\delta^{-1} \|f\|_{2\delta}  +  \|f'\|_{2\delta}) \int_{-\infty}^\infty \frac{2 t^2}{(t^2 + 1)^2} \1_{\{|t|> \delta/\eta \}}\, \mathrm dt + 2\delta^{-3}  \eta
	\end{align*}
	which implies the uniform convergence of $\lambda \mapsto \operatorname{Im}g'(\lambda + \mathrm i \eta)$ on $K_\delta$ as $\eta \downarrow 0$ as
	\begin{equation*}
		\int_{-\infty}^\infty \frac{\eta|t|^3}{(t^2 + 1)^2} \1_{\{|t| \leq \delta/\eta \}} \, \mathrm dt  \leq 2 \eta \big(1 + \log(\delta/\eta) \big).
	\end{equation*}
	\paragraph{Convergence of $\operatorname{Re}(g')$} Let $\lambda \in K_{2\delta}$. Using that $\int_{-\infty}^\infty \frac{t^2 -\eta^2}{(t^2 + \eta^2)^2} \, \mathrm dx = 0$,
	we obtain
	\begin{align*}
		\operatorname{Re}(s'(\lambda + \mathrm i \eta)) &= \int_{-\infty}^\infty \frac{t^2 - \eta^2}{[t^2 + \eta^2]^2} f(\lambda + t) \, \mathrm dt = \int_0^\infty  \frac{t^4 -\eta^2t^2}{(t^2 + \eta^2)^2} \frac{f(\lambda + t) + f(\lambda - t) - 2f(\lambda)}{t^2} \, \mathrm dt.
	\end{align*}
	We estimate
	\begin{align*}
		\Big|\frac{f(\lambda + t) + f(\lambda - t) - 2f(\lambda)}{t^2}\Big| \leq \1_{\{|t|\leq \delta\}} \|f''\|_{\delta} + \1_{\{|t|>\delta\}} \big( \delta^{-2} f(\lambda+t) + \delta^{-2} f(\lambda-t) + 2f(\lambda) t^{-2} \big)
	\end{align*}
	as well as
	\begin{equation*}
		\sup_{|t|>\delta} \frac{3\eta^2t^2 + \eta^4}{(t^2 + \eta^2)^2} = \sup_{|t|>\delta/\eta} \frac{3t^2 + 1}{(t^2 + 1)^2} \leq 4 \eta^{2} \delta^{-2}
	\end{equation*}
	for $\eta< \delta$ and obtain
	\begin{align*}
		&\Big|\operatorname{Re}(s'(\lambda + \mathrm i \eta)) - \int_0^\infty \frac{f(\lambda + t) + f(\lambda - t) - 2f(\lambda)}{t^2} \, \mathrm dt \Big| \\
		&=  \Big| \int_0^\infty  \frac{3\eta^2t^2 + \eta^4}{(t^2 + \eta^2)^2} \frac{f(\lambda + t) + f(\lambda - t) - 2f(\lambda)}{t^2} \, \mathrm dt \Big| \\
		&\leq \|f''\|_{\delta} \int_0^\infty  \frac{3\eta^2t^2 + \eta^4}{(t^2 + \eta^2)^2} \, \mathrm dt + 4\eta^2 \delta^{-4} \int_0^\infty \big(f(\lambda+t) + f(\lambda-t)  + 2\delta^2 f(\lambda)t^{-2} \big)  \1_{\{|t|>\delta\}} \, \mathrm dt \\
		&\leq \eta \|f''\|_{\delta} \int_0^\infty  \frac{3t^2 + 1}{(t^2 + 1)^2} \, \mathrm dt + 8 \eta^2 \delta^{-4} + 8 \eta^2 \delta^{-3} \|f\|_{2\delta}
	\end{align*}
	which yields the uniform convergence of $\lambda \mapsto \operatorname{Re}g'(\lambda + \mathrm i \eta)$ on $K_{2\delta}$ as $\eta \downarrow 0$.
\end{proof}
\begin{remark}
	\label{Remark: Derivative of real part of ST}
	We point out that the formula for $\operatorname{Re}(g')$ given above can be interpreted in terms of the Poisson-semigroup.
	The function $\operatorname{Im}g$ solves the Laplace equation
	\begin{equation*}
		\begin{cases}
			(\partial_\lambda^2 + \partial_\eta^2)\operatorname{Im}g(\lambda + \mathrm i \eta) = 0  &\text{ for }\lambda \in \R, \eta>0\\
			\operatorname{Im}g(\lambda + \mathrm i0)  = \pi f(\lambda) &\text{ for }\lambda \in \mathbb R
		\end{cases}
	\end{equation*}
	i.e.\ we (at least formally) have $\pi^{-1}\operatorname{Im} g(\lambda + \mathrm i \eta) =  (e^{\eta T}f)(\lambda)$ where $T = - \sqrt{-\partial_\lambda^2}$ is the generator of the Poisson-semigroup.
	By the Cauchy-Riemann differential equations, we have at least formally
	\begin{equation*}
		\partial_\lambda \operatorname{Re}(g(\lambda)) =  \partial_\eta \operatorname{Im}g(\lambda + \mathrm i \eta)|_{\eta = 0} = \pi \cdot (Tf)(\lambda).
	\end{equation*}
	The integral formula for $\partial_\lambda \operatorname{Re}(g(\lambda))$ given in Proposition \ref{Proposition: Limit of ST} coincides with the known integral formula for $T$.
\end{remark}

\section{Expansion around $p=1$}
\label{Subsection: Expansion around p=1}
For the proof of Theorem \ref{Theorem: small p expansion in dual of Ck}, we treated $p$ as a continuous parameter and introduced the measure $\hat \nu_p$ corresponding to a sprinkling of $\mathbb G$ with the sprinkling distributions \eqref{def_Pp}.
Similar to Theorem \ref{Theorem: small p expansion in dual of Ck}, one can derive an expansion of $\hat \nu_p$ around $p\in [0, 1)$: Namely, we can expand
\begin{equation}
	\label{Equation: Expansion around general p}
	\hat \nu_{p+h} = \hat \nu_p + h \phi_p + h^2 R_{p, h}, \quad p, h \in [0, 1] \text{ s.t. } p+h\in [0, 1]
\end{equation}
where $\phi_p = \partial_p \hat \nu_p$ is a distribution of order at most one whose restriction to $C_c^2(\R)$ is given by
\begin{equation*}
	f \mapsto \phi_p(f) \coloneqq \frac{1}{\pi}\int_{\C \setminus\R} h_p(z) \overline{\partial} f_{\psi}(z) \,  \mathrm dz
\end{equation*}
with
\begin{equation*}
	h_p \coloneqq \mathbb E_p\bigg[\frac{\partial_z (\theta_\mathbb og_{\operatorname{s}} r)}{1 - \theta_\mathbb o g_{\operatorname{s}} r} + \sum_{\lambda \in \mathrm S_*}\frac{1}{\lambda-z}\Big| \mathrm G\neq \emptyset \bigg] \quad \text{ where } \quad g_{\operatorname{s}}(z) \coloneqq \big\langle \delta_\mathbb o, (A-z)^{-1} \delta_\mathbb o \big\rangle
\end{equation*}
and where $\mathrm T$ is, as usual, independent of $A$, i.e\ $g_{\operatorname{s}}$ and $r$ are independent.
For every $p \in [0, 1)$ and $a<b$, we have the bound 
\begin{equation*}
	\|R_{p, h}\|_{C_c^{8}(a, b)} \leq C |b-a|(1-p)^{-1}
\end{equation*}
for some $C>0$ only depending on $\mathbb E[m^2|\mathrm G\neq \emptyset]$ and the constant $c(a, b)$ from Assumption \eqref{Assumption: Assumption for small p expansion}. The proof is almost identical to the proof of Theorem \ref{Theorem: small p expansion in dual of Ck}, one solely needs to replace Lemmas \ref{Proposition: C1 distance estimate} and \ref{Proposition: Estimate on difference of ST} with the following Lipschitz estimates.
\begin{lemma}
	\label{Lemma: General Lipschitz estimates}
	We have for all $p, h\in [0, 1]$ with $p+h\in [0, 1]$, for all $z\in \C^+$ and all $a<b$,
	\begin{equation*}
		\| \hat \nu_{p+h} - \hat \nu_p \|_{(C^1_c(a, b))'} \leq 4 h |b-a| \mathbb E[\tilde m], \qquad		|\hat s_{p+h}(z) - \hat s_{p}(z)| \leq 5 h \mathbb E[\tilde m] \operatorname{Im}(z)^{-1}.
	\end{equation*}
\end{lemma}
We shall prove Lemma \ref{Lemma: General Lipschitz estimates} at the end of this subsection. However, for $p\neq 0$ the Expansion \eqref{Equation: Expansion around general p} seems to be of limited utility as one typically cannot explicitly compute $\phi_p$ in this case. This changes, however, under the Assumption \eqref{Assumption: Detetministic attached graphs} if we expand around $p=1$. 
In the following, we show how one can then calculate the Stieltjes transforms of $\hat \nu_1$ and $\phi_1$ and derive a bound on the remainder. Notice that for $p=1$ the function $r$ as well as the set $\mathrm S_*$ are deterministic.
\begin{prop}
	\label{Propostion: s for p=1 and detAG}
	Assume that \eqref{Assumption: Detetministic attached graphs} holds with $p=1$. Then the Stieltjes transform $\hat s_1$ of $\hat \nu_1$ is given by
	\begin{equation*}
		\hat s_1(z) = \hat s_0(z+ r(z))(1 + \partial_z r(z)) + \sum_{\lambda \in \mathrm S_*}\frac{1}{\lambda-z}.
	\end{equation*}
	Moreover, we have $g_{\operatorname{s}}(z) = g(z + r(z))$.
\end{prop}

\begin{proof}
	Its is sufficient to show the statement in finite volume.
	Writing $A$ in the block form \eqref{Equation: blockform adjacency matrix} yields with the Schur complement formula that $G = (A-z)^{-1}$ is of the form
	\begin{equation*}
		G = 
		\begin{pmatrix}
			(B - z- E RE^*)^{-1}  &  * \\
			*					  &R + R E^*(B- z- ER E^*)^{-1} E R \\	
		\end{pmatrix}
	\end{equation*}
	where $R = (T-z)^{-1}$. As in the proof of Lemma \ref{Lemma: attach n-blocks to one vertex}, one calculates
	\begin{align*}
		&\big[E R E^*]_{xy}  = \delta_{xy} r_x(z)\\
		&\big[R + R E^*(B-z - ER E^*)^{-1} E R\big]_{(xt)(xt)} = \mathrm R^{x}_{tt} + \big[(\mathrm R^{x} \varphi^x)_t\big]^2\Big[\big(B-z -  r_x \big)^{-1}\Big]_{xx}.
	\end{align*}
	where neither $\mathrm R^x$ nor $r_x$ depend on $x$ as $p=1$. We obtain $g_{\operatorname{s}}(z) = G_{\mathbb o\mathbb o}(z)  = g(z + r(z))$
	for all $z\in \mathbb C^+$, as well as
	\begin{align*}
		\hat s_1(z) &= \mathbb E_1\Big[\frac{1}{|\mathbb V|}\sum_{x\in \mathbb V} G_{xx} + \sum_{t \in \mathrm V^x} G_{(xt)(xt)} \Big] - \hat s_{1, \operatorname{TF}}(z) \\
		&= s_0(z + r(z))  +   \mathbb E_1\Big[\frac{1}{|\mathbb V|}\sum_{x\in \mathbb V} \sum_{t \in \mathrm V^x} \mathrm R^{x}_{tt} + \frac{1}{|\mathbb V|}\sum_{x\in \mathbb V}\big[(B-z -  r(z))^{-1}\big]_{xx} \sum_{t \in \mathrm V^x}  \big[(\mathrm R^{x} \varphi^x)_t\big]^2 \Big] - \hat s_{1, \operatorname{TF}}(z) \\
		&=  s_0(z + r(z))  +    q(z) +  (\partial_z r(z)) s_0(z + r(z)) - \hat s_{1, \operatorname{TF}}(z) \\
		&=   \hat s_0(z+ r(z))(1 + \partial_z  r(z)) + \sum_{\lambda \in \mathrm S_*}\frac{1}{\lambda-z},
	\end{align*}
	where we used Lemma \ref{Lemma: ST of tuning fork contribution} for the last equality.
\end{proof}

\begin{cor}
	\label{Corollary: Total repulsion for p=1 and detAG}
	Assume that \eqref{Assumption: Detetministic attached graphs} holds and assume that $\lambda \in \mathrm D$ is admissible. Assume that  $\Sigma_0 \coloneqq \operatorname{supp}(\nu_0)$ is bounded. Then there exist some $\varepsilon>0$ such that $\hat\nu_p([\lambda-\varepsilon, \lambda+\varepsilon]) \leq 8M(1-p)$,
	where $M$ is such that $m=M$ conditionally on $\mathrm G\neq \emptyset$.
\end{cor}
\begin{proof}
	Since $\lambda$ is admissible, we have $\lim_{\eta \downarrow 0} |\tilde r(\lambda + \mathrm i \eta)| = \infty$. Since $\Sigma_0$ is bounded, there hence exists some $\varepsilon>0$ such that $I \coloneqq [\lambda-\varepsilon, \lambda+\varepsilon]$ does not contain any other elements from $\mathrm D$ and such that $t + \tilde r(t) \notin \Sigma_0$ for all $t\in I \setminus \{\lambda\}$. With Proposition \ref{Propostion: s for p=1 and detAG} we hence obtain $\lim_{\eta \downarrow 0}\operatorname{Im}\hat s_1(t + \mathrm i \eta) = 0$ for all $t\in I \setminus \{\lambda\}$ i.e.\ $\hat \nu_1(I \setminus \{\lambda\}) = 0$. On the other hand, we have $\hat \nu_1(\{\lambda\}) = 0$ by Corollary \ref{Proposition: Tuning fork contribution for deterministic AG} and hence $\hat \nu_1(I) = 0$. The statement follows from the estimate $\| \hat F_{p} - \hat F_1 \|_\infty \leq 4M (1-p)$,
	which we shall show in the proof of Lemma \ref{Lemma: General Lipschitz estimates}.
\end{proof}

\begin{figure}
	\centering
	\includegraphics[width=0.7\linewidth]{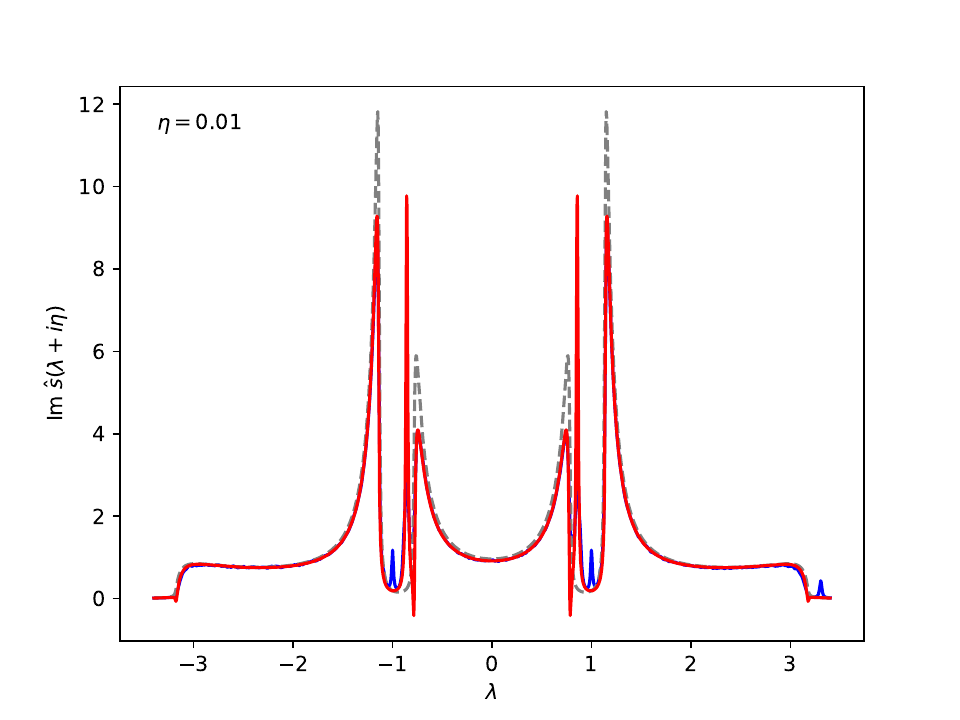}
	\caption{
		Imaginary part of the Stieltjes transform of $\hat \nu$ for $\mathbb G = \mathbb B_3$ and $\mathrm P(\mathrm G = \text{one line of length 2})=0.9$, $\mathrm P(\mathrm G = \emptyset )=0.1$. Simulation in blue, approximation as given by Proposition \ref{Propositon: Expansion around p=1} in red, graph of $\operatorname{Im} \hat s_1$ as given by Proposition \ref{Propostion: s for p=1 and detAG} in grey. By Corollary \ref{Corollary: Total repulsion for p=1 and detAG}, the mass of small intervals around the eigenvalues $\pm 1$ of the line of length two is of order $\mathcal O(1-p)$.
	}
	\label{fig:expansionaround1}
\end{figure}

\begin{prop}
	\label{Propositon: Expansion around p=1}
	Assume that \eqref{Assumption: Detetministic attached graphs} holds. We then have
	\begin{equation}
		\label{Equation: formula for h1}
		h_1(z) = \mathbb E_0\Big[\frac{\partial_z \big(g(z + r(z))r(z) \big) }{1 + g(z + r(z)) r(z)} \Big] + \sum_{\lambda \in \tilde{\mathrm S}_*}\frac{1}{\lambda-z}.
	\end{equation}
	Moreover, there exists some constant $C>0$ such that the family of remainders $(R_{p, h})_h$ in the Expansion \eqref{Equation: Expansion around general p} satisfies for $p\in (0, 1]$
	\begin{equation}
		\label{Equation: Estimate restterm for p=1}
		\|R_{p, h}\|_{(C^{10}_c(a, b))'} \leq C p^{-1} |b-a|.
	\end{equation}
\end{prop}

\begin{proof}
	By Lemma \ref{Lemma: attach n-blocks to one vertex}, we have
	\begin{equation*}
		g_{\operatorname{s}} = \theta_\mathbb og_{\operatorname{s}} + \frac{\theta_\mathbb o g^2_{\operatorname{s}} r}{1- \theta_\mathbb o g_{\operatorname{s}}r} = \frac{\theta_\mathbb o g_{\operatorname{s}}}{1- \theta_\mathbb o g_{\operatorname{s}}r} \quad \iff \quad \theta_\mathbb o g_{\operatorname{s}} = \frac{g_{\operatorname{s}}}{1 + g_{\operatorname{s}}r}.
	\end{equation*}
	which yields after a small calculation
	\begin{equation*}
		\frac{\partial_z(\theta_\mathbb og_{\operatorname{s}}r)}{1-\partial_z(\theta_\mathbb og_{\operatorname{s}}r)} = \frac{\partial_z(g_{\operatorname{s}}r) }{1+\partial_z(g_{\operatorname{s}}r)}.
	\end{equation*}
	The identity in \eqref{Equation: formula for h1} hence follows with Proposition \ref{Propostion: s for p=1 and detAG}. For the estimate in \eqref{Equation: Estimate restterm for p=1}, we notice that we have by Lemma \ref{Lemma: attach n-blocks to one vertex} 
	\begin{equation*}
		\theta^x G_{xy} = \theta_x G_{xy} + \frac{\theta_x G_{xy} \theta_x G_{xx} r}{1- \theta_x G_{xx} r} \quad \iff \quad \theta_x G_{xy} = \theta^x G_{xy} (1 - \theta_x G_{xx}r)
	\end{equation*}
	and hence $\mathbb E[| \theta_{xy} G_{xy}|^2 ] \leq C \eta^{-2} \mathbb E[| \theta_{y}^x G_{xy}|^2 ]$.
	If we hence estimate
	\begin{equation*}
		\mathbb E[| \theta_{y}^x G_{xy}|^2] = \mathbb E[| \theta_{y} G_{xy}|^2 | \chi_x = 1] \leq p^{-1} \mathbb E[| \theta_{y} G_{xy}|^2],
	\end{equation*}
	we obtain as in the proof of Proposition \ref{Proposition: Estimate on second derivative of ST} $\partial_p^2 \hat s_p(z) \leq C p^{-1} \eta^{-9}$.
	Applying the Helffer--Sjöstrand formula as in the proof of Proposition \ref{Propositio: Small p expansion in finite volume} yields \eqref{Equation: Estimate restterm for p=1}.
\end{proof}

For an illustration of the Expansion \eqref{Equation: Expansion around general p} in $p=1$ for \eqref{Assumption: Detetministic attached graphs}, see Figure \ref{fig:expansionaround1}.
It is left to show the Lipschitz estimates given in Lemma \ref{Lemma: General Lipschitz estimates}.
\begin{proof}[Proof of Lemma \ref{Lemma: General Lipschitz estimates}]
	Its is sufficient to show the statement in finite volume. We may assume without loss of generality that $h>0$.
	Let $\chi^p_x \coloneqq \1_{\{U_x \leq p\}}$ where $(U_x)_{x\in \mathbb V}$ are iid $\mathcal U[0, 1]$ distributed. Let $A^p$ denote the adjacency matrix of the graph $\mathcal G_p$ defined as $\mathcal G$ in the beginning of Subsection \ref{Subsection: Small p expansion in finite volume}, but with $(\chi_x)_x$ being replaced by $(\chi^p_x)_x$ and let $A^p$ denote its adjacency matrix. Let $\mathcal G_{p, h}$ be the graph obtained from $\mathcal G_{p+h}$ by deleting all edges contained in
	\begin{equation*}
		\big\{  \{x, (xt) \}:\, x \in \mathbb V,\, \, \chi_x^{p} = 0, \, \chi_x^{p+h} = 1, \,  t\in \tilde{\mathrm V}_x \text{ with }\tilde{\varphi}^x_t = 1\big\},
	\end{equation*}
	let $A^{p, h}$ denote its adjacency matrix and let
	\begin{equation*}
		\mu  \coloneqq \mathbb E\Big[\frac{1}{|\mathbb V|}\sum_{\lambda \in \sigma(A)} \gamma(A^{p, h}, \lambda) \delta_\lambda\Big]
	\end{equation*}
	We then have
	\begin{equation*}
		\|\hat\nu_{p}+\hat\nu_{p, \operatorname{TF}}-\mu \|_{\operatorname{TV}} = \Big \| \mathbb E\Big[\frac{1}{|\mathbb V|} \sum_{x\in \mathbb V} \1_{\{ p < U_x \leq p+h \}}  \sum_{\lambda \in \mathrm S^x} \gamma(\mathrm T^{x}, \lambda) \delta_\lambda  \Big] \Big \|_{\operatorname{TV}} \leq h \mathbb E[\tilde m]
	\end{equation*}
	as well as
	\begin{align*}
		\big|(\hat\nu_{p+h}+\hat\nu_{p+h, \operatorname{TF}}-\mu)((-\infty, \lambda])\big| &\leq \mathbb E\Big[\Big|\frac{1}{|\mathbb V|}\sum_{\lambda \in \sigma(A^{p+h})} \gamma(A^{p+h}, \lambda) \1_{\{\lambda \leq t\}} - \frac{1}{|\mathbb V|}\sum_{\lambda \in \sigma(A^{p, h})} \gamma(A^{p, h}, \lambda) \1_{\{\lambda \leq t\}} \Big|\Big] \\
		&\leq \frac{1}{|\mathbb V|} \mathbb E\Big[\operatorname{rank}(A^{p+h} - A^{p, h}) \Big] \\
		&= \frac{1}{|\mathbb V|} \mathbb E\Big[\sum_{x \in \mathbb V} \1_{\{ p < U_x \leq p+h \}} \tilde 2m_x^* \Big] \leq 2h \mathbb E[\tilde m]
	\end{align*}
	where we applied \cite[Theorem A.43]{BS16} for the second inequality. Using Proposition \ref{Proposition: Tuning fork contribution}, one easily shows that  $\big\|\hat\nu_{p+h, \operatorname{TF}}- \hat\nu_{p, \operatorname{TF}} \big\|_{\operatorname{TV}} \leq  h  \mathbb E[\tilde m]$.
	Combining the above yields $\| \hat F_{p+h} - \hat F_p \|_\infty \leq 4 h \mathbb E[\tilde m]$,
	and hence the claimed estimate in the $(C^1_c(a, b))'$ norm, see the proof of Lemma \ref{Proposition: C1 distance estimate}. The estimate on the Stieltjes transform can be shown as in the proof of Lemma \ref{Proposition: Estimate on difference of ST}, using the coupling of the $(\chi^p_x)_x$ which we introduced above.
\end{proof}

\section{List of symbols} \label{sec:symbols}

\begin{center}
\begin{tabular}{@{}llc@{}}
\toprule
\\[-1em]
	\textbf{Symbol} & \textbf{Description} & \textbf{Defined in} \\[0.1em]
	\midrule
	$\mathbb G$ & Core graph & Sec.~\ref{Subsection: notation spinkled graphs}\\
	$\mathbb V$ & Vertex set of $\mathbb G$ & Sec.~\ref{Subsection: notation spinkled graphs}\\
	$\mathbb V^+$ & Set of vertices of $\mathbb G$ with $m_x \neq 0$ & Sec.~\ref{Subsection: notation spinkled graphs}\\
	$\mathbb V_\lambda$ & Site percolation on $\mathbb V$ with percolation parameter $\mathbb P(r_x(\lambda) \neq \infty)$ & Sec.~\ref{Subsection: results atoms and the Anderson-percolation representation} \\
	$\mathcal G$ & Sprinkled graph & Sec.~\ref{Subsection: notation spinkled graphs}\\
	$\mathcal V$ & Vertex set of $\mathcal G$ & Sec.~\ref{Subsection: notation spinkled graphs}\\
	$\mathrm G^x$ & Attached graph in $x \in \mathbb V$  & Sec.~\ref{Subsection: notation spinkled graphs} \\
	$\varphi^x$ & Indicator function of set of vertices in $\mathrm G^x$ connected to $x$ & Sec.~\ref{Subsection: notation spinkled graphs} \\
	$m_x$ & Number of vertices contained in $\mathrm G^x$ & Sec.~\ref{Subsection: notation spinkled graphs} \\
	$m_x^*$ & Number of vertices contained in $\mathrm G^x$ connected to $x$ & Sec.~\ref{Subsection: notation spinkled graphs}
	\\[0.1em]
	\midrule
	\\[-1em]
	$B$ & Adjacency matrix of $\mathbb G$ & Sec.~\ref{Subsection: notation spinkled graphs} \\
	$A$ & Adjacency matrix of $\mathcal G$ & Sec.~\ref{Subsection: notation spinkled graphs}\\
	$\mathrm T^x$ & Adjacency matrix of $\mathrm G^x$ & Sec.~\ref{Subsection: notation spinkled graphs} \\
	$T$ & Block diagonal matrix with entries $\mathrm T^x$ & Sec.~\ref{Subsection: notation spinkled graphs} \\
	$H_\lambda$ & Anderson Hamiltonian on $\mathbb V_\lambda$ with potential $(-r_x(\lambda))_{x\in \mathbb V_\lambda}$ & Sec.~\ref{Subsection: results atoms and the Anderson-percolation representation}
	\\[0.1em]
	\midrule
	\\[-1em]
	$\nu_0$ & Density of states of $(\mathbb G, \mathbb o)$ & Sec.~\ref{Subsection: results support}\\
	$\nu$ & Density of states of $(\mathcal G, o)$ & Sec.~\ref{Subsection: notation spinkled graphs} \\
	$\nu_{\operatorname{TF}}$ & Tuning fork contribution to $\nu$ & Sec.~\ref{Subsection: results atoms and the Anderson-percolation representation} \\
	$\hat \nu$ & Density of states without tuning form contribution, normalized by $|\mathbb V_N|$ & Sec.~\ref{subsection: Concentration of mass}\\
	$\rho_\lambda$ & Density of states of $H_\lambda$ conditionally on $\{o\in \mathbb V_\lambda\}$ & Sec.~\ref{Subsection: results atoms and the Anderson-percolation representation} \\
	$\phi$ & Distribution with Stieltjes transform $h$ & Sec.~\ref{Subsection: Results sparse srpinkling}  \\
	$\mathrm \zeta^x$ & Counting measure of $\mathrm N^x$ & Sec.~\ref{subsection: Concentration of mass}
	\\[0.1em]
	\midrule
	\\[-1em]
	$g$ & Stieltjes transform of spectral measure of $B$ in $\mathbb o$ & Sec.~\ref{subsection: Concentration of mass} \\
	$s$ & Stieltjes transform of $\nu$ & Sec.~\ref{Subsection: Results sparse srpinkling} \\
	$\hat s$ & Stieltjes transform of $\hat \nu$ & Sec.~\ref{Subsection: Results sparse srpinkling}\\
	$h$ & Derivative of $s$ at $p=0$ & Sec.~\ref{Subsection: Results sparse srpinkling} \\
	$r_x$ & Stieltjes transform of spectral measure of $\mathrm T^x$ with respect to $\varphi^x$ & Sec.~\ref{Subsection: results atoms and the Anderson-percolation representation}
	\\[0.1em]
	\midrule
	\\[-1em]
	$\mathrm S^x$ & Spectrum of $\mathrm T^x$ & Sec.~\ref{Subsection: results atoms and the Anderson-percolation representation} \\
	$\mathrm S^x_*$ & Eigenvalues of $\mathrm T^x$ to eigenvectors that are non-orthogonal to $\varphi^x$ & Sec.~\ref{Subsection: results atoms and the Anderson-percolation representation} \\
	$\mathrm D$ & Support of $\mathrm S^x$ & Sec.~\ref{Subsection: results atoms and the Anderson-percolation representation} \\
	$\mathrm N^x$ & Zeros of $r_x$ & Sec.~\ref{subsection: Concentration of mass}\\
	$\mathcal Z$ & Set of $\lambda$ for which $g(\lambda)r(\lambda)=1$ with positive probability & Sec.~\ref{subsection: Concentration of mass} \\
	$w(\lambda)$ & Mass of $\lambda$ under spectral measure of $T$ with respect to $\varphi$ & Sec.~\ref{Subsection: results sparse sprinkling and large d}
	\\[0.1em]
	\midrule
	\\[-1em]
	$\mathrm P$ & Distribution of $(\mathrm T^x, \varphi^x)$ & Sec.~\ref{Subsection: notation spinkled graphs}
	\\
	$\mathrm E$ & Expectation with respect to $\mathrm P$ & Sec.~\ref{Subsection: notation spinkled graphs}
	\\[0.1em]
	\bottomrule
	\end{tabular}
\end{center}

{\small

\vspace{1em}

\noindent
Antti Knowles, University of Geneva, \href{mailto:antti.knowles@unige.ch}{antti.knowles@unige.ch}.
\\[0.3em]
Steffen Polzer, IST Austria, \href{mailto:steffen.polzer@ist.ac.at}{steffen.polzer@ist.ac.at}.

\end{document}